\documentclass[12pt,oneside,final]{thesis}
\usepackage{graphicx,hyperref,url}   
\graphicspath{{./figs/}} 

 \usepackage[all]{nowidow} 
\usepackage[english]{babel}
\usepackage[utf8]{inputenc}
\usepackage{cite,natbib,listing,blindtext} 
\usepackage[titletoc]{appendix} 

\usepackage{amsfonts,eufrak,dsfont,latexsym,newcent,cancel,upquote,bbm,bm,color,xcolor}   
\usepackage{syntonly,xparse,rotating,setspace,algcompatible,algorithm,algorithmicx,algpseudocode,verbatim,times}
\usepackage{textpos,booktabs,subcaption,multirow,multicol,longtable,array,mdwmath,tabularx,extramarks,eqparbox}
\usepackage[font=singlespacing, labelfont=bf]{caption} 
\usepackage{amsmath,amsthm,amssymb,amscd,thmtools,mathtools,mathrsfs,siunitx,calc}
\usepackage[shortlabels]{enumitem} 
\newlist{inlinelist}{enumerate*}{1}
\setlist*[inlinelist,1]{	label=(\arabic*), }
\usepackage{afterpage,tikz,grffile,wrapfig,scalerel,pgfplots}
\usetikzlibrary{shapes,arrows,chains}
\usetikzlibrary[calc]
\pgfplotsset{compat=newest} 
\usepgfplotslibrary{patchplots}    
 
\usepackage{dblfloatfix} 
 \usepackage{subcaption} 
 \usepackage{cases} 
 \usepackage{longtable} 
\usepackage[none]{hyphenat} 

\makeatletter
\newcommand{\customlabel}[2]{%
	\protected@write \@auxout {}{\string \newlabel {#1}{{#2}{\thepage}{#2}{#1}{}} }%
	\hypertarget{#1}{#2}
}
\makeatother

 \usepackage{fancyhdr}
 \fancypagestyle{plain}{
 	\fancyhf{} 
 	\fancyfoot[C]{\thepage} 
 	
 	}

\usepackage{etoolbox}
\usepackage{framed}
\newcommand\enclosebox[2]{%
	\BeforeBeginEnvironment{#1}{\begin{#2}}%
		\AfterEndEnvironment{#1}{\end{#2}}%
}
\enclosebox{definition}{leftbar} 
\enclosebox{remark}{leftbar} 

\makeatletter
\newcommand{\thickhline}{%
    \noalign {\ifnum 0=`}\fi \hrule height 1pt
    \futurelet \reserved@a \@xhline
}
\newcolumntype{"}{@{\hskip\tabcolsep\vrule width 1pt\hskip\tabcolsep}}
\makeatother

\makeatletter
\newcommand{\DESCRIPTION@original@item}{}
\let\DESCRIPTION@original@item\item
\newcommand*{\DESCRIPTION@envir}{DESCRIPTION}
\newlength{\DESCRIPTION@totalleftmargin}
\newlength{\DESCRIPTION@linewidth}
\newcommand{\DESCRIPTION@makelabel}[1]{\llap{#1}}%
\newcommand{\DESCRIPTION@item}[1][]{%
	\setlength{\@totalleftmargin}%
	{\DESCRIPTION@totalleftmargin+\widthof{\textbf{#1 }}-\leftmargin}%
	\setlength{\linewidth}
	{\DESCRIPTION@linewidth-\widthof{\textbf{#1 }}+\leftmargin}%
	\par\parshape \@ne \@totalleftmargin \linewidth
	\DESCRIPTION@original@item[\textbf{#1}]%
}

\makeatother

  \theoremstyle{plain}
 \newtheorem{thm}{Theorem}[section]
 \newtheorem{corr}{Corollary}[section]
 \newtheorem{lem}{Lemma}[section]
 
 \newtheorem{prop}{Proposition} 
 
 \theoremstyle{definition}
 \newtheorem{dfn}{Definition}[section]
 
\newtheorem{exmp}{Example}[section]

  \theoremstyle{remark}
 \newtheorem{rem}{Remark}

 \newtheorem{assumeA}{Assumption A.\hspace*{-1.2mm}} 
 \newtheorem{assumeB}{Assumption B.\hspace*{-1.2mm}} 
 \newtheorem{assumeC}{Assumption C.\hspace*{-1.2mm}}

\newcommand{\remove}[1]{}
\newcommand{\given}[2]{\left. #1 \right| #2}
\newcommand{\ilparenthesis}[1]{( #1 )}
\newcommand{\ilbracket}[1]{[ #1 ]} 
\newcommand{\ilset}[1]{\{ #1 \}}
\def \set#1{\left \{#1 \right \}}
\newcommand{\bracket}[1]{\left[  #1  \right]}
\newcommand{\Angle}[1]{\langle  #1 \rangle}
\def \parenthesis#1{\left (#1 \right)}
\newcommand{\norm}[1]{\| #1 \|} 
\newcommand{\abs}[1]{\left| #1 \right|} 

\DeclareMathOperator*{\argmin}{\arg\!\min}

\newcommand{\indicator}{\mathbbm{I}}

\newcommand{\diff}{d}
\renewcommand{\complement}{\mathsf{c}}

\newcommand{\real}{\mathbb{R}}

\renewcommand{\natural}{\mathbb{N}}
\newcommand{\integer}{\mathbb{Z}}

\newcommand{\field}{\mathcal{F}}
\newcommand{\nfield}{\mathcal{G}} 
\newcommand{\Prob}{\mathbbm{P}}
\newcommand{\E}{\mathbbm{E}}
\DeclareMathOperator{\Var}{\mathbbm{V}}
\newcommand{\Cov}{\mathbbm{C}}

\usepackage{relsize}
\newcommand{\transpose}{\mathsmaller{T}}
\newcommand{\tr}{\mathrm{tr}}
\newcommand{\sign}{\mathrm{sign}}

\newcommand{\Id}{\bm{I}_p} 
\newcommand{\zero}{\boldsymbol{0}}

\newcommand{\diag}{\mathrm{diag}} 
\newcommand{\Proj}{\mathscr{P}}

\newcommand{\bA}{\bm{A}}
\newcommand{\bB}{\bm{B}}

\newcommand{\bD}{\bm{D}}

\newcommand{\bG}{\bm{G}}
\newcommand{\bH}{\bm{H}}
\newcommand{\bI}{\bm{I}}
\newcommand{\bK}{\bm{K}}
\newcommand{\bL}{\bm{L}}
\newcommand{\bM}{\bm{M}}

\newcommand{\bP}{\bm{P}}
\newcommand{\bQ}{\bm{Q}}
\newcommand{\bR}{\bm{R}}
\newcommand{\bS}{\bm{S}}
\newcommand{\bU}{\bm{U}}
\newcommand{\bV}{\bm{V}}
\newcommand{\bW}{\bm{W}}
\newcommand{\bX}{\bm{X}}
\newcommand{\bY}{\bm{Y}}
\newcommand{\bZ}{\bm{Z}}

\newcommand{\bb}{\bm{b}}
\newcommand{\bc}{\bm{c}}
\newcommand{\bd}{\bm{d}}
\newcommand{\be}{\bm{e}}
\newcommand{\bg}{\bm{g}}
\newcommand{\bh}{\bm{h}}
\newcommand{\bu}{\bm{u}}
\newcommand{\bv}{\bm{v}}
\newcommand{\bw}{\bm{w}}
\newcommand{\bx}{\bm{x}}
\newcommand{\by}{\bm{y}}
\newcommand{\bz}{\bm{z}}

\usepackage{upgreek}
\renewcommand{\alpha}{\upalpha}
\renewcommand{\beta}{\upbeta}
\renewcommand{\gamma}{\upgamma}
\renewcommand{\delta}{\updelta}
\renewcommand{\epsilon}{\upepsilon}
\renewcommand{\varepsilon}{\upvarepsilon}
\renewcommand{\zeta}{\upzeta}
\renewcommand{\eta}{\upeta} 
\renewcommand{\theta}{\uptheta}
\renewcommand{\vartheta}{\upvartheta}
\renewcommand{\iota}{\upiota}
\renewcommand{\kappa}{\upkappa}
\renewcommand{\lambda}{\uplambda}
\renewcommand{\mu}{\upmu}
\renewcommand{\nu}{\upnu}
\renewcommand{\xi}{\upxi}
\renewcommand{\pi}{\uppi}
\renewcommand{\rho}{\uprho}
\renewcommand{\varrho}{\upvarrho}
\renewcommand{\sigma}{\upsigma}
\renewcommand{\tau}{\uptau}
\renewcommand{\upsilon}{\upupsilon}
\renewcommand{\phi}{\upphi}
\renewcommand{\varphi}{\upvarphi}
\renewcommand{\chi}{\upchi}
\renewcommand{\psi}{\uppsi}
\renewcommand{\omega}{\upomega} 

\newcommand{\bbeta}{\boldsymbol{\upbeta}}

\newcommand{\bDelta}{\boldsymbol{\Delta}} 

\newcommand{\bzeta}{\boldsymbol{\upzeta}}
\newcommand{\btheta}{\boldsymbol{\uptheta}}
\newcommand{\bvartheta}{\boldsymbol{\uptheta}^* }
\newcommand{\bTheta}{\boldsymbol{\Theta}}

\newcommand{\bLambda}{\boldsymbol{\Lambda}}
\newcommand{\bgamma}{\boldsymbol{\upgamma}}
\newcommand{\bGamma}{\boldsymbol{\Gamma}}

\newcommand{\bxi}{\boldsymbol{\upxi}}

\newcommand{\brho}{\boldsymbol{\uprho}}

\newcommand{\bSigma}{\boldsymbol{\Sigma}}

\newcommand{\bvarphi}{\boldsymbol{\upvarphi}}
\newcommand{\bPhi}{\boldsymbol{\Phi}}

\newcommand{\nullset}{\mathcal{N}}
\newcommand{\compactset}{\mathcal{S}} 
 
 \newcommand{\SG}{\mathrm{SG}}
 \newcommand{\SP}{\mathrm{SP}}
 \newcommand{\loss}{f}  
 
 \newcommand{\gain}{a}
 \newcommand{\timeSpan}{h}
 \newcommand{\const}{C}
 \newcommand{\firstConst}{u}
 \newcommand{\secondConst}{v}
 \newcommand{\bias}{\bbeta}
 \newcommand{\noise}{\bxi}
 \newcommand{\BIAS}{\bB}
 \newcommand{\NOISE}{\bX }
 \newcommand{\PROJECTION}{\bR}
 \newcommand{\projectionlower}{\bm{r}}
 \newcommand{\hbH}{\hat{\bH}} 

\newcommand{\hbtheta}{\hat{\btheta}}
 
\newcommand{\bbtheta}{\bar{\btheta}}
\newcommand{\hbg}{\hat{\boldsymbol{g}}}    
\newcommand{\convexPara}{\mathscr{C}}
\newcommand{\LipsPara}{\mathscr{L}}
\newcommand{\noiseBound}{\mathscr{M}}
\newcommand{\driftBound}{\mathscr{B}}
\newcommand{\gradientBound}{\mathscr{G}}
\newcommand{\Tobe}{q}
\newcommand{\ratio}{\mathscr{R}}
\newcommand{\multiple}{m}

\newcommand{\TimeMapping}[1]{m\left( #1 \right)}
\newcommand{\aDependence}{^{(a)}}
\newcommand{\aSubseqDependence}{^{(a_n)}}
\newcommand{\ConvCone}{\mathrm{Cone}}
\newcommand{\boundedFn}{F}
\newcommand{\contZ}{\breve{\bZ}} 
\newcommand{\obg}{\overline{\bg}}

\newcommand{\stateregime}{\mathsf{s}}
\newcommand{\stateInd}{\ilparenthesis{\stateregime}}
\newcommand{\start}{\text{start}}
\newcommand{\final}{\text{end}}
\newcommand{\observebx}{\mathbf{x}}

\newcommand{\window}{\mathsf{w}}
\newcommand{\VarianceError}{\bV}
\newcommand{\VarianceLimiting}{\bSigma}
\newcommand{ \jump}{\mathcal{J}}
\newcommand{\FirstState}{\mathsf{A}}
\newcommand{\SecondState}{\mathsf{B}}
\newcommand{\hypo}{\mathtt{H}}
\newcommand{\Ball}{\mathrm{Ball}}
\renewcommand{\vector}[1]{\mathrm{vec}\parenthesis{#1}}
\newcommand{\hazard}{h}
\newcommand{\gainRegime}{\gain_{\stateInd}}
\newcommand{\shrink}{\upeta_{-}}
\newcommand{\increase}{\upeta_{+}}
\newcommand{\TimeCritical}{\tilde{\upkappa}}
\newcommand{\stopTime}{\mathrm{s}}
\newcommand{\stopInd}{\ilbracket{\stopTime}}
\newcommand{\stopIndNew}{\ilbracket{\stopTime+1}}
\newcommand{\gainStop}{\gain_{\stopInd}}
\newcommand{\gainStopNew}{\gain_{\stopIndNew}}
\newcommand{\VarianceMA}{\widetilde{\bV}}
\newcommand{\VarianceG}{\overline{\bV }}

\newcommand{\velocity}{\mathsf{v}}
\newcommand{\east}{\mathrm{E}}
\newcommand{\north}{\mathrm{N}}
\newcommand{\KFmeasurementmatrix}{\bS}

\newcommand{\obH}{\overline{\bH}}
\newcommand{\oobH}{\overline{\overline{\bH}}}
\newcommand{\tbDelta}{\tilde{\bDelta}}
\newcommand{\tbu}{\tilde{\bu}}
\newcommand{\tbv}{\tilde{\bv}}
\newcommand{\tbA}{\tilde{\bA}}
\newcommand{\obB}{\overline{\bB}}
\newcommand{\obLambda}{\overline{\bLambda}}
\newcommand{\weakconverge}{\stackrel{\mathrm{dist}}{\longrightarrow}}

\newcommand{\red}[1]{\textcolor{red}{TODO: #1}}

\begin{document}
	\sloppy 
\title{ \uppercase{Error Bounds and Applications for Stochastic Approximation With Non-Decaying Gain \remove{ Stochastic Approximation With Non-Decaying Gain: Finite-Sample Error Bound, Concentration Behavior,   Data-Driven Gain-Tuning, and    Multi-Agent Applications}  }}
\author{Jingyi Zhu}
\degreemonth{February 17}
\degreeyear{2020} 
\dissertation
\doctorphilosophy
\copyrightnotice 
 
\begin{frontmatter} 
\maketitle

\begin{abstract}

	  This  work  analyzes the stochastic approximation algorithm with non-decaying gains as  applied in    time-varying problems. The    setting  is to minimize  a sequence of scalar-valued loss functions $\loss_k\ilparenthesis{\cdot}$   at sampling times   $\uptau_k$ or to locate the root of a sequence of  vector-valued functions $\bg_k\ilparenthesis{\cdot}$ at   $\uptau_k $ with respect to a parameter $\btheta\in\real^p$. The available information    is the noise-corrupted observation(s) of either $\loss_k\ilparenthesis{\cdot}$ or $\bg_k\ilparenthesis{\cdot}$ evaluated at  one or two   design points  {only}.   
	  Given the time-varying stochastic approximation setup, we apply   stochastic approximation algorithms.   The gain has to be bounded away from zero so that  the recursive estimate denoted as  $\hbtheta_k$ can  maintain its  momentum in tracking the  time-varying optimum denoted as $\bvartheta_k$. Given that $\ilset{\bvartheta_k}$  is perpetually varying,  the   best property   that $\hbtheta_k$ can   have  is  to be near the solution $\bvartheta_k$ (concentration behavior) in place  of the improbable convergence.

  Chapter~\ref{chap:FiniteErrorBound} provides   a bound for the  root-mean-squared  error   $ \sqrt{ \E \ilparenthesis{  \norm{  \hbtheta_k-\bvartheta_k  }^2  }  } $ and a bound for the  mean-absolute-deviation   $ \E \norm{  \hbtheta_k-\bvartheta_k } $.   Note that the only  assumption imposed on  $ \ilset{\bvartheta_k} $ is that the average distance between two consecutive underlying optimal parameter vectors  is bounded from above.  Overall, the bounds are  applicable  under a mild assumption on the time-varying drift and a modest restriction on    the observation noise and the bias term.   
	 After   establishing  the tracking capability in  Chapter~\ref{chap:FiniteErrorBound}, we also discuss the concentration behavior of $\hbtheta_k$ in Chapter~\ref{chap:Limiting}.  
	The weak convergence limit of the continuous interpolation of $\hbtheta_k$  is shown to follow  the trajectory of a non-autonomous ordinary differential equation. Then we apply the formula for variation of parameters to derive a computable upper-bound  for the probability that $\hbtheta_k$ deviates from $\bvartheta_k$ beyond a certain threshold.      
	Both Chapter~\ref{chap:FiniteErrorBound} and   Chapter~\ref{chap:Limiting}  are probabilistic arguments  and may not provide much guidance on the gain-tuning strategies useful for  one single experiment run. Therefore,   Chapter~\ref{chap:AdaptiveGain} discusses  a data-dependent gain-tuning strategy based on estimating the Hessian information and the noise level.
	 Overall, this work answers the questions ``what is the estimate for the dynamical system $\bvartheta_k$'' and ``how much   we can trust  $\hbtheta_k$ as an estimate for $\bvartheta_k$.''   \remove{Finally, the thesis   contains an appendix that uses  the indefinite symmetric matrix factorization to reduce the per-iteration floating-point-operations  cost of the second-order stochastic approximation algorithms from $O(p^3)$ to $O(p^2)$ where $p$ is the dimension of the  parameter $\btheta$. }

	  \remove{

	  This result complements existing big-$ O $ bound, and delivers a computable bound for practical use.

	  quantify the tracking capability of stochastic gradient (Robbins-Monro) algorithm, to enable the error bound available for model-free adaptive control of nonlinear stochastic systems with unknown dynamics.  
	  model-free adaptive control of nonlinear stochastic systems with unknown dynamics

  }

\quad

\quad


\vfill
\noindent {\bf{Primary Reader and Academic Advisor:}}  Dr.  James C. Spall\\
{\bf{Second Reader:}}    Dr. Nicolas Charon
\end{abstract}

\begin{dedication}
 \begin{center}
{\it{Dedication to my family.}}
\end{center}

\end{dedication}

\begin{acknowledgment}
I would like to express my sincere gratitude to my academic advisor,  Dr. James   Spall, for his continuing support throughout my graduate studies in the Department of Applied Mathematics and Statistics (AMS) at the Johns Hopkins University (JHU). His admirable work ethic, deep insight  in the field of  optimization and control, and kindness to his students, were  influential in making my JHU experience fruitful and enriching. 

I would like to thank Dr. Nicolas Charon for taking his   time in reviewing the thesis as a second reader, and for serving on the committee of  the three key oral exams in my graduate study. 
I also would   like to   thank  Dr. Donniell Fishkind, Dr.  Benjamin Hobbs, Dr.  Enrique Mallada, 
Dr. Daniel Robinson,   Dr.  Howard Weinert,   all the members of the committee of  my Candidacy exam, my Graduate Board Oral  exam and my dissertation defense.

I   want to thank Dr.  Avanti Athreya,  Dr. Agostino Capponi,  Dr. Nicolas Charon, Dr. Donniell Fishkind, Dr.  James Fill, Dr.  Vince Lyzinski, Dr.  Mauro Maggioni,  Dr.  Carey Priebe,  Dr.  Daniel Robinson, Dr.   James Spall,  Dr.  Minh Tang,  Dr.  Fred Torcaso,  for the helpful courses  they've designed and  delivered in the Department of AMS. I want to thank Ms. Kristin Bechtel 
 and Ms. Heather Kelm for their kindness and patience in coordinating  the academic-related issues.  
 I would like to thank Ms. Denise Link-Farajali  from the Center for Leadership Education at JHU and Ms. Kathy Ceasar-Spall for helping me in reviewing the grammatical issues and providing rephrasing  suggestions.

Last but not least, I would like to  thank my husband, Long Wang, for his friendship and  love, and for the enlightening discussions we've shared. The work in Appendix~\ref{chap:2nd} could not have been  completed timely  without his dedicated  collaboration.   I would like to thank my siblings, Jingwen Zhu and Siyi Huang, for their continuing support and trust. I would like to thank my parents, Jincheng Zhu and Qiaoming Huang, for everything they have offered     my siblings and me  throughout our lives.

This work was  supported by  the  Paul V. Renoff Fellowship from JHU,     the Charles and Catherine Counselman Fellowship from the Department of AMS, the   Acheson J. Duncan Fund for the Advancement of Research in Statistics from  the  Department of AMS,    the Office of Naval Research via Navy contract
N00024-13-D6400,  and Dr. James Spall’s JHU/APL sabbatical professorship
at  Whiting School of Engineering   in JHU.

\end{acknowledgment}

 \tableofcontents

\listoftables

\listoffigures

\chapter*{List of     Notation and Acronyms}
 \addcontentsline{toc}{chapter}{List of   Notation and Abbreviations}
\markboth{\uppercase{List of     Notation and Acronyms}}{ }

The   frequently used notation are  arranged by category in the following lists.

\subsection*{General Math Notation}

Let $A$ and $B$ denote some sets\footnote{Usually, the set of our interest in this thesis is a  subset of the $p$-dimensional Euclidean space. }, and let  $x$ and $y$ denote some canonical points.

 \begin{longtable}{cp{0.8\textwidth}} 
 	$\real^p$ &  The $p$-dimensional Euclidean space\\ 
 	$\natural$ & The set of natural numbers $ \ilset{0,1,2,\cdots} $  \\ 
\remove{ 	$\integer$ & The set of integer numbers $\ilset{0,1,-1,2,-2,\cdots}$\\          }
 	$A\subset B$ & Set $A$ is a   subset of set $B$, including the scenario where $A=B$\\ 
 	$A\subsetneq B$ & Set $A$ is a proper subset of set $B$, i.e., the possibility of $A=B$ is excluded\\
 	$A\setminus B$ & The set obtained by excluding all elements in $B$ from the set $A$\\
 	$\mathrm{int}\ilparenthesis{A}$ & The interior of set $A$ \\
 	$\partial A$ & The boundary of set $A$ \\ 
 	$\overline{A}$ & The closure of   set $A$\\ 
\remove{ 	$A^{\complement}$ & The complement of the set $A$\\        }
 	$\Ball_{\upvarepsilon}\ilparenthesis{A}$ & The $\upvarepsilon$-neighborhood of set $A$, defined as $ \ilset{x: \norm{x-y}\le \upvarepsilon, \text{for } y\in A } $\\
 	$ \Proj_{A}\ilparenthesis{x} $ & The projection of point $x$ onto set $A$, defined as $    \argmin_{y \in A } \norm{x-y} $\\ 
 	$\indicator_{A}(x)$ & The indicator function of the set $A$, which gives $1$ if $x\in A$ and   $0$ if $x\notin A$ 
 \end{longtable}

Let $\bA$ denote a matrix, and let $\bx$ denote a vector.  $\bA$ is not necessarily square unless specified otherwise.

 \begin{longtable}{cp{0.8\textwidth}}
	$\norm{\cdot}$  & The Euclidean norm if the input argument is a vector, or the spectral norm   if the input argument is a matrix\\
	$\abs{\cdot}$  & The determinant of the input   matrix in $ \real^{p\times p} $\\ 
	$\bI_p$ & The identity matrix in $\real^{p\times p}$\\
	$ \uplambda_{\min} \ilparenthesis{\bA}  $    & The smallest     eigenvalue of    $\bA\in\real^{p\times p}$\\
	$\uplambda_{\max} \ilparenthesis{\bA}$ & The largest     eigenvalue of   $\bA\in\real^{p\times p}$\\
	$\upsigma_{\min}\ilparenthesis{\bA}$ & The smallest singular value of $\bA$ \\
	$\upsigma_{\max}\ilparenthesis{\bA}$ & The largest singular value of $\bA$ \\
	$\tr\ilparenthesis{\bA}$ & The trace of   $\bA$    \\ 
	$\mathrm{cond}\ilparenthesis{\bA}$ & The condition number of $\bA\in\real^{p\times p}$\\ 
	$\mathrm{rank}\ilparenthesis{\bA}$ & The rank of $\bA$   \\
	$\vector{\bA}$ & The vectorization of   $\bA$, i.e., the concatenation of all columns  of $\bA$ \\ 
	$\bA^\transpose$ & The transpose of matrix $\bA$\\
	$\bA_1\succ \bA_2$ & The binary operator $\succ$ means that $\ilparenthesis{\bA_1-\bA_2}$ is a positive definite matrix for $\bA_1, \bA_2\in\real^{p\times p}$\\ 
	$\bA_1\succeq \bA_2$ & The binary operator $\succeq$ means that $\ilparenthesis{\bA_1-\bA_2}$ is a positive semi-definite matrix for $\bA_1, \bA_2\in\real^{p\times p}$\\
 $\diag(\bx)$ & The matrix with diagonal entries being the components of $\bx$ \\
	$\bx^{-1}$ & The component-wise inverse of vector $\bx\in\real^p$ with nonzero components\\
	$\bx^{-\transpose}$ & The transpose of $\bx^{-1}$ for $\bx\in\real^p$ with nonzero components
\end{longtable}

 Let $ \ilparenthesis{\Omega, \mathcal{A}, \Prob} $ be the probability space of interest.  
 Let  $\upomega$ denote an outcome within the  sample space $\Omega$, and let $\bx$ and $\by$ denote some random vectors.

  \begin{longtable}{cp{0.8\textwidth}}
  	$\upomega$ & An elementary event belonging to the set $\Omega$\\
  	$\Prob\ilparenthesis{A}$ & The probability of a set $A\in \mathcal{A}$\\ 
 $\E\ilparenthesis{\bx}$ & The expectation of $\bx$\\ 
 $\Var\ilparenthesis{\bx}$ & The variance (matrix) of $\bx$ \\
 $\Cov\ilparenthesis{\bx,\by}$ & The covariance (matrix) of $\bx$ and $\by$  (which is not necessarily square)\\
 $\ilparenthesis{\field_i}_{i\in I}$ & It is a filtration if $\field_i$ is a sub sigma-algebra of $\mathcal{A}$ for all $i$, and $\field_j\subset \field_k\subset\mathcal{A}$ for all $j\le k$  in the index set $I$ (which is usually $\natural$) \\
 $\stackrel{\mathrm{dist}}{\longrightarrow}$ & Converges in distribution
 \end{longtable}  

\subsection*{Notation  For  Stochastic Optimization Framework Under Nonstationary Scenarios }
The following notation  are consistently used throughout the entire thesis, except that  Section~\ref{sect:SAview} and Appendix~\ref{chap:2nd} consider the traditional stochastic approximation setting with a fixed loss function and a fixed optimum.

  \begin{longtable}{cp{0.8\textwidth}}   
  	$\btheta$ & The vector of parameters being estimated, which usually lives in $\real^p$\\ 
  	$\uptau_k$ & The sampling time at the $k$th iteration\\ 
 $\loss_k\ilparenthesis{\btheta} $ & The loss function to minimize, which  varies with the index $k$\\   $\bg_k\ilparenthesis{\btheta} $  & The gradient of the scalar-valued  loss function under the minimization setting or the underlying  vector-valued function in the root-finding setting  \\   $\bH_k\ilparenthesis{\btheta}$ & The Hessian of the scalar-valued loss function $\loss_k\ilparenthesis{\btheta}$  under the minimization setting or the Jacobian of the underlying vector-valued function $\bg_k\ilparenthesis{\btheta}$ under the root-finding setting  \\ 
   	 $\bvartheta_k$ & The minimizer of the scalar-valued loss function under the minimization setting or the root of the vector-valued function under the root-finding setting; the subscript $k$ is used to emphasize that  the optimum value changes with time $\uptau_k$\\  
  	$\hbtheta_k$ & The recursive estimate for $\bvartheta_k$ produced in the $k$th iteration of  an algorithm searching\\
  	$\hbg_k\ilparenthesis{\btheta}$ & The estimator for $\bg_k\ilparenthesis{\btheta}$ obtained at time $\uptau_k$, which will be used in the iterative SA scheme \\
  	$\bias_k\ilparenthesis{\btheta}$ and $\noise_k\ilparenthesis{\btheta}$ & The bias term and the  error term of using $\hbg_k\ilparenthesis{\btheta}$ as an estimator for $\bg_k\ilparenthesis{\btheta}$, see page~\pageref{eq:gNoisyDecomposition}   \\ 
  	$\be_k\ilparenthesis{\btheta}$ & The error term that is the sum of $\bias_k\ilparenthesis{\btheta}$ and $\noise_k\ilparenthesis{\btheta}$ (alternatively, the difference  between $\hbg_k\ilparenthesis{\btheta}$ and $\bg_k\ilparenthesis{\btheta}$), see page~\pageref{eq:gGeneral}  \\
  	$\gain_k$ or  $\gain$  & The \emph{non-decaying} or \emph{constant} positive gain appearing in the generic form of SA algorithm, see page~\pageref{eq:basicSA}\\
  	$c_k$ & The \emph{non-decaying} positive differencing magnitude  arising in constructing $\hbg_k\ilparenthesis{\cdot}$ using  FDSA/SPSA schemes, see page~\pageref{sect:SAintro} \\
  	$\bDelta_k$ & The perturbation vector arising from constructing $\hbg_k\ilparenthesis{\cdot}$  using SPSA scheme, see page~\pageref{eq:gSPSA}\\
 $C^j (\real\mapsto\real^p)$ & The set of  continuous functions that map $\real$ to $\real^p$ and are $j$th-order\footnote{ If a function that maps $\real$ to $\real^p$ is in $C^0\ilparenthesis{\real \mapsto\real^p}$, it means that it is continuous for $j\in\natural$. If a function that maps $\real$ to $\real^p$ is in $C^1\ilparenthesis{\real \mapsto\real^p}$, it means that both its first-order derivative and itself  are continuous. In general,  if a function that maps $\real$ to $\real^p$ is in $C^j\ilparenthesis{\real\mapsto\real^p}$, it means that  itself and its $i$th-order derivatives  are continuous for all $i\le j$. } continuous, see page~\pageref{subsect:Arzela}\\
 $D(\real\mapsto\real^p)$  & The set of functions that map from $\real$ to $\real^p$ and  are    right-continuous with left-hand limits, see page~\pageref{subsect:WeakConvergence}    \\ 
 $\noiseBound_k$ & The square-root of the second-moment of $\be_k$, see page~\pageref{assume:ErrorWithBoundedSecondMoment}\\
 $\driftBound_k$ & The square-root of the expected Euclidean-distance-squared between $\bvartheta_{k+1}$ and $\bvartheta_k$, see page~\pageref{assume:BoundedVariation} \\ 
 $\convexPara_k$ & The convexity parameter of $\bg_k\ilparenthesis{\cdot}$, see page~\pageref{assume:StronglyConvex} \\
 $\LipsPara_k$ & The Lipschitz parameter of $\bg_k\ilparenthesis{\cdot}$, see page~\pageref{assume:Lsmooth} \\
 $\ratio_k$ & The ratio between $\LipsPara_{k}$ and $\convexPara_k$, see page~\pageref{acronym:ratio} 
 \end{longtable}

\subsection*{Acronyms}
Below are the acronyms that will be used    in the thesis. 

\begin{longtable}{cp{0.8\textwidth}}
	SO & stochastic optimization, see page~\pageref{acronym:SO} \\
SA &   stochastic approximation, see page~\pageref{acronym:SA}  \\
SGD &  stochastic gradient descent, see page~\pageref{acronym:SG} \\ 
SG & stochastic gradient, see page~\pageref{acronym:SG}   \\ 
FD & finite difference, see page~\pageref{acronym:FD}\\ 
FDSA & finite difference  stochastic approximation, see page~\pageref{acronym:FDSA}\\   
SP & simultaneous perturbation, see page~\pageref{acronym:SP} \\
SPSA & simultaneous perturbation stochastic approximation, see page~\pageref{acronym:SPSA}\\
SP2 & simultaneous perturbation with two-measurements, see page~\pageref{eq:gSPSA}\\ 
SP1 &  simultaneous perturbation with one-measurement, see page~\pageref{eq:g1SP} \\ 
2SPSA & second-order  simultaneous perturbation stochastic approximation, see page~\pageref{acronym:2SPSA}  \\
2SG & second-order stochastic gradient, see page~\pageref{acronym:2SG}\\
E2SPSA & enhanced second-order  simultaneous perturbation stochastic approximation, see page~\pageref{acronym:E2SPSA} \\
E2SG & enhanced second-order stochastic gradient, see page~\pageref{acronym:E2SG} \\
ODE & ordinary differential equation, see page~\pageref{acronym:ODE} \\
IVP & initial value problem, see page~\pageref{acronym:IVP} \\
 RMS & root-mean-squared, see page~\pageref{acronym:RMS} \\
 MAD & mean-absolute-deviation, see page~\pageref{acronym:MAD} \\
 MSE & mean-squared-error, see page~\pageref{acronym:MSE} \\
 LMS & least-mean-squares, see page~\pageref{acronym:LMS} \\
 RLS & recursive-least-squares, see page~\pageref{acronym:RLS} \\
 ERF & empirical risk function, see page~\pageref{acronym:ERF} \\
 KF & Kalman filter, see page~\pageref{acronym:KF} \\
MISO  & multiple-input-single-output, see page~\pageref{acronym:MISO} \\
MIMO & multiple-input-multiple-output, see page~\pageref{acronym:MIMO} \\
 FLOP & floating-point-operations, see page~\pageref{acronym:FLOPs} \\
  UAV/UUV & unmanned aerial/undersea vehicle, see page~\pageref{acronym:UAV} \\ 
  i.i.d. &  independently and identically distributed \\
  a.s. & almost surely\\ 
  m.s. &mean-squared\\ 
  w.p.1.   & with probability one\\
  w.l.o.g. & without loss of generality\\
  w.r.t. & with respect to \\ 
  r.h.s. & right-hand-side\\
  l.h.s. & left-hand-side\\
  a.k.a. & also known as 
 \end{longtable}

\end{frontmatter}
 
\chapter{Introduction}\label{chap:Introduction}
This thesis focuses on a general stochastic optimization (SO)\label{acronym:SO} framework in the context of solving a nonstationary problem.   In a nutshell, we are allowed to gather \emph{noisy} zeroth- or first-order information \emph{only},     to minimize  a sequence of \emph{time-varying} scalar-valued loss functions or locate the root for a sequence of \emph{time-varying} vector-valued functions. The consideration of both the \emph{randomness} and \emph{dynamics} makes the work here different from both the traditional deterministic optimization and the classical   SO. 

The motivation for the time-varying SO   setup comes from modern needs in areas such as electrical power distribution, unmanned aerial or undersea vehicle (UAV or UUV)\label{acronym:UAV} tracking, and multi-agent problems.      To successfully track    a time-varying parameter  has long been an important topic in  these real-world problems, and both the time variation and  the noise corruption have been   notorious barriers that obstruct experimenters from making accurate statistical inference. 
Recursive stochastic approximation (SA)\label{acronym:SA} algorithms with non-decaying gain sequence   are widely used for  tracking purposes. Interestingly, the theoretical foundation for the application of   SA on the time-varying parameter identification is still developing, and so is    guidance to tune the non-decaying gain sequence.  
 This thesis is devoted to   these issues.

Before we move  forward to     the technical issues, let us further motivate why we use recursive SA schemes with non-decaying gains to handle time-varying problems in Section~\ref{sect:Motivation} and explain the challenges arising from   time-varying SO set up in Section~\ref{sect:Challenge}. Then  Section~\ref{sect:Contents} lists the contribution of this work  along with  an  overview of the upcoming chapters.

 \section{Motivation} \label{sect:Motivation}

Though    the  asymptotic properties  of    SA schemes with decaying gains  have  been    studied  mainly based on a fixed optimizer,   SA algorithms with non-decaying gains   have been widely applied   in sequential processing and online learning, where     true loss functions and, therefore,   underlying optimums  are drifting over time.
Consider a    parameter identification problem in  estimating  a sequence of time-varying unknown parameter vectors denoted as $\btheta$. 
In  adaptive tracking,  system control,  and many   estimation    applications,  the optimal value(s) of $\btheta$ at the sampling  time   $\uptau_k$ (corresponding to the   index $k$), denoted as $\bvartheta_k$,     may change over time  due to the intrinsic evolution of the underlying system.   	 When the sequence $\bvartheta_k$ varies with time,    SA  algorithms with non-diminishing gain may be         applied such that the  iterative output from    SA  algorithm, denoted as $\hbtheta_k$,  tracks  the time-varying  system parameters $\bvartheta_k$ at sampling points $\uptau_1,\uptau_2,\cdots $.

\subsection{Why Time-Varying SA Problems Are Useful?}

Stationary  systems  with   ``$\bvartheta_k= \bvartheta$ at    every sampling time $\uptau_k$'' are   well studied in   statistics and signal processing, and a voluminous literature is  available describing efficient estimation and identification of $\btheta$. However, in adaptive control of complex processes\remove{, temperature prediction in weather forecasting, } and many other real-world systems, the underlying parameters   reflecting the system characteristics $\bvartheta_k$ at time   $\uptau_k$   intrinsically vary.  Dynamic modeling is widely used in areas such as computer vision, macro- or micro-economic modeling,  feedback control systems, UAV/UUV tracking, mobile communication, radar or sonar surveillance systems, and so on. 
In addition to  the consideration of the time variability in $ \bvartheta_k $, it is also necessary to take  the randomness of the noise term. This is straightforward as noises  
  arise in almost any case where physical system measurements or computer simulations are used to approximate  a physical process.

  Let us briefly discuss an example (to appear in Chapter~\ref{chap:multiagent}) whose modeling should embrace both the time variability  of the loss function and the randomness of the noise. Consider the scenario where an agent attempts to track the coordinate of a target submarine (or aircraft) based upon the sensor (e.g., radar) readings from its surroundings.  By ``tracking'' we mean that the agent has to stay close  to   the target submarine, close enough such that the target submarine is still detectable\footnote{There exists a range outside of which   the sensor can no longer  detect the target. See \cite{kim2005tracking} for further details.  } by   its sensor reading.   The ``time variability'' comes in to play  as the coordinates of both the target submarine and the agent itself are constantly moving. The ``randomness'' is also involved as the sensor readings are corrupted by observational noise. In this real-time estimation problem, the data (sensor readings) are collected ``on the fly'' as the submarine is operating. The reading arrives in concert with the estimation of the location coordinates of the target submarine.

In addition to the submarine tracking example, 
the environment changes over time in  other practical settings. In the ever-changing environment, the solution (set) may vary with time too. 
 Such a formulation is obvious in dynamic problems such as building control systems, where the optimum may change continuously. Although less obvious, the time-varying problem also arises in settings that may appear at first glance to be static. For example, an optimum financial plan for a business or    a  family depends on the external environment, which, of course, changes over time. In some search and optimization problems, the algorithm will be explicitly designed to adapt to a changing environment (e.g., a controlled system) and automatically provide a new estimate at the optimal value. In other cases, one needs to restart the process and find a new solution. In either sense, the problem solving may never stop.

Overall, real-world problem-solving is often difficult, two common issues of which are the measurement noise  and the lack of stationarity in the solution as a result of the conditions of the problem changing over time. These two characteristics\textemdash the randomness in noises and the time variation of the loss functions\textemdash will be the central topic of this thesis. Nonetheless,   other challenges   will not be dealt with  here,  including        the  curse of dimensionality,   saddle point(s),  local optimum (versus global optimum),  the problem-specific constraints, and so on.

\subsection{Why Recursive Algorithms Are Preferred?} 
Often,  recursive algorithms   (with non-decaying gain) are used  in time-varying SA problems.
Let $\hbtheta_{k}$ represent  an estimate for $\btheta$ based on $k$ data pairs (input and corresponding output). We seek a way of computing $\hbtheta_{k+1}$ as the $\parenthesis{k+1}$-th data pair arrives.
Such schemes  are preferred  for:  (1) their ability to reveal the most  recent information,  (2) their low per-iteration memory and computation requirement, and (3) their fit for online processing.  

First of all, the recursive   form clearly  exhibits the value of a new data point.  As discussed in \cite[Sect. 3.3]{spall2005introduction}, the \emph{instantaneous}  gradient is important in time-varying systems where the aim is to estimate a sequence of $\bvartheta_k$. In such systems, it is important to place more  emphasis on the more  recent information.  Moreover,  the  recursive estimates $ \hbtheta_{k+1}$ are based on combining the current  estimate $\hbtheta_k$ with the $\parenthesis{k+1}$-th data pair in an efficient way. Granted,   the computational advantages alone are becoming less important with the advances in computing power. However, for the  \emph{real-time} applications with the flood of data accumulation, processing a bulk of data altogether from scratch after each data acquisition   may  be virtually impractical.  The sequential processing may be done to reduce the computational burden (versus the batch processing of all data) or to expose the unique impact of each datum.   

Let us take a \emph{time-invariant} system for  example: computing the (batch) ordinary least-squares solution  in basic linear regression is prohibitive (due to the expensive matrix inversion)  and whence the least-mean-squares (LMS)\label{acronym:LMS} and recursive-least-squares (RLS)\label{acronym:RLS} algorithms (see \cite[Chap. 3]{spall2005introduction} for further details).    This stress upon online  (instantaneous in real-time) processing of the incoming information regarding the system characteristics, such as estimation\footnote{It is sometimes termed as  ``filtering'' in engineering     literature.  
}\remove{ 
While filtering and prediction problems use only the past and present information, smoothing uses all the past, present and future information. Thus, smoothing is characteristically an off-line problem, but filtering and prediction can be recast as online problems. } and prediction, becomes more obvious when dealing with \emph{dynamical} systems, and consequently the typical procedure is to compute a new estimate each time when  a new measurement becomes available. Such a process corresponds  to \emph{recursive identification}, where a new estimate is computed at every sampling interval when we   need to optimize a control system online.

Aside from  the benefits brought by  the recursive form, we also note one other key   advantage of using SA algorithms. Several time-varying problems for adaptively tracking complex systems can be formulated as nonstationary extremal problems of a probabilistic nature. Under such a  setting, direct SA  methods that  do not depend on the underlying probabilistic distributions are very useful.

Further, we       mention that    a non-decaying gain should be used  to track the \emph{time-varying} target $\bvartheta_k$
in  the recursive form of   SA algorithms.  
Often, the non-decaying gain coefficient $a_k$   is set to a constant $a>0$ for its simplicity. Nonetheless,  the tuning of $\gain$  is necessary to provide sufficient impetus  for $\hbtheta_k$ to keep up with the time-varying $\bvartheta_k$.    A large $a$ helps $\hbtheta_k$ to  converge more promptly to the vicinity of $\btheta^*_k$, but a small $a$ helps the iterates  to avoid instability and divergence. 
Let us also note  that 
the  constant-gain algorithms are also frequently used in neural network training even when dealing with a \emph{time-invariant}  $\btheta^*$ due to its robustness, even though the constant-gain iterates $\hbtheta_k$ will not formally converge.

In short, the  efficient extraction of the  dynamical properties of signals and systems in a recursive form   is   central   in  system identification.  Recursive estimates are useful to adapt themselves  to the  system dynamics.

\section{Challenges}\label{sect:Challenge}

The classical SA results on the  convergence and the  rate of convergence, which are  developed based on a fixed and unique optimizer,    cannot   be   directly  transferable to the time-varying setting.  
There are several lingering concerns for applying SA recursions with non-decaying gains to the time-varying stochastic optimization setting.

\subsection{Dynamic Modeling}

New questions arise if a  setting of interests  departs  from the ``time-independent/stationary loss function $\loss\ilparenthesis{\cdot}$'' and the ``fixed unique  optimizer $\bvartheta$.'' The first question is: how to wisely characterize  the temporal changes in $\ilset{\bvartheta_k}$ such that the class of loss function $\ilset{\loss_k\ilparenthesis{\cdot}}$ is sufficiently   rich   to embrace a class  of practical scenarios and  the tracking performance  remains mathematically tractable?

Many existing works hinge on a known model for the target parameter evolution, including an ordinary differential equation (ODE)\label{acronym:ODE} model or a  random-walk model.  For example, 
Kalman filtering (KF)\label{acronym:KF} requires a linear state equation for the  $\ilset{\bvartheta_k}$ sequence and a sequence of  loss functions $\ilset{\loss_k\ilparenthesis{\cdot
}}$ in  the   quadratic form centering at $\bvartheta_k$. The particle filter requires the  conditional probabilities of  $\bvartheta_k$ to be \emph{known}, though allowing nonlinearity in the  underlying state-space model. 

 Still, the study in tracking the time-varying parameter continues because the imposed model may be invalid or is easily misspecified.  
 It is of practical interest   to circumvent the    restrictive  assumptions imposed on $\ilset{\bvartheta_k}$ or the stringent requirements on the underlying loss function sequence,  denoted as  $\ilset{\loss_k\ilparenthesis{\cdot}}$. 
We hope to set up a more general perspective in  that we  require neither a specified  linear or nonlinear  evolution for $\ilset{\bvartheta_k}$ nor  the conditional probabilities regarding $\bvartheta_k$. 

In the upcoming chapters, we   consider  the ``slowly'' time-varying target in the sense that the average  distance between successive optimizers is strictly bounded from above. Such an assumption also   includes the case where the  moving target may    change abruptly\textemdash the change may have a  drastic magnitude shift as long as it occurs sporadically.

 \subsection{Tracking Criteria}
  In addition to the issues of time-varying assumptions, we also care about such a question: what   properties   can $\hbtheta_k$  possibly  have  when the underlying parameters $\bvartheta_k$ are time-varying?

	In   real-world  applications such as  adaptive control in power-grid scheduling   and time-varying communication channels \cite{gunnarsson1989frequency}, the optimal value of the underlying parameter  is perpetually varying, so  there is no convergence per se of either $\bvartheta_k$ or the estimate $\hbtheta_k$.  
		That is, we cannot achieve the usual notion of   convergence  such     that $\norm{\hbtheta_k-\bvartheta_k}$ is arbitrarily small in a   certain  statistical sense  unless the    evolution law  of  $\bvartheta_k$ is revealed  to the agent. The best   we could hope for is that     $\hbtheta_k$ stays within a neighborhood of $\bvartheta_k$ with a high probability, which may also be termed as ``convergence to a stationary distribution'' or ``concentration.''

Regardless of the model assumptions and corresponding algorithms, understanding the tracking error is critical to the usefulness of the resulting estimates.    Chapter~\ref{chap:FiniteErrorBound}  centers  on the tracking performance in terms of controlled  root-mean-squared (RMS)\label{acronym:RMS} error $ \sqrt{\E\ilparenthesis{\norm{\hbtheta_k-\bvartheta_k}^2}}    $ or a    mean-absolute-deviation (MAD)\label{acronym:MAD} error $ \E\norm{\hbtheta_k-\bvartheta_k} $ that is bounded uniformly across $k$, and Chapter~\ref{chap:Limiting} discusses the concentration behavior in terms of  a probabilistic bound.

\subsection{Gain Tuning}

Another key issue is the question of tuning the non-diminishing gain to balance tracking accuracy and stability, under the circumstance that we have no a priori information regarding the possible disturbance acting on the system. 
In tracking problems, the gain  must be strictly bounded away from zero. Particularly, it is well known that whilst the use of a small (constant) gain decreases the magnitude of the fluctuations in $\hbtheta_k$, it also decreases the ability to track the variations in $\bvartheta_k$. 
 A larger gain   enables the resulting estimate $\hbtheta_k$  to approach promptly to the vicinity of $\bvartheta_k$, yet it may jeopardize the tracking stability.   Chapter~\ref{chap:AdaptiveGain}   tries  to  provide  some  practical guidance  on    {gain selection} based on this compromise.

	 \section{Overview of Contents and Our Contribution}\label{sect:Contents}
	 	 
 In Chapter~\ref{chap:FiniteErrorBound},   
Section~\ref{sect:ProlemSetup} sets up the time-varying SA framework, Section~\ref{sect:ModelAssumptions} motivates the model assumptions, Section~\ref{sect:UnconditionalError}  derives a computable error bound for general SA algorithms with non-decaying gains to be applied in parameter estimation along with the supporting numerical examples, and Section~\ref{sect:SpecialCases} lists some examples for  applications.   In short, this chapter  illustrates  that the bound is   favorably informative under reasonable assumptions on the evolution of the true parameter being estimated. Specifically, 
the tracking capability established   in Chapter~\ref{chap:FiniteErrorBound} differs from prior literature on error bound analysis  in the following senses:
\begin{enumerate}[(1)]
	\item   The restrictions placed  on the model of the time-varying parameter is  {mild} compared to the other assumed forms  of the  state equation. The  only  imposed assumption is that  the average  distance between two consecutive underlying parameters $\bvartheta_k$ is  strictly bounded from above. This modest assumption does not eliminate jumps in the target, and also allows the target to vary stochastically.
\item 
 
   {Biased} estimators of the gradient information may be used  in   SA algorithms, whereas most prior works are on unbiased estimators. With this extension, the tracking capability for a broad class of SA algorithms, including simultaneous perturbation stochastic approximation (SPSA)\label{acronym:SPSA} and finite difference stochastic approximation (FDSA)\label{acronym:FDSA},   is established.  \item    Many prior works are developed with  a constant gain that is tuned in advance for successful tracking and claim that the tracking error can be made smaller by decreasing the constant gain.  Our discussion reveals that the  {adaptive} gain selection should depend on the shape of the loss function, the noise level, and the drift level. Furthermore, the gain should  be neither too large nor too small. \item  The  {computable} bound applies to general  {nonlinear} SA algorithms with non-decaying gain and is valid for the entire time.  Based on this, we can  characterize the tracking performance of a large class of SA algorithms in response to the drift by determining the allowable region for the  non-decaying gain sequence. Also,  {finite-sample} analysis is possible, as the bound is not based on the vanishing gain and associated limit theorems. \end{enumerate}
	 Overall, our setup applies  to a   general scenario that allows    unbounded
	   noise, a  biased  gradient estimator, a drift (the dynamics being tracked) term without any explicit evolution model,
	 and  is useful for both finite-sample and asymptotic analysis.

In Chapter~\ref{chap:Limiting},  Section~\ref{sect:ConstantGain}  studies the weak convergence limit of the constrained SA algorithms applied in tracking time variation, and 
Section~\ref{sect:ProbBound}   quantifies the concentration behavior of the constant-gain stochastic gradient descent (SGD)\label{acronym:SGD} algorithm within finite iterations in terms of a  computable  probabilistic error  bound. 
The  concentration behavior discussed  in  
 Chapter~\ref{chap:Limiting} differs from other works on  limiting behavior  in terms of the several  subtleties which are   further discussed in Section~\ref{sect:Limiting}. 
 Section~\ref{sect:ConstantGain} develops the main result of characterizing the recursive iterates via the trajectory of a  {nonautonomous} ODE with the same initialization under proper time scaling.  
 	The bound in Section~\ref{sect:ProbBound}  is non-asymptotic as it is  derived for the actual constant gain   and not from an idealized limiting scenario based on a limit theorem for fluctuations as the constant gain goes to zero, which  is often the case in prior studies.  This is useful because we are considering the problem of continuously tracking a time-varying target.   Also, our derivation of the bound reveals  its dependence on relevant  parameters, desired accuracy, and  problem dimension.

	 Chapter~\ref{chap:AdaptiveGain} provides gain-tuning guidance based upon the observable information. 
 In addition to  the  general gain selection strategy to ensure a bounded MAD in Chapter~\ref{chap:FiniteErrorBound}, we also develop data-dependent methods to test if an abrupt jump arises and the corresponding  strategy  to  tune the gain sequence adaptively according  to the observed information.  As  the Hessian   and the observation error information needed to carry out jump detection are unknown, we  employ the    simultaneous perturbation (SP)\label{acronym:SP}  method to estimate them.   In the numerical simulation, we implement    the SGD algorithm. Results support  that    our data-dependent gain tuning strategy helps detect abrupt changes.  
 Note that many prior works are on the constant gain by assuming that the gain is tuned\footnote{	A successful constant gain is highly problem-dependent. 
 } for successful tracking, whereas we discuss data-dependent gain-tuning strategy.  Furthermore, the  adaptive step-size scheme for constrained (truncated) stochastic approximation algorithms  is useful in dynamic environments where the underlying parameters are time-varying.

Finally, Chapter~\ref{chap:multiagent} presents a problem that fits the time-varying SA problem setup and numerically illustrates the SA schemes non-decaying gain. Appendix~\ref{chap:2nd} discusses a strategy to reduce the per-iteration cost of the second-order SA algorithms from $O(p^3)$ to $O(p^2)$ using the symmetric indefinite matrix factorization.


\chapter{Preliminaries}\label{chap:Preliminaries}
This chapter lays the groundwork  for upcoming discussions. 
Section~\ref{sect:SAview}  discusses  the SO framework and
presents  the general form of SA algorithms (\ref{eq:basicSA}).  

\section{Overview of Stochastic Approximation Algorithms}\label{sect:SAview}
This section focuses  on   
SA algorithms for SO  via nonlinear root-finding.   
There are two main SO settings of   interest: one is to minimize a scalar-valued function $\loss\ilparenthesis{\cdot }$ using its  noisy evaluation  $y\ilparenthesis{\cdot} \equiv  \loss\ilparenthesis{\cdot}+\upvarepsilon\ilparenthesis{\cdot}$ evaluated at a  certain design point $\btheta$, and the other is to locate the  root(s) of a  vector-valued function $\bg\ilparenthesis{\cdot }$ using    its  corrupted observation $\bY\ilparenthesis{\cdot} \equiv \bg\ilparenthesis{\cdot} + \noise \ilparenthesis{\cdot}$ collected at a certain point $\btheta$.    
Here  $\btheta$ is   the underlying parameter vector, a collection of adjustables;   $\btheta$ typically falls within  the Euclidean $p$-space $\real^p$. Let us also denote  $\bvartheta$ as the (assumed unique) minimizer of the scalar-valued function $\loss\ilparenthesis{\cdot}$ or the (assumed unique) root of the vector-valued function $\bg\ilparenthesis{\cdot}$.

The above SO settings are distinguished from deterministic optimization in that  neither the direct evaluation of      the scalar-valued  function $\loss\ilparenthesis{\cdot}$ nor the exact observation of  the vector-valued function $\bg\ilparenthesis{\cdot}$ is    available. Sometimes, the randomness in the SO process may be due to the random choice (injected randomness) made in the search     direction as the algorithm iterates towards a solution     to avoid getting stuck.  Such scenarios commonly arise in practice.  Consider a  complex  stochastic model whose output depends on a set of parameters   $\btheta$, where the experimenter attempts to locate the value of $\btheta$ that minimizes the expected output of the model.  We are, under a majority of  circumstances,  unable to obtain a closed-form expression  or exact representation   for  the black-box model. When dealing with physical processes in actual implementations, computing the expected value of the  output for any given value of  $\btheta$ may   be impossible in general since the physical processes are governed by rules unknown to the experimenters. While deterministic optimization techniques cannot be directly transferable  to    noisy environments, SO algorithms can utilize  corrupted measurements to generate iterative estimates, denoted by  $\hbtheta_k$, at each discrete-time instance $k$\textemdash hence the term ``stochastic optimization.''

There are  many SO  algorithms: 
 random search (such as stochastic ruler, stochastic comparison, simulated annealing), 
	  SA (such as  SGD, SPSA), and so on. The focus of this thesis is on SA.   SA includes a wide range of recursive schemes (i.e., step-by-step computational methods)  with decaying  gain (i.e., the step-size approaches to zero as the iteration number increases) that iteratively  generate   $\hbtheta_k$  as an estimate for  $\bvartheta$ using the information up to  the index  $k$.

For example, given a   vector-valued function $\bg \ilparenthesis{\cdot}: \real^p\mapsto\real^p$, the basic SA algorithm for nonlinear root-finding aims  to find the root(s) of the function $\bg\ilparenthesis{\cdot} $ using the following recursive scheme: 
\begin{equation}\label{eq:basicSA}
\hbtheta_{k+1} = \hbtheta_k -  \gain_k \hbg_k \ilparenthesis{\hbtheta_k }, \,\,  k\in\natural,
\end{equation} where $\ilset{a_k}$ is a  positive gain sequence, and $\hbg_k  \ilparenthesis{\hbtheta_k}$ is the  
corrupted observation of the  vector-valued   function $\bg\ilparenthesis{\cdot}$ evaluated at $\hbtheta_k$. The details in constructing  $\hbg_k\ilparenthesis{\hbtheta_k}$ will be   discussed in the upcoming subsection.      A useful application is immediate by  letting $ \bg\ilparenthesis{\btheta} = \partial   \loss\ilparenthesis{\btheta}  / \partial\btheta$ when the iterative updating scheme  (\ref{eq:basicSA}) is used for SO via nonlinear root-finding. One of the caveats is that the roots of  the gradient equation may not be the (global) minimizer of $\loss \parenthesis{\btheta}$. 

Under certain statistical or engineering conditions, $\hbtheta_k$  converges  a.s. or in  m.s. sense to the optimum point  $\btheta^*$  as $k\to\infty$ and at a certain stochastic rate.    See \cite[Chap. 4]{spall2005introduction} for further details.  
  Given its algorithmic robustness and computational simplicity,  the recursion  (\ref{eq:basicSA}) with decaying gain $\gain_k = O(1/k)$ was pursued with great zeal by statisticians and electrical engineers as a convenient paradigm for recursive algorithms for regression, system identification, adaptive control, and so on. The subject has received  a fresh lease of life in recent years because of some new emerging application areas broadly covered under the general rubric of learning\footnote{They  encompass learning algorithms for neural networks, reinforcement learning algorithms arising from artificial intelligence and adaptive control and models of learning by boundedly rational agents in macroeconomics.} algorithms.   
  \remove{ has been a   workhorse  for decades  for  not only statistical inferences, but also for  various engineering applications ranging from   adaptive control to  machine learning. }

\subsection{General Discussion}
The basic SA algorithm for nonlinear root-finding is known as the Robbins-Monro (R-M) algorithm \cite{robbins1951stochastic}. Given a vector-valued function $\bg \ilparenthesis{\cdot }$, the R-M algorithm aims to find a root of $\bg\ilparenthesis{\cdot }$ recursively through   (\ref{eq:basicSA}), where 
$\hbg_k\ilparenthesis{\hbtheta_k}$ can  decomposed as: 
\begin{eqnarray}\label{eq:gNoisyDecomposition} 
\hbg_k \ilparenthesis{\hbtheta_k}  
&=& \bg\ilparenthesis{\hbtheta_k} + \E\ilbracket{  {\hbg_k  \ilparenthesis{\hbtheta_k} - \bg\ilparenthesis{\hbtheta_k }}   \left| {\field_k} \right. } + \set{ \hbg_k  \ilparenthesis{\hbtheta_k} -  \E\ilbracket{ {\hbg_k \ilparenthesis{\hbtheta_k}} \left| {\field_k }\right. }  }\nonumber\\
&\equiv & \bg \ilparenthesis{\hbtheta_k} + \bias_k\ilparenthesis{\hbtheta_k} + \noise_k\ilparenthesis{\hbtheta_k},\quad\text{ for }k\in\natural,
\end{eqnarray}
with  $\field_k$ being   some representation of the process history. One common choice  is  to let   $\field_k$ be  
the sigma-algebra induced by the observed quantities   up until (excluding) index $k$. Specifically, 
\begin{equation}\label{eq:filtration}
\field_0 = \upsigma\ilset{\hbtheta_0}, \,\,  \text{and } \field_k = \upsigma\ilset{\hbtheta_0,\hbg_i\ilparenthesis{\hbtheta_i},i<k} \,\text{ for }k\ge 1. 
\end{equation}   If the process history is represented as in (\ref{eq:filtration}), then the   l.h.s. of  (\ref{eq:gNoisyDecomposition}) is $\field_k$-measurable. Moreover, we can deem $\hbg_k \ilparenthesis{\hbtheta_k}$ as an estimator (in the statistical sense) of $\bg_k\ilparenthesis{\hbtheta_k}$ at fixed point $\hbtheta_k$ \cite{bickel2007mathematical}. Under such perspective,  $\bias_k\ilparenthesis{\hbtheta_k}$ represents the  {bias} of $\hbg_k  \ilparenthesis{\hbtheta_k}$ as an estimator of $\bg \ilparenthesis{\hbtheta_k}$, and  $\noise_k\ilparenthesis{\hbtheta_k}$ is termed as   the  {noise}  and is usually assumed to be  a martingale difference sequence.  
	Note  that the decomposition in (\ref{eq:gNoisyDecomposition}) is presented   mainly  for   the purpose of analysis;
in practice, the bias and noise terms are never explicitly computed or collected.    The convergence theory for the scheme (\ref{eq:basicSA}) with a general form of (\ref{eq:gNoisyDecomposition}) can be found in \cite[Chaps. 4--7]{spall2005introduction}.

SA algorithms are often categorized according to the available information.  The \emph{zeroth-order SA} includes FDSA \cite{kiefer1952stochastic}, SPSA with two-measurements \cite{spall1992multivariate}, SPSA with one-measurement \cite{spall1997one}, RDSA \cite{ermoliev1969on}, and so on. The \emph{first-order SA} includes SGD and many popular machine learning algorithms. The \emph{second-order SA} will be thoroughly reviewed in Appendix~\ref{chap:2nd}.
Two
important     SA algorithms are reviewed here:  SPSA and the stochastic gradient (SG)\label{acronym:SG}  form of
SA.

\subsection{Simultaneous Perturbation SA}\label{subsect:SPSA}
   
SPSA algorithm uses \emph{zeroth-order} information and is  especially   useful in the minimization setting. 
The recursive update  for SPSA estimates is  (\ref{eq:basicSA}), except that  $\hbg_k\ilparenthesis{\hbtheta_k}$  is substituted by $\hbg_k^{\SP 2}\ilparenthesis{\hbtheta_k}$ as below:
\begin{equation}\label{eq:gSPSA}
\hbg_k^{\SP 2} \ilparenthesis{\btheta}   =  \frac{   y \ilparenthesis{\hbtheta_k+c_k\bDelta_k} - y\ilparenthesis{\hbtheta_k-c_k\bDelta_k }  }{2c_k   }   \bDelta_k^{-1},\,\,\text{ for }k\in\natural,
	\end{equation} 
	where   the mean-zero $p$-dimensional random perturbation vector $\bDelta_k $ has a user-specified distribution satisfying conditions \cite[Sect. 7.3]{spall2005introduction},   $c_k$ is a positive scalar governing the differencing magnitude, and  $\bDelta_k^{-1}$ denotes the random vector whose individual component is the inverse of the corresponding component in $\bDelta_k$. $ \mathrm{SP2} $ in the superscript is short for ``simultaneous perturbation with two-measurements'' \cite[Sect. 7.3]{spall2005introduction}.

	\subsection{Stochastic Gradient Descent     }
	
	The
	SGD algorithm   uses \emph{first-order} information and is a foundational method when an optimization problem is converted to a root-finding problem.  It requires the availability of a  random vector   $\hbg_k^{\SG}\ilparenthesis{\cdot}$ such that $ \E\ilbracket{\given{\hbg_k^{\SG}\ilparenthesis{\btheta}}{\btheta=\hbtheta_k}} = \bg\ilparenthesis{\hbtheta_k } $. The recursion defining $\hbtheta_k$   is the same as (\ref{eq:basicSA}),   except that generic  $\hbg_k\ilparenthesis{\hbtheta_k}$ is replaced by $\hbg_k^{\SG}\ilparenthesis{\hbtheta_k}$:
	\begin{equation}\label{eq:Ystationary} 
	\hbg_k^{\SG}\ilparenthesis{\hbtheta_k} = \bg\ilparenthesis{\hbtheta_k}+\noise_k^{\SG}\ilparenthesis{\hbtheta_k},\,\,   \text{with } \E\ilbracket{ {\noise^{\SG}_k\ilparenthesis{\hbtheta_k}}   \left| {\field_k}\right. } = \zero,\,\,\text{ for }k\in\natural. 
	\end{equation}       The
	SGD algorithm  is a special case of  the R-M algorithm  (\ref{eq:basicSA})\textendash (\ref{eq:gNoisyDecomposition}),  as     $\hbg_k^{\SG}\ilparenthesis{\hbtheta_k}$ is   an  unbiased estimate of $\bg\ilparenthesis{\hbtheta_k}$.  Sometimes,  $\hbg_k^{\SG}\ilparenthesis{\hbtheta_k}$ can be obtained via \emph{deliberate} injection of mean-zero noise,     to avoid being ``stuck'' at a local solution.
	
  Through the connection between
	root-finding and optimization, SA algorithms, such as SG (\ref{eq:Ystationary}) and SPSA (\ref{eq:gSPSA}),  can be used for SO.
	
	\section{  Review on Adaptive Tracking Algorithms} \label{sect:AdaptiveReview}

	Section~\ref{sect:SAview} discusses   SA algorithms in locating  an assumed unique   $\bvartheta$ that remains the same along the entire horizon over which we carry out the optimization procedure.  However, in engineering applications about  adaptive tracking or system control, the optimum solution to the underlying parameter estimation problem generally changes. The optimal values of the model parameter may change over time because of the intrinsic evolution of the underlying process; $\bvartheta_k$ that varies with time $\uptau_k$ will substitute  for the  fixed  $\bvartheta$. Naturally, $\loss\ilparenthesis{\cdot}$ is replaced by  $\loss_k\ilparenthesis{\cdot}$.   Specific examples will be discussed in Section~\ref{sect:SpecialCases} to appear. 
	
	To track the time variability  of the sequence $\ilset{\bvartheta_k} $, it is advisable that  $a_k$ in (\ref{eq:basicSA}) should be  a constant or is non-decaying (i.e., strictly bounded away from zero).  
	It is well recognized  that  (\ref{eq:basicSA}) with constant gain $\gain_k = \gain$ generally can  track slight time variation and is of   practical usage, see   \cite[Chap. 4]{benveniste2012adaptive}. 
	
	Given that the optimizer $\bvartheta_{k}$ is drifting over time, the classical SA theories for algorithms with decaying gains presented in 	Section~\ref{sect:SAview} and relevant asymptotic properties (convergence and   normality)  are      not directly transferable to the time-varying setup.  Other notions of  ``convergence''\textemdash ``concentration'' to be more accurate\textemdash are developed for   SA algorithms with non-decaying gain, especially those with constant gain, as summarized below.

	\subsection{Assumptions on Time-Varying Target}\label{subsect:Timevaryingness}
	This section presents    many assumed forms of the nonstationary drift for the underlying optimal values of the parameters. 
		Prior works  have considered time-varying problems under the   R-M setting, with  some  hypothetical or empirical evolution forms for the underlying optimal values of the parameters $\bvartheta_k$. 
	For example, 
	\cite{ljung1990adaptation,ljung1991result,delyon1995asymptotical}  analyze a class of  recursive algorithms  after imposing   the    random-walk  assumption on $\ilset{\bvartheta_k}$. However, the random-walk model   is suboptimal because the variance of the parameter sequence will explode to infinity over time. 
	Besides,    \cite[Sect. 3.4]{ps2008adaptive} and many others assume  that the parameter can be estimated by KF, necessitating an explicit representation (e.g., linear state equation) for the evolution in $\ilset{\bvartheta_k}$. 
	Some  general forms of the error bound for nonlinear  and linear problems are discussed  in   \cite{maryak1995uncertainties,bamieh2002identification}, still, on the basis that  the knowledge-based description  (a.k.a. state equation) for $  \ilset{\bvartheta_k}$ is available.    Ref.
	\cite{kushner1995analysis}    considers  the limit as the rate of change of the functions goes to zero, yet  it   requires a Bayesian model for the changes in  $\ilset{\loss_k\ilparenthesis{\cdot}}$. Admittedly, the tracking error characterization and the inference on the resulting estimates in the aforementioned works hinge upon the model assumptions and corresponding adaptive algorithms.   To the best of our knowledge, there are no existing approaches in estimation theory that solve  a sequence of time-varying problems, under only Assumption      A.\ref{assume:BoundedVariation}  (to appear in Chapter~\ref{chap:FiniteErrorBound}) or B.\ref{assume:gSequenceRegularity}  (to appear in Chapter~\ref{chap:Limiting})   without any further  stringent state evolution assumption. Note that both A.\ref{assume:BoundedVariation}  and B.\ref{assume:gSequenceRegularity}   allow sporadic jumps in the sequence $ \ilset{\bvartheta_k} $. The relation between existing time-varying assumptions and ours is explained in Subsection \ref{subsect:boundedvariation}.

	 Things become  more complicated as  the traditional continuous dynamics via  a  differential equation and discrete switching via  jump process are not sufficient in modeling complex systems   in finance \cite{merton1976option}, physics \cite{hall1981incoherent}, or computer vision\cite{grenander1994rep}\remove{the peak and off-peak electricity supply-need}. Hybrid diffusions  can be modeled by a two-component Markov process, a continuous component (diffusion), and a discrete component (jump component).  The change detection strategy discussed in Section~\ref{sect:Detection} mainly focuses on detecting the jump component.

	\subsection{Criteria for Tracking Performance}\label{subsect:constantgain}

	This section lists  a few metrics to evaluate the tracking performance of SA algorithms with non-decaying gains, especially with constant-gains. 
	With  a perpetually varying target $\bvartheta_k$ and the  inherent observation noise  in either $y\ilparenthesis{\cdot} = \loss\ilparenthesis{\cdot} + \upvarepsilon\ilparenthesis{\cdot}$ or $\bY\ilparenthesis{\cdot} = \bg\ilparenthesis{\cdot} + \noise \ilparenthesis{\cdot}$, there is no convergence per se. 
	The  concentration argument, that $\hbtheta_k$  stays within a neighborhood of  $\bvartheta_k $ in a certain statistical sense  at   time  $\uptau_k$, is   widely used in   practical implementation.  Often, we characterize the distance between $\hbtheta_k$ and $\bvartheta_k$ using the MSE criteria $ \E\norm{\hbtheta_k-\bvartheta_k}^2 $. Note that the MSE is a family of criteria indexed by   $k$.  
	Under the mean-squared-error (MSE)\label{acronym:MSE} tracking criteria, \cite[Chap. 4]{benveniste2012adaptive} analyzes  the  tracking capability of recursive algorithms  with a constant gain by assuming additional information on the state equation is available. References 
	  \cite{eweda1985tracking,macchi1986optimization} focus on the LMS algorithm in tracking time-varying solutions and  presents an  asymptotic stochastic big-$ O $ bound  of the tracking error; however, the asymptotic bound is  valid for   linear models only       and is not computable in general because of the higher than the fourth moments of the design vector  required in the bound.  There are also some finite-iteration  error bounds developed under fairly strong assumptions. 	 For an  asymptotically stabilized target, i.e., $\lim_{ k\to\infty}\bvartheta_k=\bvartheta$, \cite{wang2014rate} studied  the quantification of the MSE bound.  
	  Chapter~\ref{chap:FiniteErrorBound} to appear  considers the problem of estimating unknown parameters and computing error bounds in a dynamic model.  The finite-sample analysis of MSE for the estimates generated from  (\ref{eq:basicSA}) will be the  central topic there. 	For finite-sample analysis, we list the distinction between    \cite{wilson2019adaptive} and our work  in Subsection~\ref{subsect:distinction}.

	A form of convergence   is possible for constant gains, typically based on limiting arguments as the gain magnitude gets small.  The notable work
	 \cite[Chaps. 2\textendash3]{kushner1984approximation},  \cite[Chaps. 7\textendash10]{kushner2003stochastic},  \cite{kushner1978stochastic,kushner1981asymptotic,kushner2003stochastic}  extensively illustrate the weak convergence method and its application in the constant-gain algorithms.    Ref.
	 \cite{kushner1981asymptotic} relates the  limiting behavior of   SA  iterates in time-varying parameter identification problem  to  the asymptotic behavior of the  limiting autonomous ODE, and thereupon establishes the theoretical foundation for the constant gain in accommodating general ``time-varying parameter'' identification problems: the estimates generated by the  constant-gain SA  algorithms tend to the true time-varying parameters, when both the constant gain and the time difference between two discrete sample points tend to zero.    
	   Later, \cite{pflug1986stochastic} states similar results for constant-gain SA algorithm applied in constrained optimization.   
	 Nonetheless, such analysis only applies to the  asymptotic behavior of   SA estimates. In reality, neither the constant gain nor the time difference can go to zero in practice: a constant gain bounded away from zero is required, and the number of iterations per unit of time has to be finite. The weakly convergence limit for estimates generated from (\ref{eq:truncatedSA1})
	  will be discussed in Chapter~\ref{chap:Limiting}, and a computable probabilistic bound will be provided under certain conditions. 	For weak-convergence argument, we list the distinction between    other prior works and our work  in Subsection~\ref{subsect:distinction2}.

	For the fixed target case, i.e., $\bvartheta_k=\bvartheta$ for all $k$, one still cannot recover convergence a.s. of     SA algorithms with a constant gain $\gain_k=\gain$ because the noise input is ``persistent'' as opposed to ``asymptotically negligible'' in the diminishing gain case. However, we are able to say something about the limiting stationary distribution of $\hbtheta_k$, which is desirably centered near $\bvartheta$.   Ref. \cite{kushner1981asymptotic} uses an ODE to approximate the asymptotic trajectory of the constant-gain SA iterates $\hbtheta_k$, which  lays the   foundation for the constant gain in accommodating time variability. It is proven that the estimate  $\hbtheta_k$ generated by the  constant-gain     SA algorithms tends to the true time-varying parameter $\bvartheta$, when both the constant gain $a$  and the time difference between two discrete sample points tend to zero.\remove{  \cite{mukherjee1996online} claims that $\hbtheta_k$ has nonnormal limiting distribution for $a>0$, i.e., $ \hbtheta_k $ oscillates around $\bvartheta$ at rate $\sqrt{a}$. 
		Later \cite{pflug1986stochastic} elaborates similar results for constant-gain   algorithm applied in constrained optimization. } However, the asymptotic theory does not provide a practical gain-selection schema  except for a vague expression ``small $a$,''\remove{ or  ``some transition points as the  gain $a$ approaching zero''} let alone the stringent assumption of $\bvartheta_k=\bvartheta$ for all $k$.   
	In addition to  the discussion on constant-gain algorithms, 
	\cite{defossez2015averaged} lends insight  into non-diminishing gain  selection to balance the bias-variance trade-off in the context of a stationary optimizer $\bvartheta$.    References
	\cite{yousefian2012stochastic,nemirovski2009robust} discuss some gain adaptation for stationary problems, as 
	the ``asymptotically optimal'' stepsize can perform poorly in the practice \cite[Sect. 4.5.3]{spall2005introduction}.  In general, a  larger value of constant gain $\gain$ helps the resulting iterates converging more quickly to the vicinity of the optimal parameter sequence $\{\bvartheta_k\}$, corresponding to the state and measurement models; yet a smaller value of $\gain $ increases the tracking stability.   
	For asymptotically fixed targets such that $ \lim_{k\to\infty}\bvartheta_k=\bvartheta $,  
	 \cite{wang2014rate} also analyzes the   MSE decomposition for $\bvartheta_k$.    Though both \cite{defossez2015averaged}    and  \cite{wang2014rate}   consider time-varying objective functions $ \ilset{\loss_k} $, the fixed or asymptotically fixed $\bvartheta$ assumption limits their application in reality.

Other than the MSE criteria and the weak convergence argument, there are other streams in quantifying the tracking performance of constant-gain SA algorithms.	 Ref. \cite{pflug1986stochastic}
analyzes the  convergence properties in the small stepsize limit and the associated functional central limit theorem for fluctuations around the deterministic ODE limit. Ref.
\cite{joslin2000law}  establishes a law of iterated logarithms.    The functional central limit theorem characterizing
a Gauss-Markov process as a limit in law of suitably scaled fluctuations
is also used for suggesting performance metrics for tracking application, see \cite{benveniste1982measure}.

	\remove{ 
	\subsection{Relation with Other Literature} 
	\red{Some connections and distinctions with other works are briefly discussed. } 
	 	\subsubsection*{State Estimation} 
	\subsubsection*{System Identification} 
	\subsubsection*{Missing Data Problem} 
	\subsubsection*{Nonlinear Dynamic Reconstruction}
	
}

\section{Supporting Materials in ODE}\label{subsect:NonautonomousSystem}

One useful method to analyze the   property  of   SA estimates is to relate  the iterates $\hbtheta_k $ to   the  trajectory  of  an initial value problem (IVP)\label{acronym:IVP}. The ODE in   this  IVP  is  determined by the average dynamics of the algorithm.   Chapter~\ref{chap:Limiting}   will discuss  the ODEs defined by   the dynamics    {projected} onto a  {compact} constraint set denoted as $\bTheta$. 
The solutions to such ODEs will be  the weakly convergence limits of the paths of constrained SA algorithms.  A basic constrained or projected  SA algorithm is
\begin{equation}\label{eq:truncatedSA1}
\hbtheta_{k+1} = \Proj_{\bTheta}      \big( \hbtheta_k   - a_k  \hbg_k\ilparenthesis{\hbtheta_k } \big),\,\, k\in\natural,
\end{equation} 
where $\bTheta \subsetneq \real^p$ is  closed and bounded, $ \Proj_{\bTheta}\ilparenthesis{\bzeta} = \argmin_{\btheta\in\bTheta}\norm{\btheta-\bzeta} $, and   $\hbg_k\ilparenthesis{\hbtheta_k }$ can also be decomposed as   (\ref{eq:gNoisyDecomposition}).  
This section reviews some supporting materials  for   ODEs that facilitate the analysis of (\ref{eq:truncatedSA1}).

\subsection{Limits of Sequence of Continuous Functions}\label{subsect:Arzela}
 The extended  Arzel\`{a}-Ascoli Theorem  reviewed in this subsection    will be useful in extracting weak-convergent subsequence whose limits satisfy the  mean  ODE   in Section~\ref{sect:ConstantGain}.   
Let  $C^j\parenthesis{\real  \mapsto \real^p }$ be  the space of functions that map from $ \real $  to $\real^p$ and are $j$th-order continuous.   Usually, $C^0\ilparenthesis{\real \mapsto\real^p}$ is compactly  written as  $C\ilparenthesis{\real\mapsto\real^p}$.   We can  similarly define $ C^j\parenthesis{  \bracket{l,r} \mapsto\real^p }  $, $ C^j\parenthesis{\left[0,\infty\right)\mapsto\real^p } $, where $l$ and $r$ are real numbers.  The metric   for  both    $C^j\parenthesis{  \bracket{l,r} \mapsto\real^p }  $   is  the supremum norm, and the metric for  $C^j\parenthesis{\real  \mapsto \real^p }$ and   $ C^j\parenthesis{\left[0,\infty\right)\mapsto\real^p } $   is    the local  supremum norm.  For example, a sequence of functions $ \ilset{\bm{f}_k\parenthesis{\cdot}}  $  in $ C \ilparenthesis{\real\mapsto\real^p} $ converges to zero if it converges to zero uniformly on every  bounded time interval within  the domain of definition.

\begin{dfn}[Equicontinuous]
	Let the function sequence $ \set{\bm{f}_k\parenthesis{\cdot}} $ indexed by $k$ be  a subset of     $ C \parenthesis{\real\mapsto\real^p} $.   The function sequence   $\ilset{\bm{f}_k\ilparenthesis{\cdot}}$  is said to be  {equicontinuous}  if  (1) $ \set{\bm{f}_k\parenthesis{0}} $ is bounded for all $k$; and  (2)
		  for each $T> 0 $ and $\upvarepsilon>0$, there exists a $\updelta>0$ such that  $ \sup_{ 0\le t-s\le \updelta, \abs{t}\le T }    \norm{  \bm{f}_k\parenthesis{t} - \bm{f}_k\parenthesis{s}  } \le \upvarepsilon $ for all $k$. 
\end{dfn}

\begin{thm}	[Arzel\`{a}-Ascoli]  \label{thm:ArzelaAscoli}
If  the function sequence $ \set{\bm{f}_k\parenthesis{\cdot}}  $ is    equicontinuous  in the function space  $C \parenthesis{\real\mapsto\real^p }$, then there exists a subsequence that converges to some  function in $C \parenthesis{\real\mapsto\real^p }$,  {uniformly} on each  {bounded} interval. 
\end{thm}

\begin{dfn}
	[Equicontinuous in the extended sense] \label{dfn:EquiContExtended} Let  $ \bm{f}_k\parenthesis{\cdot}:\real\mapsto\real^p  $ be    measurable  for every $k$. Note that $\bm{f}_k\ilparenthesis{\cdot}$ is not necessarily continuous. 
The function sequence    $ \set{\bm{f}_k\parenthesis{\cdot }}  $ is said to be equicontinuous in the extended sense if (1)  $  \set{\bm{f}_k\parenthesis{0}} $ is bounded for all $k$, and (2)  for each $T>0 $ and $\upvarepsilon>0$, there exists a $\updelta>0$ such that 
	$ \limsup_k  \sup_{ 0\le t-s\le \updelta, \abs{t}\le T } \norm{ \bm{f}_k\parenthesis{t}- \bm{f}_k\parenthesis{s} }\le \upvarepsilon $. 
\end{dfn}

\begin{thm} 	[Extended Arzel\`{a}-Ascoli]  \label{thm:ExtendedArzelaAscoli}
 If the function sequence $ \set{\bm{f}_k\parenthesis{\cdot}}  $    is equicontinuous in the extended sense, then there exists a subsequence that converges to a function in $ C\ilparenthesis{\real\mapsto\real^p} $,  {uniformly} on each  {bounded} interval. 
\end{thm}

\subsection{  Existence and Uniqueness of the  Result}\label{subsect:ExistUnique}
The regularity conditions to ensure the existence and uniqueness of the solution to an IVP reviewed in   this subsection   will be applied on the average ODE  in Chapter~\ref{chap:Limiting}.  
\remove{
\begin{dfn}[Super/Upper (Sub/Lower) Solution]
	\red{A differentiable function $\btheta\parenthesis{t}$ satisfying $\dot{\btheta}_+\parenthesis{t}>\bm{f}$ for all $t\in\left[ t_0, T\right)$. }
\end{dfn}
}

\begin{dfn}[Locally Lipschitz continuous] Let $\bm{f}$ be a $\real^p$-valued function that takes input  arguments $ \ilparenthesis{t,\btheta}  $ within the open domain $U\subseteq\real^{p+1}$. 
	$\bm{f}$ is said to be locally Lipschitz continuous  in   $\btheta$  uniformly w.r.t.    $t$, if 
	\begin{equation}   \sup_{ \parenthesis{t, \btheta_1} \neq \parenthesis{t, \btheta_2}  \in V} \frac{ \norm{ \bm{f}\parenthesis{t,\btheta_1} - \bm{f}\parenthesis{t,\btheta_2} } }{\norm{\btheta_1-\btheta_2}}<\infty,
	\end{equation}
	for  every  compact subset $V\subsetneq U$. 
\end{dfn}
\begin{dfn}\label{dfn:LipsUnif}
	We say that $\bm{f}\in C^0_\mathrm{Lips} \ilparenthesis{U\mapsto\real^p} $ with $U\subseteq\real^{p+1}$ being open, if $\bm{f}$ is zeroth-order continuous in  $  \btheta $, and is  locally Lipschitz  continuous in $\btheta$  uniformly w.r.t.  $t$. Usually, $ C^0_\mathrm{Lips} \ilparenthesis{U\mapsto\real^p} $ will be compactly written as  $ C_\mathrm{Lips} \ilparenthesis{U\mapsto\real^p} $. 
\end{dfn}

\begin{rem}
	Note that Definition  \ref{dfn:LipsUnif}  does not convey  any information  regarding  whether  $\bm{f}$ is continuous in $t$. \remove{For $j\geq    1$,  the locally Lipschitz condition can be inferred from the $j$th-order continuity in $\btheta$. }
\end{rem}

\remove{
\begin{thm}
	Suppose that $\bm{f}$ is locally Lipschitz continuous w.r.t. to $\btheta$ and uniformly in $t$. Let $\bZ\parenthesis{t}$ and $\bzeta\parenthesis{t}$ be two \emph{differentiable} functions such that 
	\begin{equation}
	\bZ\parenthesis{t_0}\le \bzeta\parenthesis{t_0}, \quad \dot{\bZ}\parenthesis{t} - \bm{f}\parenthesis{\bZ\parenthesis{t},t} \le \dot{\bzeta}\parenthesis{t} - \bm{f}\parenthesis{\bzeta\parenthesis{t}, t}, t\in\left[t_0, T\right). 
	\end{equation}
	Then we have $\bZ\parenthesis{t}\le \bzeta\parenthesis{t}$ for every $t\in\left[t_0,T\right)$. Moreover, if $\bZ\parenthesis{t}<\bzeta\parenthesis{t}$ for some $t$, this remains true for all later times.

\end{thm}

\begin{corr}
	[Uniqueness of the solution]
\end{corr}

}

Consider the following  IVP:
\begin{equation}\label{eq:generalIVP}
\begin{dcases}
& \frac{\diff}{\diff t} \btheta\ilparenthesis{t} =\bm{f}\parenthesis{t,\btheta}, \,\, t\ge t_0, \\
& \btheta\parenthesis{t_0} = \hbtheta_0   \,.
\end{dcases}
\end{equation}
\remove{\begin{numcases}{}
\frac{\diff}{\diff t} \btheta\ilparenthesis{t} =\bm{f}\parenthesis{t,\btheta}, \,\, t\ge t_0, & \label{eq:generalIVP1}\\
\btheta\parenthesis{t_0} = \hbtheta_0   \,.& \label{eq:generalIVP2} 
\end{numcases} }
When dealing with IVP (\ref{eq:generalIVP}), we   often  suppose that  $\bm{f}\in  C _\mathrm{Lips} \ilparenthesis{U\mapsto\real^p}    $ and the domain  $U\subseteq\real^{p+1}$  is      open.

\begin{thm}[Picard-Lindelf\"{o}f]
	\cite[Thm. 2.2]{teschl2012ordinary}  Suppose $\bm{f}\in C_\mathrm{Lips} \ilparenthesis{U\mapsto\real^p}$, where $U\subseteq\real^{p+1}$ is   {open}. Then there  exists a  {unique local}  solution $\bZ\parenthesis{t}\in C^1 \parenthesis{I\mapsto\real^p}$ of the IVP (\ref{eq:generalIVP}), where $I\subsetneq\real $ is some interval around $t_0$.

	For example,   let  $M$ be   the maximum of $\norm{f}$ on $\bracket{t_0,t_0+T} \times \overline{ \Ball_{\updelta}\ilparenthesis{\hbtheta_0} }  \subsetneq U$. The solution  $\bZ\ilparenthesis{t}$ exists at least for $t\in\bracket{t_0,t_0+T_0}$ and remains within  $\overline{ \Ball_{\updelta}\ilparenthesis{\hbtheta_0} } $, where $T_0    = \min\set{T,  {\updelta}/{M}} $.   The analogous result holds for $\ilbracket{t_0-T_0, t_0}$. 
\end{thm}

\begin{corr}
	\cite[Lem. 2.3]{teschl2012ordinary} Suppose $\bm{f}\in C^j  \ilparenthesis{U\mapsto\real^p}$ for  {$j\ge 1$} where $U\subseteq \real^{p+1}$ is open and $\ilparenthesis{t_0, \hbtheta_0}\in U$. Then there exists a unique local solution  $\bZ\ilparenthesis{t}\in C^{j+1}\ilparenthesis{I\mapsto\real^p} $ of   the IVP (\ref{eq:generalIVP}),   where $I\subsetneq\real $ is some interval around $t_0$. 
\end{corr}

\begin{thm}[Improved Picard-Lindelf\"{o}f]
	\cite[Thm. 2.5]{teschl2012ordinary} Suppose $\bm{f}\in  C_\mathrm{Lips} \ilparenthesis{U\mapsto\real^p}$ where $U\subseteq\real^{p+1}$ is open.  Choose $\ilparenthesis{t_0, \hbtheta_0}\in U$ and $\updelta, T>0$ such that $ \bracket{t_0,t_0+T} \times \overline{ \Ball_{\updelta}\ilparenthesis{\hbtheta_0} }  \subsetneq U$. Set
	\begin{numcases}{}
M\parenthesis{t}=\int_{t_0}^t\sup _{ \btheta\in  \Ball_{\updelta} \ilparenthesis{\hbtheta_0} } \norm{\bm{f}\parenthesis{s, \btheta}}\diff s  & \nonumber \\
\LipsPara\parenthesis{t}  = \sup_{ \btheta_1\neq \btheta_2\in \Ball_{\updelta} \ilparenthesis{\hbtheta_0}} \frac{  \norm{ \bm{f}\parenthesis{t,\btheta_1}-\bm{f}\parenthesis{t,\btheta_2} } }{\norm{\btheta_1-\btheta_2}}.  &  \label{eq:Lips(t)}
	\end{numcases} 
 Define $T_0$ as  $T_0\equiv \sup \set{ \given{ 0<t\le T }{  M\parenthesis{t_0+t}\le \updelta  } }$, which is well-defined because $M\parenthesis{t}$ is nondecreasing in $t$. 
	Suppose $L_1 \parenthesis{T_0} \equiv    \int_{t_0}^{t_0+T_0} \LipsPara\parenthesis{s}\diff 
	s <\infty$. Then there exists a   {unique local} solution $\bZ\ilparenthesis{t}\equiv \lim_{ m\to\infty} \ilbracket{\bK^m \ilparenthesis{\hbtheta_0}}\ilparenthesis{t}\in C^1 \parenthesis{  \bracket{t_0,t_0+T_0 } \mapsto \overline{ \Ball_{\updelta}\ilparenthesis{\hbtheta_0} } }$ of the IVP (\ref{eq:generalIVP}), where $\bracket{\bK\parenthesis{\btheta}	} \parenthesis{t} = \hbtheta_0 + \int_{t_0}^t \bm{f}\parenthesis{\btheta\parenthesis{s},s}\diff s$, and satisfies     $\sup_{ t_0\le t \le t_0+  T_0} \norm{ \bZ\parenthesis{t} - \bK^m\ilparenthesis{\hbtheta_0}\parenthesis{t}  } \le \frac{ L_1\parenthesis{T_0}^m }{m! } e^{L_1\parenthesis{T_0}} \int_{ t_0} ^{t_0+T_0} \norm{\bm{f}\ilparenthesis{s, \hbtheta_0}}\diff s $.   
\end{thm}

In fact, the continuity of  $\bm{f}$    is not   necessary   to ensure  the existence of a  local solution $\bZ\ilparenthesis{t}$. For $\bZ\ilparenthesis{t}$ to exist locally, all we need  are (i) $\bm{f}$ is  {measurable}, (ii)  $M\parenthesis{t}$ is  finite, and (iii) $\LipsPara\parenthesis{t}$ is  locally integrable in terms of  $\int_ I \LipsPara\parenthesis{s}\diff s <\infty$ for any compact interval $I$. However, under less stringent conditions,  the solution $\bZ\ilparenthesis{\cdot}$ may no longer  fall  in  $C^1\ilparenthesis{I\mapsto\real^p}$.

\begin{corr}[Extension Theorem] \label{corr:ExtensionThem}
	\cite[Corr. 2.6]{teschl2012ordinary} Suppose $\bracket{t_0,T}  \times \real^p  \subsetneq U$ and $\int_{t_0} ^T \LipsPara\parenthesis{s}\diff s<\infty$ where $\LipsPara\parenthesis{t} $ is defined in (\ref{eq:Lips(t)}), then $\bZ\ilparenthesis{t}$ is well-defined for all $t\in\bracket{t_0,T}$.  
	In particular, if $U=\real^{p+1}$ and $\int_{-T} ^T \LipsPara \parenthesis{s}\diff s<\infty$ for \emph{all} $T>0$, then $\bZ\ilparenthesis{t}$ is well-defined for all $t\in\real$. 
\end{corr}

In real-world applications, the  iterates are usually  confined  within   a   {compact} set $\bTheta$ as in (\ref{eq:truncatedSA1}). If an iterate ever leaves $\bTheta$, it is immediately sent back to the closest point in $\bTheta$. In accordance with the constrained SA algorithm (\ref{eq:truncatedSA1}),  we are interested in 
\begin{equation}\label{eq:constrainedODE}
\begin{dcases}
&\dot{\btheta}\ilparenthesis{t}  = \bm{f}\ilparenthesis{t,\btheta} + \bh\ilparenthesis{t}, \,\,\,\, \bh\ilparenthesis{t} \in - \ConvCone\ilparenthesis{\btheta\ilparenthesis{t}}, \\
&\btheta\ilparenthesis{0}=\hbtheta_0,
\end{dcases}
\end{equation}
where   $\bh\parenthesis{\cdot}$ is  the minimum force needed to keep $\btheta\parenthesis{\cdot}$   within $\bTheta$.   
Specifically, for $\btheta\in\mathrm{int}\ilparenthesis{\bTheta}$, $\ConvCone\parenthesis{\btheta\ilparenthesis{t}}$ contains $\zero$ only; for $\btheta\in\partial \bTheta$,   $\ConvCone\parenthesis{\btheta\ilparenthesis{t}}$ is the   convex cone generated by the set of  outward normals at $\btheta\ilparenthesis{t}$ of the faces on which $\btheta\ilparenthesis{t}$ lies, and therefore $\bh\ilparenthesis{t}$ points inward.\remove{ For further details, see \cite[Sect. 4.3]{kushner2003stochastic}. }

\subsection{  Alekseev's Formula }\label{subsect:Alekseev}
The Alekseev's formula reviewed in this subsection  will  facilitate deriving a computable probabilistic bound in  Section~\ref{sect:ProbBound}. 

\begin{dfn}
	[Fundamental Matrix] Let $\bA\ilparenthesis{t}\in\real^{p\times p}$ for all $t$. A fundamental matrix of a system of $p$ homogeneous ODEs $\dot{\bz}\ilparenthesis{t} = \bA\ilparenthesis{t} \bz\ilparenthesis{t}$  is a matrix-valued function $\bPhi\parenthesis{t}$ whose columns are linearly independent solutions of the ODE.
\end{dfn}

A useful  tool for bounding  the errors resulted from tolerable perturbations is the Alekseev's formula 
\cite{alekseev1961estimate}. Consider the IVP (\ref{eq:generalIVP})  and its perturbed system
\begin{equation}\label{eq:generalIVPperturbed}
\begin{dcases}
&\frac{\diff}{\diff t}\bzeta\ilparenthesis{t}  =\bm{f}\parenthesis{t,\bzeta} + \be\ilparenthesis{t,\bzeta}, \,\, t\ge t_0,  \\
&
\bzeta\parenthesis{t_0} = \hbtheta_0.   \,
\end{dcases}
\end{equation} 
Assume  that $\bm{f}:  \real\times\real^p \mapsto \real^p    $ appearing in both (\ref{eq:generalIVP}) and  (\ref{eq:generalIVPperturbed})   
 is measurable in $t$ and continuously differentiable in $\btheta$ with bounded derivatives uniformly w.r.t. $t$. Further assume that $\be : \real\times\real^p \mapsto \real^p    $  appearing in  (\ref{eq:generalIVPperturbed})     is measurable in $t$ and Lipschitz in $\btheta$ uniformly w.r.t. $t$.

	Let $ \btheta\ilparenthesis{t;t_0,\hbtheta_0} $  and $ \bzeta\ilparenthesis{t;t_0,\hbtheta_0} $ denote respectively the unique solutions to (\ref{eq:generalIVP}) and (\ref{eq:generalIVPperturbed})  for $t\ge t_0$ with initial condition $ \btheta\ilparenthesis{t_0;t_0,\hbtheta_0}=\bzeta\ilparenthesis{t_0;t_0,\hbtheta_0}=\hbtheta_0 $. Then for $t\ge t_0$, we have the following representation:
\begin{equation}\label{eq:Alekseev}
\begin{split}
&	 \bzeta\ilparenthesis{t;t_0,\hbtheta_0} = \btheta\ilparenthesis{t;t_0,\hbtheta_0} -\int_{t_0}^{t}\bracket{\bPhi\big(  t;s,\bzeta\ilparenthesis{s;t_0,\hbtheta_0}\big)\,\be\big( s,\bzeta\ilparenthesis{s;t_0,\hbtheta_0}\big)  }\diff s, 
\end{split}
\end{equation}
where $ \bPhi\ilparenthesis{t;s,\hbtheta_0} $ for any $\hbtheta_0\in\real^p$ is the fundamental matrix of the linearized system
\begin{equation}
\dot{\bz}\parenthesis{t} =   \left.\frac{  \partial \bm{f}\parenthesis{t, \btheta} }{\partial\btheta^\transpose} \right|_{ \btheta=\btheta  \ilparenthesis{t;s,\hbtheta_0} } \cdot \bz\parenthesis{t},  \quad \text{ for } t\ge s,
\end{equation}
such that $\bPhi \ilparenthesis{s;s,\hbtheta_0}=\bI_p$.

\subsection{Stability for  Nonautonomous System}\label{subsect:stability}
\remove{Ratnadip Adhikari, ``A Treaties on Stability of Autonomous and Non-autonomous Systems}
The notion of stability for nonautonomous system reviewed in this subsection  will be applied on the average ODE in Section~\ref{sect:ConstantGain}.

For IVP (\ref{eq:generalIVP}),  take $t_0=0$ w.l.o.g. 
\begin{dfn}
	[Equilibrium of Unconstrained Nonautonomous System] The equilibrium point $\bvartheta$ of the IVP   (\ref{eq:generalIVP})  is such that  $\bm{f}\ilparenthesis{t,\bvartheta}=\zero$ for all $t\ge t^*$ with $t^* \ge   0$.
\end{dfn}
When an equilibrium exists,       the system state  remains at $\bvartheta$  once it reaches   $\bvartheta$. 
By a suitable transformation, we can make the equilibrium point of the transformed system to be the origin $\zero$. With   abuse of notation, we  use (\ref{eq:generalIVP})  to represent the transformed system whose equilibrium  point is at the origin within the rest of this subsection. 

\begin{dfn}
	[Stable]  \label{dfn:Stable} The IVP   (\ref{eq:generalIVP}), whose  equilibrium is the origin $\zero$,   is said to be stable at $t^*$ if, for any $\upvarepsilon>0$, there exists a real number $\updelta=\updelta\ilparenthesis{\upvarepsilon,t^*}>0$ such that $ \norm{\btheta\ilparenthesis{0}}\le \updelta $ implies $ \norm{\btheta\ilparenthesis{t}}\le\upvarepsilon $ for all $t\ge t^*$.
\end{dfn}

\begin{dfn}
	[Convergent]  \label{dfn:Convergent}  The IVP   (\ref{eq:generalIVP}), whose  equilibrium is the origin $\zero$,   is said to be convergent at $t^*$ if, \remove{there exists $\updelta\ilparenthesis{t^*}$ such that for any $\upvarepsilon>0$, there exists $T(\upvarepsilon, t^*)>0$ such that   $ \norm{\btheta\ilparenthesis{0}}\le\updelta$ implies $\norm{\btheta\ilparenthesis{t}}\le\upvarepsilon $ for all $t\ge t^*+T$. }there exists a real number $\updelta = \updelta\ilparenthesis{t^*}>0$ such that $ \norm{\btheta\ilparenthesis{0}}\le \updelta$ implies $ \lim_{t\to\infty}\btheta\ilparenthesis{t}=\zero $.  
\end{dfn}

\begin{dfn}
	[Asymptotically Stable]  \label{dfn:AsymStable}  The IVP   (\ref{eq:generalIVP}), whose  equilibrium is the origin $\zero$,    is said to be asymptotically stable at time $t^*$ if it is both stable and convergent at $t^*$.
\end{dfn}

Definitions \ref{dfn:Stable}\textendash \ref{dfn:AsymStable} can be strengthened to ``uniformly stable,'' ``uniformly convergent,''  and ``uniformly asymptotically stable'' respectively,  if the  dependence on $t^*$ can be removed from the defining statements. In fact, the uniformity in time is important     to ensure the attraction region does not vanish  as time varies.

\begin{dfn}
	[Lyapunov's Stability] A set $  A  \subseteq\real^p $ is said to be  {locally stable in the sense of Liapunov} if, for each $\upvarepsilon
	>0$, there exists a $\updelta>0$ such that all trajectories starting from $\hbtheta_0\in   \Ball_{\updelta}\ilparenthesis{A} $ will never leave $  \Ball_{\upvarepsilon}\ilparenthesis{A } $. If the trajectories ultimately go to $A $, then $A $ is said to be  {asymptotically stable in the sense of Liapunov}. If this holds for all initial conditions, then the  {asymptotic stability} is said to be  {global}. 
\end{dfn}

 \remove{  
 \begin{rem}
 	We generally do not need to know the Liapunov functions themselves. We only need to know that they exist and have appropriate properties.
 \end{rem} 
}

\remove{
\begin{dfn}
	[Upper SemiContinuity of $ C\parenthesis{\btheta} $] An ``infinitesimal'' change in $\btheta$ cannot increase the number of \emph{active} constraints. Thus, loosely speaking, this property can be stated as $ \lim_{\bzeta\to\btheta} C\parenthesis{\bzeta} \subset C\parenthesis{\btheta} $. More precisely, let $ O\parenthesis{\btheta,\updelta} $ be a $\updelta$-neighborhood of $\btheta$, then \begin{equation}
	\bigcap_{\updelta>0} \mathrm{co}\bracket{  \bigcup_{ \bzeta\in O\parenthesis{\btheta,\updelta} } C\parenthesis{\bzeta}  } = C\parenthesis{\btheta} 
	\end{equation}
	where $\mathrm{co}\parenthesis{A}$ denotes the closed convex hull of the set $A$. A set-valued function $C\parenthesis{\cdot}$ satisfying above is said to be \emph{upper semi-continuous}.
\end{dfn}
}

\section{ Review on Weak Convergence} \label{subsect:WeakConvergence}
This section lays out basic facts in weak convergence theory.  
``Weak convergence'' of  the  function-valued random variables (a.k.a.  random function)  extends the notion of ``convergence in distribution'' of  the     $\real^p$-valued random variables.   Let  $D\ilparenthesis{\real \mapsto\real^p}$  be       the space of   functions that map from  $\real $ to $\real^p$ and    are  right-continuous with left-hand limits.       We are    interested in the function space $D\ilparenthesis{\real\mapsto\real^p}$  equipped  with the Skorohod topology \cite[Sect. 12]{billingsley1968convergence}\remove{\cite[p. 238]{kushner2003stochastic}}.   The exact definition of the Skorohod topology is somewhat technical\footnote{For a function $\uplambda(\cdot)$ in the space $ \Lambda_T $ of strictly increasing, continuous mappings of $\bracket{0,T}$ onto itself, first define $ \norm{\uplambda}= \sup_{s<t<T} \abs{  \log \bracket{  \parenthesis{ \uplambda(t)  - \uplambda(s)} / \parenthesis{t-s}  }   } $. The distance between $x(\cdot)$ and $y(\cdot)$ in $ D\parenthesis{\bracket{0,T}\mapsto\bracket{0,T}} $ is given by $ \inf_{\uplambda\in\Lambda_T}\ilset{ \norm{\uplambda }  + \sup_{t\in\bracket{0,T}} \abs{  x(t) - y(\uplambda(t))  }   } $.  The Skorohod topology on the space $D\parenthesis{\left[0,\infty\right)\mapsto \left[0,\infty\right)}$ or  $D\parenthesis{\real\mapsto\real}$ can be defined by requiring the convergence in the Skorohod metric on each compact interval $\bracket{0,T}$ or $\bracket{-T,T}$ for $T>0$. The metric on the product space $D\ilparenthesis{\real\mapsto\real^p}$ can be taken to be the sum of the metrics on the component spaces.     } and not essential for the upcoming proofs. 
 	Under the  Skorohod topology, $D\ilparenthesis{\real \mapsto\real^p}$ is  {separable} and  {metrizable}, and the metric is  {complete}.

	The function space  $D\ilparenthesis{\real \mapsto\real^p}$ is useful for two  reasons. First of all, the processes with paths in  $D\ilparenthesis{\real \mapsto\real^p}$ come up  naturally in applications. Moreover, the Skorohod topology in  $D\ilparenthesis{\real \mapsto\real^p}$  is an extension of the topology of uniform convergence on bounded time intervals in  $C\ilparenthesis{\real \mapsto\real^p}$,  in     that a local   $\parenthesis{k,t}$-dependent stretching or contraction of the time scale is allowed  where $k$ is the index of the function sequence  and $t$ lies within the domain of the $k$th function. 
	Therefore, this topology is weaker than that of  $C\ilparenthesis{\real \mapsto\real^p}$, so that the criteria for compactness are less stringent, even if the paths or their limits may still  lie within   $C\ilparenthesis{\real \mapsto\real^p}$.  Consequently, this property   is useful in  dealing with ``nice'' discontinuities such that the discontinuities do not appear in the limit.  Second, the convergence in such space has    many important ramifications (see further details \cite[Chap. 2]{kurtz1981approximation}). What matters the most to us is that   the convergence of a sequence of functions in  $D\ilparenthesis{\real \mapsto\real^p}$ to a continuous function in  $C\ilparenthesis{\real \mapsto\real^p}$ in the Skorohod topology is equivalent to  {convergence uniformly on each bounded time interval} in  $C\ilparenthesis{\real \mapsto\real^p}$.

	 \subsection{Weak Convergence}
	 This section reviews the notion of weak convergence of random function  by making an analogy with the weak convergence of random variable, and they will be used    in the proofs of  Subsection~\ref{subsect:WeakResult}.   
	Let us work on a common probability space $ \ilparenthesis{\Omega, \Prob,\mathcal{A}} $. 
	
	\begin{dfn}
		[Weak Convergence of Random Variables] \label{dfn:weakRV} \cite{billingsley1968convergence} A sequence of $\real^p$-valued random variables $\bX_k\ilparenthesis{\upomega}$ indexed by $k$ is said to converge    in distribution  to a $\real^p$-valued random variable $\bX\ilparenthesis{\upomega}$ if and only if $\E \boundedFn\parenthesis{\bX_k\ilparenthesis{\upomega}}\to \E \boundedFn\parenthesis{\bX\ilparenthesis{\upomega}}$ as $k\to\infty$ for every  {bounded} and  {continuous} function $\boundedFn:\real^p\mapsto\real $. 
	\end{dfn}
	The weak convergence in Definition~\ref{dfn:weakRV} is also known as ``convergence in distribution'' and   can be compactly written as $\bX_k  \weakconverge \bX$ by suppressing the $\upomega$ dependence.  
	To distinguish random variables and random processes, we represent random processes  using $\bx\parenthesis{\cdot,\cdot}$ and $\bx $ takes $t\in\real$ and $\upomega\in\Omega$ as inputs. For each fixed time $t$, $\bx\parenthesis{t,\cdot}$ is a random variable. For each fixed sample $\upomega,$ $\bx\ilparenthesis{\cdot,\upomega}$ is a function of time. We say that $\bx\parenthesis{t,\upomega}$ is measurable  if $\bx\ilparenthesis{t,\upomega}: \real\times\Omega \mapsto\real^p$ is $\mathcal{B}\parenthesis{\real}\times\mathcal{A}$-measurable, where $\mathcal{B}\ilparenthesis{\real}$ is the Borel sigma-field on the real line.  We say that a random process $\bx\parenthesis{t,\upomega}$ is (almost surely)  {continuous}  if,  for (almost) every $\upomega\in\Omega$, $\bx\parenthesis{\cdot,\upomega}$ is continuous w.r.t. $t$.                    
	Similar to writing the random variable $\bX\ilparenthesis{\upomega}$ as $\bX$, we write $\bx\ilparenthesis{\cdot,\cdot}$ compactly as $\bx\ilparenthesis{\cdot}$ by suppressing the $\upomega$ dependence; i.e., $\bx\ilparenthesis{\cdot}$ is a random function (of time $t$). The following definition of the weak convergence of random functions is a natural extension of Definition~\ref{dfn:weakRV}. 
	
	\begin{dfn}
		[Weak Convergence of Random Functions] \label{dfn:weakRF} \cite{prokhorov1956convergence}	
    The weak convergence  (a.k.a. convergence in distribution) $\bx_k\parenthesis{\cdot}\weakconverge \bx\parenthesis{\cdot}$ in $D\ilparenthesis{\real \mapsto\real^p}$ is equivalent to  $ \E \bm{\boundedFn}\ilparenthesis{\bx_k\ilparenthesis{\cdot,\upomega}}  \to \E\bm{\boundedFn} \ilparenthesis{\bx \ilparenthesis{\cdot,\upomega}}$ as $k\to \infty $ for any  bounded and continuous function $\bm{\boundedFn}$  that maps $ \bx\parenthesis{\cdot}\in D\ilparenthesis{\real \mapsto\real^p}$ to $\parenthesis{  \bx\parenthesis{t_1}^\transpose,\cdots,\bx\parenthesis{t_l}^\transpose }^\transpose \in \real^{pl}$ for $t_1,\cdots,t_l\in\real$ and  $l\ge 1$. 
     \end{dfn}

	\subsection{Tightness} \label{subsect:Tightness}
	 This section reviews the notion of tightness and   a set of sufficient conditions to prove tightness, which will be useful  in  Subsection~\ref{subsect:WeakResult}.

	The    sequence of random variables  $\set{\bX_k}$ is said to be   {tight} (a.k.a.  uniformly bounded in probability) if, for each $\updelta\in\left(0,1\right]$, there exists a compact set $\compactset_{\updelta}$ such that $\Prob\set{\bX_k\in \compactset_{\updelta}} \ge 1-\updelta$ for all $k$. 
	The Helly--Bray theorem states that a tight sequence must have  a further subsequence that converges weakly. We also define the notion of  sequential compactness by  ``each subsequence contains a further subsequence that converges weakly.''  
	
	The definition of tightness  carries over to   random functions $ \bx_k\parenthesis{\cdot} $. The exact statement of the tightness of random functions is technical and we will mainly use the following  supporting lemma to facilitate proving tightness of a sequence of random functions $\bx_k\ilparenthesis{\cdot}$ within $D\ilparenthesis{\real \mapsto\real^p}$.

	\begin{lem}\label{lem:Tightness}
		Let $\btheta_k\parenthesis{\cdot}$ be a sequence of  random processes indexed by $k$ with paths in $ D\ilparenthesis{\real \mapsto\real^p} $.   If the following two conditions hold, then we claim that $ \set{\btheta_k\parenthesis{\cdot}} $ is tight in $ D\ilparenthesis{\real \mapsto\real^p} $. 
		\begin{enumerate}
			\item  {Compact containment condition}. For each $\updelta\in\left(0,1\right]$, and for  {each} $t$ in a  {dense} subset of  $\real$, there exists a compact set $\compactset_{\updelta,t}\subsetneq\real^p $ such that
			\begin{equation}\label{eq:compact1}
			\inf_k\Prob\set{\btheta_k\parenthesis{t}\in \compactset_{\updelta,t}} \ge 1-\updelta.
			\end{equation}
			\item For each $T>0$, 
			\begin{equation}\label{eq:compact2}
			\lim_{\uptau\to 0} \limsup_{k\to\infty} \sup_{ 0\le  s\le \uptau, \abs{t}\le T  } \E\norm{\btheta_k\parenthesis{t+s}-\btheta_k\parenthesis{t}}=0.
			\end{equation}
		\end{enumerate}
	
\end{lem}
	\begin{rem}
	(\ref{eq:compact2}) does not imply the continuity of the paths of either $\btheta_k\parenthesis{\cdot}$ or any weak sense limit $\btheta\parenthesis{\cdot}$. \remove{In fact, (\ref{eq:compact2}) holds if $\set{\btheta_k\parenthesis{\cdot}}$ is a sequence of continuous-time Markov processes on a compact state space with uniformly bounded and time dependent transition rates. It immediately follows that any weak sense limit of this sequence is also a continuous-time Markov process on the same compact state space and time dependent transition functions. }
\end{rem} 
\begin{proof}[Proof of Lemma~\ref{lem:Tightness}]
	
	Under the metric defined in  \cite[Sect. 12]{billingsley1968convergence},  the metric space $ D\ilparenthesis{\real \mapsto\real^p}  $  is  {separable} and  {complete}. Prohorov's theorem \cite[p. 104]{ethier2009markov} states that tightness is equivalent to sequential compactness on a complete\footnote{A metric space is complete if every Cauchy sequence in it converges to a point in it. } separable\footnote{A topological  space is separable if it contains a countable dense subset. } metric space.  By Prohorov's Theorem, any   sequence in $D\ilparenthesis{\real \mapsto\real^p} $  is tight if and only if it is relatively compact. The result  follows from \cite[Thm. 2.7 on p. 10]{kurtz1981approximation}. 
\end{proof}

If a sequence    $ \set{\bx_k\parenthesis{\cdot}} $ on a  {complete}  {separable} metric space (mainly $D(\real\mapsto\real^p)$ in our discussion) is shown to be tight through Lemma~\ref{lem:Tightness}  then it must have a weakly convergent subsequence.    The proposition that ``if a sequence of random functions is tight, then it has a weak convergent subsequence'' is in fact an extension of the proposition that ''if a sequence of random variables is tight, then it has a subsequence that converges in distribution to some random variable.''

\begin{lem} \cite[p. 230]{kushner2003stochastic} \label{lem:weakcont}
	Suppose that a sequence of processes $\set{\btheta_k\parenthesis{\cdot}}$ is tight in $D\ilparenthesis{\real \mapsto\real^p} $ and that on each interval $\bracket{-T,T}$ the size of the maximum discontinuity goes to zero in probability as $k\to\infty$, then any weak sense limit process must have continuous paths w.p.1. 
\end{lem}

\remove{ 

\begin{thm}
	\cite[Thm. 7.3.6 on p. 233]{kushner2003stochastic} Let the processes $\btheta_k\parenthesis{\cdot}$ have paths in $D\ilparenthesis{\left[0,\infty\right)  \mapsto\real^p}     $ w.p.1. Suppose that for each $\updelta\in\parenthesis{0,1}$ and $T>0$, there is a process $\btheta_{k,\updelta,T}\parenthesis{\cdot}$ with paths in $D\ilparenthesis{\left[0,\infty\right)  \mapsto\real^p}  $ w.p.1 such that $$ \Prob\ilset{ \sup_{t\le T} \norm{  \btheta_{k,\updelta,T}\parenthesis{t} - \btheta_k\parenthesis{t}   }\ge \updelta }\le \updelta. $$
	If $\ilset{  \btheta_{k,\updelta,T}  }_{k\ge 0}$ is tight for each $\updelta$ and $T$, then $\set{\btheta_k\parenthesis{\cdot}}$ is tight.

	If $\ilset{  \btheta_{k,\updelta,T}  }_{k\ge 0}$ converges weakly to a process $\btheta\parenthesis{\cdot}$ that does  {not} depend on $\parenthesis{\updelta,T}$, then the original sequence converges weakly to $\btheta\parenthesis{\cdot}$.
	
	Suppose that for each $\updelta\in\parenthesis{0,1}$ and $T>0$, $\ilset{  \btheta_{k,\updelta,T}  }_{k\ge 0}$ converges weakly to a process $\btheta_{\updelta,T}\parenthesis{\cdot}$ and that there is a process such that the measures of $\btheta_{\updelta,T}\parenthesis{\cdot}$ and $\btheta\parenthesis{\cdot}$ on the interval $\bracket{0,T}$ are equal, except on a set $\nullset$ whose probability goes to zero as $\updelta\to 0$. Then $\set{\btheta_k\parenthesis{\cdot}}$ converges weakly to $\btheta\parenthesis{\cdot}$.

	Analogous result holds for processes with paths in $D\ilparenthesis{\real  \mapsto\real^p}  $. 
\end{thm}

}

\subsection{Skorohod Embedding} \label{subsect:Skorohod}

 This section reviews the notion of Skorohod embedding, which  will be applied     in the proofs of  Subsection~\ref{subsect:WeakResult}.  
Recall that 
$D\ilparenthesis{\real  \mapsto\real^p}   $ is a complete  and separable metric space with the   metric $d\parenthesis{\cdot,\cdot}$  that metricizes the  Skorohod topology. 
\begin{thm} 
	[Skorohod representation] \label{thm:Skorohod}  Let $\btheta_k\ilparenthesis{\cdot} \weakconverge \btheta\ilparenthesis{\cdot}$ for $\btheta_k\ilparenthesis{\cdot},\btheta\ilparenthesis{\cdot}\in D\ilparenthesis{\real  \mapsto\real^p}   $. There exists a probability space $\ilparenthesis{ \tilde{{\Omega}}, \tilde{\mathcal{B}}, \tilde{\Prob} }$ with associated   random functions $\tilde{\btheta}_k\ilparenthesis{\cdot}$ in $D\ilparenthesis{\real  \mapsto\real^p}  $ and $\tilde{\btheta}\ilparenthesis{\cdot}$ defined on it such that for each dense\footnote{A subset $B$ of a topological space $A$ is dense if every point in $A$ either belongs to $B$ or a is a limit point of $B$. } set $\compactset\subsetneq D\ilparenthesis{\real  \mapsto\real^p}  $:
	\begin{equation}
	\tilde{\Prob}\ilset{\tilde{\btheta}_k\ilparenthesis{\cdot}\in \compactset  } = \Prob\ilset{\btheta_k\ilparenthesis{\cdot}\in \compactset}, \,\,  \tilde{\Prob}\ilset{\tilde{\btheta}\ilparenthesis{\cdot}\in \compactset} = \Prob\ilset{\btheta\ilparenthesis{\cdot}\in \compactset},
	\end{equation} and $d\ilparenthesis{\tilde{\btheta}_k\ilparenthesis{\cdot},\tilde{\btheta}\ilparenthesis{\cdot}}\to 0$ w.p.1. 
\end{thm}
W.l.o.g., we   suppose that the probability space is  carefully chosen so that weak convergence is equivalent to convergence w.p.1 uniformly on bounded time intervals. Note that the use of the Skorohod representation itself does not imply that the original sequence $\btheta_k\ilparenthesis{\cdot}$ converges  w.p.1.

\chapter{Tracking Capability  }\label{chap:FiniteErrorBound}

  This chapter focuses on the analysis of the tracking performance of  adaptive-gain SA algorithms. 
 Time-varying optimization  problems arise  frequently, including in deterministic nonlinear programming \cite{bertsekas2009convex}, e.g., the method of penalty functions involves selecting a growing\footnote{With this perspective, it is easier to study the behavior of a solution at infinity and also estimate the strategy of the choice of the sequence of penalty coefficients. } sequence of the penalty coefficients and solve for the constrained minimization problem sequentially for each iteration.    
 Aside from the underlying controllable parameter $\btheta$, the loss function $\loss\ilparenthesis{\cdot}$ may also  depend on some other factors, such as  time. Following the  motivations discussed     in  Section \ref{sect:Motivation}, we are mainly interested in situations  where the noisy information of the time-varying  loss functions  can be collected at sampling time $\uptau_k$ corresponding to discrete index $k$.  That is, this chapter considers  a sequence of loss functions $\loss_k\ilparenthesis{\btheta}$ at sampling times  $\uptau_k$, instead  of  one single loss function $\loss\ilparenthesis{\cdot}$ that remains unchanged.  Moreover,  only a small number (either one or two) of   noisy observations    pertaining to  $\loss_k\ilparenthesis{\cdot}$  are revealed   at the sampling  time  $\uptau_k$.  Such setup contrasts with    \cite{popkov2005gradient,simonetto2017time} in that they assume       \emph{noise-free} observations. It  is also different from \cite{wilson2019adaptive},  where  as  many  sequential measurements as needed   can be collected at each discrete time instance. In this setting, we only require at most    two parallel measurements. The meanings of ``sequential'' and ``parallel'' will be explained in Subsection \ref{subsect:distinction}. 
  
This chapter is dedicated to showing  the tracking capability of SA algorithms with non-decaying gain as applied in a  time-varying framework, where  a  sequence of loss  functions $\ilset{\loss_k\ilparenthesis{\cdot}}$
   changes  along  $\ilbracket{\uptau_0,\uptau_K}$   of our interest with  $K$ being the last sampling index, and  a slowly  time-varying optimum.   By ``slowly'' we mean that  the average  distance between successive optimizers is strictly bounded from above on average; infrequent jumps are allowed in such a   setting.
Section~\ref{sect:ProlemSetup} presents     the problem setup  and Section \ref{sect:ModelAssumptions} discusses the model assumptions. 
Section \ref{sect:UnconditionalError} establishes the tracking capability by computing  the  error bound for MAD and RMS.     Section \ref{sect:SpecialCases} discusses some special cases of (\ref{eq:basicSA}) for nonlinear root-finding.

\section{Problem Formulation}\label{sect:ProlemSetup}
	
This section  introduces    necessary      concepts  arising in the parameter estimation   and  states   the    target-tracking problem.  
 
\subsection{Basic Setup of Time-Varying SA Problems}\label{subsect:ProblemSetup}

	We consider the   problem of estimating a time-varying parameter $\ilset{\bvartheta_k}$  that varies ``slowly''  and formulate the problem from an  online convex optimization perspective.   Each $\loss_k$ in the  sequence of convex functions $ \ilset{\loss_k\parenthesis{\cdot}}  $ indexed by   $k$  is a differentiable mapping  from $\real^p$ to $\real$.   Recall that  each     index $k$ corresponds to the actual time $\uptau_k$. Within this chapter, let us suppose that the sampling frequency is bounded from above; i.e., the actual time elapsed between two consecutive samples,  $   \uptau _{k+1} - \uptau_ k $,  is bounded from below by  zero     for all $k\in\natural$. Note that the sampling intervals $  \uptau _{k+1} - \uptau_ k $ need \emph{not} remain constant across $k$.

	 Our goal is to efficiently track  the value(s) of  $\btheta$ that minimizes  instantaneous scalar-valued  loss function (sequence) $\loss _k(\cdot)$:
\begin{equation}\label{eq:Minimization}
\text{Find }\,
\bvartheta_k  \equiv \arg \min_{\btheta\in\Theta} \loss _k(\btheta)  \,\text{for each }   k\in\natural.  
\end{equation}  
Let $\hbtheta_k$, whose recursive scheme  will soon be discussed in the next subsection,  represent our best  possible estimate for parameter $\bvartheta_k$ at time $\uptau_k$. The experimenter does not   know the exact   functional form of $\loss_k\ilparenthesis{\cdot}$, but can receive     instant  feedback  
immediately  after the decision $\hbtheta_k$ is selected.    Usually, the  instant     feedback regarding $\loss_k\ilparenthesis{\cdot}$ at a design point $\hbtheta_k$  is either a  noisy realization of the cost or a noisy evaluation of the gradient information 
\begin{numcases}{}
y_k\parenthesis{\btheta}=\loss_k\ilparenthesis{\btheta}+\upvarepsilon_k\ilparenthesis{\btheta},& \label{eq:y} \\
\bY_k\ilparenthesis{\btheta}= \frac{\partial\loss_k\ilparenthesis{\btheta}}{\partial\btheta} + \noise_k  \ilparenthesis{\btheta}.  & \label{eq:Y} 
\end{numcases}   A comprehensive summary of gradient estimation methods available through the mid-1990s is \cite{spall1994developments}  and some recent detailed analysis of some of these methods is given in \cite{blakney2019compare}. 
 Note that (\ref{eq:Y}) differs from  (\ref{eq:Ystationary}) in that   both of the  terms on the r.h.s.  of  (\ref{eq:Y})     vary with   $k$ whereas only $\noise_k\ilparenthesis{\cdot}$ on the r.h.s. of  (\ref{eq:Ystationary}) depends on $k$.  
 \begin{rem}
 	\label{rem:clarification} 
 	We need to clarify  both $\loss_k\ilparenthesis{\cdot}$ and its minimizer $\bvartheta_k$ are \emph{deterministic} to the experimenter. Granted,  $\bvartheta_k$ itself can evolve stochastically and a common example is that   the state space model in KF involves a multivariate normal distribution. Nonetheless, the randomness in $\bvartheta_k$ will \emph{not} be taken into account while formulating the loss function at time $\uptau_k$, an example of which is (\ref{eq:LinearGradient}) to appear. That said, at time $\uptau_k $, the loss function $\loss_k\ilparenthesis{\cdot}$ is formulated in a way that $\bvartheta_k$ is deemed as a fixed value, and only the measurement noise $\upvarepsilon_k\ilparenthesis{\cdot}$ in (\ref{eq:y}) or  $\noise_k\ilparenthesis{\cdot}$ (\ref{eq:Y}) is taken into consideration. 
 \end{rem}

The    general setting  (\ref{eq:Minimization})\textendash (\ref{eq:Y})    subsumes many     target tracking scenarios  where $\bvartheta_k$   represents the locations of the targets being pursued by one or more agents.  	  The agents are  expected  to utilize the immediate  feedback  via    either (\ref{eq:y}) or (\ref{eq:Y})  to improve their estimates $\hbtheta_k$ for  parameter $\bvartheta_k$ in an online fashion. Often, at each sampling time  $\uptau_k$,  only a \emph{few} (either  one or two in our discussion) noisy measurements,  either in the form of (\ref{eq:y}) or   (\ref{eq:Y}), can be gathered, and  the evaluation point is at the agents' disposal.  In the defense applications,   the agents only observe the target's location when   necessary, because   frequent emission of radar signals   inevitably  and undesirably  reveals  the agents' position.  
Such a setting promotes the ``few measurements at each time''  requirement.

\subsection{SA Algorithm with Non-Decaying Gain}\label{sect:SAintro}
We are   interested in characterizing     SA algorithms \cite[Eq. (6.5) on p. 157]{spall2005introduction}, namely, (\ref{eq:basicSA}) with a  non-diminishing step size. Note that  $\gain_k$ has to be  strictly  bounded away from zero, and advance tuning is required (see more details in Algorithm~\ref{algo:basicSA}). Let us briefly discuss  the several  forms of $\hbg_k\ilparenthesis{\hbtheta_k}$  here, corresponding to the  two feedback forms  (\ref{eq:y}) and (\ref{eq:Y}), respectively.

	\begin{itemize}
	\item When the feedback takes  the form of (\ref{eq:y}), the agent is  allowed to collect only  \emph{one} measurement at a certain point at its disposal at every sampling instance $\uptau_k $. Then (\ref{eq:basicSA}) may include the one-measurement SPSA under further assumptions.	
	Specifically, $y_k\ilparenthesis{\cdot}$  will be evaluated at the  design point $ \hbtheta_k + c_k \bDelta_k $, where $ \bDelta_k   $ is a $p$-dimensional random  vector with  zero-mean satisfying conditions listed in \cite[Sect. 7.4]{spall2005introduction},  and $c_k$ is a small   positive number strictly bounded away from zero.  For one-measurement SPSA, $\hbg_k\ilparenthesis{\hbtheta_k}$ in  (\ref{eq:basicSA}) will be substituted by $	\hbg_k  ^{\SP 1}\ilparenthesis{\hbtheta_k}$ 
	  computed as in  \cite{spall1997one}: 
	\begin{equation}\label{eq:g1SP}
	\hbg_k  ^{\SP 1}\ilparenthesis{\hbtheta_k}\equiv 	\frac{y_k\ilparenthesis{\hbtheta_k+c_k\bDelta_k}}{c_k} \bDelta_k^{-1}, 
	\end{equation}
	where $ \mathrm{SP1} $ in the superscript is short for ``simultaneous perturbation with one-measurement'' \cite[Sect. 7.3]{spall2005introduction}.

	If the agent is allowed to collect only  \emph{two} measurements, then   (\ref{eq:basicSA}) may include the two-measurement SPSA under further assumptions.	Specifically,  the agent can evaluate  $y_k\ilparenthesis{\cdot}$   at two design points $\hbtheta_k+c_k\bDelta_k$ and $\hbtheta_k-c_k\bDelta_k$, where $\bDelta_k$ satisfies the same condition mentioned above. For two-measurement SPSA,  $\hbg_k\ilparenthesis{\hbtheta_k}$ in (\ref{eq:basicSA}) will be replaced by   $\hbg_k^{\SP 2}\ilparenthesis{\hbtheta_k}$ discussed in Subsection~\ref{subsect:SPSA}, except that $y\ilparenthesis{\cdot}$       in (\ref{eq:gSPSA}) is substituted by $y_k\ilparenthesis{\cdot}$ that depends on $k$ as (\ref{eq:y}).

	\item  When the feedback   takes   the form of (\ref{eq:Y}), the agent is allowed to collect  only \emph{one} measurement of $\bY_k\ilparenthesis{\cdot}$ evaluated at the  decision point $\hbtheta_k$. In this case,   (\ref{eq:basicSA}) may include the well-known stochastic gradient algorithm proposed in \cite{robbins1951stochastic}. For SGD,   $\hbg_k \ilparenthesis{\hbtheta_k}$ is a direct noisy gradient measurement as in (\ref{eq:Ystationary}), except that the $\bg\ilparenthesis{\cdot}$ on the r.h.s. of (\ref{eq:Ystationary}) now has a $k$-dependence.

\end{itemize}

With a slight abuse of notation   $\bias_k \ilparenthesis{\hbtheta_k}$ and   $\noise_k \ilparenthesis{\hbtheta_k}$ appearing in (\ref{eq:gNoisyDecomposition}), we  express  $\hbg_k\ilparenthesis{\cdot}$    generically as below to 	  facilitate later discussion:
\begin{align}\label{eq:gGeneral}
\hbg_k\ilparenthesis{\hbtheta_k} &=\left.  \frac{\partial\loss_k\ilparenthesis{\btheta}}{\partial\btheta}\right| _{\btheta=\hbtheta_k}+\bias_k\ilparenthesis{\hbtheta_k} + \noise_k\ilparenthesis{\hbtheta_k} \nonumber\\
&\equiv  \bg_k\ilparenthesis{\hbtheta_k}+\be_k\ilparenthesis{\hbtheta_k}
\end{align}
where the gradient function  $ \bg_k\ilparenthesis{\btheta}\equiv \partial\loss_k\ilparenthesis{\btheta}/\partial\btheta $,  the error term $\be_k\ilparenthesis{\hbtheta_k}$ subsumes  both the bias term $\bias_k\ilparenthesis{\hbtheta_k}$ and the noise term $\noise_k \ilparenthesis{\hbtheta_k}$.  Note that $ \bg_k\ilparenthesis{\cdot} $  on the r.h.s. of (\ref{eq:gGeneral}) has $k$-dependence, whereas $\bg\ilparenthesis{\cdot}$ on the r.h.s. of (\ref{eq:gNoisyDecomposition}) does not. 	Moreover, the function $\bg\ilparenthesis{\cdot}$ for root-finding purposes in Chapter~\ref{chap:Preliminaries} may or may not be a gradient of an underlying loss function.

\subsection{Distinction Relative to  Other Finite-Sample Analysis}\label{subsect:distinction}

Among   the finite-sample performance analysis, \cite{wilson2019adaptive} is derived under a similar setup and used a comparable metric. We point out  several   differences between   \cite{wilson2019adaptive}  and our work.

\begin{itemize}
	
		\item   Ref. \cite{wilson2019adaptive}  assumes   that   multiple, e.g., $N_k$, sequential  measurements of (\ref{eq:Y}) can be gathered   at each sampling time instant $\uptau_k $. Specifically,        $N_k$ grows inversely  proportional to the desired accuracy, which may   be expensive as mentioned towards the end of \cite[Sect. 2.1]{wilson2019adaptive}.  By ``sequential'' we mean  that the $N_k$  observations at time $\uptau_k $ have to be carried out sequentially.  That is, the $(i+1)$th observation depends on the $i$-th observation for $1\le i<N_k$. Such a setting may be valid
if the underlying time-varying system is changing very slowly or if the experimenter  has a nearly unlimited amount of  computation power and does not get penalized for frequent observations.
		
In contrast, we do not allow the sequential observations at each sampling  time $\uptau_k$  and discourages excessive observations at each iteration. We   consider few parallel measurements at each $\uptau_k$,'' e.g., $N_k=1$ or $2$,   in order to readily adapt to changing conditions. 
		By ``parallel'' we mean that the evaluations at the  two design points at time $\uptau_k$ can be collected simultaneously\textemdash one does not depend  on the computation of another  one.

	\item Ref.    \cite{wilson2019adaptive} assumes that        the  immediate feedback   is in the form of   (\ref{eq:Y}) only, whereas our work allows the feedback to take the form of  either (\ref{eq:y}) or (\ref{eq:Y}).  Moreover,   \cite{wilson2019adaptive}   assumes that  the noisy gradient measurement is an unbiased estimator of  the true gradient; i.e.,  the $\be_k$  in   (\ref{eq:gGeneral}) under their setting is mean-zero, whereas we allow $\be_k$ in (\ref{eq:gGeneral})  to have a nonzero mean in general.

	\item In both tracking  criteria proposed  in \cite{wilson2019adaptive}, the randomness in $\bvartheta_k$ is not considered, and hence the randomness in $\loss_k$ is not allowed. On the contrary, our tracking performance result in Section  \ref{sect:UnconditionalError}    allows for some  randomnesses  in $\bvartheta_k$. 
	
	\item    Ref.  \cite{wilson2019adaptive} implicitly assumes that the selected gain will  enable the estimate to keep track of the moving target, and does not unveil their details in  the gain selection.     In contrast,  we   provide some  practical guidance in gain selection.

\remove{
\item Besides, there is no need to state   \cite[A.5]{wilson2019adaptive} as an assumption, because  it can be    ensured  by   A.\ref{assume:ErrorWithBoundedSecondMoment} and  A.\ref{assume:Lsmooth}.  This is also the case for  \cite[A.4]{wilson2019adaptive}. In short,  the parameter estimation part should have been  much streamlined there.  }

\end{itemize}

\section{Model Assumptions}\label{sect:ModelAssumptions}

We  now state  the   assumptions required for later derivations.	 	Throughout our discussion, the norm imposed on a vector is the Euclidean norm, and the norm imposed on  a matrix  is the matrix spectral norm, which is  the  matrix norm compatible with the Euclidean vector norm.    
The following assumptions  are in parallel with the statistical set of  conditions for the strong   convergence   in \cite[Sect. 2]{blum1954approximation} and \cite[Sect. 4.3]{spall2005introduction}, except for the  non-decaying  gain adapted for the extra restrictions on the drift and the nonstationarity explained in the next section.

\begin{assumeA}
	[Error Term Has     Bounded Second Moment] \label{assume:ErrorWithBoundedSecondMoment}
	There exists a finite number $ \noiseBound_k \equiv \sqrt{ \sup_{\btheta\in\real^p} \E\ilbracket{\norm{\be_k\ilparenthesis{\btheta}}^2}    } $ for each $k\in\natural$.  
\end{assumeA}

\begin{assumeA}[Strong Convexity]\label{assume:StronglyConvex}
	The  instantaneous loss function $\loss_k\ilparenthesis{\cdot} \in  C^1 \ilparenthesis{\real^p\mapsto\real }$   and strongly convex for all $k\in\natural$.  Moreover, $\convexPara_k$ is the  largest positive number such that $ \ilparenthesis{\btheta-\bzeta}^\transpose\ilparenthesis{ \bg_k\ilparenthesis{\btheta} - \bg_k\ilparenthesis{\bzeta} } \ge \convexPara_k\norm{\btheta-\bzeta}^2 $   holds for all $\btheta$, $\bzeta\in\real^p$, where $\bg_k\ilparenthesis{\btheta} = \partial\loss_k\ilparenthesis{\btheta}/\partial\btheta$.  
\end{assumeA}

\begin{assumeA}[Smoothness]\label{assume:Lsmooth}  For each $k\in\natural$, $\LipsPara_k$ is the smallest positive number such that $\bg_k\ilparenthesis{\cdot} \in C^0\ilparenthesis{\real^p\mapsto\real^p}$     is $\LipsPara_k$-Lipschitz.   
\end{assumeA}

\begin{assumeA}[Bounded Variation] \label{assume:BoundedVariation}
	There exists a finite number $ \driftBound_k \equiv \sqrt{   \E \ilparenthesis{ \norm{\bvartheta_{k+1}-\bvartheta_k }^2 }  } $ for each $k\in\natural$. It reduces to $\driftBound_k = \sqrt{    \ilparenthesis{ \norm{\bvartheta_{k+1}-\bvartheta_k }^2 }  }$  if the sequence $\ilset{\bvartheta_k}$ is deterministic. 
\end{assumeA}

\begin{rem} In addition to Remark~\ref{rem:clarification}, we reiterate that the randomness in $\bvartheta_k$ is \emph{not} taken into consideration while formulating $\loss_k\ilparenthesis{\cdot}$, but $\bvartheta_k$ itself is allowed to vary stochastically. 
\end{rem}

	To ease the upcoming  discussion, denote the ratio\label{acronym:ratio}   $\ratio_k\equiv\LipsPara_k/\convexPara_k$,  where $\LipsPara_k$ is defined in  A.\ref{assume:Lsmooth} and $\convexPara_k$ is defined in  A.\ref{assume:StronglyConvex}.   
The following subsections provide  additional explanations on the validity of the aforementioned    assumptions.

\subsection{Estimation of Parameters in Assumptions Will Not be Considered}\label{subsect:noestimation}
Note that this chapter  aims to show  the tracking capability of SA algorithms (\ref{eq:basicSA}) with non-decaying gains applied to the time-varying problem setup (\ref{eq:Minimization}). Furthermore, we are interested in the scenario where only a  {few} (one or two)  {noisy} observations pertaining to  $\loss_k\ilparenthesis{\cdot}$ are  revealed only at time instance $k$, and  the actual time elapsed between two consecutive sampling instances $\ilparenthesis{\uptau_{k+1}-\uptau_k}$ is bounded from below. Under such  a   setting, at every sampling instance $\uptau_k$, the agent   obtains a  limited amount of corrupted information regarding $\loss_k\ilparenthesis{\cdot}$. Resultingly,    we do not expect that there exists an  efficient strategy to estimate $\noiseBound_k $, $\convexPara_k$, $\LipsPara_k$, and $\driftBound_k$ in an online fashion. Nonetheless, the assumed availability of these parameters does not nullify the   deliverables of this chapter in   demonstrating the tracking capability of SA algorithms.

It was pointed out in 
Subsection \ref{subsect:distinction} that
 although  \cite{wilson2019adaptive} handles  the estimation in part,  \cite{wilson2019adaptive}  is based upon a different setup. Namely, they assume that   as many \emph{sequential}   estimates as needed  can  be gathered at each sampling instance $\uptau_k$, whereas our setup requires one single observation or two \emph{parallel} ones.      Furthermore, the estimation of the lower-bound of  $\convexPara_k$   based on  (\ref{eq:ConvIneq}) and the estimation of the upper-bound of $\LipsPara_k$ based on  (\ref{eq:LipsIneq})  could be largely   non-informative regarding the actual value of $\convexPara_k$ and $\LipsPara_k$.

\subsection{Relation With Online Learning Literature} 
\textbf{Connection.}
The  sequential SO set up in  Section~\ref{sect:ProlemSetup} can   be interpreted in the prototypical decision-making framework.  
The agents are  viewed as learners and targets as adversaries.
At each sampling instance $\uptau_k$, the online learner selects  an action  $\hbtheta_k$ that belongs to some convex compact action set $\bTheta\subsetneq\real^p$ and incurs a cost $\loss_k\ilparenthesis{\hbtheta_k}$,  where $\loss_k\ilparenthesis{\cdot}:\real^p\mapsto\real$ is an unknown convex cost function selected by the adversary. In response to the agent's action, the adversary also reveals   inexact feedback   to the learner.

\textbf{Distinctions.}
Different from the constraint that the variable $\btheta$ belongs to a compact domain $\bTheta$, we consider the situation where the objective function is strongly convex and $\btheta\in\real^p$.  
As opposed to  the result that allows the loss $\loss_k\ilparenthesis{\cdot}$ to be adversarial  w.r.t. the selected action $\hbtheta_k$, we consider the case where $\loss_k\ilparenthesis{\cdot}$ is deterministic or  maybe random, but has to be  autonomous. By ``autonomous'' we mean that     the values of the estimates $\hbtheta_k$ do  not affect the underlying evolution of $\bvartheta_k$.
Contrary to the strong requirement that ``all the loss functions $\loss_k\ilparenthesis{\cdot}$ have  uniformly bounded gradients,'' we consider a weaker assumption     as in  A.\ref{assume:Lsmooth}. 
Different from the goal of bounding the worst-case performance of the best estimators only through the regret formulation \cite{hazan2008adaptive}, we are interested in the tracking accuracy, i.e.,  controlling the error  $\norm{\hbtheta_k-\bvartheta_k}$ at each time   $\uptau_k$. 
Moreover, the regret   $\text{Reg}_K\equiv\sum_{k=1}^K\ilbracket{\loss_k\ilparenthesis{\hbtheta_k}-\loss_k\ilparenthesis{\bvartheta_k}}$            is minimized, where $K$ is the horizon over which we implement the recursive scheme (\ref{eq:basicSA}), under the condition that there exists a bound on the total variations of the gradients over the horizon $K$. 
Admittedly, if  A.\ref{assume:StronglyConvex}   is satisfied, then the bound on $ \sum_{k=2}^{K} \sup_{\btheta\in\bTheta}\norm{\bg_k\ilparenthesis{\btheta}-\bg_{k-1}\ilparenthesis{\btheta}}^2$ implies  the bound on $ \sum_{k=2}^K\norm{\bvartheta_{k}-\bvartheta_{k-1}}^2 $. The converse is not true.   In contrast,   we only impose A.\ref{assume:BoundedVariation} and  seek to maintain a certain tracking accuracy at each time instant.    	Note that we do not get into   online prediction   where the regret along the path is minimized. Instead, we are only interested in minimizing the most current estimation error.   
Another reason is that the error bound developed for many algorithms  therein requires knowing the functional variation $ \sum_{k=2}^K \sup_{\btheta\in\bTheta} \abs{\loss_k\parenthesis{\btheta}-\loss_{k-1}\parenthesis{\btheta}} $ in advance, which is typically unavailable.

\subsection{Error Form Allowing Many SA Algorithms}\label{subsect:ErrorSA}
Let us emphasize that  A.\ref{assume:ErrorWithBoundedSecondMoment} allows for    $\hbg_k\ilparenthesis{\hbtheta_k}$ to be     a biased estimate for $\bg_k\ilparenthesis{\hbtheta_k}$; i.e., the error term $\be_k$ in (\ref{eq:gGeneral}) can have a nonzero mean. 
Furthermore, A.\ref{assume:ErrorWithBoundedSecondMoment} enables recursion  (\ref{eq:basicSA}) to subsume a    broad class of SA algorithms, including the three important cases mentioned in Section~\ref{sect:SAintro}.

 For one-measurement  and two-measurement SPSA,   A.\ref{assume:ErrorWithBoundedSecondMoment} is  readily satisfied when the following holds:    (1)  $\bDelta_k$ is  generated by Monte Carlo under the conditions of independence, symmetry, and finite inverse moments \cite{spall1992multivariate} and (2) there exists\footnote{ In the  time-varying scenario, both $\gain_k$ and $c_k$ have to be  strictly positive for the recursive SA algorithm to be able to track the moving target.   In FDSA or SPSA, the gain sequence controlling the perturbation magnitude is set to be strictly bounded away from zero for stability, despite that  in theory a decaying gain can wash out the bias of the gradient approximation as an estimator of the true gradient.
  } $ c_{\text{lower}} $   such that $ 0< c_{\text{lower}} <c_k     $ for all $k$.

	  For R-M setting,    A.\ref{assume:ErrorWithBoundedSecondMoment}  is immediately met    because zero-mean and bounded-variance  $\be_k\ilparenthesis{\cdot}$    is a special case of   A.\ref{assume:ErrorWithBoundedSecondMoment} as per \cite[A.3 and A.4 on p. 106]{spall2005introduction}.

	\subsection{Global and Local Convexity Parameter}\label{subsect:LocalConvexity}
A direct consequence of   A.\ref{assume:StronglyConvex} is the   existence and uniqueness of the optimizer $\bvartheta_k$. Moreover,    $\bg_k\ilparenthesis{\bvartheta_k}=\zero$ becomes a necessary and sufficient condition in determining $\bvartheta_k$, and it will be used in proving the upcoming Lemma~\ref{lem:BasicInequalities}.  
Admittedly, there is a class of nonconvex problems in which    A.\ref{assume:StronglyConvex} fails to hold. Nonetheless,  A.\ref{assume:StronglyConvex} is still  valid in    many fundamental  problems such as regularized regression and many others  in \cite{bharath1999stochastic}. An incomplete list is given below. 

\begin{itemize} 	\item  
Suppose that the  loss function $\loss \ilparenthesis{\cdot}$ is in the empirical risk function (ERF)\label{acronym:ERF} form. For instance, given  data pairs $\ilparenthesis{\bx_i, z_i}$ where the covariate $\bx_i$ will be mapped by a function $\bPhi\ilparenthesis{\cdot}$  to the feature space $\real^p$, then the loss function $\loss\ilparenthesis{\cdot}$  can  be  formed as  $ \loss  \ilparenthesis{\btheta} =  n^{-1}\sum_{i=1}^n\ell\parenthesis{z_i, \btheta^\transpose \bPhi\parenthesis{\bx_i }} $, where   $\ell\ilparenthesis{\cdot,\cdot}$ denotes 
either the  squared-loss or zero-one loss, and the input  has a  nonsingular sample covariance matrix $ n^{-1}\sum_{i=1}^n\ilset{\ilbracket{ \bPhi\parenthesis{\bx_i} }^\transpose \bPhi\parenthesis{\bx_i }  }$.  Such a loss function $\loss\ilparenthesis{\cdot}$ satisfies      A.\ref{assume:StronglyConvex}. 
	\item Suppose that loss function $\loss\ilparenthesis{\cdot}$ is  the sum of  $   n^{-1}\sum_{i=1}^n\ell\parenthesis{z_i,  {\btheta^\transpose\bPhi\parenthesis{\bx_i }}} $ and a regularization term  $ \convexPara   \norm{\btheta}^2 /2 $, where $\bPhi\ilparenthesis{\cdot}$ maps the input $\bx_i$ to the intended feature space. Then   A.\ref{assume:StronglyConvex} is satisfied.

	\item 
	Suppose that the loss function $\loss \ilparenthesis{\cdot}$ is the expected least-squares  written as   $ \loss\parenthesis{\btheta} = \E\bracket{ z -  {\btheta_{\mathrm{true}}^\transpose\bPhi\ilparenthesis{x}} }^2/2 $, where the expectation is taken over the joint-distribution of $ \parenthesis{\bx,z} $.  When $ \E  \ilset{\ilbracket{ \bPhi\parenthesis{\bx_i} }^\transpose \bPhi\parenthesis{\bx_i }  }   \succ \convexPara  \bI_p $,    the loss function satisfies  A.\ref{assume:StronglyConvex}.
\end{itemize}

Note that   A.\ref{assume:StronglyConvex} can be   relaxed to  local strong convexity, as   the proofs in the upcoming section   require local convexity only. Namely,   the $\convexPara_k$ in A.\ref{assume:StronglyConvex} can be the largest positive number such that $ \ilparenthesis{\btheta-\bzeta}^\transpose\ilparenthesis{\bg_k\ilparenthesis{\btheta} -\bg_k\ilparenthesis{\bzeta}} \ge \convexPara_k\norm{\btheta-\bzeta}^2 $  for $\btheta,\bzeta$ in a small neighborhood around $\hbtheta_k$.   \remove{ In such case, the uniqueness of $\bvartheta_k$ is no longer guaranteed. Denote $\Theta_k$ as the set of minimizer(s) of the loss function $\loss_k\ilparenthesis{\cdot}$. The notion $\norm{\btheta-\bvartheta_k} $ is then generalized to $\mathrm{dist}\ilparenthesis{\btheta,\Theta_k} = \inf_{ \bvartheta_k\in\Theta_k }
	\norm{\btheta-\bvartheta_k}$. The loss function $\loss_k\ilparenthesis{\cdot}$ is said to be local strongly convex, if there exists a constant $ {\convexPara_k}^{\ilparenthesis{\hbtheta_k,\bvartheta_k}}$ such that $   \loss_k\ilparenthesis{\btheta_1}\ge \loss_k\ilparenthesis{\btheta_2}+\bg_k\ilparenthesis{\btheta_2}\ilparenthesis{\btheta_1-\btheta_2} +  {\convexPara_k}^{\ilparenthesis{\hbtheta_k,\bvartheta_k}}\norm{\btheta_2-\btheta_1}^2  $ for any $\btheta_1,\btheta_2\in \mathrm{Ball}\ilparenthesis{\hbtheta_k, \norm{\hbtheta_k-\bvartheta_k}}$.  }However, we do not intend to  dwell on the    ``multiple-minimizers'' setting. The rationale and the tracking capability of non-decaying gain SA under such a  scenario require  separate consideration. 

	\subsection{Global- and Local-Lipschitz Continuity}\label{subsect:LocalLipschitz} 
	Note that   A.\ref{assume:Lsmooth}   is   more lenient than the uniform boundedness of $\bg_k\ilparenthesis{\cdot}$ for all     $\btheta$ uniformly across $k$ appearing in  \cite[Sect. 6.3]{polyak1987introduction},  \cite{besbes2015non}, and many others. In fact,    A.\ref{assume:Lsmooth}  can be met  in the sense that it is implied by other  smoothness conditions that are used in local convergence theorems and are often satisfied in practice \cite[p. 39 and Chaps 6\textendash 7]{nocedal2006numerical}.  
	
	Let us provide an example  in machine learning applications where   A.\ref{assume:Lsmooth} is satisfied. 
	Suppose that the loss function $\loss\ilparenthesis{\cdot}$ is in the ERF form, i.e., $ \loss \ilparenthesis{\btheta} =  n^{-1}\sum_{i=1}^n\ell\parenthesis{z_i, \btheta^\transpose \bPhi\parenthesis{\bx_i }} $ for some squared- or $0$-$1$ loss function $\ell\ilparenthesis{\cdot,\cdot}$, where the   data pairs $ \parenthesis{\bPhi\ilparenthesis{\bx_i},z_i} $ are all bounded.  The Hessian of $\loss\ilparenthesis{\cdot}$ is approximately the  sample covariance matrix computed as  $ n^{-1}\sum_{i=1}^n \ilbracket{\bPhi\parenthesis{\bx_i}  \bPhi\parenthesis{\bx_i }}  $ for large $n$, hence   A.\ref{assume:Lsmooth} is satisfied. 

	Arguably,   A.\ref{assume:Lsmooth} does not hold even for a scalar-valued univariate function $g_k\ilparenthesis{\uptheta}$ with $g_k \ilparenthesis{\cdot}$ being a second- or higher-order polynomial function or the multiplicative-inverse function defined over $\uptheta\in\real$. The following two   observations help alleviate the concern regarding its appropriateness: 
	
	\begin{itemize}
		\item  In real-world applications, the parameter $\btheta$ is typically subject to physical restrictions or other technical constraints. \remove{We will discuss \red{constrained variant of (\ref{eq:AdaptiveGainAlgo}) in Subsection \ref{subsubsect:ConstrainedOpti} briefly}. }For instance, if $\btheta$ is confined within a closed and bounded region $\bTheta\subsetneq \real^p$, then a finite $\LipsPara_k  $ within $\bTheta$ is  attainable.
		
		\item In the upcoming proof, we can effectively replace the global smoothness by local smoothness   $ {\LipsPara_k} ^{\ilparenthesis{\hbtheta_k,\bvartheta_k}} $, the smallest number  such that $ \norm{\bg_k\ilparenthesis{\btheta_1 }-\bg_k\ilparenthesis{\btheta_2 }}\le    {\LipsPara_k} ^{\ilparenthesis{\hbtheta_k,\bvartheta_k}} \norm{\btheta_1-\btheta_2}    $ holds for any $\btheta_1,\btheta_2$ in    a ball centered at $\hbtheta_k$ with radius of $ \norm{\hbtheta_k-\bvartheta_k} $. 
	\end{itemize}

	In summary, both A.\ref{assume:StronglyConvex}  and A.\ref{assume:Lsmooth} can be weakened if a priori knowledge of the domain of the optimizers, denoted by  $\bTheta$, is known. If so, we can concentrate on functions that meet   A.\ref{assume:StronglyConvex}  and A.\ref{assume:Lsmooth} for $\btheta\in\bTheta\subsetneq\real^p$ and adapt the following proof for algorithm (\ref{eq:truncatedSA1}) readily. Nonetheless, the analysis in Section \ref{sect:UnconditionalError} reveals that the estimate $\hbtheta_k$ will stay close to $\bvartheta_k$ 
	with appropriate initialization and gain selection, and therefore we only require A.\ref{assume:StronglyConvex}  and A.\ref{assume:Lsmooth}  to be valid locally.

	\subsection{Interpreting Ratio of $\LipsPara_k$ and $\convexPara_k$}\label{subsect:ratio}
	Note that for $\convexPara_k$ in A.\ref{assume:StronglyConvex} and  $\LipsPara_k$ in A.\ref{assume:Lsmooth} to be well-defined, we only need $\loss_k\ilparenthesis{\cdot}$ be in  $C^1\ilparenthesis{\real^p\mapsto\real}$, i.e., continuously differentiable,  as stated in A.\ref{assume:StronglyConvex}.

	To provide a better intuition behind $\ratio_k$, let us further assume   \cite[Assumption B.5'' on p. 183]{spall2005introduction}, i.e., the loss function $\loss_k\ilparenthesis{\cdot}$ is  in $C^2\ilparenthesis{\real^p\mapsto\real}$ and is  bounded on $\real^p$. Let us denote $\bH_k\ilparenthesis{\cdot}$ as the Hessian of $\loss_k\ilparenthesis{\cdot}$, which is guaranteed to be square and positive-definite by A.\ref{assume:StronglyConvex}. By Taylor's Theorem for multivariate vector-valued function, we know $\convexPara_k$ in A.\ref{assume:StronglyConvex}  becomes $\inf_{\btheta\in\real^p}\uplambda_{\min}\ilparenthesis{\bH_k\ilparenthesis{\btheta}}$, and $\LipsPara_k$ in A.\ref{assume:Lsmooth} equals $\sup_{\btheta\in\real^p} \uplambda_{\max} \ilparenthesis{\bH_k\ilparenthesis{\btheta}}$. Note that both the inf and the sup are attainable under   \cite[Assumption B.5'' on p. 183]{spall2005introduction} as both $\uplambda_{\min} \ilparenthesis{\bH_k\ilparenthesis{\btheta}}$ and $\uplambda_{\max}\ilparenthesis{\bH_k\ilparenthesis{\btheta}}$ are continuous functions of $\btheta$. Therefore, when $\bH_k\ilparenthesis{\cdot}$ exists and satisfies certain smoothness conditions,  $\ratio_k$ 
	can be interpreted to be an  upper bound of the condition number of the Hessian because:
	\begin{equation}\label{eq:RatioCondition}
	\ratio_k =  \frac{\LipsPara_k}{\convexPara_k} = \frac{\sup_{\btheta\in\real^p} \uplambda_{\max}\ilparenthesis{\bH_k\ilparenthesis{\btheta}} }{\inf_{\btheta\in\real^p} \uplambda_{\min }\ilparenthesis{\bH_k\ilparenthesis{\btheta}}} \ge \sup_{\btheta
}  \frac{\uplambda_{\max}\ilparenthesis{\bH_k\ilparenthesis{\btheta}} }{\uplambda_{\min }\ilparenthesis{\bH_k\ilparenthesis{\btheta}}} = \sup_{\btheta} \parenthesis{  \mathrm{cond}\ilparenthesis{\bH_k\ilparenthesis{\btheta}} }.
	\end{equation}

		\subsection{Parameter Variations and Error Bounds}\label{subsect:boundedvariation}
	
		Note that  the model for $\ilset{\bvartheta_k}$ is autonomous because updating $\hbtheta_k$ by (\ref{eq:basicSA}) has  no effect on the true parameter $\bvartheta_k$.  Intuitively, 
	A.\ref{assume:BoundedVariation}  is imposed  to capture the fact that the sequence (\ref{eq:Minimization})   is changing ``slowly,'' yet it does not exclude  abrupt changes  as long as the corresponding probability   is small.  Overall, the  expected change of the optimal parameter between every two consecutive   time instants is modest.

	 Within the classical literature on linear models, \cite{farden1981tracking,eweda1985tracking}  obtained upper bounds of the limiting-time mean-square error (the deviation of the estimated value  of the parameter from the actual value) by assuming only that the speed of variation of the true system is bounded by some deterministic constant.  We port the idea over to the general nonlinear models.  
	
	In physical application to moving objects,
	A.\ref{assume:BoundedVariation} effectively sets a  maximum speed  of the target, which is generally reasonable, given the physical constraints of motion. If the target's position is denoted by $\bvartheta_k$ and it is moving at a constant speed, then no randomness arises in the sequence in $\ilset{\bvartheta_k}$ and   $ \sum_{k=1}^K\norm{\bvartheta_k-\bvartheta_{k-1}} $ is $  O(\uptau_K -\uptau_0)$. 
	As a practical example, consider a target that continues to move at  a constant speed in an adversarial manner. Likewise, we may consider the scenario when the errors are unpredictably random with  the bounded second moment.  We point out that \cite[Sect. 2.2]{wilson2019adaptive} provides other justifications for 
	A.\ref{assume:BoundedVariation}.
	Therefore, for the target tracking setting, it makes sense to characterize the parameter variations using the  {path length $ \sum_{k=1}^K\norm{ \bvartheta_k-\bvartheta_{k-1} } $} for some sequence of parameter values $ \ilset{\bvartheta_k} $. Alternative measures might   include  {functional variation $ \sum_{k=1}^K \sup _{\btheta\in\bTheta}\abs{ \loss_k\parenthesis{\btheta} - \loss_{k-1}\parenthesis{\btheta} } $} and the  {gradient variation $ \sum_{k=1}^K\sup_{\btheta\in\bTheta}\norm{\bg_k \parenthesis{\btheta} - \bg_{k-1} \parenthesis{\btheta}}^2 $}, which are usually unavailable in advance.

\section{ Tracking Performance Guarantee}\label{sect:UnconditionalError}

This section  characterizes the tracking performance  $ \norm{\hbtheta_k-\bvartheta_k} $ of    recursion (\ref{eq:basicSA}) with non-diminishing gain, where  $\hbg_k \ilparenthesis{\hbtheta_k}$  can  take  either of the   representations (\ref{eq:g1SP}) and (\ref{eq:Ystationary}).    When the target is perpetually varying, it is impossible for the agent to further reduce its distances  from   the target  beyond a certain value.    Hence, we assume the                                                necessary assumptions in Section~\ref{sect:assumptions} to facilitate our error bound analysis.  
Note that the values of  $\noiseBound_k$, $\LipsPara_k$,  and  $\driftBound_k$    are   the smallest possible positive reals such that  the  assumptions A.\ref{assume:ErrorWithBoundedSecondMoment}, A.\ref{assume:Lsmooth},  and A.\ref{assume:BoundedVariation} are valid, and the value of $\convexPara_k$ is the largest  possible real such that the assumption  A.\ref{assume:StronglyConvex} is legitimate. The  analysis here is built upon  the basis that $\noiseBound_k$ in A.\ref{assume:ErrorWithBoundedSecondMoment}, $\convexPara_k$ in A.\ref{assume:StronglyConvex}, $\LipsPara_k$ in A.\ref{assume:Lsmooth},  and $\driftBound_k$ in A.\ref{assume:BoundedVariation} are  available to the agent.

 \subsection{Supporting Lemmas}

 Lemma \ref{lem:BasicInequalities}  provides  some inequalities    that immediately follow  from the  assumptions stated in Section \ref{sect:ModelAssumptions}. They will be used in the upcoming subsection.
 \begin{lem}\label{lem:BasicInequalities}
 	For a loss function $\loss _ k $ satisfying  A.\ref{assume:StronglyConvex} and A.\ref{assume:Lsmooth},  the following inequalities hold  for all $\btheta\in\real^p$: 	 	
 	\begin{eqnarray}
 	& 	 {\convexPara _ k }  \norm{ \btheta-\bvartheta_k } ^ 2/ 2 \le \loss _ k \ilparenthesis{\btheta } - \loss _ k \ilparenthesis{\bvartheta_k} \le  {\LipsPara _ k }  \norm{\btheta-\bvartheta_k}^2/2,   \label{eq:ineq1}  \\
 	&  2 \convexPara _ k \bracket{  \loss_k \ilparenthesis{\btheta } - \loss_k\ilparenthesis{\bvartheta_k} } \le \norm{\bg_k\ilparenthesis{\btheta}} ^ 2   \le 2 \LipsPara_k \bracket{\loss_k\ilparenthesis{\btheta} - \loss_k\ilparenthesis{\bvartheta_k}},\quad\quad   \label{eq:ineq2}\\
& 	 	\loss_k\ilparenthesis{\btheta} - \loss_k\ilparenthesis{\bvartheta_k} \le \bg_k\ilparenthesis{\btheta}^\transpose\ilparenthesis{\btheta-\bvartheta_k},   \label{eq:ineq3}\\
&	\LipsPara_k\ge \convexPara_k >0.  \label{eq:ineq4}
 	\end{eqnarray}
 	 
 \end{lem}
\begin{proof}
	[Proof of Lemma \ref{lem:BasicInequalities}] 
	 Given  A.\ref{assume:StronglyConvex}, we know that for any $\btheta,\bzeta\in\real^p$: 
	\begin{equation}\label{eq:ConvIneq}
	\loss_k\ilparenthesis{\btheta} \ge \loss_k \ilparenthesis{\bzeta} + \bracket{\bg_k\ilparenthesis{\bzeta}}^\transpose  \ilparenthesis{\btheta-\bzeta } + \frac{\convexPara_k }{2} \norm{\btheta-\bzeta}^2. 
	\end{equation}
	Let $\bzeta = \bvartheta_k$ in (\ref{eq:ConvIneq}) and invoke A.\ref{assume:StronglyConvex}. We then have $\loss_k\ilparenthesis{\btheta} \ge \loss_k\ilparenthesis{\bvartheta_k } + \convexPara_k\norm{\btheta-\bvartheta_k}^2/2$ and, therefore,  the first inequality of (\ref{eq:ineq1}) holds.

	 By A.\ref{assume:Lsmooth} and the mean-value theorem \cite{rudin1976principles}, we know that for any $\btheta,\bzeta\in\real^p$:
	 \begin{eqnarray}\label{eq:LipsIneq} 
	 \loss_k\ilparenthesis{\btheta}   &=& \loss_k\ilparenthesis{\bzeta} + \int_0^1 \bracket{  \bg_k \ilparenthesis{ \bzeta + t \ilparenthesis{\btheta-\bzeta } } } ^ \transpose \parenthesis{\btheta-\bzeta }\diff t  \nonumber \\ 
	 &=& \loss_k\ilparenthesis{\bzeta} + \bracket{\bg_k\ilparenthesis{\bzeta}}^\transpose \ilparenthesis{\btheta - \bzeta } + \int_0 ^ 1 \bracket{  \bg_k\ilparenthesis{\bzeta + t \ilparenthesis{\btheta-\bzeta}}  - \bg_k\ilparenthesis{\bzeta}  }^\transpose\ilparenthesis{\btheta-\bzeta}\diff t \nonumber \\
	 &\le &  \loss_k\ilparenthesis{\bzeta} + \bracket{\bg_k\ilparenthesis{\bzeta}}^\transpose \ilparenthesis{\btheta - \bzeta }  + \int_0^1 \LipsPara_k \norm{t\ilparenthesis{\btheta-\bzeta}} \norm{\btheta-\bzeta}\diff t \nonumber\\
	 &=& \loss_k\ilparenthesis{\bzeta} + \bracket{\bg_k\ilparenthesis{\bzeta}}^\transpose\ilparenthesis{\btheta-\bzeta} + \frac{\LipsPara_k}{2} \norm{\btheta-\bzeta}^2.
	 \end{eqnarray}
	 Let $\bzeta = \bvartheta_k$ in (\ref{eq:LipsIneq}) and invoke  A.\ref{assume:StronglyConvex}, we then have $\loss_k\ilparenthesis{\btheta} \le  \loss_k\ilparenthesis{\bvartheta_k } + \LipsPara_k \norm{\btheta-\bvartheta_k}^2/2$. Hence,  the second inequality of (\ref{eq:ineq1}) holds.

	Note that  for every $\btheta\in\real^p$, (\ref{eq:ConvIneq}) holds.  Let us deem both sides of (\ref{eq:ConvIneq}) as two functions of $\btheta$.     By definition, the minimum of the l.h.s. is achieved by $\btheta=\bvartheta_k$. The minimizer of the r.h.s., which is a quadratic function of $\btheta$,  is given by $\btheta=\bzeta - \convexPara_k\bg_k\ilparenthesis{\bzeta}$.   Therefore,
	\begin{equation*}
	\begin{split}
	\loss_k\ilparenthesis{\bvartheta_k}&\ge 	\left.   \set{  \loss_k\ilparenthesis{\bzeta} + \bracket{\bg_k\ilparenthesis{\bzeta}}^\transpose \ilparenthesis{\btheta -\bzeta} + \frac{\convexPara_k}{2} \norm{\btheta-\bzeta}^2   }  \right| _{ \btheta = \bvartheta_k }\\
	&\ge 
	\left.   \set{  \loss_k\ilparenthesis{\bzeta} + \bracket{\bg_k\ilparenthesis{\bzeta}}^\transpose \ilparenthesis{\btheta -\bzeta} + \frac{\convexPara_k}{2} \norm{\btheta-\bzeta}^2   }  \right| _{ \btheta = \bzeta  - \convexPara_k^{-1} \bg_k\ilparenthesis{\bzeta }  }\\
	&= \loss_k\ilparenthesis{\bzeta} - \frac{1}{2\convexPara_k}\norm{\bg_k\ilparenthesis{\bzeta}}^2.
	\end{split}
	\end{equation*}
	Hence,  the first inequality in (\ref{eq:ineq2}) holds. 
	
	Note that (\ref{eq:LipsIneq}) holds for any $\btheta,\bzeta\in\real^p$ so we have: 
	\begin{equation*}
	\begin{split}
	\loss_k\ilparenthesis{\bvartheta_k} & \le \min_{\btheta} \set{   \loss_k\ilparenthesis{\bzeta} + \bracket{\bg_k\ilparenthesis{\bzeta}}^\transpose\ilparenthesis{\btheta-\bzeta} + \frac{\LipsPara_k}{2} \norm{\btheta-\bzeta}^2    } \\
	& = \left.  \set{   \loss_k\ilparenthesis{\bzeta} + \bracket{\bg_k\ilparenthesis{\bzeta}}^\transpose\ilparenthesis{\btheta-\bzeta} + \frac{\LipsPara_k}{2} \norm{\btheta-\bzeta}^2  } \right| _{\btheta=\bzeta - \LipsPara_k^{-1} \bg_k\ilparenthesis{\bzeta }} \\
	&=\loss_k\ilparenthesis{\bzeta} - \frac{1}{2\LipsPara_k} \norm{\bg_k\ilparenthesis{\bzeta}}^2.
	\end{split}
	\end{equation*}
	Hence, the second inequality in (\ref{eq:ineq2}) holds. 
	
  Eq. (\ref{eq:ineq3}) follows from A.\ref{assume:StronglyConvex};  specifically, $\loss_k\ilparenthesis{\bvartheta_k} \ge \loss_k\ilparenthesis{\btheta} + \bg_k\ilparenthesis{\btheta}^\transpose \ilparenthesis{\bvartheta_k-\btheta}$.

	 Eq.  (\ref{eq:ineq4}) can be readily obtained by comparing (\ref{eq:ConvIneq}) and (\ref{eq:LipsIneq}). 
\end{proof}

We also present Lemma~\ref{lem:Product} and Lemma~\ref{lem:Limit}    here in anticipation of handling the upcoming recursive inequality. 
\begin{lem}
	\label{lem:Product} Let $\ilset{x_k}$ be a sequence of scalars such that $ \abs{x_k} \le 1 $ for all $k$. Then for $1\le j \le k$ and $k\ge 1$, we have
	\begin{equation}\label{eq:Product}
	\sum_{i=j}^k \prod_{ l = i+1 }^ k \ilparenthesis{1-x _ l } x_i = 1 - \prod_{i = j}^k \ilparenthesis{1-x_i}. 
	\end{equation}
	We take the tradition that the cumulative product equals one if the starting index is no smaller than the ending index. 
\end{lem}
\begin{proof}[Proof of Lemma~\ref{lem:Product}]
	Note that
	\begin{align*}
	\prod_{l= i+1} ^ k \ilparenthesis{1-x_l} - \prod_{l=i}^k \ilparenthesis{1-x_l}  = \prod_{l=i+1}^k\set{ \ilparenthesis{1-x_l} \bracket{1-(1-x_i)}} = \prod_{l=i+1}^k\bracket{    (1-x_l) x_i     }.
	\end{align*}Thus,
	\begin{align*}
	 \prod_{l=i+1}^k\bracket{    (1-x_l) x_i     } =  \bracket{   \prod_{l=i+1}^k \ilparenthesis{1-x_l} - \prod_{l=j}^k \ilparenthesis{1-x_l}    } - \bracket{    \prod_{l=i}^k\ilparenthesis{1-x_l} - \prod_{l=j}^k \ilparenthesis{1-x_l}   }. 
	\end{align*}
	Summing the above equation over $i$ from $j$ to $k$ on  the r.h.s. collapses to yield equation (\ref{eq:Product}).  
\end{proof}
 
 \begin{lem}
 	\label{lem:Limit} Let $\ilset{x_k}$ be a scalar sequence such that  $0<\inf_k x_k\le \sup_{k} x_k < X<\infty$.  Let $\ilset{z_k}$ be a sequence such that $0\le z_k<1$. Define  $\upnu _{j,k}  \equiv (1-z_j ) \prod_{i=j+1}^k z_j $ for $j<k$. Then
 	\begin{equation}\label{eq:Limit}
 	\limsup_k \sum_{j =1}^k \upnu_{j,k} x_j  \le \limsup_k x_k. 
 	\end{equation}
 \end{lem}
 \begin{proof}[Proof of Lemma~\ref{lem:Limit}]
 Denote $\tilde{X} = \limsup_kx_k$. Then for any $\upvarepsilon>0$, there exists a finite $k_0$ such that $ x_k< \tilde{X} + \upvarepsilon $ for all $k>k_0$.  For such indices, we have
 \begin{equation*}
 \sum_{j =1}^k \upnu_{j,k} x_k < X \sum_{j =1}^{k_0 } \upnu_{j,k} +   \tilde{X } + \upvarepsilon.
 \end{equation*}
 where the inequality follows from the result in Lemma~\ref{lem:Product}. Furthermore, the term $\sum_{j =1}^{k_0 } \upnu_{j,k}$ goes to zero as $k\to\infty$. That is, there exists a finite $k_1>k_0$ such that the term $\sum_{j =1}^{k_0 } \upnu_{j,k}$ remains smaller than $\upvarepsilon/X$ for all $k>k_1$.  
 
 Therefore, for sufficiently large $k$, we have 
\begin{equation*} 
\sum_{j =1}^k \upnu_{j,k} x_k  < \tilde{X} + 2\upvarepsilon. 
\end{equation*}
Given that $\upvarepsilon>0$ is arbitrary, our desired result (\ref{eq:Limit}) holds.  
 \end{proof}

The   main theorems in this section  pertain to  a positive slack variable $\Tobe_k   $ whose allowable  domain   depends on  $\ratio_k $. For brevity, the dependence of $\Tobe_k$'s domain  on $\ratio_k$ will be suppressed wherever no confusion is introduced.  The slack variable $\Tobe_k$  can be effectively viewed  as a {hyper-parameter}, which shall be  picked   based on the smoothness parameter $\LipsPara_k$ and the strong convexity parameter $\convexPara_k$, before  
 selecting the gain   sequence $\gain_k$.  This is natural as   $\ratio_k$ pertains to the curvature information of the loss function $\loss_k\ilparenthesis{\cdot}$ presented in  (\ref{eq:RatioCondition}) when additional smoothness condition \cite[Assumption B.5'' on p. 183]{spall2009feedback} is met. 
Let us   present several lemmas to control the   slack variable $\Tobe_k$     to better  serve the upcoming  proofs on the   tracking performance.

 \begin{lem} \label{lem:determinant1} 
	If   $\ratio_k >1$ (i.e., $\LipsPara_k>\convexPara_k>0$), we have $ \convexPara_k^4\parenthesis{\Tobe  +2}^2-4\LipsPara_k^2\convexPara_k^2\parenthesis{\Tobe+1}\le 0 $ for any:
	\begin{equation}\label{eq:Tobe1}  0<\Tobe \le \Tobe_{k,1} \ilparenthesis{\ratio_k } \equiv   2 \ilparenthesis{\ratio_k^2-1} + 2\ratio_k\sqrt{\ratio_k^2-1}.
	\end{equation} 
\end{lem}

We will rewrite $\Tobe_{k,1}\ilparenthesis{\ratio_k}$ as $\Tobe_{k,1}$   wherever convenient.

\begin{proof}[Proof of  Lemma \ref{lem:determinant1}]
	 Define $
	 h_{k,1}\ilparenthesis{\Tobe  } \equiv \parenthesis{\Tobe+2}^2-4\ratio_k^2\parenthesis{\Tobe+1}= \Tobe^2+4\parenthesis{1-\ratio_k^2}\Tobe+4\parenthesis{1-\ratio_k^2} $, and it is a quadratic function of $\Tobe$.   The determinant\footnote{For a general  quadratic function $ax^2 + bx + c$ of $x$, its determinant is defined to be $\Delta\equiv b^2 - 4 ac$. } of the quadratic function $h_{k,1}\ilparenthesis{\cdot }$   is  $
	 \Delta_{k,1}\equiv 16(1-\ratio_k^2)^2-16(1-\ratio_k^2)=16\ratio_k^2\parenthesis{\ratio_k^2-1}$.  
	 If  $\ratio_k>1$, we have $\Delta_{k,1}>0$. 
	 The two real roots of $h_{k,1}\ilparenthesis{\cdot}$ are $ 2\parenthesis{\ratio_k ^2-1}- 2\ratio_k\sqrt{\ratio_k^2-1}$
	 and $ 2\parenthesis{\ratio_k^2-1}+ 2\ratio_k\sqrt{\ratio_k^2-1}(\equiv \Tobe_{k,1}) $ respectively. Given that the sum of the two real  roots is   $  4\parenthesis{\ratio_k^2-1}>0$, and the product of the two real roots is  $ 4\parenthesis{1-\ratio_k^2}<0 $, we know that the smaller root is negative and the larger root $\Tobe_{k,1}$ is positive.  Therefore, when  $\ratio_k>1$,  $ \convexPara_k^4\parenthesis{\Tobe   +2}^2-4\LipsPara_k^2\convexPara_k^2\parenthesis{\Tobe+1}\le 0$ is nonpositive for any $\Tobe$ satisfying (\ref{eq:Tobe1}). 
\end{proof}

 \begin{lem}\label{lem:determinant2}
	 If  $ 1\le \ratio_k  \le (1+\sqrt{5})/2 $, we have  
	$ \convexPara_k^4\parenthesis{\Tobe+2}^2-4\convexPara_k^2\LipsPara_k\parenthesis{\LipsPara_k-\convexPara_k}\parenthesis{\Tobe+1}>0 $ for any $\Tobe>0$. 
	  If  $\ratio_k > (1+\sqrt{5})/2$, we have
		$ \convexPara_k^4\parenthesis{\Tobe+2}^2-4\convexPara_k^2\LipsPara_k\parenthesis{\LipsPara_k-\convexPara_k}\parenthesis{\Tobe+1}>0 $ for any:
		\begin{equation}\label{eq:Tobe2}
		\Tobe> \Tobe_{k,2} \ilparenthesis{\ratio_k } \equiv     2\ilparenthesis{\ratio_k^2-\ratio_k-1}+2\sqrt{\ratio_k\ilparenthesis{\ratio_k-1}\ilparenthesis{\ratio_k^2-\ratio_k-1}}.
		\end{equation}

\end{lem}

We will rewrite $\Tobe_{k,2}\ilparenthesis{\ratio_k}$ as $\Tobe_{k,2}$   wherever convenient.

\begin{proof}[Proof of Lemma \ref{lem:determinant2}]
 	Define  $
 h_{k,2}\ilparenthesis{\Tobe } \equiv  \parenthesis{\Tobe+2}^2-4\ratio_k\parenthesis{\ratio_k-1}\parenthesis{\Tobe+1}= 
 \Tobe^2+ 4(1+\ratio_k-\ratio_k^2)\Tobe+4\parenthesis{1+\ratio_k-\ratio_k^2}$.
 The determinant of the quadratic function $h_{k,2}\ilparenthesis{\cdot }$  is $
 \Delta_{k,2}=16(1+\ratio_k-\ratio_k^2)^2-16(1+\ratio_k-\ratio_k^2)= 16\ratio_k(\ratio_k-1)(\ratio_k^2-\ratio_k-1)$. 
 
   If $ \ratio_k=1 $, we have $ \Delta_{k,2}=0 $. Here $\Tobe=-2$ is the only possibility  for   $ h_{k,2}\ilparenthesis{\Tobe}=0 $. Then  for any $\Tobe\in\real^+$ we have $ h_{k,2}\ilparenthesis{ \Tobe}>0 $.  
 	
 	  If $  1<\ratio_k<\ilparenthesis{1+\sqrt{5}}/2 $, we have $\Delta_{k,2}<0$. Then the  upward parabola $ h_{k,2}\ilparenthesis{\cdot } $ is above zero for any $\Tobe\in\real^+$.

 	 If $ \ratio_k=(1+\sqrt{5})/2 $, we have $\Delta_{k,2}=0$. Again, $h_{k,2}\ilparenthesis{\Tobe}>0$ for any $\Tobe\in\real^+$.

 	 If $  \ratio_k>\ilparenthesis{1+\sqrt{5}}/2 $, 
 	we have $\Delta_{k,2}>0$.
 	The two roots of $ h_{k,2}\ilparenthesis{\cdot}$ are 
 	      $   2\parenthesis{ {\ratio_k}^2- {\ratio_k}-1}- 2\sqrt{  \Tobe_k\parenthesis{\ratio_k-1}\parenthesis{ {\ratio_k}^2-\ratio_k-1}      } $ and $  2\parenthesis{ {\ratio_k}^2- {\ratio_k}-1}+ 2\sqrt{  \Tobe_k\parenthesis{\ratio_k-1}\parenthesis{ {\ratio_k}^2-\ratio_k-1}      } \equiv \Tobe_{k,2}$.  Since the sum of the two real  roots is $ 4(\ratio_k^2-\ratio_k-1)>0 $, and the product of the two real  roots is $ -4(\ratio_k^2-\ratio_k-1)<0 $, we know that the smaller root is negative and the larger root $\Tobe_{k,2}$ is positive.  Therefore, $h_{k,2}\ilparenthesis{\cdot }$ is positive for any $ \Tobe    $ satisfying (\ref{eq:Tobe2}).  
\end{proof}

 It is straightforward to verify that  $\Tobe_{k,1}>\Tobe_{k,2} $  holds\footnote{	 The relationship $ \Tobe_{k,1}>\Tobe_{k,2} $ on $\ratio_k>\ilparenthesis{1+\sqrt{5}}/2$ follows from $ \diff   \parenthesis{\Tobe_{k,1}-\Tobe_{k,2}}/\diff  \ratio_k>0 $ and that $ \Tobe_{k,1}-\Tobe_{k,2} $ evaluated at $\ratio_k=\ilparenthesis{1+\sqrt{5}}/2$ is approximately $ 7.35>0$. } for any $\ratio_k >\ilparenthesis{1+\sqrt{5}}/2$.

\begin{lem}  
	[Slack Variable $\Tobe_k $ Selection] \label{lem:Tobe} Let us   select $\Tobe_k  $   in the following manner:
	\begin{equation}
	\label{eq:Tobe}
	   		\Tobe_k \in \indicator_{\ilset{\ratio_k=1}} \times \parenthesis{0,\infty } +    \indicator_{\ilset{ 1<\ratio_k\le \ilparenthesis{1+\sqrt{5}}/2 }}  \times   \left( 0,\Tobe_{k,1}  \right]  +  \indicator_{ \ilset{ \ratio_k> \ilparenthesis{1+\sqrt{5}}/2  } }  \times\left(  \Tobe_{k,2}    , \Tobe_{k,1}   \right],  
	\end{equation}
	where $\Tobe_{k,1}$ and $\Tobe_{k,2}$ are defined in   (\ref{eq:Tobe1})     and  (\ref{eq:Tobe2})  respectively. Specifically, 	when  $\ratio_k =1$, let $\Tobe_k $ be any number in    $ \real^+$;  when  $1< \ratio_k \le \ilparenthesis{1+\sqrt{5}}/2$, let $\Tobe_k$ be any number in    $\left(0,\Tobe_{k,1}\ilparenthesis{\ratio_k} \right]$; when  $\ratio_k >\ilparenthesis{1+\sqrt{5}}/2 $, let $\Tobe_k$ be any number in    $    \left(  \Tobe_{k,2} \ilparenthesis{\ratio_k }   , \Tobe_{k,1}  \ilparenthesis{\ratio_k}  \right] $.   After selecting the slack variable $\Tobe_k $ from the   domain corresponding to different values of $\ratio_k$, the non-decaying gain $\gain_k$ will be  selected such that:   	\begin{eqnarray}\label{eq:GainBound1} 
	&\quad \gain_k \LipsPara_k   \in  \indicator_{\ilset{\ratio_k=1}} \times \left[  1,1+\frac{1}{\Tobe_k +1}   \right)    +     \indicator_{\ilset{1<\ratio_k\le (1+\sqrt{5})/2}} \times  \left[  \frac{1}{\Tobe_k+1},  \multiple_{k,+}\ilparenthesis{\Tobe_k   }   \right) \nonumber\\
	&     + \indicator_{\ilset{\ratio_k>(1+\sqrt{5})/2}} \times \ilparenthesis{  \multiple_{k,-}\ilparenthesis{\Tobe_k}, \multiple_{k,+}\ilparenthesis{\Tobe_k } }  
	\end{eqnarray}   
	where the mappings $\multiple_{k,\pm  }\ilparenthesis{\cdot }$  are    defined as: 
	\begin{equation}\label{eq:multiple1}
	\multiple_{k,\pm  }\ilparenthesis{\Tobe}\equiv \frac{\Tobe+2\pm  \sqrt{\Tobe^2+4\ilparenthesis{1+\ratio_k-\ratio_k^2}\Tobe+4\ilparenthesis{1+\ratio_k-\ratio_k^2}}}{2\ilparenthesis{\Tobe+1}}.
	\end{equation}
	Specifically, when   $\ratio_k=1$, let $\gain_k$ be such that    $\ilparenthesis{\Tobe_k+1}^{-1}\le \gain_k\LipsPara_k<1+\ilparenthesis{\Tobe_k +1}^{-1} $; when $ 1<\ratio_k\le (1+\sqrt{5})/2 $, let $\gain_k$ be such that $ \ilparenthesis{\Tobe_k+1}^{-1}\le \gain_k\LipsPara_k<   \multiple_{k,+}\ilparenthesis{\Tobe _k}   $; when $\ratio_k>(1+\sqrt{5})/2$, let $\gain_k$ be such that $\multiple_{k,-}\ilparenthesis{\Tobe_k}< \gain_k\LipsPara_k<\multiple_{k, +}\ilparenthesis{\Tobe_k}$. 
	
	If the slack variable $\Tobe_k $ is selected according to (\ref{eq:Tobe}) and then  the gain sequence $\gain_k$ is selected according to (\ref{eq:GainBound1}),   
	  the  following hold:
	\begin{numcases}{}
	\firstConst_k \equiv  \frac{\LipsPara_k}{\convexPara_k}  + \gain_k\convexPara_k\bracket{\ilparenthesis{\Tobe_k+1}\ilparenthesis{\gain_k\LipsPara_k-1}-1} \in\left[0,1\right), &	\label{eq:firstConst}\\
	\secondConst_k\equiv 	\frac{\gain_k\ilbracket{\gain_k\LipsPara_k\parenthesis{\Tobe_k +1}-1}}{\Tobe_k \convexPara_k} \ge 0 . &	\label{eq:secondConst} 
	\end{numcases}  
	  	
\end{lem}

\begin{proof}[Proof of Lemma~\ref{lem:Tobe}]

	 { First show that  $\firstConst_k\ge 0$ in  (\ref{eq:firstConst}) holds for any $\gain_k$ satisfying (\ref{eq:GainBound1}).} We     need to show that
	 \begin{eqnarray}\label{eq:h1}
 \tilde{h}_{k,1}(\gain_k;\Tobe_k)\equiv \convexPara_k^2\LipsPara_k\parenthesis{\Tobe_k+1}\gain_k^2-\convexPara_k^2\parenthesis{\Tobe_k+2}\gain_k+\LipsPara_k\ge 0 
	 \end{eqnarray}    
	 holds for all $k$.  
	 	The semicolon in (\ref{eq:h1}) is to emphasize that the selection of $\Tobe_k$ takes place before the selection of $\gain_k$. After fixing/picking the value of  $\Tobe_k$, $\tilde{h}_{k,1}$  in (\ref{eq:h1}) is simply a function of $\gain_k$.  
	 
	  When   $\ratio_k=1$, we pick  $\Tobe_k>0$ per (\ref{eq:Tobe}).  Then $   \tilde{h}_{k,1}(\gain_k;\Tobe_k)  \ge 0 $ for  {any  $ \gain_k $ such that  $ 0<\gain_k\LipsPara_k \le  \parenthesis{\Tobe_k+1}^{-1} $ or $ \gain_k\LipsPara_k\ge 1 $.   }
	 When   $ \ratio_k>1 $, we pick $\Tobe_k$ from  $\left(0, \Tobe_{k,1}\ilparenthesis{\ratio_k}\right]$ per (\ref{eq:Tobe}).  	 From   Lemma~\ref{lem:determinant1}  that $ \tilde{h}_{k,1}\ilparenthesis{\cdot} $ has a nonpositive determinant     as long as  (\ref{eq:Tobe1}) holds.  In this case,  the upward parabola $\tilde{h}_{k,1}\ilparenthesis{\gain_k}\ge 0$ for  any $\gain_k \in\real^+$.  In short, (\ref{eq:GainBound1}) is a sufficient (but not necessary) condition for $\firstConst_k\ge 0$. 
	 
  {Next,  we  show that $\firstConst_k<1$ in  (\ref{eq:firstConst}) holds for any  $\gain_k$ satisfying (\ref{eq:GainBound1}).} We   need to show: 
		\begin{eqnarray}\label{eq:h2}
		&\quad \quad& 	      \tilde{h}_{k,2}\parenthesis{\gain_k;\Tobe_k}\equiv  \convexPara_k^2\LipsPara_k\parenthesis{\Tobe_k+1}\gain_k^2-\convexPara_k^2\parenthesis{\Tobe_k+2}\gain_k+\parenthesis{\LipsPara_k-\convexPara_k}<0.\quad\quad 
		\end{eqnarray}

		When $\ratio_k=1$,  we pick  $\Tobe_k>0$ per (\ref{eq:Tobe}).   Then $\tilde{h}_{k,2}\parenthesis{\gain_k}<0$ when  {$ 0<\gain_k\LipsPara_k<\parenthesis{\Tobe_k+2}\parenthesis{\Tobe_k+1}^{-1} $}.

		When  $ 1< \ratio_k\le \ilparenthesis{1+\sqrt{5}}/2$, we pick  $\Tobe_k\in \left(0, \Tobe_{k,1}\ilparenthesis{\ratio_k}\right]$ per (\ref{eq:Tobe}).  Lemma~\ref{lem:determinant2} tells that   $ \parenthesis{\Tobe_k+2}^2-4\ratio_k\parenthesis{\ratio_k-1}\parenthesis{\Tobe_k+1}>0 $  for any $\Tobe_k>0 $. Therefore, we have 
			\begin{equation}\label{eq:GainBound2}
			\tilde{h}_{k,2}\parenthesis{\gain_k }<0,\quad \text{when } 0\le \frac{\multiple_{k,-}}{\LipsPara_k} <\gain_k<\frac{\multiple_{k,+}}{\LipsPara_k}<\infty. 
			\end{equation}
			where $\multiple_{k,\pm }\ilparenthesis{\cdot}$ are  defined  in  (\ref{eq:multiple1}). For the time being, suppress the parameter $\Tobe_k$ in the mappings $\multiple_{k,\pm}\ilparenthesis{\cdot}$. 	Notice that both  $\multiple_{k,-}+\multiple_{k,+}= \parenthesis{\Tobe_k +2}/\parenthesis{\Tobe_k+1} $ and $\multiple_{k,-}\multiple_{k,+}= \ratio_k\parenthesis{\ratio_k-1}/\parenthesis{\Tobe_k+1} $ are strictly positive except when $\ratio_k =1$. Hence, both $\multiple_{k,-}$ and $\multiple_{k,+}$ are strictly positive,  except that $\multiple_{-}=0$ when $\ratio_k=1$.

When  $ \ratio_k\ge \ilparenthesis{1+\sqrt{5}}/2 $, we pick  $\Tobe_k\in \left(\Tobe_{k,2}\ilparenthesis{\ratio_k}, \Tobe_{k,1}\ilparenthesis{\ratio_k}\right]$ per (\ref{eq:Tobe}).    Lemma~\ref{lem:determinant2} tells that  $ \parenthesis{\Tobe_k+2}^2-4\ratio_k\parenthesis{\ratio_k-1}\parenthesis{\Tobe_k+1}>0 $  for any $\Tobe_k$ satisfying (\ref{eq:Tobe2}). Again, we have (\ref{eq:GainBound2}).

			In short, we have $\firstConst_k<1$  for any $\gain_k$ satisfying  (\ref{eq:GainBound1}).

			  {Last, $v_k> 0$ is immediate.} This follows from $ \gain_k\LipsPara_k  \ge  \ilparenthesis{\Tobe_k+1}^{-1} $.

	Note that  $ \multiple_{k, \pm} $   is well-defined for any $ \ratio_k \ge 1 $ and the corresponding selection of $\Tobe_k$ in (\ref{eq:Tobe}). 
	Given that: \begin{equation} \label{eq:multiple2}
	0\le  \multiple_{k,-}\ilparenthesis{\Tobe_k }<   \frac{\Tobe_k/2+1}{{\Tobe_k+1}}   <  \multiple_{k,+}\ilparenthesis{\Tobe_k} \le   \frac{\Tobe_k+2}{\Tobe_k+1},
	\end{equation}
	where  the first and the last inequalities in (\ref{eq:multiple2}) become  strict equalities  only when $\ratio_k=1$, we can combine the  aforementioned scenarios  for gain-selection into a   consistent expression in (\ref{eq:GainBound1}), such that (\ref{eq:firstConst}) and (\ref{eq:secondConst}) always hold  for proper slack variable selection (\ref{eq:Tobe}) and gain selection (\ref{eq:GainBound1}). 
\end{proof}

 Lemma \ref{lem:Tobe}  discusses both the slack variable $\Tobe_k $ and the gain sequence $\gain_k$.   To facilitate reading and implementation, we summarize the implementation of (\ref{eq:basicSA}) as applied in tracking time variability in Algorithm \ref{algo:basicSA}. 
  
 \begin{algorithm}[!htbp]
 	\caption{Basic SA Algorithm With Non-Diminishing Gain Using \emph{One}-Function-Measurement   or \emph{One}-Stochastic-Gradient-Measurement Per Iteration} 
 	\begin{algorithmic}[1]  
 		\renewcommand{\algorithmicrequire}{\textbf{Input:}}
 		\renewcommand{\algorithmicensure}{\textbf{Output:}}
 		\Require $\hbtheta_0$, the best approximation available   at hand to estimate $\bvartheta_0$.  
 		\For{$k\ge 0$ or $k\in\set{0, 1,\cdots,K}$}  
 		\Require     $\convexPara_k $ per A.\ref{assume:StronglyConvex} and $\LipsPara_k$  per A.\ref{assume:Lsmooth}. 
 		\State \textbf{collect} instant feedback   immediately after each decision. 
 		\If{the feedback is in the form of (\ref{eq:y})}
 		\State \textbf{generate} a   $p$-dimensional  random vector $\bDelta_k$ satisfying conditions in \cite[Sect. 7.3]{spall2005introduction}\remove{ mean-zero symmetrically-distributed random vector $\bDelta_k$ such that each component has finite inverse moments}.   \Comment{Each component of  $\bDelta_k$ may  be   i.i.d.  Rademacher distributed.}
 		\State \textbf{pick} a  small number   $c_k$. \Comment{ $c_k$ may be the  desired minimal component-wise change of  $\hbtheta_k $.}
 		\State \textbf{compute} $ \hbg_k^{\mathrm{SP1}}\ilparenthesis{\hbtheta_k}  $ using (\ref{eq:g1SP}). 
 		\ElsIf{ the feedback is in the form of (\ref{eq:Y})}
 		\State \textbf{set} $\hbg_k^{\mathrm{RM}}\ilparenthesis{\hbtheta_k}$ as in (\ref{eq:Ystationary}). 
 		\EndIf
 		\If{$ \ratio_k   = 1 $}
 		\State \textbf{pick} any $\Tobe_k\in\real^+$. \remove{We may pick $\Tobe_k=1$.} \label{line:r1-1}
 		\State
 		\textbf{pick} $\gain_k$ such that   $ 1\le \gain_k\LipsPara_k< 1+\ilparenthesis{\Tobe_k+1}^{-1} $. \label{line:r1-2} \remove{We may pick $\gain_k = 1/\LipsPara_k$.}
 		\ElsIf{$1< \ratio_k \le (1+\sqrt{5})/2 $ }
 		\State \textbf{pick} any $\Tobe_k \in\left(0,\Tobe_{k,1}\ilparenthesis{\ratio_k} \right]$ per   (\ref{eq:Tobe1}). \label{line:r2-1}	\remove{We may pick $\Tobe _k = \Tobe_{1,+}\ilparenthesis{\ratio_k}$.}	 	\State \textbf{pick} $\gain_k $ such that $ \gain_k\LipsPara_k \in  \left[ \ilparenthesis{\Tobe_k+1}^{-1},\multiple_{k,+}\ilparenthesis{\Tobe _k  }  \right) $ per (\ref{eq:multiple1}). \label{line:r2-2}  \remove{We may pick $\gain_k=\ilparenthesis{\Tobe_k+1}^{-1}$. }
 		\ElsIf{$\ratio_k>(1+\sqrt{5})/2$}
 		\State \textbf{pick} any $\Tobe\in\left( \Tobe_{k,2}\ilparenthesis{\ratio_k } ,\Tobe_{k,1} \ilparenthesis{\ratio_k } \right]$ per (\ref{eq:Tobe1}) and (\ref{eq:Tobe2}). \label{line:r3-1} \remove{We may pick $\Tobe _k  = \Tobe_{1,+}\ilparenthesis{\ratio_k}$.}
 		 
 			\State \textbf{pick} $\gain_k $ such that $ \gain_k\LipsPara_k \in  \left( \multiple_{k,-}\ilparenthesis{\Tobe_k},\multiple_{k,+}\ilparenthesis{\Tobe _k  }  \right) $ per (\ref{eq:multiple1}). \label{line:r3-2}  \remove{We may pick $\gain_k=\ilparenthesis{\Tobe_k+1}^{-1}$. }
 	
 		\EndIf
 		\State	\textbf{update} $ \hbtheta_{k} $ using (\ref{eq:basicSA}).
 		\Ensure $\hbtheta_k$.
 		\EndFor 
 	\end{algorithmic}	 		\label{algo:basicSA}	
 \end{algorithm}

 The motivation behind the restriction on the non-diminishing gain $\gain_k $ in Lemma~\ref{lem:Tobe} is to ensure the tracking capability, with a manifestation of the shrinking recurrence coefficient in the    inequality (\ref{eq:ErrorPropagationLemma}) below.

\begin{lem}[Error propagation] \label{lem:ErrorPropagationLemma1} 
	Assume  that A.\ref{assume:ErrorWithBoundedSecondMoment}, A.\ref{assume:StronglyConvex}, A.\ref{assume:Lsmooth}, and A.\ref{assume:BoundedVariation} hold.   Let $\Tobe_k  $ be a slack variable selected according to (\ref{eq:Tobe}), the domain of which depends on different values of $\ratio_k$.  
	Let $\hbtheta_k$ be the sequence generated by (\ref{eq:basicSA}) with non-decaying gain satisfying    (\ref{eq:GainBound1}).  We have: 
	\begin{equation}\label{eq:ErrorPropagationLemma}
	\norm{ \hbtheta_{k+1} - \bvartheta_{k} } ^ 2 \le \firstConst _ k \norm{\hbtheta_k - \bvartheta_k}^2 + \secondConst _ k \norm{ \be_k \ilparenthesis{\hbtheta_k} }^2,
	\end{equation}  where the  constants $\firstConst_k$ and $\secondConst_k$ were   defined in (\ref{eq:firstConst}) and   (\ref{eq:secondConst}). 
	Both coefficients are  deterministic and can be fully determined after selecting $\Tobe_k$ and $\gain_k$ per Algorithm~\ref{algo:basicSA}. 	  Furthermore, the coefficient $\firstConst_k$ 
	is guaranteed to lie  within      $ \left[0,1\right) $, and the coefficient $\secondConst_k$   is positive.
\end{lem}

\begin{proof}[Proof of Lemma \ref{lem:ErrorPropagationLemma1}]

	   The  estimates generated by  (\ref{eq:basicSA})  satisfy the following:
	\begin{eqnarray}\label{eq:Prop1-1}
  \loss_k\ilparenthesis{\hbtheta_{k+1}}
	&=& \loss_k\ilparenthesis{\hbtheta_k - \gain_k \hbg_k\ilparenthesis{\hbtheta_k }}\nonumber \\
	&\le& \loss_k \ilparenthesis{\hbtheta_k} - \gain_k \ilbracket{\hbg_k\ilparenthesis{\hbtheta_k}}^\transpose\bg_k\ilparenthesis{\hbtheta_k} +  \frac{\gain _ k ^ 2 \LipsPara_k}{2} \norm{\hbg_k\ilparenthesis{\hbtheta_k}}^2\nonumber\\
	&=& \loss_k\ilparenthesis{\hbtheta_k} - \gain_k \ilbracket{\bg_k\ilparenthesis{\hbtheta_k}+ \be_k\ilparenthesis{\hbtheta_k}}^\transpose \bg_k\ilparenthesis{\hbtheta_k} + \frac{\gain_k^2\LipsPara_k}{2}\norm{\bg_k\ilparenthesis{\hbtheta_k} + \be_k\ilparenthesis{\hbtheta_k}}^2\nonumber\\
	&=& \loss_k\ilparenthesis{\hbtheta_k} + \gain_k \parenthesis{ {\gain_k\LipsPara_k}/{2}-1} \norm{\bg_k\ilparenthesis{\hbtheta_k}}^2 + \gain_k\ilparenthesis{\gain_k\LipsPara_k -1} \ilbracket{\bg_k\ilparenthesis{\hbtheta_k}}^\transpose \be_k\ilparenthesis{\hbtheta_k}  \nonumber\\
	&\quad&  + \frac{\gain_k^2\LipsPara_k}{2}\norm{\be_k\ilparenthesis{\hbtheta_k}}^2\nonumber\\
	&\le &   \loss_k\ilparenthesis{\hbtheta_k} + \gain_k \parenthesis{ \frac{\gain_k\LipsPara_k}{2}-1} \norm{\bg_k\ilparenthesis{\hbtheta_k}}^2 +    \frac{\gain_k^2\LipsPara_k}{2}\norm{\be_k\ilparenthesis{\hbtheta_k}}^2 \nonumber\\ 
	&\quad &  \frac{\gain_k\ilparenthesis{\gain_k\LipsPara_k -1}}{2 } \bracket{   \Tobe_k  \norm{\bg_k\ilparenthesis{\hbtheta_{k}}} ^2 + \Tobe_k^{-1} \norm{\be_k\ilparenthesis{\hbtheta_k}}^2    } \nonumber\\
	&=& \loss_k\ilparenthesis{\hbtheta_k} +  \frac{\gain_k}{2}\bracket{\ilparenthesis{\gain_k\LipsPara_k - 1 }\ilparenthesis{1+\Tobe_k}-1} \norm{\bg_k\ilparenthesis{\hbtheta_k}}^2 \nonumber\\
	&\quad & + \frac{\gain_k}{2\Tobe_k} \ilbracket{\gain_k\LipsPara_k \ilparenthesis{1+\Tobe_k    } -1 } \norm{\be_k\ilparenthesis{\hbtheta_k}}^2,
	\end{eqnarray}
	where the first inequality follows from (\ref{eq:LipsIneq}),     the second equality follows from  the fact that $\hbg_k\ilparenthesis{\hbtheta_k}$ on the r.h.s. can be decomposed as  (\ref{eq:gGeneral}), and the second inequality follows 
	from the fact that   $ \pm  2\bx^\transpose\by\le {\norm{\bx}^2+\norm{\by}^2}  $ for any $\bx,\by\in\real^p$. Note that   both $2\bx^\transpose\by\le {\norm{\bx}^2+\norm{\by}^2}$ and $-2\bx^\transpose\by\le {\norm{\bx}^2+\norm{\by}^2}$ hold,  so  the sign of the   coefficient $ \gain_k\ilparenthesis{\gain_k\LipsPara_k-1}/2 $ does not affect the validity of the inequality. 
	 The positive slack variable $\Tobe_k$, which may be deemed as hyper-parameter for selecting selection $\gain_k$, is picked according to Lemma~\ref{lem:Tobe} and the corresponding gain selection is summarized in Algorithm~\ref{algo:basicSA}.  
	
	Subtracting $\loss_k\ilparenthesis{\bvartheta_k}$ from both sides of (\ref{eq:Prop1-1}) yields:
	\begin{eqnarray}\label{eq:Prop1-2} 
&\quad & 	\loss_k\ilparenthesis{\hbtheta_{k+1}}  - \loss_k \ilparenthesis{\bvartheta_k} \le 
	\loss_k\ilparenthesis{\hbtheta_k}  - \loss_k\ilparenthesis{\bvartheta_k} +  \frac{\gain_k}{2}\bracket{\ilparenthesis{\gain_k\LipsPara_k - 1 }\ilparenthesis{1+\Tobe_k}-1} \norm{\bg_k\ilparenthesis{\hbtheta_k}}^2 \nonumber\\
	&\quad & \quad\quad\quad\quad\quad\quad\quad\quad\quad+ \frac{\gain_k}{2\Tobe_k} \ilbracket{\gain_k\LipsPara_k \ilparenthesis{1+\Tobe_k    } -1 } \norm{\be_k\ilparenthesis{\hbtheta_k}}^2.
	\end{eqnarray}
	Applying   both  inequalities in  (\ref{eq:ineq1}) to  both sides of (\ref{eq:Prop1-2}) gives: 
	\begin{eqnarray}
  \convexPara_k \norm{\hbtheta_{k+1}-\bvartheta_k}^2/2 
	&\le & \loss_k\ilparenthesis{\hbtheta_{k+1}} - \loss_k\ilparenthesis{\bvartheta_k} \nonumber\\
	&\le & 
	\loss_k\ilparenthesis{\hbtheta_k}  - \loss_k\ilparenthesis{\bvartheta_k} +  \frac{\gain_k}{2}\bracket{\ilparenthesis{\gain_k\LipsPara_k - 1 }\ilparenthesis{1+\Tobe_k}-1} \norm{\bg_k\ilparenthesis{\hbtheta_k}}^2 \nonumber\\
	&\quad & + \frac{\gain_k}{2\Tobe_k} \ilbracket{\gain_k\LipsPara_k \ilparenthesis{1+\Tobe_k    } -1 } \norm{\be_k\ilparenthesis{\hbtheta_k}}^2  \nonumber\\ 	
	&\le & \LipsPara_k\norm{\hbtheta_k-\bvartheta_k}^2  /2 + \frac{\gain_k}{2}\bracket{\ilparenthesis{\gain_k\LipsPara_k - 1 }\ilparenthesis{1+\Tobe_k}-1} \norm{\bg_k\ilparenthesis{\hbtheta_k}}^2 \nonumber\\
	&\quad & +  \frac{\gain_k}{2\Tobe_k} \ilbracket{\gain_k\LipsPara_k \ilparenthesis{1+\Tobe_k    } -1 } \norm{\be_k\ilparenthesis{\hbtheta_k}}^2,     \nonumber 
	\end{eqnarray}
	which implies	
	\begin{eqnarray}\label{eq:Prop1-3} 
\norm{\hbtheta_{k+1}-\bvartheta_k}^2	& \le& \frac{\LipsPara_k}{\convexPara_k }\norm{\hbtheta_k-\bvartheta_k}^2+ \frac{\gain_k }{\convexPara_k}\bracket{\ilparenthesis{\Tobe_k +1}\ilparenthesis{\gain_k\LipsPara_k-1}-1 }\norm{\bg_k\ilparenthesis{\hbtheta_k}}^2\nonumber\\
	&\quad& +\frac{\gain_k}{\Tobe_k\convexPara_k}\ilbracket{\gain_k\LipsPara_k\parenthesis{\Tobe_k +1} -1}\norm{\be_k\ilparenthesis{\hbtheta_k}}^2. 
	\end{eqnarray} 
 (\ref{eq:multiple2}) tells that  $ {\gain_k}{\convexPara_k}^{-1}\bracket{ \ilparenthesis{\Tobe_k +1}\ilparenthesis{\gain_k\LipsPara_k-1}-1 }
		$,  the  coefficient of $\norm{\bg_k\ilparenthesis{\hbtheta_k}}^2$ on the r.h.s. of (\ref{eq:Prop1-3}), is negative as long as  both           the   positive slack variable  $\Tobe_ k$   and the     non-diminishing step-size $\gain_k $   are  selected per (\ref{eq:Tobe}) and (\ref{eq:GainBound1}) in  Lemma~\ref{lem:Tobe}. 
	Hence, we can apply     (\ref{eq:ineq1}) and  (\ref{eq:ineq2})   to further manipulate (\ref{eq:Prop1-3}): 
	\begin{eqnarray}\label{eq:Prop1-4}
\norm{\hbtheta_{k+1}-\bvartheta_k}^2 
	&\le &  \frac{\LipsPara_k}{\convexPara_k}\norm{\hbtheta_k-\bvartheta_k}^2+  {\gain_k}{\convexPara_k}\bracket{\ilparenthesis{\Tobe_k +1}\ilparenthesis{\gain_k\LipsPara_k-1}-1 }\norm{ \hbtheta_k-\bvartheta_k  }^2\nonumber\\
	&\quad & +\frac{\gain_k}{\Tobe_ k\convexPara_k}\ilbracket{\gain_k\LipsPara_k\parenthesis{\Tobe_k +1}-1}\norm{\be_k\ilparenthesis{\hbtheta_k}}^2\nonumber\\
	&=&\frac{\LipsPara_k+\gain_k\convexPara_k^2  \bracket{\parenthesis{\Tobe_k+1}\parenthesis{\gain_k\LipsPara_k-1}-1} }{\convexPara_k} \norm{\hbtheta_k-\bvartheta_k}^2 \nonumber\\
	&\quad & +\frac{\gain_k\ilbracket{\gain_k\LipsPara_k\parenthesis{\Tobe_k +1}-1}}{\Tobe_k \convexPara_k}\norm{\be_k\ilparenthesis{\hbtheta_k}}^2\nonumber\\
	&\equiv & \firstConst_k \norm{\hbtheta_k-\bvartheta_k}^2 + \secondConst_k \norm{\be_k\ilparenthesis{\hbtheta_k}}^2.
	\end{eqnarray}
Note that both $\firstConst_k$ and $\secondConst_k$ are deterministic   once $\gain_k$ is selected following    Lemma~\ref{lem:Tobe}. Algorithm~\ref{algo:basicSA} summarizes the gain selection procedure.  Furthermore, the coefficient $\firstConst_k$    on the r.h.s. of (\ref{eq:Prop1-4}) is guaranteed to lie within $ \left[0,1\right) $   as long as    (\ref{eq:Tobe}) and  (\ref{eq:GainBound1})      hold per Lemma~\ref{lem:Tobe}. The coefficient $\secondConst_k$ on the r.h.s. of (\ref{eq:Prop1-4}) is guaranteed to be positive  per  Lemma~\ref{lem:Tobe}.
\end{proof}

Admittedly, Lemma \ref{lem:ErrorPropagationLemma1} is presented in a way that the slack variable $\Tobe_k $ also appears in the required conditions. To  present our theorem with conditions imposed solely on the gain sequence $\gain_k$, the only hyper-parameter in  (\ref{eq:basicSA}),     Lemma~\ref{lem:main1} is presented below. 
 
	\begin{lem}  \label{lem:main1}

		Let us use  the   recursion (\ref{eq:basicSA}) where   the non-diminishing gain  
		is such that $\gain_k\LipsPara_k$ lies within the shaded (green)   region (excluding the red boundaries) in Figure  \ref{fig:mainregion}.    	The region  in Figure \ref{fig:mainregion} is confined  by the lower   curve defined  to be $$ \indicator_{\ilset{1<\ratio_k\le  (1+\sqrt{5})/2 }} \times  \frac{1}{1+\Tobe_{k,1}\ilparenthesis{\ratio_k}}  + \indicator_{\ilset{ \ratio_k>(1+\sqrt{5})/2 }} \times  \multiple_{k,-}\ilparenthesis{\Tobe_{k,1}\ilparenthesis{\ratio_k} } $$ as a function of $\ratio_k$,  
and the upper curve defined to be $$\indicator_{\ilset{  1<\ratio_k\le (1+\sqrt{5})/2   }} \times  \multiple_{k,+}\ilparenthesis{0} + \indicator_{\ilset{ \ratio_k > (1+\sqrt{5})/2 }} \times  \multiple_{k,+}\ilparenthesis{\Tobe_{k,1}\ilparenthesis{\ratio_k}} $$ as a function of $\ratio_k$. 
	\begin{figure}[!htbp]
		\centering  
		\includegraphics[scale=.3]{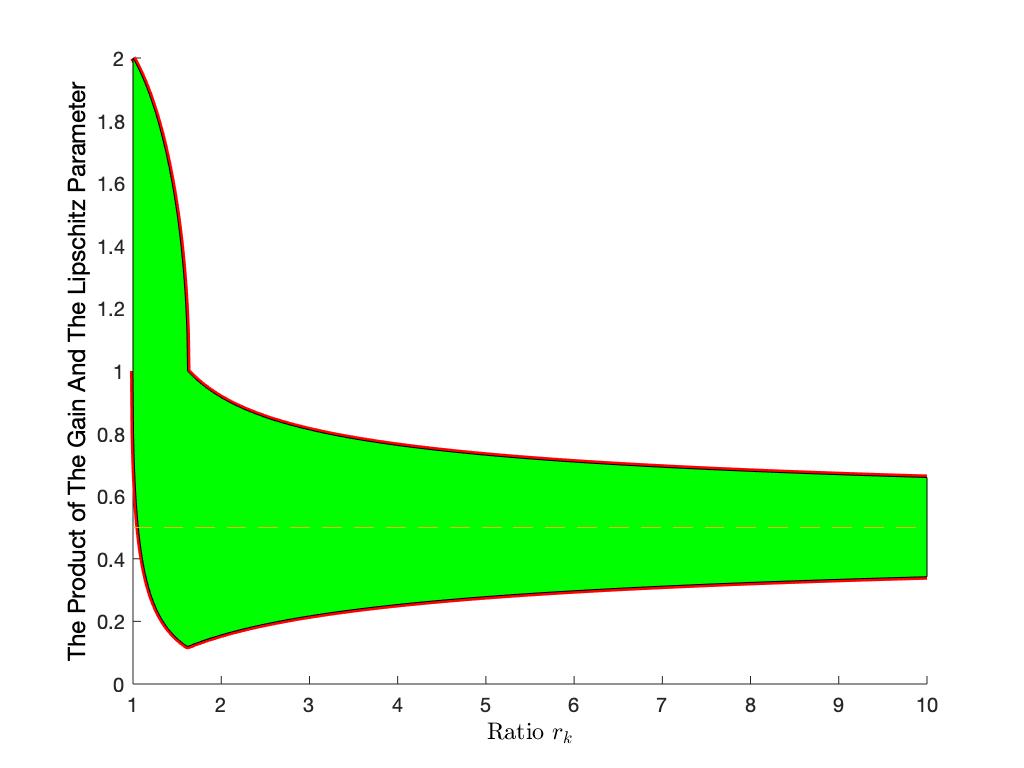}
		\caption{The Allowable Region For The Product of The Gain and The Lipschitz Parameter $\gain_k\LipsPara_k $ Under  Different Values of The Ratio $\ratio_k$. 	\label{fig:mainregion} }
	\end{figure}
		 		 	When    Assumptions    A.\ref{assume:StronglyConvex} and A.\ref{assume:Lsmooth} hold,   
and  $\gain_k$ falls within the green region indicated in Figure~\ref{fig:mainregion}, 	 we have  that (\ref{eq:ErrorPropagationLemma}) holds, 
	where $\firstConst_k$ is guaranteed to lie within $  \left[0,1\right)  $ and $\secondConst_k$ is nonnegative.

\end{lem}

\begin{proof}[Proof of Lemma \ref{lem:main1}] For $\Tobe_{k,1}\ilparenthesis{\ratio_k}$ defined in (\ref{eq:Tobe1}), we know that
$ \Tobe_{k,1} \ilparenthesis{\ratio_k} $ is  strictly increasing w.r.t.  $\ratio_k$ as $   {\partial    \Tobe_{k,1} }/{\partial \ratio_k }    $ is strictly positive on $\ratio_k>1$. Also, $\Tobe_{k,2}$ defined in (\ref{eq:Tobe2}) is also a strictly increasing function of $\ratio_k$ because   $  {\partial \Tobe_{k,2}  }/{\partial \ratio_k}    $ is strictly positive on $ \ratio_k>\ilparenthesis{1+\sqrt{5}}/2 $. 

When   $\ratio_k=1$,      lines \ref{line:r1-1}--\ref{line:r1-2} in Algorithm~\ref{algo:basicSA} immediately tell $ 1\le \gain_k\LipsPara_k<2 $.

When 
$1<\ratio_k\le \ilparenthesis{1+\sqrt{5}}/2$,   lines \ref{line:r2-1} and \ref{line:r2-2} instruct  selecting  any $\gain_k $ such that $ \ilparenthesis{\Tobe_k+1}^{-1}\le \gain_k\LipsPara_k <\multiple_{k,+}\ilparenthesis{\Tobe_k } $ for $ 0<\Tobe_k\le \Tobe_{k,1}\ilparenthesis{\ratio_k}$. 
In particular, $(1+\Tobe_{k,1}\ilparenthesis{\ratio_k})^{-1} <\gain_k\LipsPara_k< \multiple_{k,+}\parenthesis{ 0 }$, where $\multiple_{k,+}\parenthesis{0}=1+\sqrt{1+\ratio_k-\ratio_k^2}$.

When  $\ratio_k>\ilparenthesis{1+\sqrt{5}}/2 $,   lines \ref{line:r3-1} and \ref{line:r3-2} allow for  selecting  any $\gain_k $ such that $  \multiple_{k,-}\ilparenthesis{\Tobe_k} <\gain_k\LipsPara_k <\multiple_{k,+}\ilparenthesis{\Tobe_k } $ for $ \Tobe_{k,2}\ilparenthesis{\ratio_k}<\Tobe_k\le \Tobe_{k,1}\ilparenthesis{\ratio_k}$. In particular,     $  \multiple_{k,-}\parenthesis{\Tobe_{k,1}\parenthesis{\ratio_k }} < \gain_k\LipsPara_k <  \multiple_{k,+}\parenthesis{\Tobe_{k,1}\parenthesis{\ratio_k }}  $. 
\end{proof}

\subsection{A Priori    Error Bound}
\label{subsect:unconditional}

For time-varying systems,   it is unrealistic to expect the convergence results  as in classic SA settings; i.e., for   $\norm{\hbtheta_k-\bvartheta_k}$ to be arbitrarily close to zero in a certain  statistical sense. 
Our first main  result pertains to the error propagation in terms of  MAD (see Theorem \ref{thm:main1}).

	\begin{thm}
	[MAD bound     under bounded-drift assumption] \label{thm:main1}   	Assume      A.\ref{assume:ErrorWithBoundedSecondMoment},   A.\ref{assume:StronglyConvex}, A.\ref{assume:Lsmooth},   and  A.\ref{assume:BoundedVariation}. Using  recursion (\ref{eq:basicSA}) with  non-decaying gain satisfying the region specified in Lemma~\ref{lem:main1}, we have  the following recurrence on the unconditional MAD:	\begin{equation}\label{eq:PropagationLemma2}
	\E\norm{\hbtheta_{k+1}-\bvartheta_{k+1}}\le \sqrt{\firstConst_k}\E\norm{\hbtheta_k-\bvartheta_k} + \noiseBound_k\sqrt{ {\secondConst_k }}+\driftBound_k, \quad   k \in \natural.
	\end{equation}
	Furthermore, we have an asymptotic bound
	\begin{equation}\label{eq:PropagationLemma3}
		\limsup_{k\to\infty} \E\norm{\hbtheta_k-\bvartheta_k}  \le \limsup_k   \frac{\noiseBound_k   \sqrt{  {\secondConst_k }}+\driftBound_k}{1-\sqrt{\firstConst_k}}. 
	\end{equation} 
\end{thm}

\begin{proof}[Proof of Theorem~\ref{thm:main1}]  Note that $\sqrt{x^2+y^2}<x+y$ for $x,y\in\real^+$, then   from (\ref{eq:ErrorPropagationLemma})   we have the following: 
	\begin{equation}\label{eq:Prop1-5}
	\norm{\hbtheta_{k+1}-\bvartheta_k}\le \sqrt{\firstConst_k }\norm{\hbtheta_k-\bvartheta_k}+\sqrt{  { \secondConst_k }}\norm{\be_k\ilparenthesis{\hbtheta_k}},
	\end{equation} where $ \firstConst_k\in\left[0,1\right) $ and $\secondConst_k\ge 0$. Note that both $\firstConst_k $ and $\secondConst_k $ in Lemma~\ref{lem:ErrorPropagationLemma1}, though flexible, are deterministic.   	By Jensen's inequality,    A.\ref {assume:ErrorWithBoundedSecondMoment} implies that  $ \E \norm{\be_k\ilparenthesis{\hbtheta_k}}\le \noiseBound_k$ for all $\btheta\in\real^p$ and  $k\in\integer^+$.  Similarly,      A.\ref{assume:BoundedVariation} implies that     $ \E \norm{\bvartheta_{k+1}-\bvartheta_k}\le \driftBound_k $ for all $k$.  Taking the full  expectation over (\ref{eq:Prop1-5}) and   invoking   A.\ref{assume:ErrorWithBoundedSecondMoment}, we have:
	\begin{equation*}
	\E\norm{\hbtheta_{k+1}-\bvartheta_k}\le \sqrt{\firstConst_k}\E \norm{\hbtheta_k-\bvartheta_k } + \noiseBound _k\sqrt{  {\secondConst_k }}. 
	\end{equation*}
	Then  (\ref{eq:PropagationLemma2}) directly follows from triangle inequality and     A.\ref{assume:BoundedVariation} as:
	\begin{eqnarray}\label{eq:Prop1-6}
	\E\norm{\hbtheta_{k+1}-\bvartheta_{k+1}}& \le &  \E  \norm{\hbtheta_{k+1}-\bvartheta_k } + \E\norm{\bvartheta_{k+1}-\bvartheta_k}\nonumber\\
	&\le & \sqrt{\firstConst_k}\E \norm{\hbtheta_k-\bvartheta_k} + \noiseBound_k  \sqrt{  {\secondConst_k }}+\driftBound_k.
	\end{eqnarray}
	
	Define   $\upnu_{j,k} \equiv 
	\ilparenthesis{1-\sqrt{\firstConst_j}}    \prod_{i=j+1}^{k} \sqrt{\firstConst_i}$
	 for $j<k$. Then after iterating inequality (\ref{eq:Prop1-6}) back to the starting time index $0$, we have
	\begin{equation}
	\E\norm{\hbtheta_{k}  - \bvartheta_{k}} \le \parenthesis{\prod_{ j=0 }^ {k-1} \sqrt{\firstConst_i}} \E\norm{\hbtheta_0-\bvartheta_0 } + \sum_{j = 0 }^{k-1}  { \frac{  \upnu_{j,k-1}  \parenthesis{\noiseBound_j  \sqrt{  {\secondConst_j }}+\driftBound_j  }}{     1-\sqrt{\firstConst_{j}}  }},\quad k\ge 1. 
	\end{equation} Since the  $\firstConst_k$ are bounded within $\left[0,1\right)$, the leading product $\prod_{ j=0 }^ {k-1} \sqrt{\firstConst_i}$ goes to zero as $k\to \infty$. According to Lemma~\ref{lem:Product}, we know that $\sum_{j = 0 }^{k-1} \upnu_{j,k-1}   = 1- \prod_{j=0}^{k-1} \sqrt{\firstConst_j}  $, which goes to $1$ as $k\to\infty$. 
	Now we can apply Lemma~\ref{lem:Limit}  to conclude that  the asymptotic bound (\ref{eq:PropagationLemma3})     holds.  
\end{proof}

To present the theorem in terms of   RMS, we need the following lemma.

\begin{lem}
	[Triangle Inequality for RMS] \label{lem:TriangleIneqRMS} For any $p$-dimensional random vectors $\bx, \by,\bz$, we have: 
	\begin{equation*}
	\sqrt{\E\parenthesis{\norm{\bx-\by}^2}}\le \sqrt{\E\parenthesis{\norm{\bx-\bz}^2}}+\sqrt{\E\parenthesis{\norm{\bz-\by}^2}}.
	\end{equation*}
\end{lem}
\begin{proof}
	[Proof of Lemma \ref{lem:TriangleIneqRMS}]  	We will show that the r.h.s. squared is greater than or equal to the LHS squared.
	\begin{align}
&  \E\ilparenthesis{\norm{\bx-\bz}^2}+\E\ilparenthesis{\norm{\bz-\by}^2} + 2 \sqrt{  \ilbracket{   \E\ilparenthesis{\norm{\bx-\bz}^2}   } \ilbracket{\E\ilparenthesis{\norm{\bz-\by}^2}}    }\nonumber\\
	&\quad \ge  \E\ilparenthesis{\norm{\bx-\bz}^2}+\E\ilparenthesis{\norm{\bz-\by}^2} +2 \E\ilparenthesis{ \norm{\bx-\bz}\norm{\bz-\by}  }\nonumber\\
	&\quad= \E\bracket{  \ilparenthesis{ \norm{\bx-\bz} + \norm{\bz-\by}   }^2    }\nonumber\\
	&\quad \ge  \E\ilparenthesis{\norm{\bx-\by}^2}, \nonumber
	\end{align}
	where the first inequality follows from the Cauchy-Schwartz inequality,  and the last inequality follows from the triangle inequality. 
\end{proof}

\begin{thm}
	[RMS bound     under bounded-drift assumption] \label{thm:main2}  	Assume A.\ref{assume:ErrorWithBoundedSecondMoment},   A.\ref{assume:StronglyConvex}, A.\ref{assume:Lsmooth},  and A.\ref{assume:BoundedVariation}. Using  recursion (\ref{eq:basicSA}) with  non-decaying gain satisfying   the region specified in Lemma~\ref{lem:main1},   we have the following recurrence on the unconditional RMS:	\begin{equation}\label{eq:PropagationLemma4}
	\sqrt{	\E\ilparenthesis{\norm{\hbtheta_{k+1}-\bvartheta_{k+1}}^2}  }    \le 	\sqrt{ \firstConst_k 	\E\ilparenthesis{\norm{\hbtheta_{k}-\bvartheta_{k}}^2}  } + \noiseBound_k \sqrt{  {\secondConst_k }}+\driftBound_k, \quad   k \in  \natural. 
	\end{equation}
	Furthermore, we have  an asymptotic bound 
	\begin{equation}
	\label{eq:PropagationLemma5}
	\limsup_{k\to\infty}	\sqrt{	\E\ilparenthesis{\norm{\hbtheta_{k}-\bvartheta_{k}}^2 }  }    \le  \limsup_k   \frac{\noiseBound_k\sqrt{  {\secondConst_k}}+\driftBound_k}{1-\sqrt{\firstConst_k}}. 
	\end{equation} 
\end{thm}

\begin{proof}
	[Proof of Theorem~\ref{thm:main2}]
	Recall  that both $\firstConst_k $ and $\secondConst_k $ in Lemma~ \ref{lem:ErrorPropagationLemma1} are deterministic. Take the full expectation over  (\ref{eq:ErrorPropagationLemma}) and invoke    A.\ref{assume:ErrorWithBoundedSecondMoment}:
	\begin{eqnarray}\label{eq:PropagationLemma6}
	\E \ilparenthesis{\norm{\hbtheta_{k+1}-\bvartheta_k}^2}& \le  & \firstConst_k  \E \ilparenthesis{\norm{\hbtheta_k-\bvartheta_k}^2} + \secondConst_k \E\ilparenthesis{\norm{\be_k\ilparenthesis{\hbtheta_k}}^2}\nonumber\\
	&\le &  \firstConst_k  \E  \ilparenthesis{\norm{\hbtheta_k-\bvartheta_k}^2} +  {\secondConst_k}  \noiseBound_k^2,
	\end{eqnarray}  
	which   implies that: 
	\begin{eqnarray}
	\sqrt{   	\E \ilparenthesis{\norm{\hbtheta_{k+1}-\bvartheta_k}^2} } \le \sqrt{\firstConst_k }  \sqrt{\E\ilparenthesis{\norm{\hbtheta_k-\bvartheta_k}^2}} + \noiseBound _k \sqrt{  {\secondConst_k }},
	\end{eqnarray}
	because     $\firstConst_k \in\left[0,1\right)$,  $\secondConst_k\ge 0$, and that $ \sqrt{x^2+y^2}<x+y $ for $x,y\in\real^+$.   
	Then (\ref{eq:PropagationLemma4}) follows directly from Lemma~\ref{lem:TriangleIneqRMS}: 
	\begin{eqnarray}
	\sqrt{\E\ilparenthesis{ \norm{\hbtheta_{k+1}-\bvartheta_{k+1}}^2 }} &\le & \sqrt{\E\ilparenthesis{\norm{\hbtheta_{k+1}-\bvartheta_k}^2}} + \sqrt{ \E\ilparenthesis{\norm{ \bvartheta_{k+1}-\bvartheta_k}^2 } }\nonumber\\
	&\le &\sqrt{\firstConst_k }  \sqrt{\E\ilparenthesis{\norm{\hbtheta_k-\bvartheta_k}^2}} + \noiseBound _k  \sqrt{  {\secondConst_k }}  + \driftBound_k. 
	\end{eqnarray}
	
	Now following the derivation immediately after equation (\ref{eq:Prop1-6}) in the proof for Theorem~\ref{thm:main1}, we can obtain the asymptotic bound (\ref{eq:PropagationLemma5}).  
\end{proof}

\subsubsection*{Consequent Tightness} 
\label{subset:tightness}

As mentioned in Section~\ref{sect:ConstantGain},  the best we can hope for in the  time-varying scenario  is that the error term $  \ilparenthesis{\hbtheta_k-\btheta_k^*}$ hovers near zero, and that our concern centers on boundedness (input-output stability). The notion that no probability mass escapes to infinity  uniformly in $k$ is termed as tightness, which was reviewed in  Subsection \ref{subsect:Tightness}. 
\begin{equation}\label{eq:Tightness}
\lim_{M\to\infty} \sup_k \Prob\set{  \norm{\hbtheta_k-\bvartheta_{k} }\ge M  } = 0.
\end{equation}
By Chebyshev's inequality,  $ \Prob\ilset{\norm{\hbtheta_k-\bvartheta_k }\ge M } \le \E \norm{ \hbtheta_k-\bvartheta_{k}  }^2/M^2 $ holds,  where the   inequality follows from  (\ref{eq:PropagationLemma5}) in  Theorem \ref{thm:main2},  the boundedness of both $\firstConst_k$ and $\secondConst_k$ given the selection of slack variable and gain sequence in Lemma~\ref{lem:Tobe},     the assumed finiteness of $\noiseBound_k$ in A.\ref{assume:BoundedVariation}, $\convexPara_k$ in A.\ref{assume:StronglyConvex}, $\LipsPara_k$ in A.\ref{assume:Lsmooth}, and $\driftBound_k$ in A.\ref{assume:BoundedVariation}. Therefore, the boundedness in probability follows, and so does the  mean square boundedness.

\subsubsection*{Further Remarks on Slack Variables and  Gain Selection}

Let us return to   $\Tobe_k   $ (whose domain depends on $\ratio_k$)   and $\gain_k$ (whose domain depends on $\Tobe_k$) in   Lemma~\ref{lem:Tobe}.  
Recall that they are selected according to   (\ref{eq:Tobe}) and  (\ref{eq:GainBound1}) such that    $\firstConst_k$ in (\ref{eq:ErrorPropagationLemma}) is shrinking. 
There are, seemingly, many ways we may pursue to optimize the selection of both the slack variable and the gain. For example, we may consider:
\begin{enumerate}[i)]
	\item  \label{item:criteria1} minimizing  $\firstConst_k$ defined in  (\ref{eq:firstConst})  such that the previous tracking error is ``washed away'' as quickly as possible, 
 
\remove{	\item \label{item:criteria2} minimize $\secondConst_k$ defined in   (\ref{eq:secondConst})  such that the error term (including both the bias term and the noise term) plays a less dominating role in tracking, }
 
	\item  \label{item:criteria3} minimizing the r.h.s. of (\ref{eq:PropagationLemma3}) or (\ref{eq:PropagationLemma5})   such that the limiting  tracking error is as small as possible.
\end{enumerate}

Unfortunately,    \ref{item:criteria1}    has  no attainable\remove{For , the minimum is $ 1-\convexPara_k^2/(4\LipsPara_k^2) $, when  
	\begin{equation}\label{eq:PropLem7}
	\gain_k = \frac{\convexPara_k - \Tobe_1}{  2 \LipsPara_k^2 \ilparenthesis{1+\Tobe_2}  },
	\end{equation}
	and $\Tobe_1\to 0 $ and $\Tobe_2\to0$. Unfortunately, this is not attainable due to the positivity of the slack variables $\Tobe_1$ and $\Tobe_2$.  } minimizer.   For \ref{item:criteria3}, the solution depends not only on $\ratio_k$ but also on $\noiseBound_k$  and $\driftBound_k$. Worse still, the quantitative relations between them, in addition to (\ref{eq:ineq4}), also influence the result substantially. 
Also, note that  if our focus is on  finite-sample performance, then performing \ref{item:criteria3} will not benefit us in this sense much after all.
 
 Setting the nonexistence of ``optimal'' slack variable and gain selections aside, the following observation adds to the difficulty in tuning. 
A moment of reflection tells us that, even if either problem  \ref{item:criteria1} or problem  \ref{item:criteria3} is solvable,    the resulting gain may still perform poorly  because the  derivation  in Lemma~\ref{lem:ErrorPropagationLemma1}       only requires     a \emph{local} Lipschitz constant, which are  usually   smaller than the  \emph{global} one.  Similarly, Lemma~\ref{lem:ErrorPropagationLemma1} in fact requires a  \emph{local} strong convexity parameter, which is usually larger than the \emph{global} one. 

In short, we can provide neither an  optimal slack variable nor  and optimal  gain selection. Nonetheless, for ease of implementation, we may  set $\Tobe_k = \Tobe_{k,1}$  and    $\gain_k = 0.5/\LipsPara_k$ when $\ratio_k\ge (1+\sqrt{5})/2$   for simplicity\textemdash this is often smaller than what is desired due to the distinction between the global- and  local-smoothness parameter. Then, the gain strategy in Lemma~\ref{lem:main1} only ensures the parameter $\firstConst_k$ in (\ref{eq:ErrorPropagationLemma}) is shrinking so that    the  asymptotic error bounds (\ref{eq:PropagationLemma3}) and (\ref{eq:PropagationLemma5}) are valid.

	The expression on the r.h.s. of (\ref{eq:PropagationLemma3}) or  (\ref{eq:PropagationLemma5})
  conveys  information for  target tracking. The  MAD/RMS bound depends explicitly on the noise magnitude     $\noiseBound_k$ in A.\ref{assume:ErrorWithBoundedSecondMoment} and  the drift magnitude     $\driftBound_k$ in A.\ref{assume:BoundedVariation}. Besides, it   implicitly depends on $\convexPara_k$ in A.\ref{assume:StronglyConvex},  $\LipsPara _k $ in A.\ref{assume:Lsmooth} and the gain   $\gain _k$ selected by the agent(s) through both $\firstConst$ and $\secondConst$.  The first two dependencies are easy to understand, as we do expect the bound to be larger when either $\noiseBound_k$ or  $\driftBound_k$ gets larger.  The appearances of $\convexPara_k$ and $\LipsPara_k $ in the third dependency are reasonable, as the shape of $ \ilset{\loss_k} $ does impact our tracking accuracy. Interestingly, both $\firstConst_k$ in (\ref{eq:firstConst}) and $\secondConst_k $ in (\ref{eq:secondConst}) being  \emph{quadratic} functions  of $\gain_k$ inform  that     the gain $\gain_k$  for successful tracking should be ``neither too large nor too small.''

Take $\firstConst_k$ as an example and consider $\ratio_k  >1$ for all $k$. We pick $\Tobe_k= \Tobe_{k,1}$
 for all $k$. Then $\firstConst_k$ approaches the upper-bound $1$ when $\gain_k\LipsPara_k\to \multiple_{k,-
} \ilparenthesis{\Tobe_{k,1}}$ from the right or $\gain_k\LipsPara_k\to \multiple_{k,+}\ilparenthesis{\Tobe_{k,1}}$ from the left. Additionally,   $\firstConst_k$ approaches the lower-bound $0$ when $\gain_k\LipsPara_k= \ilparenthesis{ \Tobe_{k,1}/2 + 1 } \ilparenthesis{\Tobe_{k,1} + 1 }^{-1}$, which is the midpoint of $\multiple_{k,-}
\ilparenthesis{\Tobe_{k,1}}$ and $\multiple_{k, +}\ilparenthesis{\Tobe_{k,1}}$. Similarly, we can discuss $ {\secondConst}$ and potentially the coefficient $ \sqrt{ {\secondConst_k}} / \ilparenthesis{1-\sqrt{\firstConst_k}} $. However, as we do not have a definite objective towards which the slack variable and the gain selection are optimized, we no longer dwell on this topic here.

\subsubsection*{One Quick Example}\label{subsect:ExamplePriorBound}
	This example is borrowed from the  illustration  for adaptive control in \cite[Example 4.7]{spall2005introduction}. Target tracking is a common specific case of control problems. We want to track the coordinates of a time-varying multi-dimensional target, when only the noisy measurements of the distance to the moving target are available. 
To minimize the distance between the estimate $\hbtheta_k $ and the time-varying parameter sequence $\bvartheta_k$, we can formulate the loss function as $\loss _k(\btheta)=  \ilparenthesis{\btheta-\bvartheta_k}^\transpose\bH\ilparenthesis{\btheta-\bvartheta_k}/2 $. The true gradient sequence is $\bg_k(\btheta)=\bH\ilparenthesis{\btheta- \bvartheta_k } $, and the true Hessian sequence is $\bH$.

	Consider a simple case with $p=2$. We construct $\bH=\bP \bD \bP^\transpose$, where $\bP$ is   (randomly generated) orthogonal and $\bD$ is diagonal.   In our simulation,  
\begin{equation}\label{eq:SimHessian}
\bD = \begin{pmatrix}
30 & 0 \\ 0 & 5
\end{pmatrix}, \quad 
\bP = \begin{pmatrix}
0.8145 & -0.5802 \\ -0.5802 & -0.8145
\end{pmatrix}, \quad \text{ and }\quad  \bH = \begin{pmatrix}
14.9505 & -7.0884 \\ -7.0884 & 10.0495
\end{pmatrix}. 
\end{equation}

The accessible information is the noisy gradient measurement $\bY_k(\btheta)=\bg_k(\btheta)+\be_k$, where $\be_k \stackrel{\text{i.i.d.}}{\sim}  \mathrm{Normal}(\textbf{0}, \upsigma_1^2 \Id)$.  
The (unknown to the algorithm)  nonstationary drift  evolves according to:
\begin{equation}\label{eq:evolution1}
\bvartheta_{k+1}  = \begin{cases}
\bvartheta_k + \ilparenthesis{1,\,\, 1} ^\transpose + \bw_k, & \text{ for }1\le k \le 499,\\ 
\bvartheta_k +  \ilparenthesis{-1,\,\, 1}^\transpose + \bw_k, & \text{ for }500\le k \le 999, 
\end{cases}
\end{equation} with  $\bvartheta_0=\zero$ and 
$\bw_k \stackrel{\text{i.i.d.}}{\sim}  \mathrm{Normal}(\textbf{0}, \upsigma_2^2 \Id)$.  From  Theorem~\ref{thm:main1}, we take $\noiseBound_k =\sqrt{2}\upsigma_1$,   $\driftBound_k =\sqrt{2}$, and $\convexPara_k=5$, $\LipsPara_k =30$ for all $k$.    
In this experiment, both $\upsigma_1$ and $\upsigma_2$ are set to be $10$. Note that  $10$ is large  compared to the magnitude of the  deterministic trend in $\ilset{\bvartheta_k}$; i.e., heading northeast with step $ \parenthesis{1,\,\, 1} ^\transpose$ for $0\le k \le 499$ and heading northwest with step $   \parenthesis{-1, \,\,  1 } ^\transpose$ for $500\le k \le 999$.
Both  $\bvartheta_0$ and $\hbtheta_0$ are set to   $\zero$.  We implement (\ref{eq:basicSA}) for 25 trial runs, each with 1000 iterations.

\begin{figure}[!htbp]
	\centering
	\includegraphics[width=.65\textwidth]{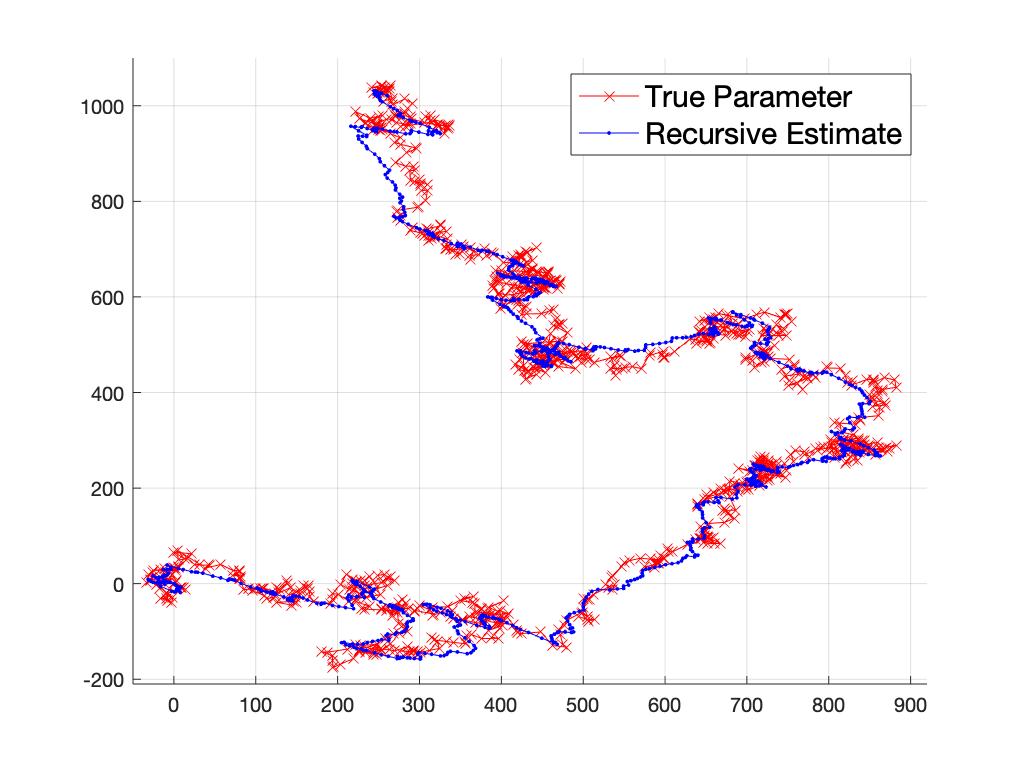} 
	\caption{The Trajectories Of The  True  Parameter $\bvartheta_k$ And  The  Recursive  Estimates $\hbtheta_k$ In  \emph{One}   Run}
	\label{fig:tracking1}
\end{figure}

\begin{figure}[!htbp]
	\centering
	\begin{subfigure}{.65\textwidth}
		\centering
		\includegraphics[width=\linewidth]{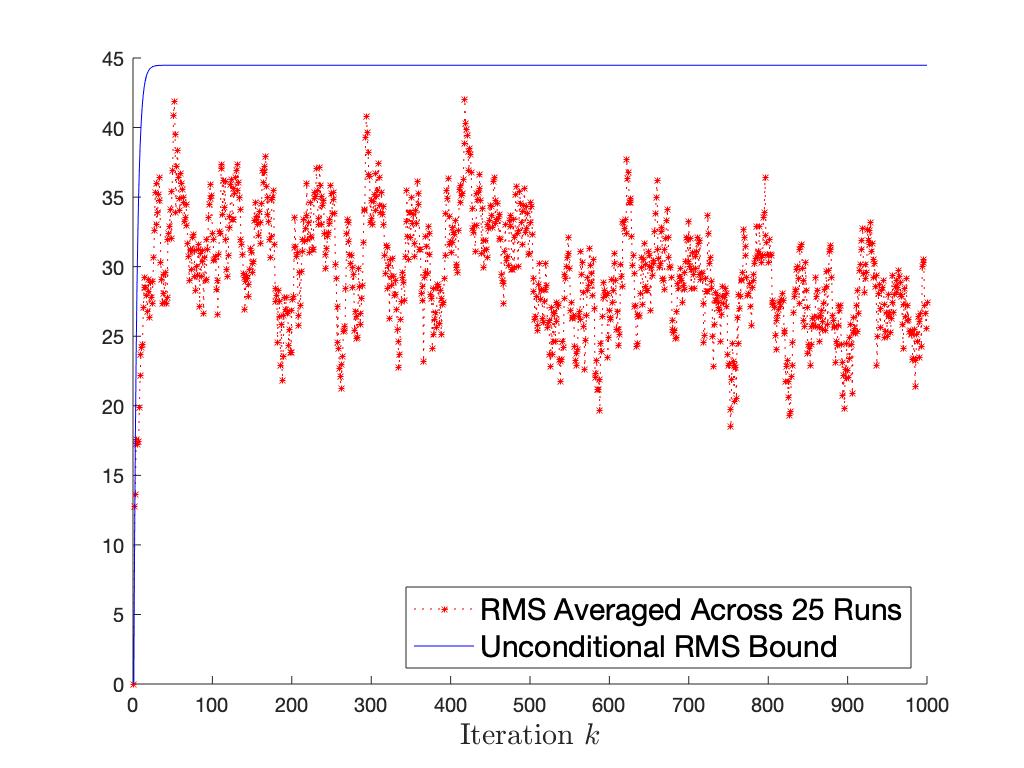}
			\caption{The Empirical RMS Averaged Across 25 Runs and The Upper Bound To The  RMS }
		\label{fig:sub1}
	\end{subfigure}\\
	\begin{subfigure}{.65\textwidth}
		\centering
		\includegraphics[width=\linewidth]{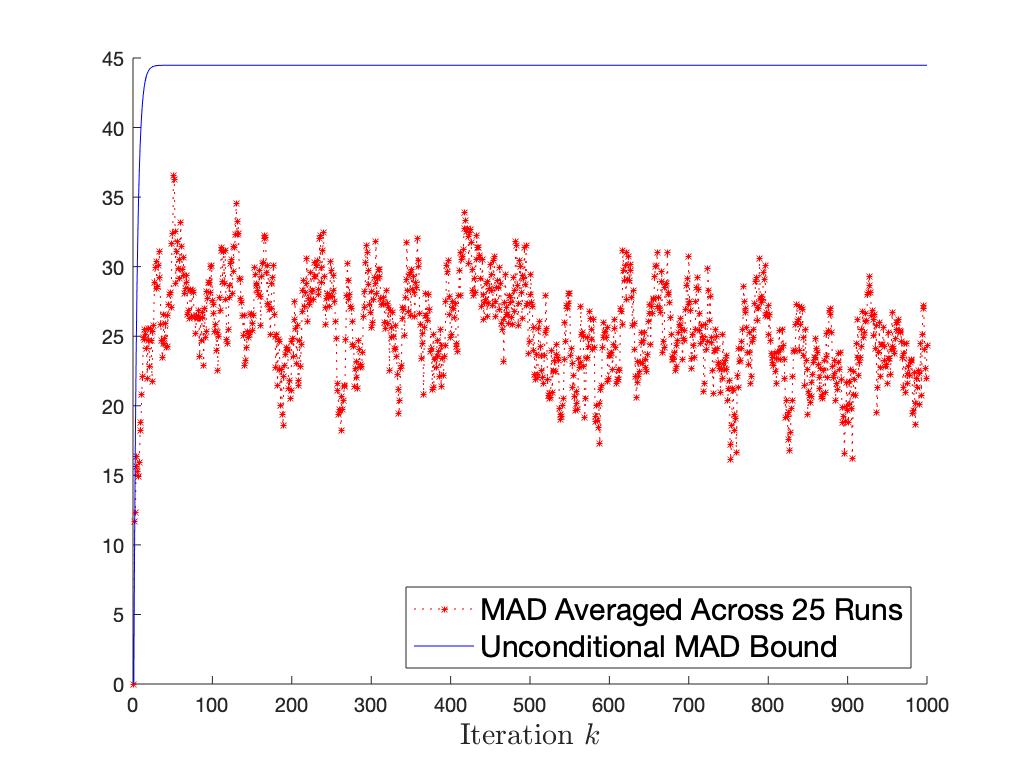}
			\caption{The Empirical MAD Averaged Across 25 Runs and  The Upper Bound To The  MAD }
		\label{fig:sub2}
	\end{subfigure}
	\caption{The Empirical Errors (Averaged Across $25$ Runs) And  The Corresponding Upper Bounds}
	\label{fig:trackingbound}
\end{figure}

Following Algorithm~\ref{algo:basicSA}, we  pick $\Tobe_k = 0.4 \Tobe_{k,1} + 0.6\Tobe_{k,2}$ and $\gain = 0.5/\LipsPara_k $.   Figure \ref{fig:tracking1}    displays how the iterates generated by (\ref{eq:basicSA}) keep up with the moving target.
Figures \ref{fig:sub1} and \ref{fig:sub2}  show the accuracy of   (\ref{eq:PropagationLemma3}) and (\ref{eq:PropagationLemma5})    in bounding the tracking error in nonstationary optimization. The empirical MAD/RMS is computed by averaging the absolute-deviation and by taking the root of the averaged squared-error across 25 trial runs. Note that the bounds   (\ref{eq:PropagationLemma3}) and (\ref{eq:PropagationLemma5}), although conservative, are quite accurate in terms of characterizing the empirical error. 
 
\subsection{  A Posterior Error Bound       }\label{subsect:conditional}
 Applications to physical systems in Subsection  \ref{subsect:unconditional} often    encourages the  ``few  measurements at a time'' requirement explained  in Section \ref{sect:ProlemSetup}. Furthermore, the physical constraints      explain  that the   drift magnitude bound $\driftBound_k$ introduced in  A.\ref{assume:BoundedVariation}    should be  knowable or estimatable (similar to  (\ref{eq:driftBoundKF}) for the multi-agent application in Chapter~\ref{chap:multiagent}) in advance  and  should also be  small  (relative to the magnitude of the observable gradient information)  in Subsection \ref{subsect:boundedvariation}.  Note that   (\ref{eq:PropagationLemma3}) to (\ref{eq:PropagationLemma5})  already average   out  the  observable information   $\field_{k}$ which was defined in (\ref{eq:filtration}). That is, the error bounds in  Subsection \ref{subsect:unconditional}   provide  the a priori  tracking performances that are  average  over all possible sample paths so as  to ensure tracking performance.
 The gain selection in Algorithm~\ref{algo:basicSA} is not impacted by the observable information $\field_k$
 
  Nonetheless, during   actual implementation,  we  hope to  react to the changes   in $\ilset{\bvartheta_k}$ as promptly as possible. Thus,  the average  performance  
 may not be informative in \emph{one} realization, though it is meaningful in  providing   gain selection guidance  to ensure tracking. In fact, we have the following upper and lower bounds  on the MAD,  $\E_{k+1}\norm{\hbtheta_k-\bvartheta_k}$,    conditioned on  $ \field_{k+1}  $ defined in (\ref{eq:filtration}), which is  the  observable  information   through  time instant $k$.

 	\begin{thm}[Conditional MAD bound]\label{thm:CondBound1}
 	Assume A.\ref{assume:ErrorWithBoundedSecondMoment}, A.\ref{assume:StronglyConvex}, and A.\ref{assume:Lsmooth}.   We have the following MAD bound conditioned on $\field_{k+1}$:
 	\begin{equation}\label{eq:CondBound1}
 	\LipsPara_k^{-1} \abs{ \norm{\hbg_k\ilparenthesis{\hbtheta_k}} - \noiseBound_k     }	\le 	\E_{k+1} \norm{\hbtheta_k-\bvartheta_k} \le \convexPara_k ^{-1}\parenthesis{  \norm{\hbg_k\ilparenthesis{\hbtheta_k }} + \noiseBound_k  }, 
 	\end{equation}
 	where $\E_{k+1}$ is the   expectation conditioned on  the  observable  information $ \field_{k+1}   $ through  time instant $k$.
 \end{thm}	

\begin{proof}[Proof of Theorem~\ref{thm:CondBound1}]
    From (\ref{eq:ineq2}) and (\ref{eq:ineq1}), we have:  \begin{equation}\label{eq:CondBound2} \LipsPara_k ^{-1}\norm{\bg_k\ilparenthesis{\btheta}}\le 
    \norm{\btheta-\bvartheta_k}\le \convexPara_k ^{-1}\norm{\bg_k\ilparenthesis{\btheta}},
    \end{equation} for any $\btheta\in\real^p$ when A.\ref{assume:StronglyConvex} and A.\ref{assume:Lsmooth} hold. 
    By the triangle inequality and the reverse triangle inequality, we have: 
    \begin{align}\label{eq:CondBound3}  & \LipsPara_k^{-1}\abs{ \norm{\hbg_k\ilparenthesis{\hbtheta_k}} - \norm{\be_k\ilparenthesis{\hbtheta_k}}    } \nonumber\\
    &\quad \le \LipsPara_k^{-1} \norm{\bg_k\ilparenthesis{\hbtheta_k}}\nonumber\\
    &\quad \le 
    \norm{\hbtheta_k-\bvartheta_k}\nonumber  \\
    &\quad   \le \convexPara_k^{-1} \norm{\bg_k\ilparenthesis{\hbtheta_k}} \nonumber\\
    &\quad \le \convexPara_k^{-1}  \bracket{ \norm{\hbg_k\ilparenthesis{\hbtheta_k}} + \norm{\be_k\ilparenthesis{\hbtheta_k }} }. 
    \end{align}
    Taking conditional expectation over (\ref{eq:CondBound3}) and invoking  A.\ref{assume:ErrorWithBoundedSecondMoment} yields (\ref{eq:CondBound1}). 
\end{proof}

To obtain a conditional   bound on the drift term, let us   consider  the following   filtration instead of (\ref{eq:filtration}). 
\begin{equation}\label{eq:SigmaAlgebra2}
\nfield_0=\field_0=\upsigma\ilset{\hbtheta_0}, \,\,\,\text{and } \nfield_k=\upsigma\ilset{\hbtheta_0 ,\hbg_i\ilparenthesis{\hbtheta_{i-1}}, \hbg_i\ilparenthesis{\hbtheta_i}, i<k } \text{ for }k\ge 1, 
\end{equation}
which is finer (richer) than  $\field_k$.  We may have the following indicators for the tracking performance.

\begin{thm}
	[Estimation for drift using two-measurements]\label{thm:CondBound1nField}  Under Assumptions A.\ref{assume:ErrorWithBoundedSecondMoment}, A.\ref{assume:StronglyConvex}, and A.\ref{assume:Lsmooth}, we have the following drift bound conditioned  on $\nfield_{k+2}$: 
	\begin{equation}\label{eq:MultipleCondBound1}
	\E\bracket{\given{\norm{\bvartheta_{k+1}-\bvartheta_k }}{\nfield_{k+2}}} \le \convexPara_{k+1}^{-1}\norm{\hbg_{k+1}\ilparenthesis{\hbtheta_k}} + \convexPara_k^{-1} \norm{\hbg_k\ilparenthesis{\hbtheta_k }} + \noiseBound_{k+1} \convexPara_{k+1}^{-1} + \noiseBound_k \convexPara_{k}^{-1},
	\end{equation}
	and 
	\begin{align}
	\label{eq:MultipleCondBound2}
	&  \E\bracket{\given{\norm{\bvartheta_{k+1}-\bvartheta_k }}{\nfield_{k+2}}}\nonumber\\
	&\quad \ge \max  \bigg\{  \LipsPara_{k+1} ^{-1} \abs{  \norm{\hbg_{k+1}\ilparenthesis{\hbtheta_k}} - \noiseBound_{k+1} } - \convexPara_k^{-1} \parenthesis{\norm{\hbg_k\ilparenthesis{\hbtheta_k}}+\noiseBound_k}, \nonumber\\
	&\quad\quad\quad\quad\quad  \LipsPara_k^{-1} \abs{\norm{\hbg_k\ilparenthesis{\hbtheta_k}}-\noiseBound_k } - \convexPara_{k+1}^{-1} \parenthesis{\norm{\hbg_{k+1}\ilparenthesis{\hbtheta_k}}+\noiseBound_{k+1}}   \bigg\} .
	\end{align} 
\end{thm}

\begin{proof}[Proof for Theorem~\ref{thm:CondBound1nField}]
		\begin{eqnarray}\label{eq:temp1}
	\norm{\bvartheta_{k+1}-\bvartheta_k}
	&\le & \norm{\bvartheta_{k+1}-\hbtheta_k}+\norm{\hbtheta_k-\bvartheta_k}\nonumber\\
	&\le & \convexPara_{k+1}^{-1} \norm{\bg_{k+1}\ilparenthesis{\hbtheta_k}} + \convexPara_k^{-1} \norm{\bg_k\ilparenthesis{\hbtheta_k }}\nonumber\\
	&\le & \convexPara_{k+1}^{-1}\parenthesis{ \norm{\hbg_{k+1}\ilparenthesis{\hbtheta_k}} + \norm{\be_{k+1}\ilparenthesis{\hbtheta_k}}  } + \convexPara_k^{-1}\parenthesis{\norm{\hbg_k\ilparenthesis{\hbtheta_k}} + \norm{\be_k\ilparenthesis{\hbtheta_k }}},
	\end{eqnarray}
	where the second inequality follows from (\ref{eq:CondBound2}). Taking the conditional expectation of (\ref{eq:temp1}) over $\nfield_{k+2}$ gives (\ref{eq:MultipleCondBound1}). 
		\begin{eqnarray}\label{eq:temp2}
	&\quad& 	\norm{\bvartheta_{k+1}-\bvartheta_k}\nonumber\\
	&\ge & \max \set{  \norm{\bvartheta_{k+1}-\hbtheta_k} -   \norm{\bvartheta_k-\hbtheta_k} ,      \norm{\bvartheta_k-\hbtheta_k} - \norm{\bvartheta_{k+1}-\hbtheta_k} }\nonumber\\
	&\ge & \max \set{  \LipsPara_{k+1}^{-1}   \norm{\bg_{k+1}\ilparenthesis{\hbtheta_k }}  - \convexPara_k^{-1}  \norm{\bg_k\ilparenthesis{\hbtheta_k }},   \LipsPara_k^{-1 } \norm{\bg_k\ilparenthesis{\hbtheta_k}} - \convexPara_{k+1}^{-1} \norm{\bg_{k+1}\ilparenthesis{\hbtheta_k }}  }\nonumber\\
	&\ge & \max  \bigg \{   \LipsPara_{k+1}^{-1} \abs{ \norm{\hbg_{k+1}\ilparenthesis{\hbtheta_k}} - \norm{\be_{k+1}\ilparenthesis{\hbtheta_k }} } - \convexPara_k ^{-1} \parenthesis{ \norm{\hbg_k\ilparenthesis{\hbtheta_k}} + \norm{\be_k\ilparenthesis{\hbtheta_k }} } ,  \nonumber\\
	&\quad &  \quad\quad\,\,  \LipsPara_k ^{-1} \abs{     \norm{\hbg_k\ilparenthesis{\hbtheta_k}} - \norm{\be_k\ilparenthesis{\hbtheta_k }}  }  - \convexPara_{k+1}^{-1} \parenthesis{\norm{\hbg_{k+1}\ilparenthesis{\hbtheta_k}  }+ \norm{\be_{k+1}\ilparenthesis{\hbtheta_k }} } \bigg \}, \quad\quad 
	\end{eqnarray} 	where the second inequality follows from (\ref{eq:CondBound2}). Taking the conditional expectation of (\ref{eq:temp2}) over  $\nfield_{k+2}$ gives (\ref{eq:MultipleCondBound2}). 
\end{proof}

\begin{corr}
	[Conditional mean tracking performance] \label{corr:MeanTrackingPerf} 	In addition to the conditions in Theorem~\ref{thm:CondBound1nField}, further assuming    $ \E\ilbracket{ \left.{ \be_{k+1}\ilparenthesis{\hbtheta_k}  }\right| {\nfield_{k+2}}} =\zero  $. , we have: 
	\begin{align}
&	\E\bracket{\given{\loss_{k+1}\ilparenthesis{\hbtheta_k}- \loss_{k+1}\ilparenthesis{\bvartheta_{k+1}}}{\nfield_{k+2}}} \nonumber\\
	&\quad \le \frac{1}{2\convexPara_{k+1}} \parenthesis{\norm{\hbg_{k+1}  \ilparenthesis{\hbtheta_k } }^2+\noiseBound_{k+1} ^2}   + \frac{\gain_k ^2\LipsPara_{k+1}}{2}\norm{\hbg_k\ilparenthesis{\hbtheta_k}}^2 - \gain_k \ilbracket{\hbg_{k+1}\ilparenthesis{\hbtheta_k}}^\transpose\hbg_k\ilparenthesis{\hbtheta_k}, 		\end{align}
	and 
	\begin{align} 	& \E\bracket{\given{\loss_{k+1}\ilparenthesis{\hbtheta_k}- \loss_{k+1}\ilparenthesis{\bvartheta_{k+1}}}{\nfield_{k+2}}}\nonumber\\
	&\quad \ge 
	\frac{1}{2\LipsPara_{k+1}}   \norm{\hbg_{k+1}\ilparenthesis{\hbtheta_k}}^2  + \frac{\gain_k\convexPara_{k+1}}{2}\norm{\hbg_k\ilparenthesis{\hbtheta_{k}}}^2 - \gain_k \ilbracket{  \hbg_{k+1}\ilparenthesis{\hbtheta_k}   }^\transpose\hbg_k\ilparenthesis{\hbtheta_k }, 
	\end{align}
	where the expectation is conditioned on $\nfield_{k+2}=\upsigma\ilset{\hbtheta_0,\hbg_i\ilparenthesis{\hbtheta_i}, \hbg_i\ilparenthesis{\hbtheta_{i-1}},i<k+2}$. 
	
\end{corr}

\begin{proof}[Proof of Corollary~\ref{corr:MeanTrackingPerf}]
	With the recursion in (\ref{eq:basicSA}), we have the following from inequality (\ref{eq:LipsIneq}): 	\begin{align}
& \loss_{k+1}\ilparenthesis{\hbtheta_{k+1}}\nonumber\\
&\quad = \loss_{k+1}\ilparenthesis{\hbtheta_k-\gain_k \hbg_k\ilparenthesis{\hbtheta_k }}\nonumber\\
	&\quad \le \loss_{k+1}\ilparenthesis{\hbtheta_k} - \gain_k   \ilbracket{\bg_{k+1}\ilparenthesis{\hbtheta_k}}^\transpose\hbg_k\ilparenthesis{\hbtheta_k } + \frac{\LipsPara_{k+1}}{2}\gain_k^2 \norm{\hbg_k\ilparenthesis{\hbtheta_k}}^2\nonumber\\
	&\quad = \loss_{k+1}\ilparenthesis{\hbtheta_k } - \gain_k \ilbracket{\hbg_{k+1}\ilparenthesis{\hbtheta_k} }^\transpose\hbg_k\ilparenthesis{\hbtheta_k } + \gain_k \ilbracket{\be_{k+1}\ilparenthesis{\hbtheta_k}}^\transpose \hbg_k\ilparenthesis{\hbtheta_k  } + \frac{\gain_k^2\LipsPara_{k+1}}{2}\norm{\hbg_k\ilparenthesis{\hbtheta_k}}^2. 
	\end{align}
	Therefore,  
	\begin{align}
	&    \E \bracket{ \given{\loss_{k+1}\ilparenthesis{\hbtheta_{k+1} }  - \loss_{k+1}\ilparenthesis{\bvartheta_{k+1}}}{\nfield_{k+2}} }\nonumber\\
	&\quad  \le   \E  \bracket{  \given{\loss_{k+1}\ilparenthesis{\hbtheta_k}-\loss_{k+1}\ilparenthesis{\bvartheta_{k+1}} }{\nfield_{k+2}} } - \gain_k  \ilbracket{\hbg_{k+1}\ilparenthesis{\hbtheta_k}}^\transpose\hbg_k\ilparenthesis{\hbtheta_k} + \frac{\gain_k ^2\LipsPara_{k+1}}{2}\norm{\hbg_k\ilparenthesis{\hbtheta_k}}^2\nonumber\\
	&\quad  \le \frac{1}{2\convexPara_{k+1}}  \E\bracket{\given{\norm{\bg_{k+1}\ilparenthesis{\hbtheta_k}}^2}{\nfield_{k+2}}} + \frac{\gain_k ^2\LipsPara_{k+1}}{2}\norm{\hbg_k\ilparenthesis{\hbtheta_k}}^2 - \gain_k \ilbracket{\hbg_{k+1}\ilparenthesis{\hbtheta_k}}^\transpose\hbg_k\ilparenthesis{\hbtheta_k}\nonumber\\ 
	&\quad  \le \frac{1}{2\convexPara_{k+1}}\parenthesis{ \norm{\hbg_{k+1}  \ilparenthesis{\hbtheta_k } }^2 + \noiseBound_{k+1}^2  } + \frac{\gain_k ^2\LipsPara_{k+1}}{2}\norm{\hbg_k\ilparenthesis{\hbtheta_k}}^2 - \gain_k \ilbracket{\hbg_{k+1}\ilparenthesis{\hbtheta_k}}^\transpose\hbg_k\ilparenthesis{\hbtheta_k}, 
	\end{align}
	where the second   last inequality follows from (\ref{eq:ineq2}).
	
		Similarly, from  the recursion in (\ref{eq:basicSA}), we have the following from inequality  (\ref{eq:ConvIneq}): 
	\begin{align}
	&  \loss_{k+1}\ilparenthesis{\hbtheta_{k+1}}\nonumber\\
	&\quad  \ge   \loss_{k+1}\ilparenthesis{\hbtheta_k}-\gain_k \ilbracket{\bg_{k+1}\ilparenthesis{\hbtheta_k}}^\transpose\hbg_k\ilparenthesis{\hbtheta_k } + \frac{\convexPara_{k+1}}{2}\gain_k ^2\norm{\hbg_k\ilparenthesis{\hbtheta_k}}^2\nonumber\\
	&\quad  = \loss_{k+1}\ilparenthesis{\hbtheta_k} - \gain_k\ilbracket{\hbg_{k+1}\ilparenthesis{\hbtheta_k}}^\transpose\hbg_k\ilparenthesis{\hbtheta_k} + \gain_k \ilbracket{\be_{k+1}\ilparenthesis{\hbtheta_k}}^\transpose\hbg_k\ilparenthesis{\hbtheta_k} + \frac{\gain_k ^2\convexPara_{k+1}}{2}\norm{\hbg_k\ilparenthesis{\hbtheta_k}}^2.  
	\end{align}	Therefore,
	\begin{align}
	&    \E\bracket{\given{  \loss_{k+1}\ilparenthesis{\hbtheta_{k+1}} - \loss_{k+1}\ilparenthesis{\bvartheta_{k+1}}  }{\nfield_{k+2}}}\nonumber\\
	&\quad \ge    \E  \bracket{  \given{\loss_{k+1}\ilparenthesis{\hbtheta_k}-\loss_{k+1}\ilparenthesis{\bvartheta_{k+1}} }{\nfield_{k+2}} } - \gain_k \ilbracket{\hbg_{k+1}\ilparenthesis{\hbtheta_k}}^\transpose\hbg_k\ilparenthesis{\hbtheta_k} + \frac{\gain_k ^2\convexPara_{k+1}}{2}\norm{\hbg_k\ilparenthesis{\hbtheta_k}}^2\nonumber\\
	&\quad \ge  \frac{1}{2\LipsPara_{k+1}}  \E\bracket{\given{  \norm{\bg_{k+1}\ilparenthesis{\hbtheta_k }}^2 }{\nfield_{k+2}}} + \frac{\gain_k \convexPara_{k+1}}{2}\norm{\hbg_k\ilparenthesis{\hbtheta_{k}}}^2 - \gain_k \ilbracket{  \hbg_{k+1}\ilparenthesis{\hbtheta_k}   }^\transpose\hbg_k\ilparenthesis{\hbtheta_k }\nonumber\\
	&\quad \ge   \frac{1}{2\LipsPara_{k+1}}   \norm{\hbg_{k+1}\ilparenthesis{\hbtheta_k}}^2 + \frac{\gain_k \convexPara_{k+1}}{2}\norm{\hbg_k\ilparenthesis{\hbtheta_{k}}}^2 - \gain_k \ilbracket{  \hbg_{k+1}\ilparenthesis{\hbtheta_k}   }^\transpose\hbg_k\ilparenthesis{\hbtheta_k }.
	\end{align}	
\end{proof}

 \subsubsection*{One Quick Example } 
Again, we use the similar setup as in Subsubsection \ref{subsect:ExamplePriorBound}, except that the evolution of  $\ilset{\bvartheta_k}$  in  (\ref{eq:evolution1}) now changes to: 
	\begin{equation}\label{eq:evolution2}
\bvartheta_{k+1}  = \begin{cases}
\bvartheta_k +  \ilparenthesis{1,\,\,1}^\transpose + \bw_k, & \text{ for }1\le k \le 499,\\ 
\bvartheta_k + 200 \ilparenthesis{  \cos(\upvarphi),   \,\,   \sin (\upvarphi ) } ^\transpose, & \text{ for } k = 500,\\ 
\bvartheta_k + \ilparenthesis{-1,\,\,1}^\transpose + \bw_k, & \text{ for }501\le k \le 999,
\end{cases}
\end{equation} with $\bvartheta_0=\zero$. Again,  $\bw_k\stackrel{\text{i.i.d.}}{\sim}  \mathrm{Normal}\ilparenthesis{\zero,\upsigma_2^2\bI_p} $, and $\upvarphi\sim\mathrm{Uniform}\bracket{0,2\uppi}$.   All the other   parameters and the gain sequence selection remain the same as Subsubsection \ref{subsect:ExamplePriorBound}.

Still following the general procedure in  Algorithm~\ref{algo:basicSA},   this time  we do not know $\driftBound_k$ a priori  and cannot proceed with the computation of  the error bounds established in Subsection \ref{subsect:unconditional}. 
Figure \ref{fig:tracking2}    displays how the iterates generated by (\ref{eq:basicSA}) keep up with the moving target using the same gain sequence as the one used in Subsection~\ref{subsect:unconditional}.
Figure \ref{fig:trackingbound2} shows how collecting two measurements at a time, i.e., collecting (\ref{eq:SigmaAlgebra2}),   to  bounding the drift term $ \norm{\bvartheta_{k+1}-\bvartheta_k } $ per  (\ref{eq:MultipleCondBound1})  in nonstationary optimization.  
\begin{figure}[!htbp]
	\centering
	\includegraphics[width=.65\textwidth]{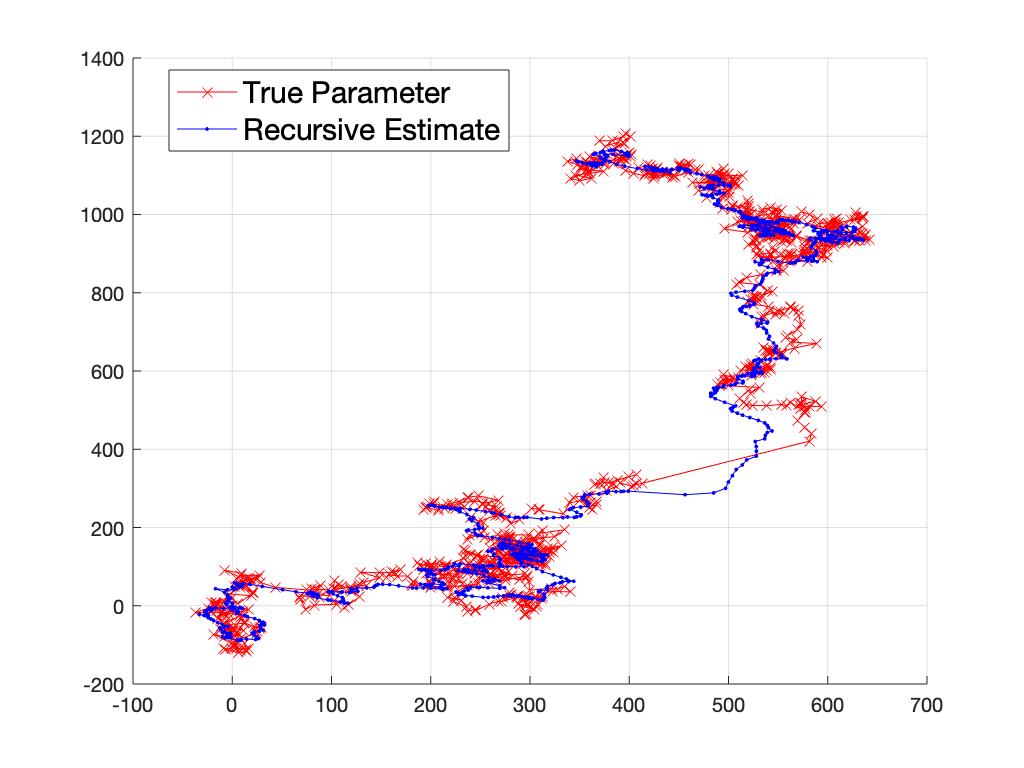} 
	\caption{The Trajectories Of The Underlying Parameter $\bvartheta_k$ And The  Recursive Estimates $\hbtheta_k$ In \emph{One}  Run}
	\label{fig:tracking2}
\end{figure}
\begin{figure}[!htbp]
	\centering
	\includegraphics[width=.65\textwidth]{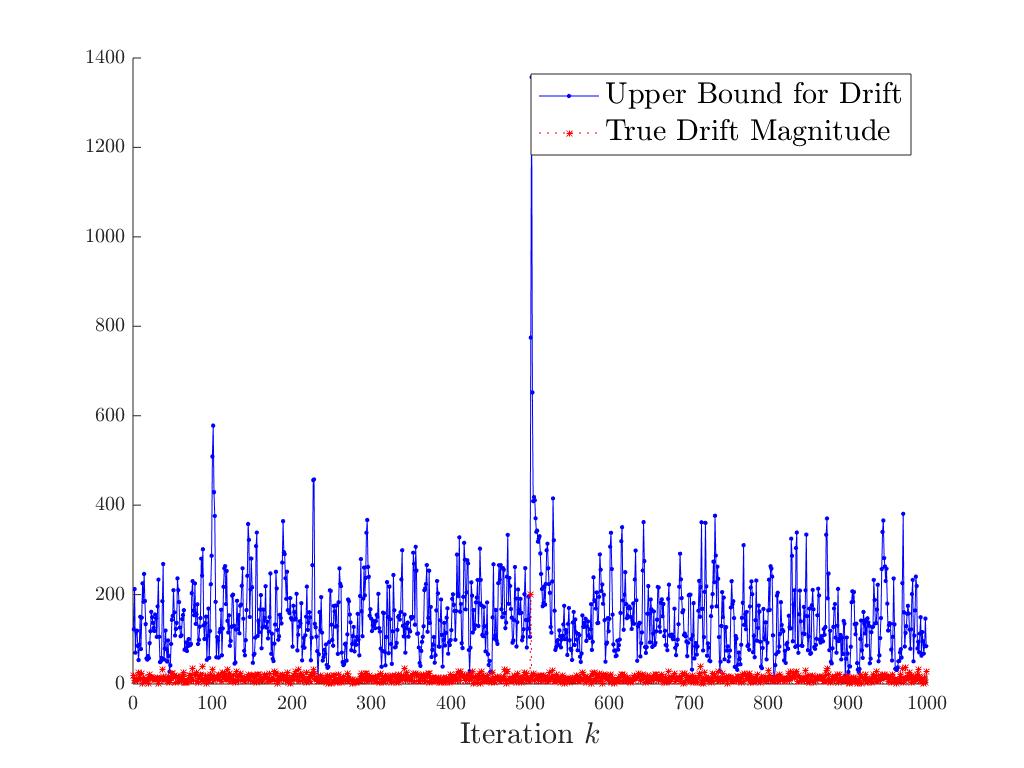} 
	\caption{Comparison Between The Upper Bound (\ref{eq:MultipleCondBound1}) Obtained From  Two Measurements And The Actual Drift  Term $\norm{\bvartheta_{k+1} - \bvartheta_k}$ }
	\label{fig:trackingbound2}
\end{figure}

 \subsubsection*{Call  For A Data-Driven Gain-Tuning Strategy}
Under the general  time-varying assumption A.\ref{assume:BoundedVariation}, 
note that (\ref{eq:MultipleCondBound1}) can   provide a rudimentary  assessment to the variation in   $\norm{\bvartheta_{k+1}-\bvartheta_k}$, and that (\ref{eq:CondBound1}) can help bounding  the conditional MAD.   However,  our gain selection strategy in Lemma~\ref{lem:Tobe} (equivalently Algorithm~\ref{algo:basicSA}) is  developed on the basis  that $\noiseBound_k  $ in A.\ref{assume:ErrorWithBoundedSecondMoment} and $\driftBound_k$ in A.\ref{assume:BoundedVariation} are relatively small   and on the purpose of  ensuring \emph{average} tracking performance as opposed to   a single sample-path.   Therefore, we may need to develop a new gain selection strategy    to be  relatively  robust   to abrupt changes  as  indicated by  $ \E_{k+2}\norm{\bvartheta_{k+1}-\bvartheta_k} $,  if any such change arises within a particular sample-path. See further details in Chapter~\ref{chap:AdaptiveGain}.

Observe from  Theorem~\ref{thm:CondBound1} that 
a larger value of $\norm{\hbg_k\ilparenthesis{\hbtheta_k}}$ is   strong evidence that $\hbtheta_k$ is further away from $\bvartheta_k$.   However, within one run of generating $\hbtheta_0,\cdots,\hbtheta_k$, we \emph{cannot} differentiate  whether or not the large value of $\norm{\hbg_k\ilparenthesis{\hbtheta_k}}$ is due to   excessive noise or due to the abrupt jump from $\bvartheta_k$ to $\bvartheta_{k+1}$. Furthermore, as mentioned in Subsection \ref{subsect:distinction}, we do not want to consider   ``multiple \emph{sequential} measurements at a time,'' especially an excessive number  of measurements (same order of the squared of the  inverse desired accuracy) as in \cite{wilson2019adaptive}.

Therefore, we will turn to a more restrictive scenario in Chapter~\ref{chap:AdaptiveGain}. 
It is desirable to obtain a testing rule under which   SA iterates can promptly detect  the 
change in $\ilset{\bvartheta_k}$ and provide guidance in gain selection.  	It is certainly advantageous to use   adaptive rules that enable the stepsize to vary with information gathered during the progress of the estimation procedure.

\section{Special Cases}\label{sect:SpecialCases}

\subsection{ Regression with Time-Varying Underlying Parameter   }\label{sect:LMS} 

\subsubsection*{Least-Mean-Squares}

 In the linear regression model,  we assume the following  measurement equation that is linear in $\bvartheta_k$:
\begin{equation}\label{eq:LinearMeasurement}
z_k=\bh_k ^\transpose \bvartheta_k+v_k,\,\, k\in \natural, 
\end{equation}
where $z_k$ is the $k$th scalar measurement of the output, $\bh_k$ is a $p\times 1$ stochastic design vector of the input or regression vector, $\bvartheta_k$ is the underlying target parameter, which evolves smoothly along the passage of time, and $v_k$ is a mean-zero disturbance sequence.  
For multiple-input-single-output (MISO)\label{acronym:MISO} system (\ref{eq:LinearMeasurement}),  the goal is to use known input values of $\bh_k$ (e.g., from a training sequence) and observed output values $z_k$ to estimate and track the underlying  MISO    system parameter $\btheta_k^*$. The time-varying function we are trying to minimize is 
\begin{equation}\label{eq:QuadraticLoss}
\loss_k\ilparenthesis{\btheta} = \frac{1}{2}\E\bracket{\ilparenthesis{z_{k+1} - \bh_{k+1}^\transpose\btheta }^2},
\end{equation}
where the expectation is taken w.r.t. the noise $v_{k+1}$ in (\ref{eq:LinearMeasurement}) and the randomness in $\bh_{k+1}$ if the $\bh_{k+1}$ is  random. Here  $\bh_k$ is the controllable input (may be random), while $z_k$ is the output that contains partial information on  $\btheta_k^*$.   

Suppose that the measurement noise $v_k$ is independent of  both $\btheta$ and $\bh_k$.  Then the derivative of the time-varying loss  function (\ref{eq:QuadraticLoss}) w.r.t. parameter $\btheta$ is
\begin{eqnarray}\label{eq:LinearGradient}
\bg_k\ilparenthesis{\btheta} &=& \frac{1}{2} \frac{\partial \E\bracket{\ilparenthesis{z_{k+1} - \bh_{k+1}^\transpose\btheta}^2}}{\partial\btheta} \nonumber\\
&=& \frac{1}{2} \frac{\partial \E  \bracket{ \ilparenthesis{ \bh_{k+1}^\transpose\bvartheta_{k+1} + v_{k+1 }  - \bh_{k+1}^\transpose\btheta  }^2  } }{\partial\btheta}\nonumber\\
&=&\begin{dcases}
\bh_{k+1}\bh_{k+1}^\transpose\ilparenthesis{\btheta-\bvartheta_{k+1}},\,\, \text{when } \bh_{k+1} \text{ is deterministic,}\\
\E\ilparenthesis{\bh_{k+1}\bh_{k+1}^\transpose}\ilparenthesis{\btheta-\bvartheta_{k+1}},\,\, \text{when } \bh_{k+1} \text{ is stochastic yet independent of }v_{k+1}.
\end{dcases}
\end{eqnarray}
Note that the expectation in the last line is w.r.t. the input-noise pair $ \ilparenthesis{\bh_{k+1},v_{k+1}} $. The randomness in $\bvartheta_{k+1}$,  if there is any, is not involved (note  that the computation of (\ref{eq:QuadraticLoss}) and (\ref{eq:LinearGradient}) is infeasible in reality due to the unavailability to carry out the expectation in (\ref{eq:QuadraticLoss}) and the unknown target $\bvartheta_{k+1}$). The most accessible information is the instantaneous gradient: 
\begin{eqnarray}\label{eq:LinearStochasticGradient}
\hbg_k\ilparenthesis{\btheta} &\equiv & \frac{1}{2} \frac{\partial\ilbracket{ \parenthesis{ z_{k+1} - \bh_{k+1}^\transpose\btheta }^2 }}{\partial\btheta}\nonumber\\
&=&\bh_{k+1}\ilparenthesis{\bh_{k+1}^\transpose\btheta-z_{k+1}}. 
\end{eqnarray} Estimate 
(\ref{eq:LinearStochasticGradient}) is a stochastic gradient due to the derivative of the argument inside the expectation operator in (\ref{eq:QuadraticLoss}). Besides, (\ref{eq:LinearStochasticGradient}) is an unbiased estimator of (\ref{eq:LinearGradient}). 

  When applying in linear regression models, $\bh_k$ represents the gradient of the predicted model output w.r.t.  the parameter $\btheta$ in the model (\ref{eq:LinearMeasurement}), and  then    the  recursion (\ref{eq:basicSA})  reduces to the LMS algorithm. Explicitly, the stochastic gradient at step $k$ is calculated as: 
\begin{eqnarray}\label{eq:LMS}
\hbtheta_{k+1} &=& \hbtheta_{k} - \gain_{k+1} \bh_{k+1} \ilparenthesis{  \bh_{k+1}^\transpose\hbtheta_{k}- z_{k+1}  } \nonumber \\
&=& \ilparenthesis{\bI_p - \gain \bh_{k+1}\bh_{k+1}^\transpose}\hbtheta_{k} +\gain_{k+1} \bh_{k+1} z_{k+1}, \,\,\, k \in\natural,
\end{eqnarray}   where $\hbtheta_0$ is chosen arbitrarily or with a priori information,  and is  assumed to have a finite second moment. 
 
Comparing (\ref{eq:LinearGradient}) and (\ref{eq:LinearStochasticGradient}), the error term in (\ref{eq:gGeneral}) becomes:
\begin{eqnarray}\label{eq:LMSerror}
 \be_k\ilparenthesis{\btheta} 
&=&\begin{dcases}
-\bh_{k+1} v_{k+1}, \quad \text{when } \bh_{k+1} \text{ is deterministic,}\\
\ilbracket{ \bh_{k+1}\bh_{k+1}^\transpose-\E\ilparenthesis{\bh_{k+1}\bh_{k+1}^\transpose} } \ilparenthesis{\btheta-\bvartheta_{k+1}} - \bh_{k+1} v_{k+1}, \\
\quad \quad\quad \text{when }\bh_{k+1}\text{ is stochastic yet independent of }v_{k+1}.
\end{dcases}
\end{eqnarray}
Here $\be_k\ilparenthesis{\btheta}$ is mean-zero as long as the measurement noise $v_k$ in (\ref{eq:LinearMeasurement}) is mean-zero (as assumed above).  
\begin{rem}
	The   change of $\bvartheta_k$ is called state evolution.  Naturally, all the randomnesses in the dynamic system, which consists of (\ref{eq:LinearMeasurement}) and the state evolution, arise from the $\ilset{\bvartheta_k}$ in the state evolution and the input-noise pair $\ilset{\bh_k,v_k}$ in the measurement equation. 
	Note that under A.\ref{assume:BoundedVariation}, the sequence $\ilset{\bvartheta_k}_{k\ge 0}$ is allowed to be either stochastic or fully deterministic.
\end{rem}

Note that Assumptions A.\ref{assume:ErrorWithBoundedSecondMoment}\textendash A.\ref{assume:BoundedVariation} listed previously for (\ref{eq:basicSA})  can be specialized to    the  LMS algorithm (\ref{eq:LMS}) as in  \cite{zhu2015error}.
The required assumptions are (1) the design vector sequence $\ilset{\bh_k}$ is random\footnote{For the case where $\bh_k$ is deterministic, see \cite{guo1990estimating}} and has a bounded $L_2$ norm uniformly across $k$, (2)  $v_k$ in (\ref{eq:LinearMeasurement})  is mean-zero and has a bounded variance of $\upsigma_{v_k}^2$, and (3) the pair $ \ilset{\bh_k,v_k} $ is independent of $\bvartheta_k$.  Immediately, $\noiseBound_k$ in A.\ref{assume:ErrorWithBoundedSecondMoment} becomes $ \sqrt{\upsigma_{v_{k+1}}^2  \ilparenthesis{\E\ilbracket{\norm{\bh_{k+1}}^2}}} $, $\convexPara_k $ in A.\ref{assume:StronglyConvex} becomes $\uplambda_{\min }\ilparenthesis{\E\ilparenthesis{\bh_{k+1}\bh_{k+1}^\transpose}}$,  and   $\LipsPara_k$ in A.\ref{assume:Lsmooth} becomes $\uplambda_{\max}\ilparenthesis{\E\ilparenthesis{\bh_{k+1}\bh_{k+1}^\transpose}}$.

Prior work on error bounds for the linear case  include \cite{farden1981tracking,macchi1986optimization,guo1995performance}. However, the bounds  therein are usually  not computable, as they require higher-order (higher than second-order) moments information of the design vector $\bh_k$. Admittedly, the error bound  (\ref{eq:PropagationLemma2}) and (\ref{eq:PropagationLemma4}) also requires information regarding  $\E\ilparenthesis{\bh_k\bh_k^\transpose}$,  but the  estimation of $\E\ilparenthesis{\bh_k\bh_k^\transpose}$ on the fly requires 2nd-oder information and that $\bh_k$'s are i.i.d. As explained  in Subsection \ref{subsect:noestimation}, we do not dwell on the estimation issues given that our problem setup only allows a \emph{few} observations. 
Also, there are  numerous works on the  {random-walk} evolvement assumption (based on a  linear model, mainly for LMS): \cite{ljung1990adaptation} and \cite[Chap. 5]{solo1994adaptive}, but they are not as informative and general as our results (\ref{eq:PropagationLemma3}) and (\ref{eq:PropagationLemma5}) that reveal the  dependency explicitly on the gain selection, the noise level, the drift level, and the second-order information. 

\subsubsection*{General Empirical Risk Minimization}

In general empirical risk minimization, given data pairs $ \ilparenthesis{\bh_k, \bz_k} $, we wish to learn a hypothesized  relationship $\bz_k\approx \bvarphi \ilparenthesis{\bh_k }$ for $\bvarphi$ chosen from a family of functions $ \ilset{\bvarphi_ {\btheta}} $ parametrized\footnote{By ``parametrized'' we mean that the mapping from $\btheta$ to $\bvarphi _{\btheta}\ilparenthesis{\cdot}$ is one-to-one. Specifically, $\btheta_1\neq\btheta_2$ implies $\bvarphi _{\btheta_1}\ilparenthesis{\cdot}\neq \bvarphi _{\btheta_2}\ilparenthesis{\cdot}$. Alternatively,  $\bvarphi _{\btheta_1}\ilparenthesis{\cdot}=  \bvarphi _{\btheta_2}\ilparenthesis{\cdot}$ implies $\btheta_1=\btheta_2$.
  } by $\btheta\in\real^p$. That is, (\ref{eq:LinearMeasurement}) becomes
\begin{equation}\label{eq:NonlinearMeasurement}
\bz_k = \bvarphi_{\bvartheta_k} \ilparenthesis{\bh_k } + \bv_k, \quad k\in\natural
\end{equation} where $\bvartheta_k$ is the underlying target parameter which evolves smoothly along the passage of time, and $\bv_k$ is a  disturbance sequence with a mean of $\zero$.  
Note that  the function form of  $\bvarphi_{\btheta}\ilparenthesis{\cdot}$ allows for both the linear representation as in   (\ref{eq:LinearMeasurement}) and nonlinear form, and $\bh_k$ is not necessarily in $\real^p$ due to the potentially nonlinear mapping $\bvarphi_{\btheta}\ilparenthesis{\cdot}$. For the  multiple-input-multiple-output (MIMO)\label{acronym:MIMO} system   (\ref{eq:NonlinearMeasurement}), we aim to use the known input values $\bh_k$ and observed output values $\bz_k$ to estimate and track the underlying MIMO system parameter $\bvartheta_k$. 
 The time-varying function we  are trying to 
minimize is 
\begin{equation}\label{eq:NonlinearLoss}
\loss_k\ilparenthesis{\btheta}  = \frac{1}{2} \E\bracket{  \norm{  \bz_{k+1} - \bvarphi_{\btheta} \ilparenthesis{\bh_{k+1}}  } ^2 },
\end{equation} where the expectation in (\ref{eq:NonlinearLoss}) is taken w.r.t. the noise $\bv_k$ in (\ref{eq:NonlinearMeasurement}) and the randomness in $\bh _{k+1}$ if $\bh_{ k+1}$ is random.  Here  $\bh_k$ is the controllable input which may be random, while $\bz_k $ is the output that contains partial information on  $\btheta_k^*$.

Suppose that the measurement noise $\bv_k$ in (\ref{eq:NonlinearMeasurement}) is independent of both $\btheta$ and $\bh_k$, then the derivative of the time-varying loss function  (\ref{eq:NonlinearLoss}) w.r.t. $\btheta$ is 
\begin{align}\label{eq:NonlinearGradient}
\bg_k\ilparenthesis{\btheta } &=  - 
\E\bracket{   \parenthesis{ \frac{\partial\bvarphi_{\btheta}\ilparenthesis{\bh_{k+1}}}{\partial\btheta}  }^\transpose   \parenthesis{\bz_{k+1} - \bvarphi _{\btheta}\ilparenthesis{\bh_{k+1}}}        } \nonumber\\
&= 
\E\bracket{   \parenthesis{ \frac{\partial\bvarphi_{\btheta}\ilparenthesis{\bh_{k+1}}}{\partial\btheta}  }^\transpose   \parenthesis{ \bvarphi _{\btheta}\ilparenthesis{\bh_{k+1}}  -  \bvarphi _{\bvartheta_{k+1}}\ilparenthesis{\bh_{k+1}}          - \bv _{k+1}   }        }  \nonumber\\
&=\begin{dcases}
\parenthesis{ \frac{\partial\bvarphi_{\btheta} \ilparenthesis{\bh_{k+1}}}{\partial\btheta}  }^\transpose \ilparenthesis{  \bvarphi_{\btheta} \ilparenthesis{\bh_{k+1}}    - \bvarphi_{\bvartheta_{k+1}}\ilparenthesis{\bh_{ k+1}}  },\quad \text{ when }\bh_{k+1} \text{ is deterministic}, \\ 
\E\bracket{   \parenthesis{ \frac{\partial\bvarphi_{\btheta}\ilparenthesis{\bh_{k+1}}}{\partial\btheta}  }^\transpose   \parenthesis{ \bvarphi _{\btheta}\ilparenthesis{\bh_{k+1}}  -  \bvarphi _{\bvartheta_{k+1}}\ilparenthesis{\bh_{k+1}}      }        } , \\
\quad\quad\quad\quad\quad \quad\quad\quad\quad\quad\quad\quad \text{ when } \bh_{ k+1} \text{ is stochastic yet independent of }\bv_{k+1},
\end{dcases}  
\end{align}
where we have assumed that the differentiation interchanges with the integral (expectation), and  the expectation  is w.r.t. the data pair $ \ilparenthesis{\bh_{k+1},\bz_{k+1}} $.  The randomness in $\bvartheta_{k+1}$, if there  is any, is not involved. 
  Oftentimes, the joint distribution of $ \ilparenthesis{\bh_k,\bz_k} $
 is unknown. The accessible information is the instantaneous gradient:
 \begin{equation}\label{eq:NonlinearStochasticGradient}
 \hbg_k \ilparenthesis{\btheta} = \parenthesis{ \frac{\partial\bvarphi_{\btheta}\ilparenthesis{\bh_{k+1}}}{\partial\btheta}  }^\transpose   \ilparenthesis{ \bvarphi _{\btheta}\ilparenthesis{\bh_{k+1}}  - \bz_{k+1} } . 
 \end{equation}
 Hence,   SA recursion at step $k$ is calculated as: 
 \begin{align}\label{eq:ERM}
 \hbtheta_{k+1} &= \hbtheta_k -\gain_k \hbg_k \ilparenthesis{\hbtheta_k} \nonumber\\
 &= \hbtheta_k  - \gain_k      \parenthesis{\left.  \frac{\partial\bvarphi_{\btheta}\ilparenthesis{\bh_{k+1}}}{\partial\btheta} \right|_{\btheta=\hbtheta_k }  }^\transpose   \ilparenthesis{ \bvarphi _{\hbtheta_k}\ilparenthesis{\bh_{k+1}}  - \bz_{k+1} } , \quad k\in\natural
 \end{align}
 with an initialization $\hbtheta_0$ being deterministic or stochastic but with the finite second moment.
 We immediately see that the LMS algorithm (\ref{eq:LMS}) is a special case of the general principle of  empirical risk minimization (\ref{eq:ERM}). Comparing (\ref{eq:NonlinearGradient}) and (\ref{eq:NonlinearStochasticGradient}), the error term in (\ref{eq:gGeneral}) becomes
\begin{equation}
\be_k\ilparenthesis{\btheta} = \begin{dcases}-
 \parenthesis{ \frac{\partial\bvarphi_{\btheta}\ilparenthesis{\bh_{k+1}}}{\partial\btheta}  }^\transpose    {   \bv_{k+1} } , \text{ when } \bh_{k+1} \text{ is deterministic}\\
 \bracket{   \parenthesis{ \frac{\partial\bvarphi_{\btheta}\ilparenthesis{\bh_{k+1}}}{\partial\btheta}  }^\transpose   \parenthesis{ \bvarphi _{\btheta}\ilparenthesis{\bh_{k+1}}  -  \bvarphi _{\bvartheta_{k+1}}\ilparenthesis{\bh_{k+1}}      }        } \nonumber\\
 \quad - 
\E\bracket{   \parenthesis{ \frac{\partial\bvarphi_{\btheta}\ilparenthesis{\bh_{k+1}}}{\partial\btheta}  }^\transpose   \parenthesis{ \bvarphi _{\btheta}\ilparenthesis{\bh_{k+1}}  -  \bvarphi _{\bvartheta_{k+1}}\ilparenthesis{\bh_{k+1}}      }        } -    \parenthesis{ \frac{\partial\bvarphi_{\btheta}\ilparenthesis{\bh_{k+1}}}{\partial\btheta}  }^\transpose   \bv_{k+1},\nonumber\\
\quad\quad\quad\quad\quad \quad\quad\quad\quad\quad\quad\quad \text{ when } \bh_{ k+1} \text{ is stochastic yet independent of }\bv_{k+1},
\end{dcases}
\end{equation}
Here $\be_k\ilparenthesis{\btheta}$ has a mean of $\zero$ as long as the measurement noise $\bv_k$ in (\ref{eq:NonlinearMeasurement}) has a mean of 
$\zero$.

Now the Assumptions A.\ref{assume:ErrorWithBoundedSecondMoment}\textemdash A.\ref{assume:BoundedVariation}  listed for the general SA algorithm (\ref{eq:basicSA}) can be specialized for (\ref{eq:NonlinearStochasticGradient}). 
The required assumptions are   (1)  the family of functions $\ilset{\bvarphi_{\btheta}}$ is parametrized by $\btheta$, and every the second-order partial derivatives of  $\bvarphi_{\btheta}\ilparenthesis{\cdot}$ w.r.t. $\btheta$, which is a 3-dimensional matrix (a.k.a. tensor), are continuous in $\btheta$, \remove{(0)the design vector sequence $\ilset{\bh_k}$ is random and has a bounded $L_2$ norm uniformly across $k$, }(2)  $\bv_k$ in (\ref{eq:NonlinearMeasurement})  has a mean of $\zero$ and a covariance matrix $\bSigma_{\bv_k}$ with bounded entries,   (3) the pair $ \ilset{\bh_k,\bv_k} $ is independent of $\bvartheta_k$.   Immediately, $\noiseBound_k$ in A.\ref{assume:ErrorWithBoundedSecondMoment} becomes $$\sup_{\btheta\in\real^p} \sqrt{ \tr \Bigg(   \parenthesis{ \frac{\partial\bvarphi_{\btheta}\ilparenthesis{\bh_{k+1}}}{\partial\btheta}  }\parenthesis{ \frac{\partial\bvarphi_{\btheta}\ilparenthesis{\bh_{k+1}}}{\partial\btheta}  }^\transpose  \bSigma_{\bv_{k+1} }    \Bigg)         }, $$ $\convexPara_k $ in A.\ref{assume:StronglyConvex}  and   $\LipsPara_k$ in A.\ref{assume:Lsmooth} become $ \inf_{\btheta} \uplambda_{\min} \ilbracket{\partial\bg_k\ilparenthesis{\btheta}/\partial\btheta} $ and $ \sup_{\btheta} \uplambda_{\max} \ilbracket{\partial\bg_k\ilparenthesis{\btheta}/\partial\btheta} $, where $\bg_k\ilparenthesis{\cdot}$ is defined in  (\ref{eq:NonlinearGradient}). We omit the detailed expression here as it involves the notion of tensor and the definition of multiplying a tensor by a matrix, which is not the focus here.

\subsection{General Adaptive Algorithms }\label{subsect:StaticKF}
Subsection~\ref{sect:LMS}  discusses the scenario where the $\ilset{\bvartheta_k}$ evolution is unknown, and the general form of SA algorithm (\ref{eq:basicSA}) is used to track the time variation. Nonetheless, if the evolution law   is partially revealed, it should be taken into consideration in the time-varying parameter estimation along the lines of \cite[Eq. (3.19) on p. 84]{spall2005introduction}. For example, the prediction step in KF, which is similar to (\ref{eq:predictionstep}) to appear, makes direct use of the linear state-space model. 

\subsubsection*{Static Kalman Filtering }

In general, the classical KF algorithm cannot be rearranged  as a special case of (\ref{eq:basicSA}).  
 Here is an exception: consider  the case when there are no dynamics, that is, 
\begin{numcases}{}
\text{Static Model:}\,\,\bvartheta_k  =   \bvartheta, \label{eq:static}&\nonumber\\
\text{Measurement:}\,\,\bz_k=\bH_k\bvartheta_k+\bv_k, & \label{eq:MIMO} 
\end{numcases} and  where the observation $\bz_k\in\real^{p
'}$ (usually $1\le p'\ll p$), the   matrix  $\bH_k\in\real^{p'\times p}$ is known,  the independent sequence $\set{\bv_k}$ satisfies    $\E\ilparenthesis{\bv_k} = \zero$ and $\Cov\ilparenthesis{\bv_k}=\bR_k$ (which is symmetric). Moreover, $\bvartheta_0$ is random with a known mean and a known variance $\bP_0$. 
\begin{rem}
	The general framework of Kalman filtering that allows time-varying $\bvartheta_k$ (a nonzero $\driftBound_k$) pertains to second-order derivative w.r.t. $\btheta$ and second-order noise statistics for both the modeling noise  and the measurement noise, and cannot be put into the first-order SA algorithm framework (\ref{eq:basicSA}).  Hence it is not discussed here. 
\end{rem}  Then the prediction of the state estimation and the covariance estimate from the KF are 
\begin{equation}\label{eq:StaticKFpred}
\begin{dcases}
\hat{\btheta}_ {\given{k}{k-1}} = \hat{\btheta}_{k-1},\quad \text{with }\hat{\btheta}_0 = \E\ilparenthesis{\btheta_0 }, \\
 {\bP}_{\given{k}{k-1}} =  {\bP}_{k-1},\quad\text{with } {\bP}_0 = \E \ilbracket{  \ilparenthesis{\hbtheta_0- \btheta_0} \ilparenthesis{\hbtheta_0- \btheta_0} ^\transpose }, 
\end{dcases}
\end{equation} 
and the updating step is
\begin{numcases}{} 
\bK_k =  {\bP}_{\given{k}{k-1}}\bH_k^\transpose\ilparenthesis{   \bH_k {\bP}_{\given{k}{k-1}} \bH_k^\transpose + \bR_k    }^{-1},  & \label{eq:StaticKFgain}\\
\hat{\btheta}_k = \hat{\btheta}_{\given{k}{k-1}} + \bK_k\ilparenthesis{\bz_k - \bH_k \hat{\btheta}_{\given{k}{k-1}}},  & \label{eq:StaticKFupdate}\\
 {\bP}_k = \ilparenthesis{\bI-\bK_k\bH_k} {\bP}_{\given{k}{k-1}}.&\label{eq:StaticKFcov}
\end{numcases}
Above updating formulas imply:
\begin{equation}
 {\bP}_{k+1}^{-1} =  {\bP}_k^{-1} + \bH_{k}^\transpose\bR_k^{-1} \bH_k \quad\text{and}\quad \bK_k = {\bP}_k\bH_k^\transpose\bR_k^{-1}. 
\end{equation}
For  MIMO system (\ref{eq:MIMO}), the time-varying function we are trying to minimize at each sampling instance $\uptau_k$ is
\begin{equation}\label{eq:LossStaticKF}
\loss_k\ilparenthesis{\btheta } = \E\bracket{ \frac{1}{2}  \ilparenthesis{\bz_{k+1}- \bH_{k+1}\btheta}^\transpose\bR_{k+1}^{-1} \ilparenthesis{\bz_{k+1}- \bH_{k+1}\btheta}}, 
\end{equation}
where the expectation is taken w.r.t. the noise $\bv_k$ in (\ref{eq:MIMO}).  Suppose that the measurement noise $\bv_k$ in (\ref{eq:MIMO}) is independent of the valuation point $\btheta$, then the \emph{instantaneous} gradient  of the time-varying loss function 
(\ref{eq:LossStaticKF})   can be obtained by taking the derivative of the quantity inside of the expectation operator in (\ref{eq:LossStaticKF}) w.r.t. $\btheta$: 
\begin{eqnarray}\label{eq:LossStaticKFgradient}
\hbg_k\ilparenthesis{\btheta} &=&  
\frac{1}{2} \frac{\partial\ilbracket{   \ilparenthesis{\bz_{k+1}- \bH_{k+1}\btheta}^\transpose\bR_{k+1}^{-1} \ilparenthesis{\bz_{k+1}- \bH_{k+1}\btheta}}}{\partial\btheta} \nonumber\\ &=&\bH_{k+1}^\transpose\bR_{k+1}^{-1} \ilparenthesis{\bH_{k+1}\btheta-\bz_{k+1}}\nonumber\\
&=& {\bP}_{k+1}^{-1} \bK_{k+1} \ilparenthesis{\bH_{k+1} \btheta-\bz_{k+1}},  
\end{eqnarray}
whose error term   defined in (\ref{eq:gGeneral})         is 
\begin{eqnarray}\label{eq:LossStaticKFerror}
\be_k \ilparenthesis{\btheta} = \hbg_k\ilparenthesis{\btheta} - \bg_k\ilparenthesis{\btheta}= \bH_{k+1}^\transpose\bR_{k+1}^{-1} \bH_{k+1}\ilparenthesis{\btheta - \bvartheta_{k+1}} - \bH_{k+1}^\transpose\bR_{k+1}^{-1} \bv_{k+1}.
\end{eqnarray}
\begin{rem}
The loss function construction (\ref{eq:LossStaticKF}), the stochastic gradient form (\ref{eq:LossStaticKFgradient}), and the error form (\ref{eq:LossStaticKFerror}) for MIMO model (\ref{eq:MIMO}) are   natural extensions of   (\ref{eq:QuadraticLoss}), (\ref{eq:LinearStochasticGradient}), and (\ref{eq:LMSerror}) for MISO model (\ref{eq:LinearMeasurement}). 
\end{rem}
Then   (\ref{eq:StaticKFupdate}) becomes
\begin{eqnarray}
\hat{\btheta}_{k+1} &= & \hat{\btheta}_k + \bK_{k+1} \ilparenthesis{\bz_{k+1}- \bH_{k+1} \hat{\btheta}_k}\nonumber\\
&=&    {\btheta}_k - \hat{\bP}_{k+1} \hbg_k \ilparenthesis{\hat{\btheta}_k}, 
\end{eqnarray}
which aligns with the SGD algorithm (\ref{eq:basicSA}) where $\hbg_k\ilparenthesis{\hbtheta_k}$ is replaced  by  (\ref{eq:Ystationary}),  except  that  the scalar gain $\gain_k$ is replaced by the matrix gain $\hat{\bP}_{k+1}$. 

Similar to Subsection \ref{sect:LMS}, the assumptions A.\ref{assume:ErrorWithBoundedSecondMoment}--A.\ref{assume:BoundedVariation} can be specialized for the static Kalman filter algorithm (\ref{eq:StaticKFupdate}). That is, as long as   the observation matrix $\bH_k\in\real^{p'\times p}$ has full (row) rank and has a bounded $\ell_2$ norm uniformly across $k$, $\bv_k$ in (\ref{eq:MIMO}) is mean-zero and has a nonsingular\footnote{ If the observations are nearly perfect, then $\bR_k$ is close to $\zero$. We do not dive into schemes in handling the consequent computational instability here.   } covariance matrix $\bR_k$ uniformly for all $k$, and $\bv_k$ is independent of $\bvartheta_k$. Specifically, $\noiseBound_k$ in A.\ref{assume:ErrorWithBoundedSecondMoment} becomes $ \sqrt{    \tr \ilparenthesis{\bH_{k+1} ^\transpose \bR_{k+1}^{-\transpose}    \bH _{k+1} } } $, $\convexPara_k$ in A.\ref{assume:StronglyConvex}  becomes $\uplambda_1\ilparenthesis{  \bH_{k+1}^\transpose\bR_{k+1}^{-1} \bH_{k+1} }$,    $\LipsPara_k$ in A.\ref{assume:Lsmooth} becomes $\uplambda_p\ilparenthesis{  \bH_{k+1}^\transpose\bR_{k+1}^{-1} \bH_{k+1} }$,  and $\driftBound_k$ in A.\ref{assume:BoundedVariation} reduces to 0 in the static model (\ref{eq:static}).

\subsubsection*{General  Dynamic Model With \emph{Known} Evolution}
 As mentioned before, the classical KF and extended KF (EKF) algorithm cannot be rearranged as a special case of (\ref{eq:basicSA}). Here we mention a simple tracking algorithm (\ref{eq:}) that  does not deal with the matrix multiplication and matrix inversion arising in computing the Kalman gain in KF/EKF, and provides a tracking error bound in Proposition~\ref{prop:KF}.
 
\remove{ 
	Time-varying  MIMO  systems  have   been extensively studied, but it is usually assumed that  the parameter process evolves either deterministic continuously or stochastically due to some mean-zero Gaussian disturbance. These models assume that parameter changes are small when they occur, which allows for more tractability in establishing convergence or error bounds. }

In many applications, we hope to  estimate a time-varying quantity that evolves with time  according to a nonlinear state equation:
\begin{equation}\label{eq:DynamicModel}
\bvartheta_{k+1} = \bm{f}_k\ilparenthesis{\bvartheta_k} + \bw_k, \quad k=0,1,2,\cdots,
\end{equation}
where the evolution function form of $\bm{f}_k\ilparenthesis{\cdot}:\real^p\mapsto\real^p$ is \emph{known}, and  $\bw_k$ is a mean-zero stochastic process. The  incomplete information about the $\bvartheta_k$ is available through observations $\bz_k$ in the following form:
\begin{equation}\label{eq:NonlinearMIMO}
\bz_k=\bh_k\ilparenthesis{\bvartheta_k}+\bv_k,\quad k = 1, 2, \cdots,
\end{equation} where the  measurement function form of  $\bm{h}_k\ilparenthesis{\cdot}:\real^p\mapsto\real^p$ is \emph{known}, and $\bv_k$ is a mean-zero stochastic process. Equation 
(\ref{eq:NonlinearMIMO})  is naturally a nonlinear extension of (\ref{eq:MIMO}).  
Following \cite[Eq. (3.19) on p. 84]{spall2005introduction}, we proceed the time-varying parameter estimation via
\begin{equation}\label{eq:predictionstep}
\begin{dcases}
 \text{Prediction step:}\quad  & \hbtheta_{\given{k+1}{k}} = \bm{f}_k \ilparenthesis{\hbtheta_k}, \\
\text{Updating step:}\quad  & \hbtheta_{k+1} = \hbtheta_{\given{k+1}{k}} + \gain_{k+1} \ilparenthesis{ \bz_{k+1} - \bh_{k+1}\ilparenthesis{ \hbtheta_{\given{k+1}{k}} }},
 \end{dcases}
\end{equation} where $\gain_{k+1}$ is a scalar gain satisfying certain conditions.  Combined, the recursion for $\hbtheta_k$ is:
\begin{equation}\label{eq:}
\hbtheta_{k+1} = \bm{f}_k\ilparenthesis{\hbtheta_k} + \gain_{k+1} \ilbracket{\bz_{k+1} - \bm{h}_{k+1}\ilparenthesis{\bm{f}_k\ilparenthesis{\hbtheta_k}}},\quad k = 0,1,2,\cdots. 
\end{equation}

\begin{prop}\label{prop:KF}
	Consider the state equation (\ref{eq:DynamicModel}) and the measurement equation (\ref{eq:NonlinearMIMO}). Assume that the following conditions hold. 
	\begin{enumerate}[(a)]
		\item \label{item:wk} The noise process  $\bw_k$ has a mean of $\zero$ and a covariance matrix of $\bQ_k$. The sequence $ \ilset{\bw_k} $ is an independent sequence. 
		
		\item   \label{item:vk}   The noise process  $\bv_k$ has a mean of $\zero$ and a covariance matrix of $\bR_k$. The sequence $ \ilset{\bv_k} $ is an independent sequence.

		\item \label{item:independence} All the noises in $\ilset{\bw_k}$ are independent of all the noises in $ \ilset{\bv_k} $. 
		\item \label{item:KFcontinous}
		For every $k$, the squared-matrix-valued functions of $\btheta$,   $ \diff  \bm{f}_k\ilparenthesis{\btheta}  / \diff\btheta^\transpose \equiv \dot{\bm{f}}_k\ilparenthesis{\btheta }  $ and  $ \diff  \bm{h}_k\ilparenthesis{\btheta}  / \diff\btheta^\transpose \equiv \dot{\bm{h}}_k\ilparenthesis{\btheta} $, are continuous w.r.t. $\btheta$. 
		\item \label{item:KFsign}  $ \dot{\bm{f}}_k\ilparenthesis{\btheta }  $  is either  positive-definite or negative definite for all $\btheta$. This also holds  for $ \dot{\bm{h}}_k\ilparenthesis{\btheta } $ and $ \bm{h}_{k+1} \ilparenthesis{\bm{f}_k \ilparenthesis{\btheta}} $. 
	\end{enumerate} When discussing positive/negative definiteness, we may write  $ \dot{\bm{f}}_k$ ($\dot{\bm{h}}_k$) instead of $ \dot{\bm{f}}_k\ilparenthesis{\btheta }  $ ($ \dot{\bm{h}}_k\ilparenthesis{\btheta }  $). 
Following the rational explained in Subsection \ref{subsect:ratio},  denote  $     \LipsPara_{k}^{\bm{f}} \equiv  \uplambda_p\ilparenthesis{ |{ \dot{\bm{f}}_k }|  } $,  $  \convexPara_k^{\bm{f}} \equiv  \uplambda_1\ilparenthesis{ |{ \dot{\bm{f}}_k }|  }$, $\ratio_k ^{\bm{f}} \equiv  \LipsPara_{k}^{\bm{f}} /\convexPara_k^{\bm{f}}   $ Similarly define $\LipsPara_{k}^{\bm{h}} $, $\convexPara_ {k}^{\bm{h}} $, and $\ratio_{k}^{\bm{h}} $. 
Suppose $\gain_k$ is picked such that  
\begin{equation}\label{eq:KFgain}
    \frac{  \convexPara_ {k}^{\bm{f}}       }{  \LipsPara_{k+1}^{\bm{h}}  \LipsPara_{k}^{\bm{f}}      }            \le 
\sign \ilparenthesis{\dot{\bm{h}}_{k+1} \dot{\bm{f}}_k} \gain_{k+1} <      \frac{  \convexPara_ {k}^{\bm{f}}  +1      }{  \LipsPara_{k+1}^{\bm{h}}  \LipsPara_{k}^{\bm{f}}      }, 
\end{equation} where $\sign\ilparenthesis{\cdot}$ is positive/negative if the argument square matrix is positive/negative definite.  Then we have an asymptotic bound
\begin{equation}\label{eq:KFasymp}
\limsup_{k\to\infty}  \E\ilparenthesis{\norm{ \hbtheta_k-\bvartheta_k }^2} \le \limsup_{k}  \frac{     \gain_{k+1}^2\tr\ilparenthesis{\bR_k} +       \ilparenthesis{    \abs{\gain_{k+1}}  \LipsPara_{k+1}^{\bm{h}} 
		+ \ilparenthesis{-1}^{ \sign\ilparenthesis{  \dot{\bm{h}}_{k+1} \dot{\bm{f}}_k   } }   } ^2    \tr\ilparenthesis{\bQ_k}     }{  1-  \ilparenthesis{  \abs{\gain_{k+1}}       \LipsPara_{k+1}^{\bm{h}}       \LipsPara_{k}^{\bm{f}}     -   \convexPara_{k}^{\bm{f}}                   } ^ 2}.  
\end{equation}
\end{prop}
\begin{proof}  Notice that
	\begin{align}\label{eq:KFgeneral1}
	&	\hbtheta_{k+1} - \bvartheta_{k+1} \nonumber\\
	&\,\, = \bm{f}_k \ilparenthesis{\hbtheta_k }+ \gain_{k+1}  \ilbracket{\bz_{k+1} - \bh_{k+1}\ilparenthesis{\bm{f}_k\ilparenthesis{\hbtheta_k}}} - \bvartheta_{k+1} \nonumber\\
	&\,\,   =     \bm{f}_k \ilparenthesis{\hbtheta_k }+\gain_{k+1}    \ilbracket{  \bm{h}_{k+1} \ilparenthesis{\bvartheta_{k+1}} + \bv_{k+1}    - \bh_{k+1}\ilparenthesis{ \bm{f}_k\ilparenthesis{\hbtheta_k}  }}    -  \ilparenthesis{\bm{f}_k\ilparenthesis{\bvartheta_k} + \bw_k} \nonumber\\
	&\,\,  = \ilbracket{\bm{f}_k\ilparenthesis{\hbtheta_k}-\bm{f}_k\ilparenthesis{\bvartheta_k}} +\gain_{k+1}   \ilbracket{  \bh_{k+1}\ilparenthesis{  \bm{f}_k\ilparenthesis{\bvartheta_k} + \bw_k  } - \bh_{k+1}\ilparenthesis{\bm{f}_k\ilparenthesis{\hbtheta_k}}     } + \gain_{k+1}   \bv_{k+1}  - \bw_k \nonumber\\
	&\,\,  = \left. \frac{\diff\bm{f}_k\ilparenthesis{\btheta}}{\diff\btheta}\right| _{\btheta =\eta _1 \hbtheta_k+\ilparenthesis{1-\eta_1}\bvartheta_k  } \ilparenthesis{\hbtheta_k-\bvartheta_k } +  \gain_{k+1}   \bv_{k+1}  - \bw_k   \nonumber\\
	&\,\, \quad + \gain_{k+1}   \set{ \left. \frac{ \diff\bh_{k+1}\ilparenthesis{\bx} }{\diff \bx} \right|  _{\bx = \eta_2 \ilparenthesis{ \bm{f}_k\ilparenthesis{\bvartheta_k} + \bw_k } + \ilparenthesis{1-\eta_2} \bm{f}_k\ilparenthesis{\hbtheta_k }  }  \bracket{ \bw_k -       \parenthesis{  \bm{f}_k\ilparenthesis{\hbtheta_k} - \bm{f}_k\ilparenthesis{\bvartheta_k}  }       }    } \nonumber\\
	&\,\,  = \left. \frac{\diff\bm{f}_k\ilparenthesis{\btheta}}{\diff\btheta}\right| _{\btheta =\eta _1 \hbtheta_k+\ilparenthesis{1-\eta_1}\bvartheta_k  } \ilparenthesis{\hbtheta_k-\bvartheta_k } +  \gain_{k+1}   \bv_{k+1}  - \bw_k   \nonumber\\
	&\,\, \quad + \gain_{k+1}   \bracket{ \left. \frac{ \diff\bh_{k+1}\ilparenthesis{\bx} }{\diff \bx} \right|  _{\bx = \eta_2 \ilparenthesis{ \bm{f}_k\ilparenthesis{\bvartheta_k} + \bw_k } + \ilparenthesis{1-\eta_2} \bm{f}_k\ilparenthesis{\hbtheta_k }  }  \parenthesis{ \bw_k -   \left.  \frac{\diff \bm{f}_k }{\diff\btheta} \right| _{\btheta= \eta_3 \hbtheta_k + \ilparenthesis{1-\eta_3}\bvartheta_k   }   \ilparenthesis{\hbtheta_k-\bvartheta_k}    }    } \nonumber\\
	&\,\,  \equiv  \parenthesis{   \left. \frac{\diff\bm{f}_k\ilparenthesis{\btheta}}{\diff\btheta}\right| _{\btheta = \tilde{\btheta}_1 }        - \gain_{k+1}            \left. \frac{ \diff\bh_{k+1}\ilparenthesis{\bx} }{\diff \bx} \right|  _{\bx =  \tilde{\btheta}_2 }                \left.  \frac{\diff \bm{f}_k }{\diff\btheta} \right| _{\btheta= \tilde{\btheta}_3  }      } \ilparenthesis{\hbtheta_k-\bvartheta_k} \nonumber\\
&\,\,\quad 	+ \gain_{k+1}  \bv_{k+1}  - \parenthesis{\bI -\gain_{k+1}  \left.  \frac{\diff \bm{f}_k }{\diff\btheta} \right| _{\btheta= \tilde{\btheta}_3  }    }  \bw_k
	\end{align}
	where the second   equation uses both (\ref{eq:DynamicModel}) and (\ref{eq:NonlinearMIMO}), and   the fourth equation uses mean-value theorem and assumption \ref{item:KFcontinous}. In the last line, we use $\tilde{\btheta} _1\equiv \eta _1 \hbtheta_k+\ilparenthesis{1-\eta_1}\bvartheta_k$,   $\tilde{\btheta}_2 \equiv \eta_2 \ilparenthesis{ \bm{f}_k\ilparenthesis{\bvartheta_k} + \bw_k } + \ilparenthesis{1-\eta_2} \bm{f}_k\ilparenthesis{\hbtheta_k } $, and   $ \tilde{\btheta}_3\equiv   \eta_3 \hbtheta_k + \ilparenthesis{1-\eta_3}\bvartheta_k$.    Let us square (\ref{eq:KFgeneral1}) and taking expectations over the randomness in  $\bvartheta_0$, $\bw_0$,$\cdots$,$\bw_k$, $\hbtheta_0$, $\bv_1$,$\cdots$,$\bv_{k+1}$, we have the following 
	\begin{align}
& \E  	\ilparenthesis{\norm{\hbtheta_{k+1} - \bvartheta_{k+1}}^2}  \nonumber\\
&\,\,= \E   \parenthesis{\norm{  \parenthesis{\left. \frac{\diff\bm{f}_k\ilparenthesis{\btheta}}{\diff\btheta}\right| _{\btheta = \tilde{\btheta}_1 }        - \gain_{k+1}            \left. \frac{ \diff\bh_{k+1}\ilparenthesis{\bx} }{\diff \bx} \right|  _{\bx =  \tilde{\btheta}_2 }                \left.  \frac{\diff \bm{f}_k }{\diff\btheta} \right| _{\btheta= \tilde{\btheta}_3  }  } \ilparenthesis{\hbtheta_k-\bvartheta_k}}^2 }\nonumber\\
&\,\,\quad  + \gain_{k+1}^2 \tr\ilparenthesis{\bR_k}  + \E \parenthesis{  \norm{  \parenthesis{\bI-\gain_{k+1}\left. \frac{ \diff\bh_{k+1}\ilparenthesis{\bx} }{\diff \bx} \right|  _{\bx =  \tilde{\btheta}_2 }  } \bw_k   }^2  }\nonumber\\
&\,\,\le         \ilparenthesis{  \abs{\gain_{k+1}}       \LipsPara_{k+1}^{\bm{h}}       \LipsPara_{k}^{\bm{f}}     -   \convexPara_{k}^{\bm{f}}                   } ^ 2   \E	\ilparenthesis{\norm{\hbtheta_{k} - \bvartheta_{k}}^2} +  \gain_{k+1}^2\tr\ilparenthesis{\bR_k} \nonumber\\ 
	&\,\,\quad     \ilparenthesis{    \abs{\gain_{k+1}}  \LipsPara_{k+1}^{\bm{h}} 
 + \ilparenthesis{-1}^{ \sign\ilparenthesis{  \dot{\bm{h}}_{k+1} \dot{\bm{f}}_k   } }   } ^2    \tr\ilparenthesis{\bQ_k}. 
	\end{align}
because of  assumptions \ref{item:wk}\textendash \ref{item:KFsign}, and the gain selection (\ref{eq:KFgain}). Furthermore,  the coefficient $ \ilparenthesis{  \abs{\gain_{k+1}}       \LipsPara_{k+1}^{\bm{h}}       \LipsPara_{k}^{\bm{f}}     -   \convexPara_{k}^{\bm{f}}                   } ^ 2 $ is guaranteed to be in $\left[0,1\right)$  when (\ref{eq:KFgain}) holds. 
Now following the derivation immediately after equation (\ref{eq:Prop1-6}) in the proof for Theorem~\ref{thm:main1}, we can obtain the asymptotic bound (\ref{eq:KFasymp}).  
\end{proof}

 \section{Brief Summary}

 Note that in time-varying scenarios as in Section~\ref{sect:ProlemSetup}, the concentration result, instead of the improbable convergence,  is the best we can hope for: $ \norm{\hbtheta_k-\bvartheta_k} $ can be made  small in certain statistical sense as $k$ gets  large, where $\bvartheta_k$ is the time-varying target and $\hbtheta_k$ is the corresponding SA estimate.  
 Under the model assumptions listed in Section~\ref{sect:ModelAssumptions}, i.e., the observational noise level $\noiseBound_k  $ in A.\ref{assume:ErrorWithBoundedSecondMoment},  the strong convexity parameter $\convexPara_k$ in A.\ref{assume:StronglyConvex}, the  Lipschitz continuity $\LipsPara_k $ in A.\ref{assume:Lsmooth},  
 and the expected drift magnitude $\driftBound_k $ in A.\ref{assume:BoundedVariation} are \emph{known},
 we may implement Algorithm~\ref{algo:basicSA}.  
 The tracking performance 
 of  non-diminishing gain SA algorithms is guaranteed by  a computable  bound on MAD/RMS presented in    (\ref{eq:PropagationLemma2}) and in  (\ref{eq:PropagationLemma4}), which is useful in   the analysis of finite-sample performance.
 The practical aspects  of the finite-sample error bound for the recursion  (\ref{eq:basicSA}) is  listed below. 
 \begin{itemize}
 	\item The restrictions placed  on the model of the time-varying parameter is  {lenient} compared to other assumed form of state equation. The  only  imposed assumption is that  the average  distance between two consecutive underlying parameters is  strictly bounded from above. This modest assumption does not eliminate jumps in the target, and also allows the target to vary stochastically.
 	
 	 	\item A.\ref{assume:ErrorWithBoundedSecondMoment} allows $\hbg_k\ilparenthesis{\btheta}$ to be a biased estimator of $\bg_k\ilparenthesis{\btheta}$. Therefore,  our discussion embraces many  SA algorithms,  including the special case of the SGD algorithm (\ref{eq:basicSA}) where $\hbg_k\ilparenthesis{\hbtheta_k}$ is substituted by  (\ref{eq:Ystationary}) discussed in \cite{zhu2016tracking} and SPSA in \cite{spall1992multivariate}. 
 	 	 	\item The gain selection strategy in Lemma~\ref{lem:Tobe} or Algorithm~\ref{algo:basicSA}  may provide  some guidance in real-world  gain-tuning. Moreover, the MAD/RMS bound informs us that the gain $\gain_k$ can be neither  too large nor too small\textemdash this contrasts with most prior works that claim the tracking error can be made smaller by decreasing the constant stepsize $\gain$. This is intuitive, as  the ability to track time variations in $\bvartheta_k$ is lost if the step-size is made too small. 
 	 	
 	\item 
 	Both error bounds are computable as long as we have access to the noise level, the drift level, and the Hessian of the underlying loss function.    
 	The case of interest requires the strong convexity of the time-varying loss function (sequence), but  our  tracking error bound  is favorably informative under reasonable assumptions on the evolution of the true parameter being estimated.   Note that our quantification for tracking capability within finite-iterations of the non-diminishing gain SA algorithm in terms of a   computable error bound, can also apply to the general nonlinear  SA literature.  These two characteristics make our discussion different from \cite{eweda1985tracking,wilson2019adaptive}.

 	\item 
 	To the best of our knowledge, there are no existing approaches in estimation theory that produce  a sequence of estimates for  a time-varying minimization/root-finding problem, under only the     A.\ref{assume:BoundedVariation}  without any further  stringent state evolution assumption.

 \end{itemize}
 
 In a nutshell, the iterate $\hbtheta_k $ provides an estimate of the optimum point $\bvartheta_k$ with a certain accuracy, and the tracking errors of using $\hbtheta_k$ as an estimate for $\bvartheta_k$ is stable for all time.


\chapter{Concentration Behaviors   }\label{chap:Limiting}

Chapter \ref{chap:FiniteErrorBound} develops a      MAD/RMS bound  for   SA algorithm (\ref{eq:basicSA}) with non-decaying  gain  on the basis that the sampling frequency is bounded from above; i.e., the actual time elapsed between two consecutive samples, $  \ilparenthesis{  \uptau _{k+1}- \uptau_k   }$, is   strictly bounded away  from zero  for all $k\in\natural$. This chapter will  utilize  the weak convergence argument (reviewed in Section \ref{subsect:WeakConvergence})  to analyze the continuous-time interpolation  of     SA iterates  as the gain sequence approaches zero, and correspondingly   $ \ilparenthesis{\uptau _{k+1}- \uptau_k } $ goes to zero at the same order of  rate. The  requirement  that the sampling frequency (the number of samples per unit time) has to grow as  the gain sequence decreases is needed to closely follow  the perpetually varying target. Even though we analyze the weak convergence limit as the gain     sequence  goes    to zero and the number of samples per unit time grows  inversely proportional to the gain sequence,  the gain sequence needs not to  go to zero in the actual implementation.  

Many prior work  on weak convergence is  developed   on the basis that certain averages\footnote{For example, for every $\btheta$, $ \lim_{k\to\infty}  \abs{ \sum_{i=k} ^{k+j} \gain_i \ilbracket{  \bg_i\ilparenthesis{\btheta} - \bar{\bg} \ilparenthesis{\btheta} } } \to \zero  $ for every $j\in\natural$. }  of  the dynamics  $\bg_k\ilparenthesis{\cdot}$, denoted by   $\bar{\bg}\ilparenthesis{\cdot}$,  do  not depend on time.  This assumption is appropriate if the observed data is a stationary process that evolves on a time scale that is faster  than what is  implied  by the gain sequence.  By ``faster'' we mean that   the gain $\gain_k$ is often very small compared to the time interval at which successive sets of  observations are available. For example, in astronomy,    the meteorological observations may be available every few  hours, while   the stars in the sky, in fact, changes every few  seconds.  
On the contrary, we consider the case where the underlying $\bvartheta_k$   evolves on  a time-scale that is comparable  with what is implied by the gain sequence. By ``comparable'' we mean that   the gain $\gain_k$ is comparable to the time difference between two consecutive observations of the moving target such as submarines and  aircraft. For example, the target submarine/aircraft changes its coordinate every few  seconds, and the agent that needs to track the target also need to adjust its tracking direction every few  seconds.             In this scenario, there is no mismatch between the  model time step (the time difference between two different values of $\bvartheta_k$) and the time interval between the observation (the time difference between two consecutive noisy observations). \remove{Resultingly, $\hbg_k\ilparenthesis{\hbtheta_k}$ is available at every tick of the model clock. }As a result, 
the   mean ODE (to be defined momentarily) is indeed {time-dependent}.

To supplement the tracking capability results  in Chapter  \ref{chap:FiniteErrorBound}, this chapter characterizes the concentration behavior of the estimates using the trajectory of a nonautonomous ODE   via  a    weak convergence argument   and develops a  probabilistic bound.  By ``concentration''  we mean that the recursive estimates spend a majority of time arbitrarily close to some point. Namely, with an arbitrarily high probability and a  small   $\upvarepsilon$, the limit process is concentrated in a  $\upvarepsilon $-neighborhood of some limit set of the mean ODE, if a  limit set exists.     The result in Section  \ref{sect:ConstantGain}   unveils the behavior of the estimates for a small gain sequence and a finite iteration number.  
  Then Section \ref{sect:ProbBound} provides a computable probabilistic bound to supplement the concentration result in Section \ref{sect:ConstantGain}, but under more stringent assumptions.

\section{Concentration Behavior of Constant-Gain Algorithm}\label{sect:ConstantGain}

This section studies the     concentration behavior  of the SA  sequence via   the    properties  of an ODE that represents the   dynamics of the algorithm.   We are to  establish the     following proposition: with an arbitrarily fixed (usually high) probability, for   small gain,  the underlying data change with time should be on a scale that is  {commensurate} with what is  determined by the gain, the iterates are  concentrated in an arbitrarily small neighborhood of some limit set (if one exists) of the mean ODE.

\subsection{Basic Setup and Truncated SA Algorithm}

 Consider the following constrained minimization problem: 
\begin{equation}\label{eq:RootFinding}
\text{for each }k,  \text{ find   }\btheta\in\bTheta \text{ s.t. }\remove{\bg_k\parenthesis{\btheta}=\zero \text{ or }} \norm{\bg_k\parenthesis{\btheta}} \text{ is minimized},
\end{equation}
where $\bg_k\ilparenthesis{\cdot}: \real^p\mapsto\real^p$ is a continuous mapping for each $k$,  $\norm{\cdot}$ is the vector Euclidean norm, and $\bTheta\subset\real^p$ is compact.
Constraints are common in daily applications  due to safety or economic concerns.  
   If  the  (assumed unique)   root  of the vector-valued function $\bg_k\parenthesis{\cdot}  $ lies within $\bTheta$ for all $k$, then (\ref{eq:RootFinding}) is equivalent to  a root-finding problem.  Nonetheless, this   general root-finding problem  handles a sequence of functions $\bg_k\ilparenthesis{\cdot}$ that varies with time   $\uptau_k$, whereas       the  R-M setting \cite{robbins1951stochastic}  deals with locating the root for a single function $\bg\ilparenthesis{\cdot}$ that has no $k$-dependence.  One common application of  (\ref{eq:RootFinding}) is immediate by letting $\bg_k\parenthesis{\btheta}=   \partial \loss _k \ilparenthesis{\btheta} / \partial\btheta$,  where $\loss_k\parenthesis{\btheta
 }$ is a sequence of time-varying loss functions to be minimized. In this case, the problem setup (\ref{eq:RootFinding}) becomes the same as   Subsection \ref{subsect:ProblemSetup}.

Different from the unconstrained SA algorithm in  Chapter \ref{chap:FiniteErrorBound},  this section  discusses    the  projected  SA    algorithm (\ref{eq:truncatedSA1}) per the problem setup (\ref{eq:RootFinding}).

\begin{rem}
	Although this chapter primarily discusses the  constant-gain algorithm, this subsection will define  terms using non-decaying gain. 		This general definition is in anticipation of   further  discussion on  the adaptive gain in Chapter  \ref{chap:AdaptiveGain}, where $\inf _{k} a_k >0$, and $a_k$ is not necessarily constant across $k$. 
\end{rem}

To facilitate later discussion,
we  introduce a projection term $\boldsymbol{\upeta}_k $ and rewrite $-\hbg_k\ilparenthesis{\hbtheta_k}$ as $\bgamma_k$. Then  the projected SA algorithm (\ref{eq:truncatedSA1}) can be rearranged as a stochastic difference equation with a small step size $\gain_k$: 
\begin{equation}\label{eq:truncatedSA2}
 \hbtheta_{k+1}=\hbtheta_k+\gain _k\bgamma_k+a_k\boldsymbol{\upeta}_k, 
\end{equation}
where $a_k \boldsymbol{\upeta}_k = \Proj_{\bTheta}\ilparenthesis{\hbtheta_k+a_k\bgamma_k} - \ilparenthesis{\hbtheta_k+a_k\bgamma_k} $.
That is, if $\ilparenthesis{\hbtheta_k+a_k\bgamma_k}$ is not in $\bTheta$, $a_k\boldsymbol{\upeta}_k$ is   the  vector that takes  $ \ilparenthesis{\hbtheta_k+a_k\bgamma_k}$ back to $\bTheta$ with the shortest Euclidean norm; otherwise, $\boldsymbol{\upeta}_k=\zero$.

\subsection{Rewrite Projected   SA   Algorithm (\ref{eq:truncatedSA2})  as a Stochastic Time-Dependent Process}\label{subsect:Interpolation}
To examine the   behavior of the sequence of estimates $ \ilset{\hbtheta_k } $\remove{and true parameters $\ilset{\bvartheta_k }$ as gain approaching zero}, we  construct  a continuous-time interpolation of the discrete sequence $\ilset{\hbtheta_k} $.  
A natural time scale for the interpolation is   the    gain sequence $a_k $.   \remove{This interpolation facilitates  the effective exploitation of the time scale differences between the iterate process $\ilset{\hbtheta_k}$ and the driving nonstationary process $\set{\bvartheta_k }$.  }With appropriate interpolation, a suitably constructed sequence from  the iterates in  (\ref{eq:truncatedSA1}) will converge to the appropriate limit set  of an ODE determined by the average dynamics. If we further impose the  Lyapunov stability assumption on the   mean ODE, then   the SA estimates ``concentrates'' around the stable point (if it  exists) of the  corresponding ODE.

\begin{enumerate}[i)]
	\item Define   $t_k=\sum_{i=0}^{k-1} a_i   $. Further,  define the time-mapping function $\TimeMapping{t} $ over the domain $ \left[t_0,\infty\right) $ as:  $
	\TimeMapping{t}  =  
	k$   if $t\in\left[t_k,t_{k+1}\right) \text{ for }k\ge 0$. 
	\item  Define the following time-dependent step function:
		\begin{eqnarray}\label{eq:Zbar0} 
 	 {\bZ} \parenthesis{t,\upomega}  &=&    \hbtheta_{\TimeMapping{t}}\ilparenthesis{\upomega}   \,\,\,\, \text{ for  }t\ge t_0. 
	\end{eqnarray}  Note that  $ {\bZ} \parenthesis{\cdot,\upomega}$ is in $D\ilparenthesis{\real\mapsto\real^p}$, the space of functions that are  right-continuous and have left-limits endowed with the  {Skorohod topology}  \cite[Sect. 14]{billingsley1968convergence}.  
		\remove{  The subindex ``$0$'' is used in anticipation of the  {shifted}  version to appear.

		To facilitate later upcoming  discussion (and later chapters), let us also define the  {shift} process 
	$	\bZ_{\uptau}\parenthesis{t} =  \bZ _0 \parenthesis{\uptau + t}. 
	\remove{\bZ_k\parenthesis{t} = \bZ_0\parenthesis{t_k+t},}$, where 
		$\set{\bZ_{\uptau}\parenthesis{\cdot}}_{\uptau }$ is a sequence of univariate $\real^p$-valued continuous functions indexed by the shifting amount  $\uptau$. We are particularly interested in a specific subsequence where $\uptau = t_k $.  
		The limiting behavior of a further pathwise convergent subsequences  will  capture   the  asymptotic properties  of the sequence $ \ilset{\hbtheta_k}  $  for  large  $k$. }

	\item   For $t\ge t_0$, $\bgamma_i$ is the noisy observation of $\bg_i\ilparenthesis{\hbtheta_i}$, and $\boldsymbol{\upeta}_i $ is such that $  a_i \boldsymbol{\upeta}_i = \Proj_{\bTheta}\ilparenthesis{\hbtheta_i+a_i\bgamma_i} - \ilparenthesis{\hbtheta_i+a_i\bgamma_i}$. 
	Define   
		\begin{equation}	\label{eq:Gamma0}
	\bGamma\ilparenthesis{t,\upomega} = \indicator_{\set{ t\ge t_0  }} \sum_{i=0}^{\TimeMapping{t}-1} a_i \bgamma_i\ilparenthesis{\upomega},
	\end{equation}
	and   define $     \NOISE\parenthesis{\cdot,\upomega} $, $\BIAS\parenthesis{\cdot,\upomega}$, $\PROJECTION\parenthesis{\cdot,\upomega}$ analogously to $\bGamma\parenthesis{\cdot,\upomega}$, but using   $\noise_k$, $\bias_k$, $\boldsymbol{\upeta}_i $  in place of $\bgamma_i$ respectively.  
\item   Now  we may write (\ref{eq:truncatedSA2}) equivalently as 
	\begin{eqnarray}\label{eq:truncatedSA3} 
 \bZ\parenthesis{t+t_k} \remove{
	&= \bZ_0\parenthesis{t_k+t}\\
	&= \bZ_k\parenthesis{t_0-t_k} + \bracket{\bGamma_0\parenthesis{t_k}+\bGamma_k\parenthesis{t}} + \bracket{\bH_0\parenthesis{t_k}+\bH_k\parenthesis{t}}\\
	&= \bZ_k\parenthesis{t_0} + \bGamma_k\parenthesis{t}+\bH_k\parenthesis{t}\\
	&\quad + \bracket{\bZ_k\parenthesis{t_0-t_k} - \bZ_k\parenthesis{t_k}}+ \bGamma_0\parenthesis{t_k} + \bH_0\parenthesis{t_k}\\ }
&=&\hbtheta_k+\ilbracket{ \bGamma\parenthesis{t+t_k}- \bGamma\ilparenthesis{t_k}  }+\ilbracket{\PROJECTION\parenthesis{t+t_k} - \PROJECTION\ilparenthesis{t_k}   }\nonumber \\
&= & 
\hbtheta_k+\sum_{i=k}^{\TimeMapping{t_k+t}-1} a_i\parenthesis{\bgamma_i+\boldsymbol{\upeta}_i}\quad \text{ if }t\ge t_0, 
\remove{&\quad 	\bZ_k\parenthesis{t}\\
	\remove{&= \bZ_0\parenthesis{t_k+t}\\
		&= \bZ_0\parenthesis{0}  +  \bGamma_0\parenthesis{0}+ \bH_0\parenthesis{0}\\
		&= \bZ_k\parenthesis{t_0-t_k} + \bracket{\bGamma_0\parenthesis{t_k}+\bGamma_k\parenthesis{t}} + \bracket{\bH_0\parenthesis{t_k}+\bH_k\parenthesis{t}}\\
		&= \bZ_k\parenthesis{t_0} + \bGamma_k\parenthesis{t}+\bH_k\parenthesis{t}\\
		&\quad + \bracket{\bZ_k\parenthesis{t_0-t_k} - \bZ_k\parenthesis{t_k}}+ \bGamma_0\parenthesis{t_k} + \bH_0\parenthesis{t_k}\\ }
	&= \hbtheta_k+\bGamma_k\parenthesis{t}+\bH_k\parenthesis{t}\\
	&=&  \begin{dcases}
	\hbtheta_k+\sum_{i=k}^{m\parenthesis{t_k+t}-1} a_i\parenthesis{\bgamma_i+\boldsymbol{\upeta}_i}&\text{ if }t\ge 0,\\
	\hbtheta_k-\sum_{i=m\parenthesis{t_k+t}}^{k-1}a_i\parenthesis{\bgamma_i+\boldsymbol{\upeta}_i}&\text{ if }t<0.
	\end{dcases}} 
\end{eqnarray}

The above   expression lays the foundation for constructing the continuous-time  version of the generalized SA. 

\end{enumerate}

Often, we    let $t_0=0$ w.l.o.g.

\subsection{ Model   Assumptions}\label{sect:assumptions}

In the remaining subsections,  we focus on      the  SA algorithm (\ref{eq:truncatedSA1}) with  {constant} gain  $a_k = a>0$. While implementing the    recursive algorithms with  a  constant gain $a$ 	in    time-varying problems,   the  convergence in some distributional sense as the iteration index $k\to\infty$ is the best we can hope for, see Section~\ref{sect:AdaptiveReview}.

Given that this section considers the behavior of $\hbtheta_k$ for \emph{different} values of \remove{ and a moderate value of $ka$} the constant gain $a>0$,  a  superscript $\parenthesis{a}$ is included to emphasize the dependency on \emph{different} values of the  constant gain $a$ within Section \ref{sect:ConstantGain}.   Specifically, $\hbtheta_k\aDependence$, $\bZ  \aDependence  \parenthesis{\cdot}$,  $\bGamma   \aDependence   \parenthesis{\cdot}$, $\BIAS  \aDependence   \parenthesis{\cdot}$, $\NOISE  \aDependence   \parenthesis{\cdot}$ will be used in Section \ref{sect:ConstantGain}  (and in this section  {only}) to   represent $\hbtheta_k$, $\bZ\parenthesis{\cdot}$, $ \bGamma\parenthesis{\cdot} $, $\BIAS\parenthesis{\cdot}$, and $\NOISE\parenthesis{\cdot}$ defined in Subsection \ref{subsect:Interpolation}.   However, the  initialization $\hbtheta_0$ should be  independent of $a$.

Let $\field_{t}\aDependence $ be the  linear space  spanned by $\ilset{\hbtheta_j\aDependence, j \le m\parenthesis{t}}$\remove{ $ \ilset{\hbtheta_0, \bgamma_j, j<m(t)} $, or spanned by $ \ilset{\hbtheta_j, \bgamma_{j-1}, j\le m\parenthesis{t}} $ equivalently}, the information available  {up until} the  discrete time index   $\TimeMapping{t}$. Let $\E_t\aDependence$ represent the expectation conditioned on $\field_t\aDependence$.

\begin{assumeB}\label{assume:UniformIntegrability}
	The sequence of random variables $ \ilset{\bgamma_k\aDependence} $ (indexed both by  time index   $k$ and by  gain $a$)  is  uniformly   integrable. That is, $ \lim_{ N\to\infty}  \sup_{k, a}  \E \ilparenthesis{  \indicator_{\ilset{ \norm{\bgamma_k\aDependence}\ge N }} \cdot \norm{\bgamma_k \aDependence} } = 0  $.  
\end{assumeB}

\begin{assumeB}
	\label{assume:Bias} \remove{The bias term satisfies either of the following: (i) $ \E_{t_k} \aDependence\ilparenthesis{\bias_k\aDependence} =\zero $, (ii) $\E_{t_k} \norm{\bias_k\aDependence}  = O(c_k^2)$ where $c_k$ is the finite difference interval. Furthermore, if (ii) holds, we should have $\lim_{k\to \infty}   \gain_k/c_k^2=0$. }

	{The bias  sequence satisfies $ \lim_{j\to \infty} \lim_{k\to \infty} j^{-1} \sum_{i =k}^{k+j-1} \E_{t_k}\aDependence  \bias_i\aDependence=\zero $ in expectation   for all $\gain>0$. } 
\end{assumeB}

\begin{assumeB}\label{assume:samplingfrequency}    Assume that  $ \ilparenthesis{\uptau_{k+1} - \uptau_k  } \propto \gain $, where $\uptau_k  $ was defined as    the actual time    corresponding to the  sampling instance $k$   in Chapter~\ref{chap:FiniteErrorBound}.
\end{assumeB}
\begin{rem}
	Here, the dependency of $\uptau_k$ on $\gain$ is suppressed. For a fixed value of $\gain$, the corresponding sampling frequency  $\uptau_{k+1}-\uptau_k$ is fixed for all $k$. 
\end{rem}

 \begin{assumeB}\label{assume:gSequenceRegularity}
 	The  sequence of measurable functions $\ilset{\bg_k  \parenthesis{\cdot}}$  of argument $\btheta\in\bTheta\subseteq\real^p$  are continuous uniformly in $k$.   Furthermore,   $\LipsPara$ is the smallest positive real such that $\bg_k\ilparenthesis{\cdot}$ are $\LipsPara$-Lipschitz continuous in $\btheta$ for all $k$.   Lastly, the time variability of the sequence $\ilset{\bg_k  \ilparenthesis{\cdot}}$ is such that $ \remove{  \E_{t_k }   } \norm{\bg_{k+1}  \ilparenthesis{\btheta}-\bg_{k}   \ilparenthesis{\btheta}} \propto (\uptau_{k+1} - \uptau_k )  $ for all $\btheta$
 \end{assumeB} 
	\begin{rem}
	We do not put the $\ilparenthesis{\gain}$ dependence on $\bg_k\ilparenthesis{\cdot}$, as the underlying time variability of  the sequence $\ilset{\bg_k\ilparenthesis{\cdot}}$ is \emph{not} affected by how we implement the tracking algorithm.
\end{rem}
\remove{
\begin{rem}
	In you conjecture, one $\omega\in\Omega$ still generates the entire sequence of $\btheta_k^*$. Focusing on each $\omega$, time-varying/nonstationary is sufficient as a characterization.  
\end{rem}
}
Let us  make   a few remarks regarding the aforementioned assumptions. 
\begin{itemize}
	
	\item 
A sufficient condition for   B.\ref{assume:UniformIntegrability}   is $ \sup_{k,a} \E\norm{\bgamma_k\aDependence}^{1+\upvarepsilon} <\infty $ for some $\upvarepsilon>0$. The commonly-used value of $\upvarepsilon$ is $1$. Namely, $ \sup_{k,a} \E\norm{\bgamma_k\aDependence}^{2} <\infty $.

 \item   \remove{Part (i) of  B.\ref{assume:Bias} is intuitive and is applicable when dealing with $\hbg_{\SG}\ilparenthesis{\cdot}$ as in (\ref{eq:Ystationary} ).

 	Part (ii) of B.\ref{assume:Bias} is pertinent to most of the zeroth-order algorithms, including SPSA defined in (\ref{eq:gSPSA}),  one-measurement SPSA defined in (\ref{eq:g1SP}), finite difference SA discussed in \cite[Chap. 6]{spall2005introduction}, and etc.  }
 In the traditional SA setup where the time variability is not pertinent, we may  impose stronger assumptions on the bias term than B.\ref{assume:Bias}; i.e., $\bias_k\aDependence$ is  assumed to represent a bias that is asymptotically unimportant in the sense that  $ \E\norm{\bias_k \aDependence}\to \zero$ as $k\to\infty$ for all $a$.   However, as explained in a footnote in Subsection \ref{subsect:ErrorSA}, for FDSA or SPSA where the $\hbg_k\ilparenthesis{\btheta}$ is a biased estimator for $\bg_k\ilparenthesis{\btheta}$,   the differencing interval is not allowed to decrease to zero while applied to time-varying tracking problems. Although this prevents the estimator from being ``almost unbiased,'' this non-diminishing bias   is  preferred to the otherwise slower convergence and ``noisier'' behavior as the variance of the effective noise is inversely proportional to the square of a differencing interval.  In short,  the bias   term $\bias_k\aDependence$  here could be persistent.

 \item B.\ref{assume:samplingfrequency} is a manifestation that the  sampling frequency grows inversely proportional to the constant gain $\gain$.  
 Namely,   the number of iterates per unit time has to grow inversely proportional to  the gain magnitude.

\item  B.\ref{assume:gSequenceRegularity}   and  Corollary~\ref{corr:ExtensionThem} ensures that the solution to the mean ODE (to appear in (\ref{eq:gInterpolation}) later) exists on the entire real line. 
\end{itemize}

\subsection{Main Results}
\label{subsect:WeakResult}

This subsection illustrates how a    {nonautonomous} differential equation can be associated with the  SA  iterations $\hbtheta_k$ generated from the  projected  SA algorithm (\ref{eq:truncatedSA1}) with constant gain $\gain$.  We will show that,  for almost all $\upomega\in\Omega$, all the sample paths  $ \ilset{\bZ\aDependence\parenthesis{\cdot, \upomega}} $   are  equi- (in fact Lipschitz-) continuous in the extended sense\remove{, thus the set of piecewise linear interpolations is also equicontinuous}.  Then the extended Arzel\`{a}-Ascoli Theorem  \ref{thm:ExtendedArzelaAscoli} can    be applied to extract convergent subsequences whose limits satisfy the mean  ODE.  	The path $\bZ\aDependence\parenthesis{\cdot,\upomega}$ will closely follow the solution to the ODE (\ref{eq:LimitingODE}) on any  finite interval, with an arbitrarily high probability (uniformly w.r.t. all  initial conditions within $\bTheta$) as $a\to 0$.  \remove{    ``as $T\to\infty$ and $a\to 0$, the probability goes to $1$'' }
The limit  of a  pathwise convergent  sequence of the   process $\ilset{\bZ\aDependence \parenthesis{\cdot,\upomega }}$  will satisfy the mean ODE.

\begin{lem}\label{lem:tight}
	Assume B.\ref{assume:UniformIntegrability}. The interpolated sequences $\ilset{\bZ  \aDependence  \parenthesis{\cdot,\upomega}}$ and $ \ilset{\PROJECTION  \aDependence  \parenthesis{\cdot,\upomega}} $  are tight\footnote{Tightness of a sequence of random processes was reviewed in Subsection \ref{subsect:Tightness}.}  for all $\upomega\in\Omega$.

\end{lem}

\begin{proof}[Proof of Lemma~\ref{lem:tight}]
 	To show  the uniform boundedness (which immediately implies tightness) using  the uniform-integrability assumption B.\ref{assume:UniformIntegrability}, let us first   truncate the noisy observation sequence  $\set{\bgamma_k\ilparenthesis{\upomega}}$.   For any  $N>0$, define the truncated random variables
 \begin{equation*}
 \bgamma_{k,N}  \aDependence\ilparenthesis{\upomega}= \begin{dcases}
 \bgamma_k\aDependence \ilparenthesis{\upomega}, &\text{if } \norm{\bgamma_k\aDependence\ilparenthesis{\upomega}}\le N,\\
 \zero,&\text{otherwise}. 
 \end{dcases}
 \end{equation*}
 
 	Fix a threshold $\upvarepsilon>0$, which may be arbitrarily small. 
 Fix a time $T> 0 $, which may be arbitrarily large. For some time interval $\updelta>0$, we have:
 \begin{align}\label{eq:GammaTight1} 
 &  \E\parenthesis{\sup_{s\le \updelta, \,0\le t \le T - s} \norm{\bGamma  \aDependence  \parenthesis{t+s,\upomega}-\bGamma  \aDependence   \parenthesis{t,\upomega}}}   \nonumber\\ 
&\quad \le  \E\parenthesis{ \sup_{s\le \updelta, \,0\le t \le T - s} \sum_{i= \TimeMapping{t}}^{\TimeMapping{t+s}-1} a\norm{\bgamma_i\aDependence\ilparenthesis{\upomega}}  }  \nonumber\\ 
& \quad \le \E  \bracket{\sup_{s\le \updelta, \, 0\le t \le T - s}   \sum_{i= \TimeMapping{t}}^{  \TimeMapping{t+s}-1}    \parenthesis{  a\norm{\bgamma_{i,N}\aDependence\ilparenthesis{\upomega}}+a\norm{\bgamma_i\aDependence \ilparenthesis{\upomega}-\bgamma_{i,N}\aDependence \ilparenthesis{\upomega } }  }  }.  
 \end{align} where   both inequalities follow  from the triangle inequality. 
 On one hand, 	 while    deeming  $a \sum_{i=\TimeMapping{t}}^{\TimeMapping{t+s}-1}\norm{\bgamma_{i,N}\aDependence\ilparenthesis{\upomega}} $  as a function of the time $t$,  it  can change values only at   multiples of $a$. Furthermore, at every such  point, the value can change by at most $Na$, which goes to zero as $\gain$ decreases to zero.  On the other hand,      Assumption B.\ref{assume:UniformIntegrability} implies  that $ \sup_{j, a} \E \norm{ \bgamma_j\aDependence\ilparenthesis{\upomega}-\bgamma_{j,N} \aDependence \ilparenthesis{\upomega} }  \to 0  $ as $N\to\infty$, which further implies that
 \begin{align}\label{eq:UnifIntConsequence1}
 &    \E \bracket{ \sup_{s\le \updelta, \,0\le t \le T - s} \sum_{i= \TimeMapping{t}}^{ \TimeMapping{t+s}-1 } a    \norm{\bgamma_i\aDependence\ilparenthesis{\upomega}-\bgamma_{i,N}\aDependence\ilparenthesis{\upomega}}    }\nonumber\\
 &\quad \le    \E\bracket{     \sup_{ a,\,  \TimeMapping{t}\le i \le \TimeMapping{t+\updelta}-1 } \parenthesis{a\cdot \frac{\updelta}{a}    \cdot \norm{\bgamma_i\aDependence\ilparenthesis{\upomega}-\bgamma_{i,N}\aDependence\ilparenthesis{\upomega}}}  } \nonumber \\
&\quad \le      \updelta \sup_{ a,  \,  \TimeMapping{t}\le i \le \TimeMapping{t+\updelta}-1  } \E\norm{\bgamma_i\aDependence\ilparenthesis{\upomega}-\bgamma_{i,N}\aDependence\ilparenthesis{\upomega}}  \nonumber\\
   &\quad \to 0 \,\,\, \text{as }N\to\infty. 
 \end{align}   
 Combing above observations, we know that for any given threshold $\upvarepsilon>0$, there exists a  {finite} $N\ilparenthesis{\upvarepsilon} $ and a finite $\gain\ilparenthesis{\upvarepsilon}$ that depend on $\upvarepsilon$,  such that both terms on the r.h.s. of   (\ref{eq:GammaTight1}) are less than $\upvarepsilon/2$. Given the arbitrariness of $\upvarepsilon>0$, we claim that  the paths of $\ilset{\bGamma  \aDependence   \parenthesis{\cdot,\upomega}}$ indexed by $a$ are asymptotically continuous in $t$ w.p.1,\remove{ i.e.,  \begin{equation} {eq:UnifIntConsequence3}
 	\lim_{\updelta\to 0} \limsup_{a\to 0}  \Prob\set{    \sup_{s\le\updelta,\,0\le t\le T-s}  \norm{\bGamma  \aDependence  \parenthesis{t+s}-\bGamma  \aDependence  \parenthesis{t}}  \ge \updelta} =0,
 	\end{equation}
 	Since (\ref{eq:UnifIntConsequence3}) holds for every $\updelta>0$,  we have the following by  \cite[Thm. 4.2.1 and 4.2.3]{chung2001course} } and
 \begin{equation}\label{eq:UnifIntConsequence2}
 \lim_{\updelta\to 0} \limsup_{a\to 0} \E\parenthesis{     \sup_{ s\le \updelta, \,0\le t\le T-s }   \norm{   {\bGamma}  \aDependence  \parenthesis{t+s,\upomega} - {\bGamma}  \aDependence  \parenthesis{t,\upomega}     }  } = 0.
 \end{equation}

 	Moreover, the statement   in  (\ref{eq:UnifIntConsequence2}) also holds  if  $\PROJECTION^{\parenthesis{a}}\parenthesis{\cdot,\upomega}$ replaces   $\bGamma^{\parenthesis{a}}\parenthesis{\cdot,\upomega}$    because 
 \begin{equation}\label{eq:ProjectionBound}
 \norm{\PROJECTION^{\parenthesis{a}}\parenthesis{t+s,\upomega}-\PROJECTION^{\parenthesis{a}}\parenthesis{t,\upomega}} \le a \sum_{i= m\parenthesis{t}}^{m\parenthesis{t+s}-1} \norm{\bgamma_i\aDependence\ilparenthesis{\upomega} }.
 \end{equation}
 
 	Also, the statement   in (\ref{eq:UnifIntConsequence2}) also holds if  $\bZ  \aDependence  \parenthesis{\cdot,\upomega}$  replaces $ \bGamma  \aDependence  \parenthesis{\cdot,\upomega} $  because  
 \begin{equation}\label{eq:RecursionConstantGain}
 \bZ  \aDependence  \parenthesis{t,\upomega}=\hbtheta_0+\bGamma  \aDependence  \parenthesis{t,\upomega}+\PROJECTION  \aDependence  \parenthesis{t,\upomega}.
 \end{equation} 
 
 In short, with (arbitrarily) high probability, the processes $\PROJECTION^{\parenthesis{a}}\parenthesis{\cdot ,\upomega}$ and  $\bZ  \aDependence  \parenthesis{\cdot,\upomega}$ change   slightly on the small time interval $ \bracket{t,\, t+s} \subset \bracket{0,T}$. Therefore,  the condition (\ref{eq:compact2}) for  Lemma~\ref{lem:Tightness} in proving tightness is met for the  random processes $ \ilparenthesis{\bZ  \aDependence  \parenthesis{\cdot,\upomega},\PROJECTION  \aDependence  \parenthesis{\cdot,\upomega}} $ in the  functional space  $D\parenthesis{\real\mapsto\real^{2p}}$. 
 Given that $\ilset{\bvartheta_k}$ is constrained within $\bTheta$ and is  independent of $a$, and $\PROJECTION  \aDependence  \parenthesis{t_0}=\zero$ at a valid initialization $\hbtheta_0\in\bTheta$,   the condition (\ref{eq:compact1})  is also  met. By Lemma \ref{lem:Tightness},  we claim     the tightness of the sequence of random processes $ \ilset{\bZ  \aDependence  \parenthesis{\cdot,\upomega},\PROJECTION  \aDependence  \parenthesis{\cdot,\upomega}} $.    
\end{proof}

 Now we proceed to  extract and characterize  a proper sequence of continuous-time interpolations  $ \ilparenthesis{\bZ  \aDependence  \parenthesis{\cdot,\upomega}, \PROJECTION   \aDependence   \parenthesis{\cdot,\upomega}} $.  
By  Prohorov's theorem   \cite[p. 104]{ethier2009markov},   we can extract a convergent sequence  $ \ilparenthesis{\bZ \aSubseqDependence \parenthesis{\cdot,\upomega}, \PROJECTION \aSubseqDependence \parenthesis{\cdot,\upomega} } $ as $n\to\infty$   based on  Lemma~\ref{lem:tight}. Let  $ \parenthesis{\bZ\parenthesis{\cdot,\upomega},\PROJECTION\parenthesis{\cdot,\upomega}} $ be the process such that  on the space $ D\parenthesis{\real\mapsto\real^{2p}} $, we have 
\begin{equation}\label{eq:WeakConvLimit}\ilparenthesis{\bZ  \aSubseqDependence  \parenthesis{\cdot,\upomega},\PROJECTION  \aSubseqDependence  \parenthesis{\cdot,\upomega}} \text{ converges weakly to }\parenthesis{\bZ\parenthesis{\cdot,\upomega},\PROJECTION\parenthesis{\cdot,\upomega}} \text{ as }n\to\infty. 
\end{equation}
\remove{Equivalently, for any bounded and continuous function $\boundedFn\parenthesis{\cdot}$ defined on the functional space $D\ilparenthesis{\real\mapsto\real^{2p}}$, we have $ \E \boundedFn\ilparenthesis{\bZ  \aSubseqDependence  \parenthesis{\cdot,\upomega}, \PROJECTION\aSubseqDependence\parenthesis{\cdot,\upomega}} \to \E \boundedFn\parenthesis{\bZ\parenthesis{\cdot,\upomega}, \PROJECTION\parenthesis{\cdot,\upomega}} $ as $n\to\infty$. }

\begin{lem}\label{lem:UnifEquiCont} Assume  B.\ref{assume:UniformIntegrability}.   There exists a sequence $ a_n\to 0 $ as $n\to\infty$, such that for all $\upomega\notin \nullset$ where $\nullset$  is a set of null measure, the weak convergence limit of $ \ilset{\bZ\aSubseqDependence\parenthesis{\cdot, \upomega},\PROJECTION\aSubseqDependence\parenthesis{\cdot, \upomega}} $, which was denoted as  $\parenthesis{\bZ\parenthesis{\cdot,\upomega},\PROJECTION\parenthesis{\cdot,\upomega}}$ in  (\ref{eq:WeakConvLimit}), is   equicontinuous  in the extended sense  (see Definition \ref{dfn:EquiContExtended}). 
\end{lem}

\begin{proof}[Proof of Lemma~\ref{lem:UnifEquiCont}]
First of all, the weak convergence limit $\parenthesis{\bZ\parenthesis{\cdot,\upomega},\PROJECTION\parenthesis{\cdot,\upomega}}$   defined in (\ref{eq:WeakConvLimit}) have continuous paths w.p.1. thanks to  Lemma~\ref{lem:weakcont}.

We now  show the  equi- (in fact Lipschitz-) continuity of paths of  the  weak sense limit  $\parenthesis{\bZ\parenthesis{\cdot,\upomega},\PROJECTION\parenthesis{\cdot,\upomega}}$ w.p.1.
From (\ref{eq:UnifIntConsequence2}), 
we know that 
\begin{equation}\label{eq:UnifIntConsequence3}
\lim_{\updelta\to 0} \limsup_{a\to 0}  \Prob\set{ \upomega:   \sup_{s\le\updelta,\, 0\le t\le T-s}  \norm{\bGamma  \aDependence  \parenthesis{t+s,\upomega}-\bGamma  \aDependence  \parenthesis{t,\upomega}}  \ge \upvarepsilon} =0,\,\,\, \forall\upvarepsilon>0, T>0. 
\end{equation}  Therefore, for each $T>0$, there exists a random variable $\LipsPara \parenthesis{T,\upomega}<\infty$ (which is  {independent} of the sequence index $a$)    such that for $0\le t\le t+s\le T$, $ \norm{\bGamma  \aDependence  \parenthesis{t+s,\upomega}-\bGamma  \aDependence  \parenthesis{t,\upomega}}\le \LipsPara\parenthesis{T,\upomega}s $ w.p.1.  Therefore, the sequence of random processes $ \ilset{\bGamma  \aDependence  \parenthesis{\cdot,\upomega}} $ (indexed by $\gain$) is locally Lipschitz continuous w.p.1 uniformly for all  $a$,   thanks to B.\ref{assume:UniformIntegrability}.  
Given   (\ref{eq:ProjectionBound}) and (\ref{eq:RecursionConstantGain}),  both $\PROJECTION  \aDependence  \parenthesis{\cdot,\upomega}$ and $\bZ  \aDependence  \parenthesis{\cdot,\upomega}$ is also locally Lipschitz continuous  w.p.1 uniformly across   $a$.   The result then follows.

	In summary,  for stochastic processes  $ \ilparenthesis{\bZ  \aDependence  \parenthesis{\cdot,\upomega},\PROJECTION  \aDependence  \parenthesis{\cdot,\upomega}} $ indexed by $\gain$,  there exists a   sequence    $a_n$, which goes to $0$ as $n\to\infty$,  such that the weak convergence limit  of $ \ilparenthesis{\bZ  \aSubseqDependence  \parenthesis{\cdot,\upomega},\PROJECTION  \aSubseqDependence  \parenthesis{\cdot,\upomega}} $, denoted as $\parenthesis{\bZ\parenthesis{\cdot,\upomega},\PROJECTION\parenthesis{\cdot,\upomega}}$, is  {  equicontinuous} in the extended   sense   w.p.1.  
\end{proof}

\begin{rem}  Under given assumptions, we know that the weak limit of the sequence  $ \ilparenthesis{\bZ  \aSubseqDependence  \parenthesis{\cdot,\upomega},\PROJECTION  \aSubseqDependence  \parenthesis{\cdot,\upomega}} $ indexed by $n $   is equicontinuous in the extended sense w.p.1, and the error vanishes  for almost all $\upomega$   along that sequence $\gain_n$ only.   
\end{rem}

 For succinctness, we will suppress $\upomega$ whenever appropriate.

\begin{thm} \label{thm:ODE}  	Assume B.\ref{assume:UniformIntegrability}, B.\ref{assume:Bias}, B.\ref{assume:samplingfrequency},  B.\ref{assume:gSequenceRegularity}.
	The weak convergence limit  $ \parenthesis{\bZ\parenthesis{\cdot,\upomega},\PROJECTION\parenthesis{\cdot,\upomega}} $  in (\ref{eq:WeakConvLimit})  satisfies the    mean ODE: 
	\begin{equation}\label{eq:LimitingODE}
	\dot{\btheta}= - \obg\parenthesis{t, \btheta}+\projectionlower\parenthesis{t}, {\quad \projectionlower\parenthesis{t}\in -\ConvCone\parenthesis{\bZ\parenthesis{t}}},
	\end{equation}
	where  $\projectionlower\parenthesis{t}$ is the  adjustment needed to keep $\btheta\parenthesis{t}$   within $\bTheta$, the notion of the convex cone was introduced immediately after (\ref{eq:constrainedODE}),  and $ \obg\parenthesis{t,\btheta}$  is the limit of 
	\begin{equation}\label{eq:gInterpolation}
	\bg\parenthesis{t,\btheta} =   { \sum_{k= 0}^\infty \bracket{   \indicator_{\ilset{t_k\le t<t_{k+1}}} \cdot \parenthesis{ \frac{t_{k+1}-t}{a} \bg_k\parenthesis{\btheta}+ \frac{t-t_k}{a} \bg_{k+1} \parenthesis{\btheta}  }  }}	\end{equation}   as $\gain \to 0$. 
\end{thm}

\begin{proof}
	First, from the result in Lemma  \ref{lem:UnifEquiCont} and  Theorem  \ref{thm:ArzelaAscoli},   the process $ \parenthesis{\bZ\parenthesis{\cdot,\upomega},\PROJECTION\parenthesis{\cdot,\upomega}} $ defined in  (\ref{eq:WeakConvLimit})       is  continuous uniformly on each bounded interval, and in fact has  Lipschitz continuous paths w.p.1.

	 Now we  characterize its limit of measures of the process $ \parenthesis{\bZ\parenthesis{\cdot,\upomega},\PROJECTION\parenthesis{\cdot,\upomega}} $ on appropriate path space  such that the limit measure induces a process on the path space supported on some set of limit trajectories of the ODE (\ref{eq:LimitingODE}).    
		In addition to  $ \bGamma  \aDependence  \parenthesis{\cdot,\upomega}$, $ \BIAS  \aDependence  \parenthesis{\cdot,\upomega}$, $ \NOISE  \aDependence  \parenthesis{\cdot,\upomega}$,  and $ \PROJECTION  \aDependence  \parenthesis{\cdot,\upomega} $ defined  in (\ref{eq:Gamma0}), we also  define \begin{equation}\label{eq:Ga(t)}
		{\bG}  \aDependence  \parenthesis{t,\upomega} =   \int_{0}^t\bg\ilparenthesis{s,\bZ  \aDependence  \parenthesis{s,\upomega}}\diff s \,\,\text{for }t\ge 0 ,
		\end{equation} and $$\brho  \aDependence  \parenthesis{t,\upomega} = { a\sum_{i=0}^{m\parenthesis{t}-1} \bg_i \ilparenthesis{\hbtheta_i\ilparenthesis{\upomega}}
		-  \int_0^t\bg\ilparenthesis{s,\bZ  \aDependence  \parenthesis{s,\upomega}}\diff s}\,\, \text{for }t\ge 0.$$

	By the decomposition (\ref{eq:gGeneral}), specifically,
	\begin{equation}\label{eq:gNoisyDecompositionDependsona}
	\hbg_k\aDependence \ilparenthesis{\hbtheta_k\aDependence} = \bg_k \ilparenthesis{\hbtheta_k\aDependence} + \bias_k\aDependence\ilparenthesis{\hbtheta_k\aDependence} + \noise_k \aDependence \ilparenthesis{\hbtheta_k\aDependence}, 
	\end{equation} \remove{ note that $\bg_k$ does not depend on gain $\gain$ }
	 and following the construction in Subsection \ref{subsect:Interpolation}, we can rewrite (\ref{eq:truncatedSA3}) as: 
	\begin{equation}
	\bZ  \aDependence  \parenthesis{t} = \hbtheta_0-\bG  \aDependence  \parenthesis{t} - \brho  \aDependence  \parenthesis{t}  -  \BIAS   \aDependence  \parenthesis{t} - \NOISE  \aDependence  \parenthesis{t}+\PROJECTION  \aDependence  \parenthesis{t}.
	\end{equation}

	Furthermore,   define the following random processes with paths in $D \parenthesis{\real\mapsto\real^p}$:
	\begin{eqnarray}\label{eq:Wt} 
	\bW  \aDependence  \parenthesis{t} &=& \bZ  \aDependence  \parenthesis{t}-\hbtheta_0+ \bG  \aDependence  \parenthesis{t}-\PROJECTION  \aDependence  \parenthesis{t} \nonumber \\
	 &=& -  \NOISE   \aDependence  \parenthesis{t}- \BIAS  \aDependence  \parenthesis{t}- \brho \aDependence  \parenthesis{t}. 
	\end{eqnarray}
		Our goal   is to show that for each existing sequence $a_n$ arising in Lemma~\ref{lem:UnifEquiCont},  $\bW  \aSubseqDependence  \parenthesis{t}$ converges weakly to  a martingale process w.r.t. $\field_{  {t}}\aSubseqDependence$ spanned by $ \ilset{\hbtheta_j\aSubseqDependence,   j\le \TimeMapping{t}} $ as $n\to\infty$.   Thanks to  \cite[Thm. 7.4.1 on p. 234]{kushner2003stochastic},  we only need to show that for any   time $t\ge 0$ and time-interval $\updelta\ge 0$, and for  any integer $S>0$,  
		\begin{equation}\label{eq:WtMartingale1}
	\E \set{ \boundedFn\ilparenthesis{\hbtheta_{j_1}  \aSubseqDependence  ,\cdots , \hbtheta_{j_S}  \aSubseqDependence  , \boldsymbol{\upeta}_{j_1}  \aSubseqDependence , \cdots, \boldsymbol{\upeta}_{j_S}  \aSubseqDependence } \bracket{\bW  \aSubseqDependence  \parenthesis{t+\updelta}-\bW  \aSubseqDependence  \parenthesis{t}}}\stackrel{n\to\infty}{\longrightarrow}\zero, 
	\end{equation}
	where $  {j_s} \in    \ilset{  0, 1,\cdots,  \TimeMapping{t} } $  for all $1\le s \le S$, and $ \boundedFn\parenthesis{\cdot}$ is any bounded and continuous   function that maps $ \real^{2pS} $ to $\real$. Given (\ref{eq:Wt}), we only need to show the following in order to show (\ref{eq:WtMartingale1}): 
	\begin{numcases}{}
 	\E \set{ \boundedFn \ilparenthesis{\hbtheta_{j_1}  \aSubseqDependence  ,\cdots , \hbtheta_{j_S}  \aSubseqDependence  , \boldsymbol{\upeta}_{j_1}  \aSubseqDependence , \cdots, \boldsymbol{\upeta}_{j_S}  \aSubseqDependence } \bracket{   \NOISE  \aSubseqDependence  \parenthesis{t+\updelta}-\NOISE  \aSubseqDependence  \parenthesis{t}}}\stackrel{n\to\infty}{\longrightarrow}\zero, &	\label{eq:WtMartingale2-1}\\
	 	\E \set{ \boundedFn  \ilparenthesis{\hbtheta_{j_1}  \aSubseqDependence  ,\cdots , \hbtheta_{j_S}  \aSubseqDependence  , \boldsymbol{\upeta}_{j_1}  \aSubseqDependence , \cdots, \boldsymbol{\upeta}_{j_S}  \aSubseqDependence } \bracket{\BIAS \aSubseqDependence  \parenthesis{t+\updelta}-\BIAS  \aSubseqDependence  \parenthesis{t}}} \stackrel{n\to\infty}{\longrightarrow}\zero, &	\label{eq:WtMartingale2-2}\\
 	\E \set{ \boundedFn \ilparenthesis{\hbtheta_{j_1}  \aSubseqDependence  ,\cdots , \hbtheta_{j_S}  \aSubseqDependence  , \boldsymbol{\upeta}_{j_1}  \aSubseqDependence , \cdots, \boldsymbol{\upeta}_{j_S}  \aSubseqDependence } \bracket{\brho  \aSubseqDependence  \parenthesis{t+\updelta}-\brho  \aSubseqDependence  \parenthesis{t}}} \stackrel{n\to\infty}{\longrightarrow}\zero. & \label{eq:WtMartingale2-3}
	\end{numcases}

	Let us first show that the r.h.s. of   (\ref{eq:WtMartingale2-1}) is   $\zero$.  By construction (\ref{eq:gNoisyDecompositionDependsona}), $\noise_k\aDependence = \hbg_k\aDependence\ilparenthesis{\hbtheta_k\aDependence}-\E_{t_k}\aDependence \ilbracket{\hbg_k\aDependence\ilparenthesis{\hbtheta_k\aDependence}}$. Hence, $ \ilset{\NOISE\aSubseqDependence\ilparenthesis{s}, s\le t} $ is  $\field _{t}\aSubseqDependence$-measurable. Furthermore,  the process $\NOISE\aSubseqDependence\ilparenthesis{t}$ is an $\field_{t}\aSubseqDependence$-martingale. Then  by iterated conditioning and (\ref{eq:gNoisyDecompositionDependsona}):
	\begin{align} 
	&  	\E \set{ \boundedFn  \ilparenthesis{\hbtheta_{j_1}  \aSubseqDependence  ,\cdots , \hbtheta_{j_S}  \aSubseqDependence  , \boldsymbol{\upeta}_{j_1}  \aSubseqDependence , \cdots, \boldsymbol{\upeta}_{j_S}  \aSubseqDependence }  \bracket{\NOISE   \aSubseqDependence  \parenthesis{t+\updelta}-\NOISE  \aSubseqDependence  \parenthesis{t}}} \nonumber  \\
	&\quad =  	\E \parenthesis{    \E\set{\given{ \boundedFn\ilparenthesis{\hbtheta_{j_1}  \aSubseqDependence  ,\cdots , \hbtheta_{j_S}  \aSubseqDependence  , \boldsymbol{\upeta}_{j_1}  \aSubseqDependence , \cdots, \boldsymbol{\upeta}_{j_S}  \aSubseqDependence } \bracket{\NOISE  \aSubseqDependence  \parenthesis{t+\updelta}-\NOISE  \aSubseqDependence  \parenthesis{t}}}{\field_{{t}}  \aSubseqDependence  } }       } \nonumber  \\
	&\quad = \E\set{       \boundedFn \ilparenthesis{\hbtheta_{j_1}  \aSubseqDependence  ,\cdots , \hbtheta_{j_S}  \aSubseqDependence  , \boldsymbol{\upeta}_{j_1}  \aSubseqDependence , \cdots, \boldsymbol{\upeta}_{j_S}  \aSubseqDependence } \cdot  \E\bracket{    \given{     \NOISE  \aSubseqDependence  \parenthesis{t+\updelta}-\NOISE\aSubseqDependence\parenthesis{t}          }{\field_{{t}}  \aSubseqDependence  }  }          }\nonumber \\
	&\quad = \E\set{        \boundedFn  \ilparenthesis{\hbtheta_{j_1}  \aSubseqDependence  ,\cdots , \hbtheta_{j_S}  \aSubseqDependence  , \boldsymbol{\upeta}_{j_1}  \aSubseqDependence , \cdots, \boldsymbol{\upeta}_{j_S}  \aSubseqDependence } \cdot  \E\bracket{    \given{     \gain_n  \sum_{i = \TimeMapping{t}} ^ { \TimeMapping{t+\updelta} -1 } \noise_i  \aSubseqDependence      }{\field_{{t}}  \aSubseqDependence  }  }          }    \nonumber    \\
	&\quad=\E\set{\bracket{    \boundedFn  \ilparenthesis{\hbtheta_{j_1}  \aSubseqDependence  ,\cdots , \hbtheta_{j_S}  \aSubseqDependence  , \boldsymbol{\upeta}_{j_1}  \aSubseqDependence , \cdots, \boldsymbol{\upeta}_{j_S}  \aSubseqDependence }  }\cdot\zero}\nonumber \\
	&\quad=\zero.  
	\end{align} 
	
We then show that the r.h.s. of  (\ref{eq:WtMartingale2-2}) goes to    $\zero$ as the gain $\gain_n\to 0$.  By iterated conditioning, 
	\begin{align} 
	& \quad  \norm{     	\E \set{ \boundedFn    \ilparenthesis{\hbtheta_{j_1}  \aSubseqDependence  ,\cdots , \hbtheta_{j_S}  \aSubseqDependence  , \boldsymbol{\upeta}_{j_1}  \aSubseqDependence , \cdots, \boldsymbol{\upeta}_{j_S}  \aSubseqDependence }\bracket{\BIAS   \aSubseqDependence  \parenthesis{t+\updelta }-\BIAS  \aSubseqDependence  \parenthesis{t}}}      }\nonumber\\
	&    =      \norm{   	\E \parenthesis{    \E\set{\given{ \boundedFn  \ilparenthesis{\hbtheta_{j_1}  \aSubseqDependence  ,\cdots , \hbtheta_{j_S}  \aSubseqDependence  , \boldsymbol{\upeta}_{j_1}  \aSubseqDependence , \cdots, \boldsymbol{\upeta}_{j_S}  \aSubseqDependence } \bracket{\BIAS  \aSubseqDependence  \parenthesis{t+\updelta}-\BIAS  \aSubseqDependence  \parenthesis{t}}}{\field_{{t}}  \aSubseqDependence  } }       }        }     \nonumber     \\
	& =\norm{      \E\set{       \boundedFn   \ilparenthesis{\hbtheta_{j_1}  \aSubseqDependence  ,\cdots , \hbtheta_{j_S}  \aSubseqDependence  , \boldsymbol{\upeta}_{j_1}  \aSubseqDependence , \cdots, \boldsymbol{\upeta}_{j_S}  \aSubseqDependence } \cdot 
			 \E\bracket{    \given{     \BIAS  \aSubseqDependence  \parenthesis{t+\updelta}-\BIAS  \aSubseqDependence  \parenthesis{t}          }{\field_{{t}}  \aSubseqDependence  }  }          }    }               \nonumber    \\
	& =\norm{        \E\set{  \boundedFn \ilparenthesis{\hbtheta_{j_1}  \aSubseqDependence  ,\cdots , \hbtheta_{j_S}  \aSubseqDependence  , \boldsymbol{\upeta}_{j_1}  \aSubseqDependence , \cdots, \boldsymbol{\upeta}_{j_S}  \aSubseqDependence }  \cdot  \E\bracket{\given{   \gain_n  \sum_{i=\TimeMapping{t}}^{\TimeMapping{t+\updelta}-1}  \bias_i   \aSubseqDependence   }{\field _{{t}}  \aSubseqDependence  }} }          }\nonumber \\
	& \le \norm{  \E \bracket{   \boundedFn   \ilparenthesis{\hbtheta_{j_1}  \aSubseqDependence  ,\cdots , \hbtheta_{j_S}  \aSubseqDependence  , \boldsymbol{\upeta}_{j_1}  \aSubseqDependence , \cdots, \boldsymbol{\upeta}_{j_S}  \aSubseqDependence }    }    } \cdot \norm{   \E       \set{    \E\bracket{\given{   \gain_n  \sum_{i=\TimeMapping{t}}^{\TimeMapping{t+\updelta}-1}  \bias_i   \aSubseqDependence   }{\field _{{t}}  \aSubseqDependence  }}      }   },  
	\end{align} 
	where the second term $\norm{   \E       \ilset{    \E\ilbracket{\given{   \gain_n  \sum_{i=\TimeMapping{t}}^{\TimeMapping{t+\updelta}-1}  \bb_i   \aSubseqDependence   }{\field _{{t}}  \aSubseqDependence  }}      }   }$    goes to $\zero$ as $n\to\infty$ by  B.\ref{assume:Bias}.

	Let us finally   show   that  the r.h.s. of  (\ref{eq:WtMartingale2-3}) goes to  $\zero$ as the gain $\gain_n \to 0$.    Observe from our construction (\ref{eq:gInterpolation}) and by B.\ref{assume:gSequenceRegularity}, $\brho\aDependence\parenthesis{t}\to \zero$ uniformly in $t$ as $a\to 0$. Consequently, $\norm{\brho  \aDependence  \parenthesis{t+\uptau} - \brho  \aDependence  \parenthesis{t}} \to 0$ uniformly in $t$ for any time interval $\updelta$ as $a\to 0$. Specifically,	
	\begin{align} 
	&   \norm{\brho  \aDependence  \parenthesis{t+\updelta} - \brho  \aDependence  \parenthesis{t}} \nonumber \\
	&\quad= \norm{a  \sum_{i= m\parenthesis{t}}^{m\parenthesis{t+\updelta} - 1} \bg_i\ilparenthesis{\hbtheta_i} -   \int_ {t} ^{t+\updelta} \bg\ilparenthesis{s,\bZ  \aDependence  \parenthesis{s}}\diff s} \nonumber\\
	&\quad=\norm{ a  \sum_{i= m\parenthesis{t}}^{m\parenthesis{t+\updelta} - 1} \bg_i\ilparenthesis{\hbtheta_i}  - \frac{a}{2}   \sum_{i= m\parenthesis{t}}^{m\parenthesis{t+\updelta} - 1} \bracket{   \bg_i \ilparenthesis{\hbtheta_i }  + \bg_ {i+1}  \ilparenthesis{\hbtheta_i } }}\nonumber\\
	&\quad= \norm{\frac{a}{2} \sum_{i= m\parenthesis{t}}^{m\parenthesis{t+\updelta} - 1} \bracket{  \bg_{i} \ilparenthesis{\hbtheta_i} - \bg _{i+1} \ilparenthesis{\hbtheta_i }   }}\nonumber\\
	&\quad\le  \frac{a}{2} \cdot \frac{\updelta}{a} \sup_{  m\parenthesis{t}\le i \le m\parenthesis{t+\updelta} - 1}\norm{  \bg_i \ilparenthesis{\hbtheta_i} -   \bg_{i+1}\ilparenthesis{\hbtheta_i}   } \nonumber\\  
	&\quad  =  \frac{\updelta}{2} \sup_{  m\parenthesis{t}\le i \le m\parenthesis{t+\updelta} - 1}\norm{  \bg_i \ilparenthesis{\hbtheta_i} -   \bg_{i+1}\ilparenthesis{\hbtheta_i}   }.  
	\end{align} 
Then	by iterated conditioning and replacing $\gain$ by $\gain_n$, we have:
		 \begin{align} 
	& \norm{     	\E \set{ \boundedFn\ilparenthesis{\hbtheta_{j_1}  \aSubseqDependence  ,\cdots , \hbtheta_{j_S}  \aSubseqDependence  , \boldsymbol{\upeta}_{j_1}  \aSubseqDependence , \cdots, \boldsymbol{\upeta}_{j_S}  \aSubseqDependence } \bracket{\brho  \aSubseqDependence  \parenthesis{t+\updelta}-\brho  \aSubseqDependence  \parenthesis{t}}}      }\nonumber \\
	&\,\,=      \norm{   	\E \parenthesis{    \E\set{\given{ \boundedFn \ilparenthesis{\hbtheta_{j_1}  \aSubseqDependence  ,\cdots , \hbtheta_{j_S}  \aSubseqDependence  , \boldsymbol{\upeta}_{j_1}  \aSubseqDependence , \cdots, \boldsymbol{\upeta}_{j_S}  \aSubseqDependence } \bracket{\brho  \aSubseqDependence  \parenthesis{t+\updelta}-\brho  \aSubseqDependence  \parenthesis{t}}}{\field_{{t}}  \aSubseqDependence  } }       }        }       \nonumber   \\
	&\,\,=\norm{      \E\set{       \boundedFn \ilparenthesis{\hbtheta_{j_1}  \aSubseqDependence  ,\cdots , \hbtheta_{j_S}  \aSubseqDependence  , \boldsymbol{\upeta}_{j_1}  \aSubseqDependence , \cdots, \boldsymbol{\upeta}_{j_S}  \aSubseqDependence } \cdot  \E\bracket{    \given{     \brho  \aSubseqDependence  \parenthesis{t+\updelta}-\brho  \aSubseqDependence  \parenthesis{t}          }{\field_{{t}}  \aSubseqDependence  }  }          }    }           \nonumber        \\
	&\,\,\le \norm{  \E \bracket{   \boundedFn \ilparenthesis{\hbtheta_{j_1}  \aSubseqDependence  ,\cdots , \hbtheta_{j_S}  \aSubseqDependence  , \boldsymbol{\upeta}_{j_1}  \aSubseqDependence , \cdots, \boldsymbol{\upeta}_{j_S}  \aSubseqDependence }    }    } \nonumber\\
	&\quad\quad  \cdot \norm{   \E       \set{    \E\bracket{\given{     \brho  \aSubseqDependence  \parenthesis{t+\updelta}-\brho  \aSubseqDependence  \parenthesis{t}   }{\field _{{t}}  \aSubseqDependence  }}      }   } \nonumber\\
	&\,\,  \to \zero,
	\end{align} 
	where the last line goes to zero because  of  B.\ref{assume:samplingfrequency},  B. B.\ref{assume:gSequenceRegularity}, and   above observation.

	To characterize the limit process  of sequence, we also define the following in parallel to (\ref{eq:Wt}):   
	\begin{equation}\label{eq:Wt1}
	\bW\parenthesis{t} = \bZ\parenthesis{t} - \hbtheta_0+ \bG\parenthesis{t} - \PROJECTION\parenthesis{t}
	\end{equation} where $\bZ\parenthesis{\cdot}$ and $\PROJECTION\parenthesis{\cdot}$ are defined in (\ref{eq:WeakConvLimit}), and $\bG(t)$ is defined by replacing $\bg\ilparenthesis{\cdot,\cdot}$ in    (\ref{eq:Ga(t)})    with  $\obg\ilparenthesis{\cdot,\cdot}$. It then follows that  $\bW\parenthesis{t}$ is a function of $ \set{\bZ\parenthesis{s},\PROJECTION\parenthesis{s},s\le t}$. 	The validity of  (\ref{eq:WtMartingale2-1})--(\ref{eq:WtMartingale2-3}) gives rise to (\ref{eq:WtMartingale1}) given the decomposition in (\ref{eq:Wt}).

	Note that we have shown  the validity of  (\ref{eq:WtMartingale1}) and that $\bW\aSubseqDependence\ilparenthesis{t,\upomega}$ is uniformly integrable. By Theorem~\ref{thm:Skorohod}, there exists a probability space $\ilparenthesis{\widetilde{\Omega}, \widetilde{\field},\widetilde{\Prob}}$ with processes $ \ilparenthesis{  \widetilde{\bZ} \aSubseqDependence\ilparenthesis{\cdot}, \widetilde{\PROJECTION}\aSubseqDependence\ilparenthesis{\cdot}, \widetilde{\bW}\aSubseqDependence\ilparenthesis{\cdot} } $ and $ \ilparenthesis{\widetilde{\bZ}\ilparenthesis{\cdot},\widetilde{\PROJECTION}\ilparenthesis{\cdot},\widetilde{\bW}\ilparenthesis{\cdot}} $, which have the same distribution as the processes $ \ilparenthesis{\bZ\aSubseqDependence\ilparenthesis{\cdot}, \PROJECTION\aSubseqDependence\ilparenthesis{\cdot},\bW\aSubseqDependence\ilparenthesis{\cdot}} $ and $ \ilparenthesis{\bZ\ilparenthesis{\cdot},\PROJECTION\ilparenthesis{\cdot},\bW\ilparenthesis{\cdot}} $ on $ \ilparenthesis{\Omega,\field,\Prob} $. As the way $\hbtheta_k$ and $\boldsymbol{\upeta}_k$  appear in (\ref{eq:Zbar0}) and (\ref{eq:Gamma0}),  $ \tilde{\hbtheta}_k $ and $\widetilde{\boldsymbol{\upeta}}_k$ can be similarly defined on the probability space $\ilparenthesis{\widetilde{\Omega}, \widetilde{\field},\widetilde{\Prob}}$.
	 Furthermore,  $ \ilparenthesis{  \widetilde{\bZ} \aSubseqDependence\ilparenthesis{\cdot}, \widetilde{\PROJECTION}\aSubseqDependence\ilparenthesis{\cdot}, \widetilde{\bW}\aSubseqDependence\ilparenthesis{\cdot} } $ converges to $ \ilparenthesis{\widetilde{\bZ}\ilparenthesis{\cdot},\widetilde{\PROJECTION}\ilparenthesis{\cdot},\widetilde{\bW}\ilparenthesis{\cdot}} $ w.p.1 under the (Skorohod) topology within $D\ilparenthesis{\left[0,\infty\right)\mapsto\real^{3p}}$. Recall the statement ``the convergence of a sequence of functions in $D\ilparenthesis{\left[0,\infty\right)\mapsto\real^p}$ to a continuous function in $C\ilparenthesis{\left[0,\infty\right)\mapsto\real^p}$ in the Skorohod topology is equivalent to  {convergence uniformly on each bounded time interval} in $C\ilparenthesis{\left[0,\infty\right)\mapsto\real^p}$'' from Section  \ref{subsect:WeakConvergence}. Therefore, we know that \begin{eqnarray}\label{eq:tildeWtMartingale1}
\remove{&&	\E \set{ \boundedFn\ilparenthesis{\widetilde{\hbtheta}_{j_1}  \aSubseqDependence  ,\cdots,\widetilde{\hbtheta}_{j_S}  \aSubseqDependence , \widetilde{\boldsymbol{\upeta}}_{j_1}  \aSubseqDependence  , \cdots,   \widetilde{\boldsymbol{\upeta}}_{j_S}  \aSubseqDependence  } \bracket{\widetilde{\bW } \aSubseqDependence  \parenthesis{t+\updelta}-\widetilde{\bW}  \aSubseqDependence  \parenthesis{t}}}=\zero, \nonumber\\
&\implies &} \E \set{ \boundedFn\ilparenthesis{\widetilde{\hbtheta}_{j_1}  \aSubseqDependence  ,\cdots,\widetilde{\hbtheta}_{j_S}  \aSubseqDependence , \widetilde{\boldsymbol{\upeta}}_{j_1}  \aSubseqDependence  , \cdots,   \widetilde{\boldsymbol{\upeta}}_{j_S}  \aSubseqDependence  } \bracket{\widetilde{\bW }    \parenthesis{t+\updelta}-\widetilde{\bW}     \parenthesis{t}}}=\zero,
	\end{eqnarray}
	where 	  $  {j_s} \in    \ilset{  0, 1,\cdots,  \TimeMapping{t} } $  for all $1\le s \le S$, and $ \boundedFn\parenthesis{\cdot}$ is any bounded and continuous real-valued function. 
	Since the processes are within the expectation operator $\E$ in (\ref{eq:tildeWtMartingale1}), the underlying probability space is no longer relevant. 
Now based upon  the weak convergence (\ref{eq:WeakConvLimit}) derived from Lemma \ref{lem:tight} and the Skorohod embedding argument,  we claim that $ \ilparenthesis{\bZ\aSubseqDependence\ilparenthesis{\cdot,\upomega},   \PROJECTION\aSubseqDependence\ilparenthesis{\cdot,\upomega}, \bW\aSubseqDependence\ilparenthesis{\cdot,\upomega} } $ converges weakly to  $ \ilparenthesis{\bZ\ilparenthesis{\cdot,\upomega},\PROJECTION\ilparenthesis{\cdot,\upomega},\bW\ilparenthesis{\cdot,\upomega}} $ as $n\to\infty$
uniformly on each interval $\bracket{0,T}$. 
Together with  the uniform integrability of every element in the sequence  $\ilset{\bW\aSubseqDependence}$, we have 
		\begin{equation}\label{eq:WtMartingale3}
	\E \set{ \boundedFn\ilparenthesis{\hbtheta_{j_1}  \aSubseqDependence  ,\cdots , \hbtheta_{j_S}  \aSubseqDependence  , \boldsymbol{\upeta}_{j_1}  \aSubseqDependence , \cdots, \boldsymbol{\upeta}_{j_S}  \aSubseqDependence } \bracket{\bW\parenthesis{t+\updelta}-\bW\parenthesis{t}}}=\zero 
	\end{equation} where $ \boundedFn\parenthesis{\cdot}$ is any bounded and continuous real-valued function of its arguments and 	  $  {j_s} \in    \ilset{  0, 1,\cdots,  \TimeMapping{t} } $  for all $1\le s \le S$. By \cite[Thm. 7.4.1 on p. 234]{kushner2003stochastic},  (\ref{eq:WtMartingale3}) implies that  $\bW\parenthesis{\cdot}$ is  a martingale, 
	\begin{equation}
	\E\bracket{\given{\bW\ilparenthesis{t+\updelta}-\bW\parenthesis{t}}{\bZ\parenthesis{s},\PROJECTION\parenthesis{s},s\le t}} = \zero.
	\end{equation}
	
		Combined with the   weak convergence result  (\ref{eq:WeakConvLimit})   derived from  Lemma   \ref{lem:tight}, 	  Lemma \ref{lem:UnifEquiCont} shows  that both $\bZ\parenthesis{\cdot}$ and $\PROJECTION\parenthesis{\cdot}$ have Lipschitz continuous paths w.p.1. By     \cite[Thm. 4.1.1 on p. 98]{kushner2003stochastic}, (\ref{eq:Wt1}) implies      that $\bW\parenthesis{\cdot}$ is a constant w.p.1. Since $\bW\parenthesis{0}=\zero$, we have $ \bW\parenthesis{t} =\zero $ for all $t$. Ultimately,      for all $\upomega\notin \nullset$ where the null set  $\nullset$  is specified in Lemma \ref{lem:UnifEquiCont}, we have  	\begin{equation}
		\bZ\parenthesis{t,\upomega} = \bZ\parenthesis{0,\upomega}  -\int_0^t \obg\parenthesis{s,\bZ\parenthesis{s,\upomega}}\diff s + \PROJECTION\parenthesis{t,\upomega}
		\end{equation}   
		Note that $\PROJECTION\parenthesis{0,\upomega}=\zero$ and $\bZ\ilparenthesis{t,\upomega} \in \bTheta$ for all $t$. Namely, the process $\PROJECTION\ilparenthesis{\cdot,\upomega}$ is constructed to balance  the dynamics $\obg\ilparenthesis{\cdot,\bZ\parenthesis{\cdot,\upomega}}$ at each time $t$, so that  $\bZ\ilparenthesis{\cdot,\upomega}$ is within $\bTheta$ for all time. Specifically, $\boldsymbol{\upeta}_k  = \zero$ if $\hbtheta_{k+1} \in \bTheta^0$ and  $\boldsymbol{\upeta}_k   \in -  \ConvCone\ilparenthesis{\hbtheta_{k+1}}$ if    otherwise.  Therefore, for $s>0$, $ \norm{ \PROJECTION\parenthesis{t+s,\upomega} - \PROJECTION\parenthesis{t,\upomega} } \le \int_t^{t+s} \norm{\obg\parenthesis{u,\bZ\parenthesis{u,\upomega}}} \diff u $, and therefore $\PROJECTION\parenthesis{\cdot, \upomega}$ is Lipschitz-continuous for all $\upomega\notin \nullset $. By \cite[Thm. 4.3.1 on p. 109]{kushner2003stochastic}, we may write $\PROJECTION\parenthesis{t }= \int_0^t \projectionlower\parenthesis{s }\diff s ${ where $\projectionlower\parenthesis{t }\in- \ConvCone\parenthesis{\bZ\parenthesis{t}}$ for almost all $t$}.  
\end{proof}

In summary, Theorem~\ref{thm:ODE}    deals with the limit of the sequence of the measures induced by  the processes $\bZ\aDependence\ilparenthesis{\cdot}$ on the appropriate path space, and  the limit measure corresponds to  a process on the path space supported on some set of limit trajectories of the ODE (\ref{eq:LimitingODE}).  Moreover, when 
B.\ref{assume:gSequenceRegularity}   
 holds, the solution to ODE (\ref{eq:LimitingODE}) is well-defined on the entire real line due to Corollary~ \ref{corr:ExtensionThem}.

Theorem~\ref{thm:ODE} informs us that   the process $\ilset{\hbtheta_k}$ is shown to spend nearly all of its time arbitrarily close to the the limit set $ \compactset_{\bTheta} \equiv  \lim_{t\to\infty}  \cup _{\hbtheta_0\in\bTheta} \set{\btheta\parenthesis{s},s\ge t: \btheta\parenthesis{0} = \hbtheta_0} $,  where $\btheta\ilparenthesis{\cdot}$ is the solution to a time-dependent ODE (\ref{eq:LimitingODE}). Unfortunately, since the driving term $\obg\ilparenthesis{t,\btheta}$ depends on both $t$ and $\btheta$,   we do not have much information regarding the  limit set $ \compactset_{\bTheta}   $.

Under the special case that $\bg_k\ilparenthesis{\cdot}$ varies with time yet $\bvartheta_k=\bvartheta$ for all $k$ as in  \cite{wang2014rate}, the   limit set is a singleton invariant set   given that the trajectories are bounded within $\bTheta$ \cite{gucken1983nonlinear}. 
Under this special setting, we have the following corollary of ``$\hbtheta_k$ will spend nearly all of its time in a small neighborhood of  $ \compactset_{\bTheta}$ with an arbitrarily high probability,''  which immediately follows from Theorem~\ref{thm:ODE}.

\begin{corr}\label{corr:ODE}
	Assume B.\ref{assume:UniformIntegrability}, B.\ref{assume:Bias}, B.\ref{assume:samplingfrequency},  B.\ref{assume:gSequenceRegularity}. Further, suppose that the   limit set of the  time-varying ODE (\ref{eq:LimitingODE}) is a unique point $\bvartheta$  and is  asymptotically stable   in the sense of Liapunov (discussed in Subsection \ref{subsect:stability}).  Then 
		for any $\upvarepsilon>0$, the fraction of time that $\bZ^a\parenthesis{\cdot}$ will stay within   the $\upvarepsilon$-neighborhood of the limit set $   \ilset{\bvartheta}  $, on $ \bracket{0,T} $ grows  to one in probability  as $a\to 0$ and $T\to\infty$.   
		Specifically,  there exist $\upvarepsilon_k\to 0$, $T_k\to\infty$ such that $ \lim_{k\to\infty} \Prob   \set{  \sup_{t\le T_k}   \norm{\bZ\aDependence\ilparenthesis{t}-  \bvartheta} \ge \upvarepsilon_k     } = 0       $.   
		
			 \remove{Let there be a unique limit point $\btheta^*$ of the autonomous ODE, and it is asymptotically stable in the sense of Liapunov. Then for each $\updelta>0$, $T>0$, $ \lim_k \Prob\set{ \sup_{t\le T}  \norm{  \btheta^k \parenthesis{t} - \btheta } >\updelta } = 0 $. More generally, there are $\updelta_k\to 0$, $T_k\to\infty$ such that $ \lim_k \Prob\ilset{ \sup_{t\le T_k } \text{dist} \ilparenthesis{ \btheta^k\parenthesis{t}, L_{\bTheta} } \ge \updelta_k } = 0 $. }
			
\end{corr}

\remove{

\begin{proof}[Proof of Corollary~\ref{corr:ODE}] 
\remove{
 Suppose that there exists a set $\compactset$ that is asymptotically stable in the sense of Liapunov and that for small $\updelta>0$, $\limsup_k \Prob\ilset{\upomega:  \norm{\hbtheta_k\aDependence\ilparenthesis{\upomega}- \bvartheta} \le \updelta } >\upmu$ for some $\upmu>0$ and small $a$. Then, the problem of escaping from some larger neighborhood of $\compactset$ is one in the theory of large deviations.  It will be a rare event for small $a$.  }
 
 Let us  prove it by contradiction. Suppose otherwise  that   the statement in Corollary~\ref{corr:ODE} is false.  
 We can then  extract a suitable sequence for which   $\liminf_{k\to{\infty}}\Prob   \set{  \sup_{t\le T_k}  \mathrm{dist}\bracket{\bZ\aDependence\ilparenthesis{t}, \bvartheta} \ge \upvarepsilon_k     } $   is positive\remove{and use the weak-convergence result in Theorem~\ref{thm:ODE}  to get a contradiction}.  Let $T>0$, and \remove{(extracting a  subsequence if necessary) }work with the left shifted random processes $ \ilset{\bZ\aDependence\parenthesis{\cdot-T}, \PROJECTION\aDependence\parenthesis{\cdot-T} } $. For any $\updelta>0$, the time required for the solution of the ODE to reach and remain in a $\updelta$-neighborhood of $\ilset{\bvartheta}$   is bounded for all the  initial condition in $\bTheta$. Since the left-shift amount $T$ is arbitrary, this yields that the solution to the ODE on $\left[0,\infty\right)$ is in $N_{\updelta}\parenthesis{\compactset_{\bTheta}}$ for each $\updelta>0$. 
\end{proof}

}

Unfortunately, we cannot write out an explicit expression for the rate at which the probability $ \Prob   \set{  \sup_{t\le T_k}  \norm{\bZ\aDependence\ilparenthesis{t}- \bvartheta} \ge \upvarepsilon_k     } $ goes to zero. This is why we impose more  assumptions and develop   Section~\ref{sect:ProbBound}.

\section{Probabilistic Bound}\label{sect:ProbBound}

The previous section shows that a proper continuation of the  estimate $\ilset{\hbtheta\aDependence_k}$  converges weakly   to the trajectory of the   mean  ODE (\ref{eq:LimitingODE}) as the constant gain $\gain\to 0$. Under the special case that the limit set of the ODE is a singleton, we have the concentration result in Corollary~\ref{corr:ODE}.  A natural question that follows is the concentration rate. Unfortunately, we cannot determine the distribution of $\ilparenthesis{ \hbtheta_k-\bvartheta_{k}}/\sqrt{\gain} $
due to the unknown evolution of $\bvartheta_k$. 

Therefore, this section instead develops a computable upper bound of the probability that $\hbtheta_k$ generated from  constant-gain SGD algorithm  deviates from the trajectory of the IVP (to appear). The constant-gain SGD recursion   \cite[Chap. 5]{spall2005introduction} is simply replacing $\hbg_k\ilparenthesis{\cdot}$ in (\ref{eq:basicSA}) by $\hbg^{\mathrm{SG}}\ilparenthesis{\cdot}$ in (\ref{eq:Ystationary}) and replacing $\gain_k$ by $\gain$:
\begin{equation}\label{eq:ConstantGainSGD}
\hbtheta_{k+1} = \hbtheta_k - \gain \hbg_k^{\SG}\ilparenthesis{\hbtheta_k},
\end{equation}
where $\hbg_k^{\SG}\ilparenthesis{\cdot}$ is     an unbiased estimator for $\bg_k\ilparenthesis{\cdot}$.  
The main  theoretical result on  the finite-horizon behavior to appear  is quite similar to that in \cite{ljung1983theory}:
\begin{equation}
\forall T<\infty, \forall \upvarepsilon>0, \quad  \Prob\set{  \max_{0\le k \le T/\gain } \norm{\hbtheta_k-\bZ\parenthesis{t_k}}>\upvarepsilon  } \le C\parenthesis{\upvarepsilon,\gain, T},
\end{equation}
where for fixed $T<\infty$, $C\parenthesis{\upvarepsilon,a,T}$ tends to zero as $a$ tends to $0$,   $\bZ\parenthesis{t}$ denotes the solution to the  IVP (to be defined momentarily).  This result asserts that $\ilset{\hbtheta_k}_{k\ge  0 }$ is a  {perturbed} discrete-time approximation of the  {nonautonomous} ODE with discretization step $a$.  
The main tool in establishing the connection is the formula for 
variation of parameters reviewed in Subsection~\ref{subsect:Alekseev}. 
\begin{rem}
	If convenient stability assumptions (similar to Corollary~\ref{corr:ODE})  are satisfied by the IVP, there exists a corresponding statement for infinite $T$ (see  \cite{derevitskii19742} and \cite[Corr. 2 on p. 43]{benveniste2012adaptive} for further details).   For succinctness, we discuss the finite-time performance with $T<\infty$ only. 
\end{rem}

We should point out that the availability of a computable probabilistic bound  requires more stringent assumptions than those  imposed in the previous section. Specifically, this section presents a finite-time probabilistic bound on the accuracy of the estimate (for tracking a discrete-time varying target) coming from a constant-gain SGD algorithm (\ref{eq:ConstantGainSGD}).  \cite{zhu2016tracking} provides the tracking error bound, whereas  the probabilistic bound presented in \cite{zhu2018prob}   characterizes the   behavior of  the estimates during the process of tracking and can be used to characterize the uncertainty via confidence regions.

 \subsection{Basic  Setup}\label{set:SG}
We follow the problem setup (\ref{eq:Minimization}) in Chapter~\ref{chap:FiniteErrorBound}, i.e., our goal is to estimate the time-varying value(s) for $\btheta$ that minimize the instantaneous scalar-valued loss function $\loss_k\ilparenthesis{\cdot}$. Unlike Chapter \ref{chap:FiniteErrorBound} that considers (\ref{eq:basicSA}) in general, we consider the special case (\ref{eq:ConstantGainSGD}).

Consider the following IVP: 
\begin{equation}\label{eq:IVP}
\begin{dcases}
&\frac{\diff}{\diff t}\btheta\parenthesis{t} = - \bg\parenthesis{t,\btheta},\quad t\ge t_0 ,  \\
&\btheta\parenthesis{t_0 }=\hbtheta_0,
\end{dcases}
\end{equation}
and its \emph{perturbed} system 
\begin{equation}\label{eq:PerturbedIVP}
\begin{dcases}
&\frac{\diff}{\diff t}\bzeta\parenthesis{t} = - \bracket{\bg\parenthesis{t,\bzeta}+\bxi \parenthesis{t,\bzeta}},\quad t\ge t_0,\\
&\bzeta\parenthesis{t_0}=\hbtheta_0,
\end{dcases}
\end{equation}  where both $\bg\ilparenthesis{\cdot,\cdot}$ and $\bxi \ilparenthesis{\cdot,\cdot} $ are maps from $\real\times\real^p$ to $\real^p$.

Consider    $\bg\ilparenthesis{\cdot,\cdot}$ defined in  (\ref{eq:gInterpolation}). This    $\bg\parenthesis{t,\btheta}$ is measurable   in $t$ and  continuously differentiable in $\btheta$ with bounded derivatives uniformly w.r.t. $t$.  
The $\bxi\ilparenthesis{\cdot,\cdot}$ function can be similarly defined by substituting $\bxi_k\ilparenthesis{\cdot}$ for  $\bg_k\ilparenthesis{\cdot}$    in (\ref{eq:gInterpolation}). Such  
$\bxi\parenthesis{\cdot, \cdot}$ is  measurable in $t$ and Lipschitz in $\btheta$ uniformly  w.r.t. $t$. Now that     $\bxi\parenthesis{t,\bzeta}$ is  the linear interpolation of the measurement noise $\bxi_k \parenthesis{\upomega}$ at sample point $\upomega$  (the dependence on $\upomega$ is suppressed), we will analyze the behavior of $\bzeta\parenthesis{t}$ at each  sample point, i.e., with a fixed $\upomega$, the system (\ref{eq:PerturbedIVP}) is effectively deterministic at a given sample point $\upomega$.   

\subsection{Model Assumptions}

\begin{assumeB}\label{assume:Noise} The   sequence $\set{\bxi _ k}_{k=0}^{K-1}$ is mutually \emph{independent}, not necessarily identically distributed random vectors with mean $\zero$ and bounded magnitude $ \norm{\bxi_k }\le \noiseBound  $ for all $k$ almost surely. The value of the error does not depend on the evaluation point $\btheta$.  
\end{assumeB}

\begin{assumeB}\label{assume:TwiceDiff}  The function $\bg _k$ is continuously differentiable. Furthermore, $\convexPara$ is the smallest positive real such that $\bg_k\ilparenthesis{\cdot}$ satisfies  $   \ilparenthesis{\btheta - \bzeta }^\transpose\ilparenthesis{\bg_k\ilparenthesis{\btheta}-\bg_k\ilparenthesis{\bzeta}} \ge \convexPara\norm{\btheta - \bzeta }^2    $ for all $\btheta,\bzeta\in\real^p$ and    all $k$. 
\end{assumeB}
 {
	Denote       $\bH_k\parenthesis{\btheta}=\partial\bg_k\parenthesis{\btheta}/\partial\btheta^\transpose$.  The following statements regarding   $\bg_k$, $\bH_k$, and $\btheta_k^*$ should be interpreted in the  a.s.   sense  if   randomness is involved. }

\begin{assumeB}\label{assume:Drift}  The magnitude of the (discrete-time) varying gradient function  is strictly bounded:   $\norm{\bg_k\parenthesis{\btheta}}\le \gradientBound $  for all $k$ and $\btheta$.     
\end{assumeB}

Here are some implications of the assumptions.

  \begin{itemize}
  	
  	\item  Under B.\ref{assume:Noise}, the noise term does not depend on $\btheta$ at all. We may use $\bxi \ilparenthesis{t}$ and $\bxi\ilparenthesis{t,\btheta}$ interchangeably for the rest of our discussion.

  	Note that   the function $\bxi \parenthesis{t}$ is, in fact,  random, since it depends on the specific sample point  $\upomega\in\Omega$ of the stochastic process $ \set{\bxi_k \parenthesis{\upomega}}_{k\ge 0} $. Namely, only one trajectory is under consideration for deterministic $\bxi\parenthesis{t}$, whereas the \emph{average} performance of a collection of all possible realizations of trajectories of $\set{\bxi_k \parenthesis{\upomega}}$ has to be taken into account. For succinctness, we suppress the dependence of $\bxi\parenthesis{t}$ on $\upomega$. 
  	
  	\item 
One direct consequence of  B.\ref{assume:TwiceDiff} and  B.\ref{assume:Drift}   is $\norm{\btheta_{k+1}^*-\btheta_k^*} \le 2 \convexPara^{-1} \gradientBound$ from (\ref{eq:ineq2}) and (\ref{eq:ineq1}),  i.e.,   the change of the optimal parameter between every two consecutive discrete time instances is strictly bounded. This resembles assumption  A.\ref{assume:BoundedVariation}.  

\item 
Under B.\ref{assume:gSequenceRegularity}, 
the IVP (\ref{eq:IVP}) has a unique solution over $\ilbracket{0,T}$ for any finite $T$. To see this, notice that $\bg\ilparenthesis{t,\btheta}$ shares a common Lipschitz constant  $\LipsPara$
w.r.t. $\btheta$ for every $t$. Therefore, the existence, uniqueness, and extensibility (to $t\in\real$) follow immediately from Corollary~\ref{corr:ExtensionThem} \remove{
	\cite[Thm. 2.17]{teschl2012ordinary}}.  
Furthermore, the Lagrange stability of the solution to (\ref{eq:IVP}) follows from the Gronwall-Bellman inequality 
\cite[Lemma 1 on p. 35]{bellman1953stability}.

\item   Under B.\ref{assume:gSequenceRegularity} and B.\ref{assume:Noise}, the IVP (\ref{eq:PerturbedIVP}) admits a unique solution. To see this, notice that the driving term, $-\bg\ilparenthesis{t,\btheta}-\bxi\ilparenthesis{t,\btheta}$ is piecewise continuous in $t$, and is  Lipschitz continuous  w.r.t. $\btheta$, the global existence and uniqueness follow directly from \cite[Thm. 3.2 on p. 93]{khalil2002noninear}.

\end{itemize}

\subsection{Main Results}

Let us mention one  caveat  before we present the main results.  It is desirable to increase $\gain $ for maintaining tracking momentum, whereas it is necessary to decrease $\gain$ for better tracking accuracy when $\bvartheta_k$ is fixed at one value. Nonetheless, the gain selection is not the central topic here; it was touched on in the previous chapter and will be further discussed in Chapter~\ref{chap:AdaptiveGain}. We assume that the  pre-determined  gain $a$ enables  the SGD algorithm (\ref{eq:ConstantGainSGD})  with a constant gain  $\gain$  to keep track  of the target $\btheta_k^*$. Without a carefully-tuned  constant gain $\gain$, once the estimate $\hbtheta_k$ deviates significantly from  the target $\btheta_k^*$, it is likely to lose it ever after. The following discussion is based upon the availability of a tuned gain $a>0$.

Let us first discuss a lemma to handle the noise term later on.

\begin{lem}\label{lem:Noise}Assume B.\ref{assume:Noise}.
	For an arbitrarily fixed  $\updelta>0$ and finite time-horizon $T>0$, 
	\begin{equation}
	 \Prob\set{\upomega:\norm{\int_0^T\bxi\parenthesis{s}\diff s}>\updelta}=\parenthesis{p+1} \exp\parenthesis{-\frac{\updelta^2}{ a\noiseBound \parenthesis{T\noiseBound/2+\updelta/3} }},
	 \end{equation}
	where   the r.h.s. approaches $0$ exponentially as $\gain \to 0$, and   the sample point $\upomega\in\Omega$ determines\footnote{Once a $\upomega$ is picked, the entire sequence $\bxi_k $ is determined.} the entire measurement noise sequence  $ \set{\bxi_k \parenthesis{\upomega }} _{k\ge 0}$.
\end{lem}
\begin{proof}[Proof of Lemma~\ref{lem:Noise}]
	Without loss of generality, assume that $ \TimeMapping{T} $ defined in Subsection~\ref{subsect:Interpolation} equals  $T/a \equiv K$. Denote    the  variance  statistics of the sum as:
	\begin{align}
&	\upnu_K    \equiv\max\left\{  \norm{   \Var\parenthesis{\bxi_0 } + 2\sum_{k=1}^{K-2}\Var\parenthesis{\bxi_k } + \Var\parenthesis{\bxi_ {K-1}  }   },  \right.   \nonumber\\
&\quad\quad\quad\quad\quad \left.  \E\bracket{\norm{\bxi _0}^2  + 2 \sum_{k=1}^{K-2}  \norm{\bxi_k }^2 + \norm{\bxi_ {K-1}}^2 }  \right\}  
	\end{align}  
	Under B.\ref{assume:Noise},  
	$\norm{\Var\parenthesis{\bxi_k}}   \le \tr\bracket{   \E \parenthesis{ \bxi_k \bxi_k^\transpose }  }  \le \noiseBound ^2 $, and $\E\norm{\bxi_k}^2 \le \noiseBound^2$ for all $0\le k \le K-1$. Therefore, $\upnu_K\le 2K\noiseBound^2 $.  By  \cite[Thm. 1.6.2 on p. 13]{tropp2015introduction}, we have
	\begin{align}
	\label{eq:eErrorBound}
	&	 \Prob\set{\upomega:\norm{\int_0^T\bxi \parenthesis{s}\diff s}>\updelta} \nonumber\\
	&\quad   = \Prob\set{\upomega:\frac{a}{2}\norm{\bxi_0\parenthesis{\upomega}+ 2\sum_{k=1}^{K-2} \bxi_k\parenthesis{\upomega} + \bxi_{ K-1}\parenthesis{\upomega} }>\updelta} \nonumber\\ 
	&\quad \le \parenthesis{p+1} \exp\parenthesis{-\frac{\updelta^2}{ a\noiseBound \parenthesis{T\noiseBound/2+\updelta/3} }}.
	\end{align}
	The r.h.s. of (\ref{eq:eErrorBound})   approaches $0$ exponentially as $a$ approaches $ 0$  for  fixed  $\updelta$.  
\end{proof}

 Now we present the main theorem in computing the probabilistic bound. 

\begin{thm} \label{thm:ProbBound} Assume  B.\ref{assume:gSequenceRegularity},  B.\ref{assume:Noise}, B.\ref{assume:TwiceDiff}, and  B.\ref{assume:Drift}. 
	For an arbitrarily fixed    finite   $T>0$, and for any   threshold $\upvarepsilon>0$, \begin{align}
	&	\Prob\set{ \max_{0\le k\le T/\gain } \norm{\hbtheta_k-\btheta \parenthesis{t_k }}>\upvarepsilon }\nonumber\\   
	&\quad \le \parenthesis{p+1} \exp\bracket{   -\frac{  \parenthesis{\upvarepsilon  e^{\convexPara} /2 }^2   }{a \noiseBound \parenthesis{ T\noiseBound/2 + \upvarepsilon e^{\convexPara} / 6   } }   }, 
	\end{align}
	where  the r.h.s. approaches $0$ exponentially as $\gain \to 0$, and   $\btheta\ilparenthesis{t}$ is the solution to (\ref{eq:IVP}). 
\end{thm}

\begin{proof}[Proof of Theorem~\ref{thm:ProbBound}]
	Define the following time-dependent continuous function:
	\begin{eqnarray}\label{eq:Zbar} 
	\contZ \parenthesis{t }  = \sum_{k=0}^{\infty} \set{  \bracket{  \frac{t_{k+1}-t}{a} \bZ\ilparenthesis{t_k} + \frac{t-t_k}{a} \bZ\ilparenthesis{t_{k+1}} }\cdot\indicator_{\ilset{t_k\le t<t_{k+1}}}  }, 
	\end{eqnarray}  
	where $\bZ\ilparenthesis{t}$ was defined in (\ref{eq:Zbar0}). Note that $\contZ \ilparenthesis{t}$ is the linear interpolation of $\ilset{\hbtheta_k}$ at times $t_k=ka$. 
	We have
	\begin{eqnarray}
	\contZ\ilparenthesis{t_{k+1}} = \hbtheta_{k+1} = \contZ\ilparenthesis{t_k} -\gain \ilbracket{  \bg\ilparenthesis{t_k,\contZ\ilparenthesis{t_k}} + \bxi\ilparenthesis{t_k,\contZ\ilparenthesis{t_k}} }. 
	\end{eqnarray}
	
	To establish a connection between $\contZ\ilparenthesis{t_k}$ and $\btheta\ilparenthesis{t_k}$, we invoke the triangle inequality and analyze the behavior of two terms: (1) $ \max _{0\le t_k\le T} \norm{\btheta\ilparenthesis{t_k}-\bzeta\ilparenthesis{t_k}} $ where $\bzeta\ilparenthesis{t}$ is the solution to (\ref{eq:PerturbedIVP}), and (2) $ \max_{ 0\le t _k\le T} \norm{\contZ\ilparenthesis{t_k}-\bzeta\ilparenthesis{t_k}} $.
	\begin{enumerate}[(1)]
\item 	First,  consider   term $ \max _{0\le t_k\le T} \norm{\btheta\ilparenthesis{t_k}-\bzeta\ilparenthesis{t_k}} $. The difference between the solution to the system (\ref{eq:IVP}) and the solution to the perturbed system (\ref{eq:PerturbedIVP}) can be handled by the Alekseev's formula reviewed in Subsection~\ref{subsect:Alekseev}. All the necessary conditions to invoke the  Alekseev's formula  are met: $\bg\ilparenthesis{t,\btheta}$ is  continuously differentiable w.r.t. $\btheta$ under B.\ref{assume:TwiceDiff} and the construction (\ref{eq:gInterpolation}),  the magnitude of $\partial \bg\parenthesis{t,\btheta}/\partial\btheta$ is uniformly  bounded under B.\ref{assume:gSequenceRegularity}, and $\bxi\parenthesis{t }$   does not depend on $\btheta$ under B.\ref{assume:Noise}.     Let us invoke the  (uniform) bound on the norm of the fundamental matrix \remove{$ \bPhi\ilparenthesis{t;s,\hbtheta_0} $  }provided  in       \cite[Thm. 1]{brauer1966perturbations}: 
	\begin{align}
	&\norm{\btheta\parenthesis{t;t_0,\hbtheta_0}-\zeta\parenthesis{t;t_0,\hbtheta_0}} \nonumber  \\
	&\quad\le \big \| \int_{t_0}^t \exp\set{ \int_s^t\bracket{\uplambda_p\parenthesis{\frac{\partial \bracket{-\bg\parenthesis{s,\btheta}}}{\partial\btheta^\transpose}}}   }  \cdot\bxi \parenthesis{s}\diff s \big\|  \nonumber \\
	&\quad\le e^{-\convexPara} \norm{\int_{t_0}^t\noise\parenthesis{s,\bzeta\ilparenthesis{s;t_0,\hbtheta_0}}\diff s },
	\end{align}
	where the second inequality uses   B.\ref{assume:TwiceDiff} and  $ \uplambda_{\max} \ilset{ -\partial\ilbracket{\bg\ilparenthesis{t,\btheta} } / \partial\btheta^\transpose }\le -\convexPara $.  The notions $ \btheta\ilparenthesis{t;t_0,\hbtheta_0} $ and $\bzeta\ilparenthesis{t;t_0,\hbtheta_0}$ are to emphasize the dependence of the initialization of $\hbtheta_0$ at $t_0$ in  Alekseev's formula  reviewed in Subsection~\ref{subsect:Alekseev}. Besides, B.\ref{assume:gSequenceRegularity} implies    $ \uplambda_{\min } \ilset{ -\partial\ilbracket{\bg\ilparenthesis{t,\btheta} } / \partial\btheta^\transpose }\ge  -\LipsPara $.     Therefore, for  arbitrarily given threshold $\upvarepsilon >0$, \begin{align}\label{eq:firstpart}
	&\Prob\set{\max_{0\le t_k \le T}\norm{\btheta \parenthesis{t_k}-\bzeta \parenthesis{t_k}}>\upvarepsilon } \le \Prob\set{ e^{-\convexPara}  \norm{  \int_{0 }^T \bxi\parenthesis{s } \diff s  }>\upvarepsilon  }. 
	\end{align} 
	
\item 	Now consider the term $ \max_{0\le t_k \le T} \norm{\contZ\parenthesis{t_k}-\bzeta\parenthesis{t_k}}  $. 	\remove{Denote $\bm{f}\parenthesis{t,\bzeta}= - \bg\parenthesis{t,\bzeta}-\bxi \parenthesis{t,\bzeta }$.  The solution to (\ref{eq:PerturbedIVP}) has a second-order derivative of  $\ddot{\bzeta}\parenthesis{t}= \partial \bm{f} /\partial t +  \parenthesis{\partial \bm{f} /\partial \bzeta^\transpose} \cdot \bm{f}\in\real^{p } $,    whose magnitude is bounded by $\sqrt{2}\noiseBound  + 2\gradientBound + \parenthesis{\gradientBound+\noiseBound}\LipsPara$.   }
	\cite[Thm. 212A]{butcher2016numerical} shows that   $ \max_{0\le t_k \le T} \norm{\contZ \parenthesis{t_k}-\bzeta\parenthesis{t_k}}  $ is   bounded from above by $O\parenthesis{a}$ with a bounded constant term\remove{ of  $ \LipsPara^{-1} \parenthesis{e^{\LipsPara T}-1} \max_{t}\norm{\dot{\bzeta}\parenthesis{t}  }   $ which is bounded}.  Therefore, the difference between the linear interpolation of the  noisy discretization $ \hbtheta_{k+1}= \hbtheta_k - a\hbg_k\ilparenthesis{\hbtheta_k} $ and the perturbed system $ \dot{\bzeta}\parenthesis{t }=-\bg\ilparenthesis{t,\bzeta}-\bxi \parenthesis{t,\bzeta} $   diminishes  to zero as the discretization interval $ a$ approaches $ 0 $.
\end{enumerate}

Any sample point $\upomega$ in the intersection of     the event  $     \set{ \upomega: \max_{0\le t_k \le T} \norm{\btheta\parenthesis{t_k,\upomega}  - \bzeta\parenthesis{t_k,\upomega}}<\upvarepsilon/2 } $ and $ \set{ \upomega:  \max_{0\le t_k \le T}\norm{ \bzeta\parenthesis{t_k,\upomega}-\contZ\parenthesis{t_k ,\upomega}}  <\upvarepsilon/2} $  must fall within the event $ \set{\upomega: \max_{0\le t_k \le T} \norm{\contZ\parenthesis{t_k,\upomega} - \btheta\parenthesis{t_k,\upomega}} <\upvarepsilon  } $. Part (2) establishes that $\max_{0\le t_k \le T}\norm{\bzeta\parenthesis{t_k} - \contZ\parenthesis{t_k}}<\upvarepsilon/2$ is valid almost surely as long as the constant gain $a\remove{\le \upvarepsilon\LipsPara\parenthesis{e^{\LipsPara T}-1}^{-1} / \ilparenthesis{2\max_t \norm{\ddot{\bzeta}\parenthesis{t}}}
}$ is smaller than a certain threshold specified in \cite[Thm. 212A]{butcher2016numerical}. Combined, for \emph{certain} gain $a>0$ satisfying this condition, the probability that $\hbtheta_k$ deviates from $\btheta\parenthesis{ka}$ is bounded from above by  \begin{equation*}
C\parenthesis{a,\upvarepsilon, T} = \parenthesis{p+1} \exp\bracket{   -\frac{  \parenthesis{\upvarepsilon  e^{\convexPara} /2 }^2   }{a \noiseBound\parenthesis{ T\noiseBound/2 + \upvarepsilon e^{\convexPara} / 6   } }   }. 
\end{equation*}
\end{proof}

\subsection{Further Remarks}\label{subsect:distinction2}

This section analyzes the recursive iterates via the solution to an IVP. Some subtleties are worth mentioning.   In the  classical setting of decaying gain and fixed  underlying parameter to be identified, the stationary point of the limiting autonomous ODE  is shown to be the limit point of general SA algorithms under certain conditions \cite[Sect. 4.3]{spall2005introduction}. However, for constant-gain algorithm designed to minimize a time-varying objective function $\loss_k\parenthesis{\cdot}$, it is not justified to transfer the terminologies, such as the concept of equilibrium, from  an autonomous ODE to nonautonomous ODE (where the forcing term has explicit dependence on time).  \remove{Autonomous system generates \emph{semigroups} while nonautonomous system generates \emph{cocycles}. \emph{Cocycle attractor} for nonautonomous system is a proper generalization of the \emph{semigroup attractor} for autonomous systems. Such a cocycle attractor consists of a \emph{family} of equivariant compact subsets of the state space instead of just a \emph{single} positively invariant subset as for a semigroup attractor.}The recursive estimates never settle  if  the underlying parameter is  perpetually  time-varying.\remove{Furthermore, the Lyapunov stability and asymptotic stability are mostly applicable for autonomous systems\textemdash extra caution is needed in handling the nonautonomous system. }
Many prior works on   tracking    problems assume that the time-varying objective function $\loss _k\parenthesis{\cdot}$ and its gradient function, evaluated at the values within  the allowable region, have fixed limiting values\footnote{	For a continuously differentiable  function   that has a limit as $t\to\infty $, i.e., $ f\parenthesis{t}\stackrel{t\to\infty}{\longrightarrow}\ell $, it is \emph{not} necessarily the case that $ f'\parenthesis{t}\stackrel{t\to\infty}{\longrightarrow}  0 $. 		
	However, 	if a function $f\in C^1\parenthesis{\real}$ satisfies both  $f\parenthesis{t}\stackrel{t\to\infty}{\longrightarrow}  \ell$  and $ f'\parenthesis{t}\stackrel{t\to\infty}{\longrightarrow} \ell' $, then we can safely conclude that $ \ell'=0 $.  }. Such assumption essentially forces the slowly time-varying parameter to converge to a limit  for large $k$, and the limit point of recursive estimates will eventually coincide with the equilibrium of limiting autonomous ODE. However, this condition may not fit practical scenarios.  Also note  that this section focuses  on the SGD algorithm  (\ref{eq:ConstantGainSGD}), where  direct unbiased measurement of the unknown gradient is available.

 {

\subsection{One Quick Example}
This subsection provides a synthetic study in tracking a jump process  to illustrate the effects of the noise, the drift, and the gain on the tracking capability.

We     aim to track a  jump process    $\btheta_k^*$. For every $k$, $\btheta_{k+1}^*$   remains  the same as $\btheta_k^*$ with  a   probability  of   $0.9995$, and  $\btheta_{k+1}^*=\btheta_k^*+\bv_k$ with   a  probability  of   $0.0005$, where $\bv_k$ is independent and  uniformly distributed on a spherical  disc, with a  radius of $\gradientBound$, centered at the origin.  The time-varying loss function is  $ \loss _k\parenthesis{\btheta}= \E\norm{\btheta-\btheta_k^*}^2/2 $, and the corresponding gradient function  is  $ \bg_k\parenthesis{\btheta} = \btheta-\E \btheta_k^* $   per discussion on \cite[p. 70]{spall2005introduction}.  However, the accessible information is the noisy gradient evaluation $ \hbg_k\parenthesis{\btheta} = \bg_k\parenthesis{\btheta} + \bxi_k  $, where $ \bxi_k$ 
follows a truncated normal distribution with mean $0$, positive definite matrix $\bSigma$, and truncation bounds  $ \bracket{l,u} $ on each component of $\be_k$.   This distribution satisfies B.\ref{assume:Noise}.

To illustrate, we pick $ p=2 $, $ \hbtheta_0=\btheta_0^*=\zero $, $ a=0.1 $, $ T=5000 $, $\gradientBound=50$,   $ \bSigma=\upsigma^2\bI $ with $ \upsigma=1 $, and the truncated  normal with $l=-3$ and  $u=3$. Figure \ref{figure1} is the scatter plot  of a single realization of $\hbtheta_k$ and the underlying jump process $\btheta_k^*$. It is visually obvious that       the estimates $ \hbtheta_k $ are  capable of tracking the  time-varying jump process $ \btheta_k^* $. In terms of the tracking speed and accuracy, when a jump in the  $\ilset{\btheta^*_k}$ sequence occurs, it takes \emph{at most} $29$ iterations for the estimates $\hbtheta_k$ to fall within the ball with a radius of $2$ and a center of the newest value of  $\btheta^*$. 

 \begin{figure}[!htbp]
	\centering
	\includegraphics[width=0.65\textwidth]{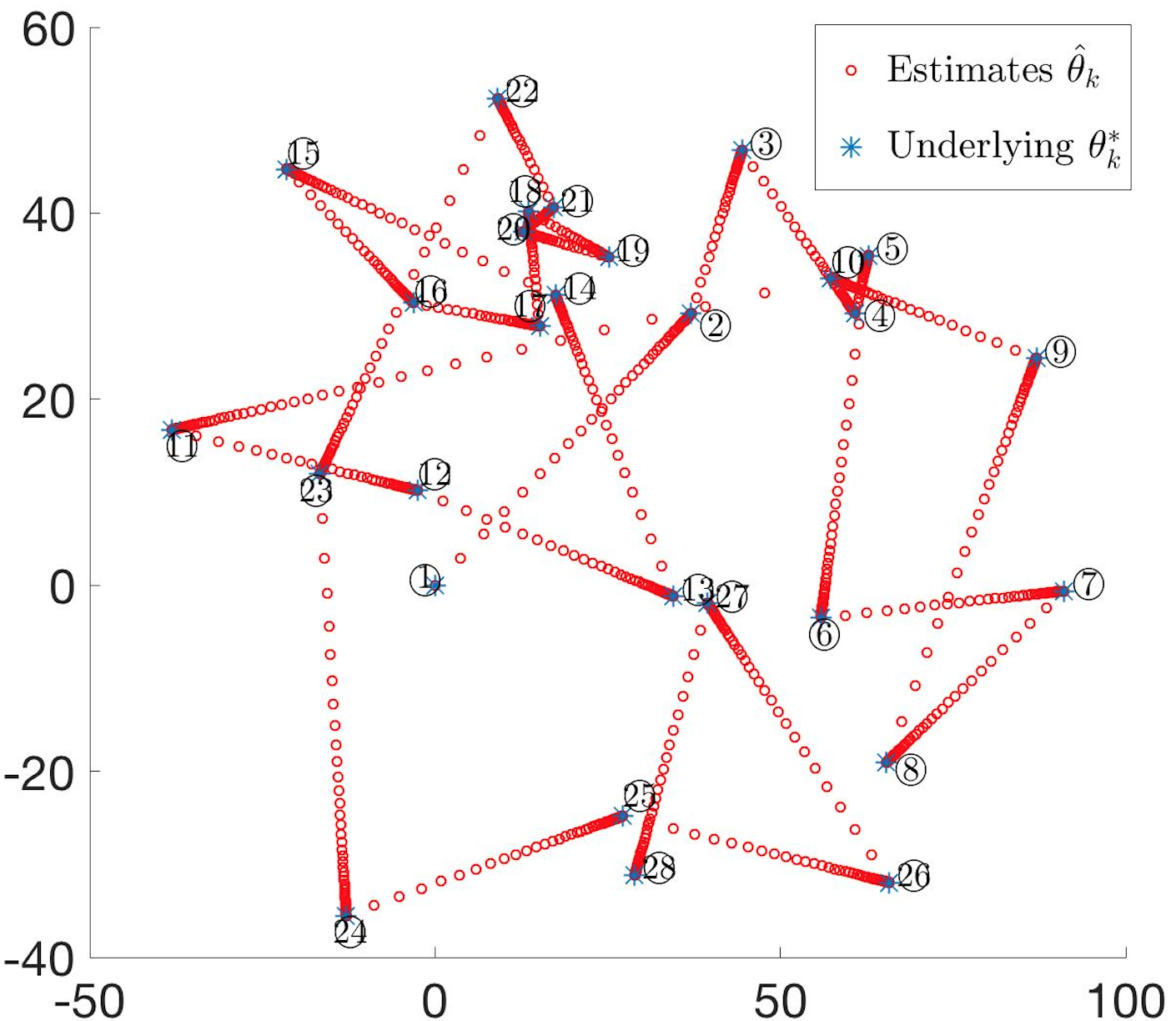} 
	\caption{The underlying time-varying jump process $\btheta_k^*$  and $\contZ \parenthesis{t}$ generated by $\hbtheta_k$, with $a=0.1$. For all $k$,  we have $ \contZ (t_k)=\hat{\bm{\uptheta}}_k$. The number in the circles corresponds to the counter of the jumps. }
	\label{figure1}
\end{figure}

We also  run $100,000$ replicates for $a=0.1$, and the empirical probability is listed in Figure \ref{figure3}. The empirical probability  for the event $ \max_{0\le t_k \le T}\norm{\hbtheta_k-\btheta\parenthesis{t_k}}>4$ happening is $ 0.88 $. Any $\upvarepsilon\ge 7 $ gives an empirical probability of zero. Note that the magnitude of $\upvarepsilon$ is still small compared to the possible jump magnitude $\gradientBound=50$.  This phase transition (the probability is either very close to one or very close to zero) may be attributed to these two main reasons: (1) there is a certain (unknown) stability region for the constant gain $a$ and (2) the probability bound in Theorem \ref{thm:ProbBound} is \emph{not} \emph{uniformly tight} for all $\upvarepsilon$. Overall, the trajectory of $\bZ\parenthesis{t}$, the linear continuation of $ \hbtheta_k $, can be characterized by the trajectory of $ \btheta\parenthesis{t} $, the solution to IVP (\ref{eq:IVP}).

\begin{figure}[thpb]
	\centering
	\includegraphics[width=0.7\textwidth]{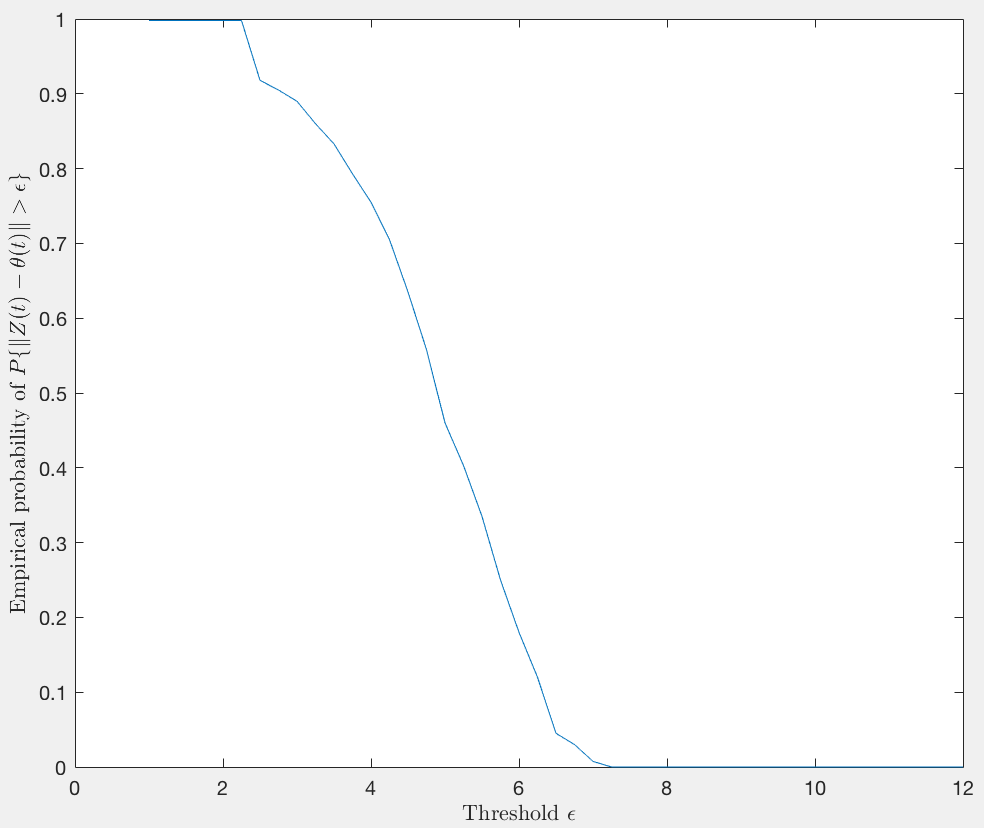} 
	\caption{The empirical probability that $\bZ\parenthesis{t}$ deviates from $\btheta\parenthesis{t}$ by at least $\upvarepsilon$ as a function of  $\upvarepsilon$. Note that  $ \bm{Z}(t_k)=\hat{\bm{\uptheta}}_k, \forall k$.  }
	\label{figure3}
\end{figure}

\section{Concluding Remarks}\label{sect:Limiting}

Our work investigates a class of stochastic approximation algorithms that allows for  time-varying loss functions and  nonlinear dynamics.  
In the nonstationary scenario, we cannot expect ``convergence'' for the constant-gain algorithm, due to a combination of observation noise and nonnegative gain. The best we can hope for is to get into a neighborhood of the optimizer (sequence).
The practical  implication of the weak convergence result and the  probabilistic bound for SA-like tracking algorithms are listed below. 
\begin{itemize}
	\item  
	Our framework does  {not} require an explicit model for  the time variations of $\btheta_k^*$ because they are typically  unknown in reality. Instead, we   ``bury'' the variations of $\btheta_k^*$ in either B.\ref{assume:gSequenceRegularity} or B.\ref{assume:Drift} as applied to  $\bg_k\ilparenthesis{\cdot}$. The analysis of  the time-varying framework is more challenging, as the classical  SA  techniques cannot be applied.

	\item The time-varying assumption imposed on the loss function $\loss _k$  is useful and  {necessary} when the underlying system is   time-varying,   when successive iterations are performed on different components of the independent variable (e.g., the alternating minimization procedure), or when the experimental procedure varies with $k$, or when specific variance reduction methods (e.g. stratified sampling) are employed, and so on.  
	
	\item Many prior works impose assumptions on the noise process and the time-varying sequences so that the dynamics ``average out'' to a function that does not depend on time. However, this is rarely the case in applications.  In our case, the    mean ODE can be time-dependent (nonautonomous).
	\item The result in Theorem~\ref{thm:ODE} and Corollary~\ref{corr:ODE} informs us that, the smaller the step-size, the better $\hbtheta_k$ approximates the trajectory of $\btheta\ilparenthesis{t}$. However, with the smaller   $\gain$, the number of steps to simulate the time-varying ODE on $\ilbracket{0,T}$ with fixed $T$ grows as $T/\gain$.

	\item Theorem~\ref{thm:ProbBound} characterizes the probabilistic behavior of the recursive SGD estimates over a  finite-time period. Realistically, we cannot achieve many asymptotic (as $k\to\infty$) properties of the recursive estimates, as all algorithms have to stop within finite time.

	\item To guarantee tracking stability,  there exists an upper-bound on the gain sequence. Chapter \ref{chap:FiniteErrorBound} informs us  that, when the sampling frequency is fixed,  there exists a  {lower-bound} on the gain sequence for tracking capability and robustness consideration.  
	
As in  the gain-selection guidance conveyed in the previous chapter, this chapter also informs us that the gain sequence for tracking perpetually varying target should be neither too large nor too small. The trajectory of $\btheta\ilparenthesis{t}$, which is the solution  to the ODE (\ref{eq:LimitingODE}) or (\ref{eq:IVP}), does not coincide with  the true $\bvartheta_k$ sequence at every $\uptau_k$, although they are close, see \cite{wiggins2003introduction}.

	 \remove{A common way of viewing this in application is to view time as ``frozen'' and look at equilibra of the frozen time vector field. These ``instantaneous'' fixed points are given by $\bm{f}\ilparenthesis{t,\btheta}=\zero$. If we can find a point $\ilparenthesis{\bar{t},\bar{\btheta}}  $ such that  $\bm{f}\ilparenthesis{\bar{t},\bar{\btheta}} =\zero$ and $ \left. \frac{\partial \bm{f}\ilparenthesis{t,\btheta}}{\partial\btheta} \right| _{  \ilparenthesis{t,\btheta} = \ilparenthesis{\bar{t},\bar{\btheta}}  }  $ is nonsingular, then by the implicit function theorem, we can find a function $\bar{\btheta}\ilparenthesis{t}$ with $\bar{\btheta}\ilparenthesis{\bar{t} }\bar{\btheta}$ such that $ \bm{f}\ilparenthesis{t,\bar{\btheta}\ilparenthesis{t}} = \zero $ for $t$ in some interval around $\bar{t}$. However, the ``frozen time equilibria'' are not solutions of the nonautonomous vector field.   In fact, it is immediate that if     $\bar{\btheta}\ilparenthesis{t}$ is a solution of the nonautonomous vector field, then it must be constant in time, i.e., $\dot{\bar{\btheta}}\ilparenthesis{t}=\zero$.    }

\end{itemize}

 We should mention that in the results (\ref{eq:PropagationLemma3}), (\ref{eq:PropagationLemma5}), and  (\ref{eq:Tightness}) back in Chapter~\ref{chap:FiniteErrorBound}, the limit is taken  over the iteration number $k$ given the adaptive gain selected according to Algorithm~\ref{algo:basicSA}.   Here, the result in Theorem~\ref{thm:ProbBound}  is valid for the entire time-frame, and the maximization is taken  over the iteration number given a fixed constant gain $\gain$.  The probabilistic bound   in Theorem~\ref{thm:ProbBound} provides a general sense of the likelihood of $\hbtheta_k$ staying close to $\bvartheta_k$ for a constant gain $\gain$  under    Assumptions B.\ref{assume:gSequenceRegularity},  B.\ref{assume:Noise}, B.\ref{assume:TwiceDiff}, and  B.\ref{assume:Drift}.    Besides, in   Theorem~\ref{thm:ODE} here, the weak convergence limit is taken  over the constant gain $\gain$. It should be interpreted that for some nonzero constant gain $\gain$, which needs not go to zero, the continuous interpolation of  the estimates $\hbtheta_k$ will stay ``close'' (in the sense of weak limit) to the ODE (\ref{eq:LimitingODE}) under the conditions  therein and when  the underlying data should change with time   at a rate that is  {commensurate} with what is  determined by the gain.

\remove{ 
\subsection{Convergence rates} 

Let $\tilde{\btheta}^k\ilparenthesis{t}$ for $t\ge t_k$ be the solution to $\dot{\btheta}\ilparenthesis{t}=\projectionlower\ilparenthesis{\btheta\ilparenthesis{t}}$  with $\tilde{\btheta}\ilparenthesis{t_k} = \hbtheta_k$. As before, we can show that for any $T>0$, 
\begin{equation}
\sup_{t\in\ilbracket{0,T}} \norm{ \bar{\btheta} \ilparenthesis{t_k+t}  - \tilde{\btheta}^k\ilparenthesis{t_k+t} } \to 0
\end{equation} as $k\to \infty$. Thus $\dot{\btheta}\ilparenthesis{t}=\projectionlower\ilparenthesis{\btheta\ilparenthesis{t}}$  captures the ``typical'' behavior of (\ref{eq:basicSA}). The ``fluctuation'' part is $ \bar{\btheta} \ilparenthesis{t_k+\cdot }  - \tilde{\btheta}^k\ilparenthesis{t_k+\cdot } $ for $k\ge 0$. 
One can show a ``functional central limit theorem'' to the effect that $ \parenthesis{\bar{\btheta} \ilparenthesis{t_k+\cdot }  - \tilde{\btheta}^k\ilparenthesis{t_k+\cdot } }/\sqrt{a_k}$ converges in distribution to a Gauss-Markov process, i.e., a linear system driven by a Gaussian noise.

In the constant step-size case, analogous results are available in the limit as the constant step-size $a$ tends to zero.

}

\chapter{Data-Dependent Gain-Tuning}\label{chap:AdaptiveGain}

 In time-varying SA problems, the gains  in the recursive schemes must be strictly bounded away from zero  to accommodate the time variability in the target values  $\ilset{\bvartheta_k}$. This characteristic distinctively differs from  the classical SA algorithms with diminishing gains  that place lesser weights on more recent information.   In general, the SA algorithms with non-decaying gain $a_k$, such as     Algorithm \ref{algo:basicSA}   in Chapter~\ref{chap:FiniteErrorBound}, are    capable of tracking time-varying targets.   Nonetheless, the optimal value of $a_k $ depends on the   knowledge of the drift   $\ilparenthesis{\bvartheta_ {k+1}-\bvartheta_k}$,   which we do not know.   Therefore, we have to provide an estimate  for the step-size $\gain_k$ on top of the estimation of $\bvartheta_k$ in Chapter~\ref{chap:FiniteErrorBound}.  Often, a constant gain, $\gain_k=\gain$ for all $k$, is used in (\ref{eq:basicSA}), for both the ease of implementation and the consequent tracking algorithm robustness. It has been  observed   that the constant-gain   SGD  algorithm (\ref{eq:ConstantGainSGD}) is  capable of tracking a time-varying target under certain conditions \cite{ljung1977analysis}. However, the tracking performance is rather sensitive  to   the constant gain  $a$, and gain tuning remains  an unsettled practical issue \cite{kushner1995analysis} \cite[p. 160]{benveniste2012adaptive}.

Recall that the gain selection strategy in Algorithm~\ref{algo:basicSA} requires knowledge of both the Lipschitz constant $\LipsPara_k$ and the convexity parameter  $\convexPara_k$, which may be unknown  in practical applications. Though it may be possible to estimate these   parameters by collecting multiple observations at each time instant $k$, such  ``multiple sequential measurements  at a time''  implementation   is contradictory with   the general SA philosophy\footnote{In history, there were attempts to approximate $\loss\ilparenthesis{\btheta}$ by averaging several i.i.d.  measurements of $y\ilparenthesis{\btheta}$. However, this approach turns out to be theoretically inefficient and numerically prohibitive. The cost of  obtaining noisy measurements used to approximate $\loss\ilparenthesis{\btheta}$ at a single point could have been allocated to help to minimize $\loss\ilparenthesis{\btheta}$\textemdash after all, minimization is the primary objective.    } of ``averaging across iterations'' and  the time-varying setting. It is also prohibitive due to  the computational overhead (and possibly equipment cost) within each iteration. Instead, we consider a more restrictive  time-varying scenario summarized in C.\ref{assume:Hybrid},  C.\ref{assume:Stationary}, or C.\ref{assume:Jump} (to appear). With more stringent assumptions, we can detect regime change using  Algorithm \ref{algo:ChangeDetection},   adapting  the gain sequence correspondingly using Algorithm \ref{algo:adaptiveGain}. The main advantages  here are   that we do not require  $\LipsPara_k$ and $\convexPara_k$ to be known in advance for gain-tuning purposes.

\section{Detecting Jumps/Changes}
\label{sect:Detection}
 This section  proposes a method  for jump/change detection. 
 We consider    a special  case summarized in   Assumption C.\ref{assume:Hybrid} to appear, which   is  motivated by the hybrid systems mentioned in Subsection \ref{subsect:Timevaryingness}.  Hybrid systems are  routinely  modeled by a finite  number of \emph{diffusion}s with different drift and diffusion coefficients, and a random \emph{jump} process modulates these diffusions  with a known transition matrix. Since     the diffusion   and  the jump structures  are   rarely  available to the experimenter, we set aside the diffusion component and abstract the jump component  via  C.\ref{assume:Hybrid}. 
To keep it simple, we consider the  constant-gain SGD algorithm (\ref{eq:ConstantGainSGD}), where the gain $\gain$ requires   advance tuning.

\subsection{Basic Change Detection Setup}

In a typical  change-detection setup,  we receive a sequence of observations 
$\observebx_1,\observebx_2,\cdots$, which are realizations of a sequence of  random variables $\bx_1,\bx_2,\cdots$. Several   number of abrupt change points $\upkappa_1,\upkappa_2,\cdots$ divide the sequence of random variables  into segments, where the observations within each segment are i.i.d.\footnote{Although the assumption of independent observation between change points may seem restrictive, this is not the case since a statistical model can usually be fitted to the observations to model any dependence, with change detection then being performed on the independent residuals \cite{gustafsson2000adaptive}. }. That is, 
\begin{equation}
\bx_i\sim\begin{cases}
F_0, & \text{ if } i\le \upkappa_1,\\
F_1, & \text{ if } \upkappa_1+1\le i\le \upkappa_2,\\
F_2, & \text{ if } \upkappa_2+1\le i\le \upkappa_3, \\
\vdots & 
\end{cases}
\end{equation}
for some set of distributions $ \set{F_0, F_1,\cdots} $, and that $ F_i \neq F_{i+1} $ for all $i$. The goal of this section is to estimate the set of change points $\ilset{\upkappa_i}$.

\subsubsection*{Detection Criteria}
The performance of online change detection algorithms is typically measured by two criteria \cite{basseville1993detection}. Take the situation where the length of observations is fixed at $K$ and  there is only one possible change point $\upkappa$ for example.  The first criterion is  the ``average run length,''  $\mathrm{ARL}_0\equiv \E\parenthesis{\given{\hat{\upkappa}}{\bx_i\sim F_0 \text{ for }i\le K }}$,
which is defined as the average number of observations until a changepoint is detected, when the algorithm is run over a sequence of observations with no changepoints (i.e., false positive).  A false positive is said to have occurred if $\hat{\upkappa}<\upkappa$.  The second criterion is  the  ``mean detection delay,'' $ \mathrm{ARL}_1\equiv \E\parenthesis{\given{\hat{\upkappa}-\upkappa}{   \bx_i \sim F_0 \text{ for }i\le \upkappa \text{ and } \bx_i\sim F_1 \text{ for } \upkappa < i \le K  }} $, is defined as the average number of observations between a changepoint occurring and the change being detected (i.e., a mean delay).  
In general, an acceptable value of $\mathrm{ARL}_0$ is chosen before attempting to minimize the detection delay. This is analogous to the Neyman-Pearson testing setup, where a Type-II error is minimized subject  to  the Type-I error being  bounded from above.

\subsubsection*{Relation With Control Chart and Change Detection  }\label{subsect:DetectionDifficulty}

Note that  a great deal of difficulty in our setting comes from that   we only have access to $\ilset{\hbtheta_k}$ instead of $\ilset{\bvartheta_k}$ itself. 

\textbf{Connection.}
The challenge of the online 
monitoring involves a sequence of changes of unknown and
varying magnitude at unknown time instances.  Furthermore,  there is no universal criterion for  accessing the   detection performance in an online  monitoring framework.

\textbf{Distinction.}
A majority of the change detection literature assumes direct access, though it may be noisy, of the underlying process $\bvartheta_k$. However, we only get to access $\hbtheta_k$, whose explicit distributional relation  with the time-varying $\bvartheta_k$ is   {unknown}. As a consequence, our proposed change detection strategy inevitably has lower power compared to the scenario where we can observe $\bvartheta_k$ directly.

\subsection{Model Assumptions} \label{subsect:Hybrid}

We consider a simplified scenario for the hybrid diffusions mentioned at the end of  Sect. \ref{subsect:Timevaryingness}. 
  Initially,  the jump process rests at one of its states/regimes, denoted as $\bvartheta_{\ilparenthesis{\stateregime}}$, and the continuous component evolves per the diffusion process (with associated drift and diffusion).  	Then after a random duration of time $\ilbracket{\start_{\stateregime},\final_{\stateregime} }$, a   jump occurs. The discrete process then switches to a new state $\bvartheta_{\ilparenthesis{\stateregime+1}}$, and,  accordingly,  the diffusion process changes its drift and diffusion matrix within another random duration of time $\ilbracket{\start_{\stateregime+1},\final_{\stateregime+1} }$ with $\start_{\stateregime+1} = \final_{\stateregime} + 1 $. The jump component will remain in this state/regime until the next jump, and the diffusion/oscillation component will not change its drift or  diffusion matrix until a new jump takes place, and so on.

\begin{assumeC}
	[Regime-specified quadratic function form]
	\label{assume:gRegimeForm}
	 For  $k\in \ilbracket{\start_{\stateregime}, \final_{\stateregime}}$ with $\stateregime\in\integer$,
	the loss function takes the form of  $ \loss_k\ilparenthesis{\btheta} =   \ilparenthesis{\btheta-\bvartheta_{\stateInd}}^\transpose \bH _ {\stateInd} \ilparenthesis{\btheta-\bvartheta_{\stateInd}}/2 $, and    the gradient takes the form of  $ \bg_k\ilparenthesis{\btheta} = \bH _ {\stateInd} \ilparenthesis{\btheta-\bvartheta_{\stateInd}} $ for some symmetric and positive-definite matrix $\bH_{\stateInd}$.  
\end{assumeC}

\begin{assumeC}
	[Error $\noise_k$ is zero-mean and bounded-variance]
	\label{assume:ErrorUnbiased}  For $ k\in \ilbracket{\start_{\stateregime} ,  \final_{\stateregime}} $   the sequence $\ilset{\noise_k}$    is i.i.d. with mean $\zero$  and a  bounded covariance matrix of   $\VarianceError _ {\stateInd}$. That is, the observation noise $\noise_k $ in (\ref{eq:Ystationary}) depends on the state/regime \emph{only}.
\end{assumeC}

\begin{assumeC}[Abstraction of jump component]
	\label{assume:Hybrid}
	The lengths  $\ilparenthesis{\final_{\stateregime}-\start_{\stateregime}} $  of the random durations for $\bvartheta_k = \bvartheta_{\stateInd}$ are i.i.d. with geometric distribution having  a mean of $\jump^{-1}$, where $\jump$ is the jump/change probability  (usually less than $5\%$).   Furthermore, assume that $ \ilparenthesis{\final_{\stateregime} - \start_{\stateregime}} >\window \equiv   \max\set{ p+1,   \jump^{-1}/10 }  $ w.p.1.
\end{assumeC}

\begin{assumeC}[Abstraction of general trend-stationary system, including both jump and diffusion components]
	\label{assume:HybridRelaxed} 
	In addition to C.\ref{assume:Hybrid},      
 let  $\driftBound'$ be  the smallest number such that 
	the within-regime oscillation is restricted by   $ \norm{\bvartheta_k-\bvartheta_{\stateInd}} \le \driftBound '$  	 w.p.1.  for $\start_{\stateregime}\le k \le \final_{\stateregime}$ and for all   $\stateregime\in\natural$. Further assume that $\driftBound'$ is no larger than $\driftBound$, where $\driftBound$  is the smallest number such that the cross-regimes jump is restricted by $\norm{ \bvartheta_{ \ilparenthesis{\stateregime+1}}-  \bvartheta_{\stateInd}}  \le \driftBound $.
\end{assumeC}

Let us also provide some remarks regarding these  assumptions.
\begin{itemize}\item   In a majority of  applications to identification and adaptive system theory \cite{widrow1977stationary,ljung1977positive}, 
	 a positive-definite matrix $\bH$ such that $ \bg\ilparenthesis{\btheta}  = \bH \ilparenthesis{\btheta-\bvartheta} $ does exist; i.e., C.\ref{assume:gRegimeForm} holds.
	
	Following the discussion in Subsection \ref{subsect:ratio}, denote  $ \convexPara_{\stateInd}\equiv \uplambda_{\min }\ilparenthesis{\bH_{\stateInd}}  $ and  $ \LipsPara_{\stateInd}\equiv \uplambda_{\max}\ilparenthesis{\bH_{\stateInd}} $. 
	\item  Given  that $ \E\ilparenthesis{\noise_k }=\zero $ and  $\Var\parenthesis{\noise_k} = \VarianceError_{\stateInd}$ for all $ k\in \ilbracket{\start_{\stateregime}, \final_{\stateregime}}  $ as in C.\ref{assume:ErrorUnbiased}, we have $ \E\ilparenthesis{\norm{\noise_k}^2}  = \tr\ilparenthesis{\VarianceError_{\stateInd}} $. In fact, C.\ref{assume:ErrorUnbiased}  is an  abstraction for  ``the diffusion/oscillation component will not change its drift or  diffusion matrix until a new jump takes place.''

	\item   C.\ref{assume:gRegimeForm} can
	be relaxed to \cite[Assumption B(ii)]{pflug1986stochastic}, i.e., $ \bg_k\ilparenthesis{\btheta} = \bH_{\stateInd} \ilparenthesis{\btheta-\bvartheta_{\stateInd}} +O\ilparenthesis{\norm{\ilparenthesis{\btheta-\bvartheta_{\stateInd}}}^2}  $. 
	C.\ref{assume:ErrorUnbiased}  can be relaxed to \cite[Assumption A(iii)]{pflug1986stochastic};  i.e., $ \norm{\noise_k\ilparenthesis{\hbtheta_k}}^2  $ can be upper bounded by  a linear function of $ \norm{\hbtheta_k-\bvartheta_k} $ w.p.1.  We use    stronger assumptions  in order to present the results more elegantly.

	\item 
	C.\ref{assume:Hybrid} captures the jump part of the  ``diffusion and jump and so on and so forth'' nature of the hybrid system, and discards  the oscillation part for the time being. C.\ref{assume:Hybrid} does not require that the number of states is finite,  as long as  the jump probability is small. Even though our detection algorithm (summarized in Algorithm \ref{algo:ChangeDetection} to appear) is developed based on  C.\ref{assume:Hybrid}, our numerical result supports that it is also robust to the case where  the following C.\ref{assume:HybridRelaxed} holds. C.\ref{assume:HybridRelaxed}  is less stringent than C.\ref{assume:Hybrid} and captures both the oscillation and the jump components for hybrid systems.
	
	\item The assumption $ \final_{\stateregime}-\start_{\stateregime}>\window$ is imposed so that the duration of each regime  $\bvartheta_{\stateInd}$ is   sufficiently long such  that 
	\begin{enumerate}[(i)]
		\item  the normal approximation \cite{pflug1986stochastic}
		of constant-gain estimates takes effect; \item the full rank of our pooled variance estimate (\ref{eq:variancepooled}) to appear  is ensured;  \item  a  sufficient amount of data can be gathered to compute the needed statistics (\ref{eq:T2Statistic})   to appear     for the $p$-dimensional problem. 
		
		\item When C.\ref{assume:Hybrid} holds, we  say that a  (regime) change arises at time $\ilparenthesis{k+1}$ if  $\bvartheta_{k+1}$  differs   from $\bvartheta_k$. When C.\ref{assume:HybridRelaxed} holds, we say that a change arises at time $\ilparenthesis{k+1}$   if $\bvartheta_k\in \Ball_{\driftBound'}\ilparenthesis{\bvartheta_{\stateInd}}$ and $ \bvartheta_{k+1} \in \Ball_{\driftBound'} \ilparenthesis{ \bvartheta_{ \ilparenthesis{\stateregime+1} } } $.
	\end{enumerate}
	
	When  $ \final_{\stateregime}-\start_{\stateregime} $ is sufficiently long, depending on the starting value $\hbtheta_{\start_{\stateregime}}$, the process $\hbtheta_k$ given by (\ref{eq:ConstantGainSGD}) may  first show  a phase of steadily approaching the solution $\bvartheta_{\stateInd}$,  and then shows  the oscillation around $\bvartheta_{\stateInd}$ without further approaching $\bvartheta_{\stateInd}$. We will call them   the  \emph{transient phase} (known as search phase in  \cite{pflug1988stepsize}) and the \emph{steady-state phase} (known as stationary/convergence phase in  \cite{pflug1988stepsize}) throughout the rest of our discussion.
\end{itemize}

\subsection{Base Case: One \emph{Unknown} Change Point Occurs For $\upkappa \le K$ }

 Let us start with a simplified scenario where $ \bvartheta_i  = \bvartheta_{\FirstState}$ for $ 1\le i \le \upkappa $ and $ \bvartheta_j =\bvartheta_{\SecondState} $ for  $\upkappa<  j \le  K $, such that there is a \emph{single} hypothesized change point at time $1<\upkappa < K $.     At each time instant $k$, we   test  the null hypothesis 
 \begin{equation} \label{eq:nullhypo}
 \hypo _0:  \bvartheta_1= \cdots = \bvartheta_{K} \remove{\in\Ball_{\FirstState}},  
 \end{equation} 
 versus the alternative hypothesis
 \begin{equation*}
  \hypo_1:    \bvartheta_1=\cdots = \bvartheta_k, \bvartheta_k\neq \bvartheta_{k+1}, \bvartheta_{k+1} = \cdots = \bvartheta_K\text{ for some $k$ such that $2\le k \le K-2$. } 
 \end{equation*}
During the first regime $k\le \upkappa$,  let    $ \bg_k\ilparenthesis{\btheta} $ be $ \bH_{\FirstState} \ilparenthesis{\btheta-\bvartheta_{\FirstState}} $ and $\Var\ilparenthesis{\noise_k} = \VarianceError_{\FirstState}$.
During the second regime $\upkappa < k \le K$, let  $ \bg_k\ilparenthesis{\btheta} $ be $ \bH_{\SecondState} \ilparenthesis{\btheta-\bvartheta_{\SecondState}} $ and $\Var\ilparenthesis{\noise_k} = \VarianceError_{\SecondState}$.    
\cite{pflug1988stepsize}  shows  that if the gain  is held to a constant $\gain$,     constant-gain SA estimate behaves differently compared to decaying-gain SA estimates, in that the estimates ultimately converge to a  region of radius $ O(\sqrt{a}) $ that contains $\bvartheta$  and then oscillates in that region without further approaching $\bvartheta$.  The steady-state covariance of $\hbtheta_k$ is\textemdash for a small value of  $\gain$\textemdash approximately equal to $\gain \VarianceLimiting $ where $\VarianceLimiting$ is the solution of $ \bH\VarianceLimiting+\VarianceLimiting\bH=\VarianceError $ and can be  given by   \cite{walk1977invariance}:
\begin{equation}\label{eq:LimitingVariance}
\VarianceLimiting  = \int_0^\infty e^{ t \bH } \VarianceError e ^{t \bH^\transpose} \diff t, \quad \text{or,} \quad 
\vector{ \VarianceLimiting }= \parenthesis{  \bI \otimes  \bH + \bH^\transpose\otimes \bI   } ^{-1}   \vector {\VarianceError }. 
\end{equation} 
When Assumption C.\ref{assume:Hybrid} holds, we    expect that $\hbtheta_k$ will  quickly reach the steady-state phase within each regime after a short period of a transient phase  provided that the gain $\gain$ is pre-tuned carefully.  Immediately, 
 $ \gain \cdot \tr\ilparenthesis{\VarianceLimiting}  $ is   approximately equal to $ \E\ilparenthesis{\norm{\hbtheta_k-\bvartheta}^2} $ during the steady-state phase. Hence, for $k\le \upkappa$, we expect   $  {\hbtheta_k} $ to be \emph{approximately} normally distributed with a mean of  $\bvartheta_{\FirstState}$ and  a variance  matrix  $\gain \VarianceLimiting_{\FirstState} $ given by $ \bH_{\FirstState} \VarianceLimiting_{\FirstState} + \VarianceLimiting_{\FirstState} \bH_{\FirstState} = \VarianceError_{\FirstState} $.  Let us   ignore the transient behavior after the jump point $\upkappa$ for the time being.   We expect that,  for $\upkappa < k \le K$,  $\hbtheta_k$ is going  to be approximately normally distributed with a mean of $\bvartheta_{\SecondState}$ and a variance matrix $\gain \VarianceLimiting_{\SecondState}$ given by  $ \bH_{\SecondState} \VarianceLimiting_{\SecondState} + \VarianceLimiting_{\SecondState} \bH_{\SecondState} = \VarianceError_{\SecondState} $. Of course, neither $\bH_{\FirstState}$ ($ \bH_{\SecondState} $) nor $\VarianceError_ {\FirstState}$ ($ \VarianceError_ {\SecondState} $) is known in reality.   
Since  the      information   $\VarianceError_{\FirstState}$ ($ \VarianceError_{\SecondState} $)  and $ \bH_{\FirstState} $   ($\bH_{\SecondState}$) needed to construct $\VarianceLimiting_{\FirstState}$ ($\VarianceLimiting_{\SecondState}$) are \emph{not} revealed to the agent(s), we can \emph{not} take  full advantage of the multivariate-normal approximation or  to detect the regime states,

 \subsubsection*{Test Statistic}
The \emph{multivariate Behrens\textendash Fisher problem} deals with testing the equality of means from two multivariate normal distributions when the dispersion matrices are \emph{unknown} and potentially \emph{unequal}. It inherits all the difficulties arising in  the univariate Behrens\textendash Fisher problem, including estimating the dispersion matrix using data, and the distributional approximation. 
Define: 
\begin{eqnarray}\label{eq:mean}
\bbtheta_{i:j} &=& \frac{1}{j-i+1}  \sum_{l=i}^{j} \hbtheta_l
\remove{ \bW _k &=& \frac{1}{K-2}\set{  \sum_{i=1}^k \bracket{\ilparenthesis{  \hbtheta_i -  \bbtheta_{0, k }   }   \ilparenthesis{  \hbtheta_i -  \bbtheta_{0, k }   } ^\transpose}   + \sum_{i=k+1}^K    \bracket{\ilparenthesis{  \hbtheta_i -  \bbtheta_{k, K }   }   \ilparenthesis{  \hbtheta_i -  \bbtheta_{k, K }   } ^\transpose }    } },\quad \text{for }i\le j,
\end{eqnarray}
and 
\begin{eqnarray}
\bW_{1:k} & = &  \frac{1}{k\parenthesis{k-1}}   \sum_{i=1}^k \bracket{\ilparenthesis{  \hbtheta_i -  \bbtheta_{1: k }   }   \ilparenthesis{  \hbtheta_i -  \bbtheta_{1:k }   } ^\transpose}, \label{eq:variance1}   \\ 
\bW_{k+1:K} &=&  \frac{1}{\ilparenthesis{K-k}\ilparenthesis{K-k-1}} \sum_{i=k+1}^K    \bracket{\ilparenthesis{  \hbtheta_i -  \bbtheta_{k+1:  K }   }   \ilparenthesis{  \hbtheta_i -  \bbtheta_{k+1:K }   } ^\transpose },  \label{eq:variance2}    \\ 
\bW _k &=& W_{1:k}  + W_{k+1:K}.  \label{eq:variancepooled}
\end{eqnarray}
When Assumption C.\ref{assume:Hybrid} holds, $\bW_k$\remove{have  $K>p+1$ and $ 2\le k \le K-2 $ to} is ensured   to have   full rank. To investigate a possible jump/change occurring after observation $k$, we use the following statistic for testing a difference between pre-change and post-change data at an \emph{assumed} change point $k$ as: 
 \begin{equation}\label{eq:T2Statistic}
T_{1:k:K}^2 \equiv   \parenthesis{   \bbtheta_{1:k }   -\bbtheta_{k+1:K }    }  ^\transpose \bW_k^{-1}  \parenthesis{   \bbtheta_{1:k }   -\bbtheta_{k+1:K }    }, \quad k=2,\cdots,K-2. 
\end{equation}  
\begin{rem}
	Though (\ref{eq:T2Statistic}) shares  some similarities with the Hotelling $T^2$ statistic, it is  fundamentally different in that the Hotelling $T^2$ statistic is \emph{not} robust to unequal covariance matrices. 
\end{rem}

 \subsubsection*{Distribution of Test Statistics}

 The main   issue  in applying  (\ref{eq:T2Statistic}) to  detecting change for streaming data  in an online fashion is that,  the probability of rejecting 
 the  null  via the $T^2$ test statistic defined in (\ref{eq:T2Statistic})  depends on the \emph{unknown} dispersion matrices $\VarianceLimiting_{\FirstState}$ and $\VarianceLimiting_{\SecondState}$ under the null hypothesis (\ref{eq:nullhypo}) that $\bvartheta_{\FirstState}$ equals  $\bvartheta_{\SecondState}$ as in  (\ref{eq:nullhypo}).  In practice, this dependency compromises the statistical  inference when   the underlying true  dispersion matrices $\VarianceLimiting_{\FirstState}$ and $\VarianceLimiting_{\SecondState}$   significantly deviate from each other or when the sample size is not sufficiently large  to estimate them accurately. Below are some existing remedies.

 The first remedy is to use $T_{1:k:K}^2$  with an  approximation of its  degrees of freedom  \cite{yao1965approximate}:
 \begin{align}\label{eq:YaoApproximation}
 	&T_{1:k:K}^2 \sim \frac{\upnu_{1:k:K} p }{\upnu_{1:k:K} -p+1} F_{p,\upnu_{1:k:K} -p+1}, \nonumber\\
 	&\text{ with }
 	\upnu_{1:k:K} = \set{  \frac{1}{k} \bracket{  \frac{\bd_{k}  ^\transpose    \bW_k^{-1 }  \bW_{1:k} \bW_k^{-1} \bd_{k} }{  \bd_{k}  ^\transpose    \bW_k^{-1 }    \bd_{k}    }        }^2       + \frac{1}{K-k} \bracket{  \frac{\bd_{k}  ^\transpose    \bW_k^{-1 }  \bW_{k+1:K} \bW_k^{-1} \bd_{k} }{  \bd_{k}  ^\transpose    \bW_k^{-1 }    \bd_{k}    }        }^2         } ^{-1}, \quad\quad 
 	\end{align}
 	where $F_{\cdot, \,\cdot}$ (with two positive inputs) denotes the probability distribution function for $F$-distribution with given degrees of freedoms,  $ \bd_{k} =  \bbtheta_{1:k }   -\bbtheta_{k+1:K }  $. 
 	In addition to the approximation in (\ref{eq:YaoApproximation}), there are several others, including Johansen's approximation \cite{johansen1980welch}, and Nel and Van der Merwe's  approximation \cite{nel1986solution}. 
 	
 	 The second remedy follows from \cite{krishnamoorthy2004modified} which  proposed another approximation where the approximated degrees of freedom is   guaranteed to be nonnegative:  
 	\begin{align}\label{eq:YuApproximation}
 	&	T_{1:k:K}  ^2\sim \frac{\upnu_{1:k:K}  p }{\upnu_{1:k:K} -p+1} F_{p,\upnu_{1:k:K} -p+1}, \nonumber\\
 	& \text{ with }
 	\upnu_{1:k:K} = \frac{p+p^2}{  *    } ,
 	\end{align} 
 	where  the $*$ in (\ref{eq:YuApproximation})  is 
 	\begin{align*}
 	* &=  \frac{1}{k-1} \set{     \tr\bracket{  \ilparenthesis{  \bW_{1:k} \bW_k  ^{-1}}^{2}  }    + \bracket{ \tr\parenthesis{  \bW_{1:k} \bW_k^{-1} }  }^2        }   \nonumber\\
 	&\quad + \frac{1}{K-k-1} \set{     \tr\bracket{  \ilparenthesis{  \bW_{k+1:K} \bW_k  ^{-1}}^{2}  }    + \bracket{ \tr\parenthesis{  \bW_{k+1:K} \bW_k^{-1} }  }^2        }. 
 	\end{align*}
 	The approximation in (\ref{eq:YuApproximation}) has the best known size and power since 2004.

  \subsubsection*{A  Change Detection Strategy  for   Base Case}
 
 If the change point were known a priori to be at $\upkappa $, then $T_\upkappa^2$ will be the generalized likelihood ratio test statistic for testing a change between pre-$\upkappa$ and post-$\upkappa$ data. If the change point is unknown in advance, the maximum over all possible split points, $ \max_{2\le k \le K-2} T_k^2 $, is the generalized likelihood ratio test statistic for change in the mean. The maximizing index $\hat{\upkappa}  = \arg\max _{2\le k \le K-2} T_k^2$ is the maximum likelihood estimate  of the change/jump point.  
  Now  that there is a \emph{single} assumed change point up until time index $K$, a natural estimate for $2\le k\le K-2$ is 
 \begin{equation}\label{eq:tauBase}
 \hat{\upkappa}  = \arg\max _{2\le k \le K-2} T_{1:k:K}^2.
 \end{equation}
  This statistic fits well  for a \emph{single} change point  in a \emph{fixed} sample of size $K$. Unfortunately, even for such a  simplified base case, we are not able to accurately provide the 
 $\mathrm{ARL}_0$ and $\mathrm{ARL}_1$ for strategy 
 (\ref{eq:tauBase}), because   the distribution of  the maximum over a range of $T_{1:k:K}^2$ statistics gets extremely 
 complicated. 
 
   All we can conclude is that, after $\hat{\upkappa}$ is computed, we may use the \emph{approximated} $P$-value, denoted as $\hazard_{1:\hat{\uptau}:K}$,   to serve as a proxy for the probability of \emph{incorrectly} announcing  a  change arises,  whereas, in fact,  no change occurs.  Specifically,  $\hazard_{1:\hat{\upkappa}:K}$ is calculated by    $ F_{p,\upnu_{\hat{\upkappa}}, -p+1 } $
 evaluated at $ \ilparenthesis{\upnu_{\hat{\upkappa}} - p + 1 } T_{1:\hat{\uptau}:K}^2 / \parenthesis{\upnu_{\hat{\upkappa}}  p } $,
  \begin{equation}\label{eq:HazardRate}
 \hazard_{1:\hat{\upkappa}:K} \equiv 
 \int_{ \frac{ \upnu_{\hat{\upkappa}} -p+1 }{  \upnu_{\hat{\upkappa}} p  }  T_{1:\hat{\upkappa}:K}^2     }^{\infty }  \frac{  \Gamma\bracket{\frac{  \upnu_{\hat{\upkappa}} + 1 }{2}}  }{\Gamma \parenthesis{\frac{p}{2}} \Gamma\parenthesis{\frac{\upnu_{\hat{\upkappa}} -p+1}{2}} } \parenthesis{\frac{p}{\upnu_{\hat{\upkappa}}  -p+1} } ^{\frac{p}{2}}  \frac{x ^{ \frac{p-2}{2} } }{  \bracket{1+\parenthesis{\frac{p}{\upnu_{\hat{\upkappa}} -p+1}}x}^{\frac{ \upnu_{\hat{\upkappa}}+1 }{2}}   } \diff x.  \end{equation}

 \subsection{Building Block: Multiple \emph{Unknown}  Change Points  For  the Data Stream}
 When  we need to detect \emph{multiple} \emph{unknown} change points for a  data \emph{stream}, the problem   gets even more unwieldy. 
 First, we do \emph{not} know how many change points are upcoming  beforehand. 
 Second,   the   detection has to be performed on a stream of data $ \ilset{\hbtheta_k} $ in an \emph{online} fashion, which causes excessive storage and computational overhead for active monitoring as the stream gets longer and longer. Even for the base case where there is  \emph{one} unknown change point, the naive strategy of computing $ T^2_{1:j:k} $ for every  $j<k-1$ whenever a new $\hbtheta_k$ comes in is unrealistic.  The computational burden becomes increasingly heavy,  as the datastream grows  larger and larger  and as the number of possible change points increases. It is again unrealistic to achieve the change detection goal promptly, i.e., correctly announce $k$ to be a change point immediately after observing the information up till time $\uptau_k$, not to mention that the probabilistic error for $\hat{\upkappa}_{\stateregime+1}$ hinges upon that for   $\hat{\upkappa}_{\stateregime}$.

  To avoid further complications, we impose C.\ref{assume:Hybrid} for the following reasons.
  \begin{enumerate}[(i)]
  	\item It is assumed that the period of each regime should be sufficiently long, so that we can  gradually accrue  confidence in making detection decisions  within a certain time-frame. 
  	\item The random duration   $\ilparenthesis{ \final_{\stateregime} - \start_{\stateregime} }$ is assumed to be bounded from below by $\window> 0$ w.p.1. Then at each time instant $k\ge 2\window$, we can  use   a \emph{fixed} amount of data,  $\ilset{  \hbtheta_{k-2\window+1}, \hbtheta_{k-2\window+2},\cdots,\hbtheta_{k-1}, \hbtheta_k  }$, to test whether a change arose  at time index $k-\window$. 
  	
  \end{enumerate} 
  	
  	The ``elbow'' (i.e., the hazard rate at this point is lower than its two adjacent time points) point  on the curve of  $P$-value defined in \ref{eq:HazardRate}
  gets identified as a change point. See the details summarized in Algorithm \ref{algo:ChangeDetection}.   
 	\begin{algorithm}[!htbp]
 	\caption{Change Detection: Constant-gain SGD Algorithm (\ref{eq:ConstantGainSGD}) Using \emph{One}-Measurement at a Time } 
 	\begin{algorithmic}[1]  
 		\renewcommand{\algorithmicrequire}{\textbf{Input:}}
 		\renewcommand{\algorithmicensure}{\textbf{Output:}}
 	\Require a window size $\window$ (based on dimension $p$ and $\jump$),  a constant gain $\gain $, and a $P$-value threshold $\upalpha$. 
 		\State \textbf{set} $\hbtheta_0$, the best approximation available   at hand to estimate $\bvartheta_0$. 
 		\For{$ 1\le k  < 2\window $}
 		\State \textbf{update} $ \hbtheta_{k}\gets \hbtheta_{k-1} - \gain \hbg_{k-1}^{\SG}\ilparenthesis{\hbtheta_{k-1} } $. 
 		\EndFor
 		\For{$k\ge 2\window$}
 		\State \textbf{compute} $  \bbtheta_{k-2\window+1:k-\window} $ and $     \bbtheta_{k-\window+1:k}  $ per (\ref{eq:mean}).
 		\State \textbf{compute} $   \bW_{ k-2\window+1:k-\window}  $, $   \bW_{k-\window+1:k} $  per (\ref{eq:variance1}) to (\ref{eq:variance1}). 
 		\State  \textbf{compute} $ T_{k-2\window+1: k-\window: k } $ per (\ref{eq:T2Statistic}). 
 		\State \textbf{compute} corresponding   $P$-value $\hazard_k\gets  \hazard_{k-2\window+1:k-\window : k} $  per (\ref{eq:HazardRate}).  
 		\If{$k>2\window+1$}
 		\If{$\hazard_{k-1}< \min \set{\upalpha, \hazard_{k-2} , \hazard_{k} } $} 
 		\Ensure  $k-1$ as  a change point. 
 		\EndIf
 		\EndIf
 		\EndFor
 	\end{algorithmic}
 	\label{algo:ChangeDetection}
 \end{algorithm}     
\begin{rem}
	 By fixing the window $\window$,  Algorithm \ref{algo:ChangeDetection} has  \emph{constant} computational complexity and a \emph{fixed} amount of memory. 
\end{rem}

 \begin{prop}
 	\label{prop:Hazard}
 	Under C.\ref{assume:ErrorUnbiased}, C.\ref{assume:gRegimeForm}, and C.\ref{assume:Hybrid},   
 	the  probability  of incorrectly detecting that   a  change happened when,  in fact,  no change did occur, is  approximately  $ \hazard_{\hat{\upkappa}-\window+1:\hat{\uptau}:\hat{\upkappa}+\window} $, where the function  $\hazard$ (which takes three inputs) is defined in (\ref{eq:HazardRate}), and  $\hat{\upkappa}$ (suppressing the numbering if there are multiple identified points) is identified by  Algorithm \ref{algo:ChangeDetection}. 
 \end{prop}
 Proposition \ref{prop:Hazard} follows directly from the distribution of the test statistic (\ref{eq:T2Statistic}) in the  multivariate Behrens\textendash Fisher problem, and the approximation hinges upon the imposed assumptions: the Hessian matrix of $\loss_k\ilparenthesis{\cdot}$ remains constant within each regime under C.\ref{assume:gRegimeForm};   the observation errors  are i.i.d. mean zero  within each regime under   C.\ref{assume:ErrorUnbiased}. Moreover,  under C.\ref{assume:Hybrid}, after ignoring the transient phase between regimes,  we assume that the estimates $\hbtheta_k$ are approximately normally distributed around   $\bvartheta_{\stateInd}$ where $\stateregime$ is such that $ \start_{\stateregime}\le k \le \final_{\stateregime} $. 
 
 	\subsection{An Example for  Detecting Regime Change }

  Let us consider an example similar to Subsubsection \ref{subsect:ExamplePriorBound}, yet different in the sense that we no longer have access to $\convexPara_k$, $\LipsPara_k$, $\noiseBound _k$ and $\driftBound_k$ as defined in Section \ref{sect:UnconditionalError}.  Again, consider a simple case with $p=2$, where the   (unknown) nonstationary drift evolves according to:  
 \begin{equation}\label{eq:model1}
 \bvartheta_{k+1} = \begin{cases}
 \bvartheta_k , \quad \text{with a  probability of 99.9\%}, \\
 \bvartheta_k + 500 \parenthesis{ \cos\ilparenthesis{\upvarphi_k},\,\, \sin\ilparenthesis{\upvarphi_k}}^\transpose,  \quad \text{with a  probability of 0.1\%},
 \end{cases}
 \end{equation} 
  with $\bvartheta_0=\zero$ and  $\upvarphi_k  \stackrel{\mathrm{i.i.d.}}{\sim } \mathrm{Uniform}\ilbracket{0,2\uppi}$.   The observation error $\noise_k$ is again i.i.d. $\mathrm{Normal}\ilparenthesis{\zero, \upsigma_1^2\bI_p}$, and the Hessian matrix is again given in (\ref{eq:SimHessian}). Here, we use a constant gain $\gain =1/30$, which is the inverse of the Lipschitz continuity parameter of the gradient.

  Following Algorithm~\ref{algo:ChangeDetection},  we pick the window size $\window$ to be $25$, as we are expecting a jump to arise every $ 1/(.1\%) = 1000 $ iterations on average and the dimension $p=2$. Figure \ref{fig:detection1} shows    the \emph{true} jump point, at time $241$, $2412$, and $4644$ in red circles, and the identified jump point  (a very successful identification in this case)  in enlarged black stars. 
 \remove{  Unfortunately, the two identified jump points  still have large $p$-values as high as $ \ilparenthesis{0.2, 0.3} $, 
 	which is due to the \emph{transient} behavior of $\hbtheta_k$ of moving from one regime to the next regime. }

 \begin{figure}[!htbp]
 	\centering
 	\includegraphics[width=.65\linewidth]{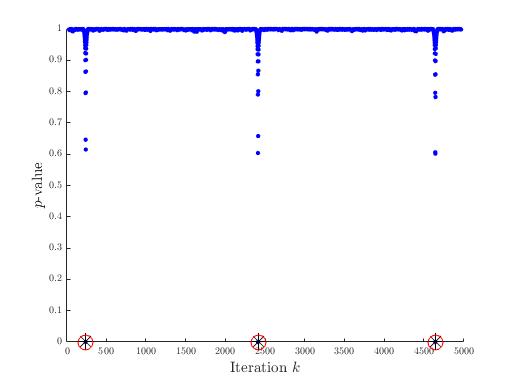}
 	\caption{Change detection using the ``elbow''-point of  the $P$-value curve, when $\ilset{\bvartheta_k}$ evolves according to (\ref{eq:model1})}
 	\label{fig:detection1}
 \end{figure}

 As it turns out, even though  the detection algorithm listed in Algorithm  \ref{algo:ChangeDetection} is proposed based on   C.\ref{assume:Hybrid} where only the ``jump'' structure is captured, it is robust to the case where C.\ref{assume:HybridRelaxed} is met. Figure \ref{fig:detection2} below shows  how Algorithm \ref{algo:ChangeDetection} detects the jump points $605$, $1051$, $2189$, $3300$, and $4522$ when the (unknown) nonstationary drift is evolved according to:  
 \begin{equation}\label{eq:model2}
 \bvartheta_k \text{ is i.i.d. uniformly distributed within }   \ilset{\given{\btheta}{\norm{\btheta-\bvartheta_{\stateInd} } \le 50 }  }, \text{ for } \start_{\stateregime}\le k \le \final_{\stateregime},
 \end{equation}
 where the jump probability of the sequence $\ilset{\bvartheta_{\stateInd}}$ is again $0.1\%$. 
 
  \begin{figure}[!htbp]
 	\centering
 	\includegraphics[width=.65\linewidth]{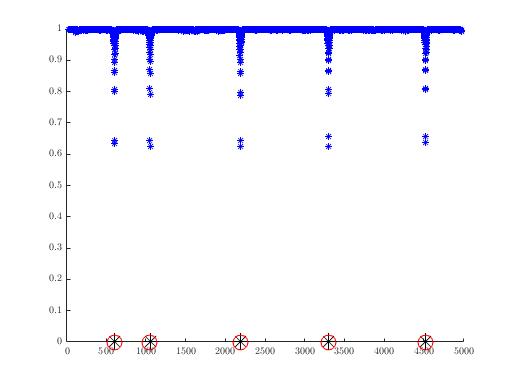}
 	\caption{Change detection using  the ``elbow''-point of the $P$-value curve, when $\ilset{\bvartheta_k}$ evolves according to (\ref{eq:model2})}
 	\label{fig:detection2}
 \end{figure}

  \subsection{Further Remarks}
 It is natural to envision    the adaptation  of a  distribution-free, nonparametric test statistic, whose   distribution under the null is  \emph{independent} of the data   to streamline the above change detection procedure and to  produce   a desired false alarm rate (FAR)  to be maintained for any stream. Granted,    \cite{ross2011nonparametric,lepage1971combination,kifer2004detecting} discussed several univariate distribution-free test statistics, aiming  to  detect  a change in the location and/or scale parameter of  a stream of random variables. However, we do not consider the  nonparametric method due to the expensive computational cost of computing ranks or the depth of the  minimum spanning tree  for the  multivariate extension \cite{friedman1979multivariate} in an online fashion  and the low power of a general nonparametric test that uses a distribution-free  test statistic.

 This section discusses a strategy to detect the jump components in the time-varying sequence $ \ilset{\bvartheta_k} $ by making use of the constant-gain SGD estimates $\hbtheta_k$. Nonetheless, Algorithm~\ref{algo:ChangeDetection} still requires   keeping track of the last  $2\window$ SGD estimates, and can only identify the jump after $\window$ iterations, even though it has a favorable detection power. The upcoming section discusses a way to identify jumps \emph{instantly} by imposing more stringent assumptions than C.\ref{assume:Hybrid}.

 \section{Gain Adaptation }\label{sect:GainOneRegime}

  We still focus on   the constant-gain recursion (\ref{eq:ConstantGainSGD})  and  propose a  gain-adaptation strategy  using the available information  $ \field_{k+1}  $ at time $\uptau_k$.   
  In general, there is no guarantee that the constant-gain SGD estimates    will converge to a  fixed   $\bvartheta$. Under weak conditions,  \cite{dieuleveut2017bridging} shows that  constant-gain SGD estimate $ \hbtheta_k  $        exhibits positive variance (uniformly bounded away from zero) for all $k$. Moreover,  \cite{dieuleveut2017bridging} also proves  that the constant-gain SGD iterates converge to their unique stationary distribution exponentially fast in $k$. 
 The results of  \cite{dieuleveut2017bridging} are  consistent with practical experience that  the constant-gain SA makes rapid progress in approaching $\bvartheta$, yet it remains in the neighborhood of $\bvartheta$ afterward.  Therefore, we are motivated to perform the following: 
  \begin{enumerate}[a)]
  	\item  increase the stepsize for faster cross-regime adaptation once a jump is detected;
  	\item reduce the stepsize in a controlled  manner to further reach the vicinity of $\bvartheta_{\stateInd}$ once the iterates are determined to oscillate around $\bvartheta_{\stateInd}$.
  \end{enumerate}

  Different from the previous section where Algorithm~\ref{algo:ChangeDetection} announces a change arises after $\window$ observations after the change point, this section aims to make the announcement as \emph{soon} as possible. Contrary to the previous section where a constant gain is used throughout the entire optimization process by disregarding whether we have observed a jump, this section   proposes  a method  to control  the non-diminishing step-size based upon  the observable information $\field_k$ defined in (\ref{eq:filtration}), aiming to achieve  better   performance within each regime and \emph{faster} adaptation between different regimes. 
 
 \subsection{Model   Assumptions }

 \begin{assumeC}	[Abstraction of a  regime that lasts a long duration of time]  \label{assume:Stationary}
 	Assume the following: 
 	\begin{enumerate}[(i)]
 			\item  $\bvartheta_k=\bvartheta$ for all $k$. 
 			
 			\item  $\bg_k\ilparenthesis{\btheta}   =  \bH\ilparenthesis{\btheta-\bvartheta}$ for some symmetric and positive-definite matrix $\bH$ for all $k$.

 		\item  $\noise_k$ are i.i.d. with    mean $\zero$ and    bounded covariance $\VarianceError$ for all $k$.

 	\end{enumerate} 
 \end{assumeC}

C.\ref{assume:Jump} is   more general than C.\ref{assume:Stationary}. 

\begin{assumeC}
	[Abstraction of regime-switch] \label{assume:Jump}
	Assume the following:
	\begin{enumerate}[(i)]
			\item  $\bg_k\ilparenthesis{\btheta} = \bH\ilparenthesis{\btheta-\bvartheta_k}$ for some symmetric and positive-definite matrix $\bH$ fo all $k$.

		\item $\noise_k $ are i.i.d. with   mean $\zero$ and   bounded covariance $\VarianceError$ for all $k$.

		\item  	Suppose  C.\ref{assume:Hybrid} holds.

	\end{enumerate}
\end{assumeC}

Let us first provide some discussions on the  Assumption C.\ref{assume:Stationary}.
\begin{itemize}
	\item At   first glance, 
	C.\ref{assume:Stationary} reduces the time-varying problem (\ref{eq:Minimization}) to the classical SA problem of minimizing a fixed  loss function  reviewed in Section~\ref{sect:SAview}. Nonetheless, this stationarity assumption is imposed to facilitate the discussion of the base case    in Subsection~\ref{subsect:BaseCase}. Later on, the 
	exposition  in Subsection~\ref{subsect:BaseCase} will be readily extended to a more general case in Subsection~\ref{subsect:BuildingBlock}.

	\item  The Markovian process $\ilset{\hbtheta_k}$ generated from     (\ref{eq:ConstantGainSGD}) 
	does \emph{not} approach  $\bvartheta$ beyond a certain distance, when  the gain  is held  constant.  C.\ref{assume:Stationary} is imposed so that we can devise a strategy to decide when to decrease the gain  to further approach $\bvartheta$.

	\item 
	Another reason to impose C.\ref{assume:Stationary}, for the time being, is as follows. Gain selection is not  a problem exclusive to nonstationarity tracking. For the stationary/fixed  setting  in classic SA literature,   the decaying sequence $ \gain_k =O(1/k) $ may not be desirable for practical usage, even though it is proven to be asymptotically optimal (in minimizing the trace of the limiting covariance of $\hbtheta_k$). Worse still, even when we pick the decaying gain sequence  with the  $O(1/k)$ decaying rate, the constant sitting in front of $ 1/k $ still drastically  affects our estimation  if it is  misspecified and is very sensitive to the initialization. In fact, (1) the absolute value of the gain plays a more important role than the convergence rate to zero, especially when we only have limited resources to run a \emph{finite} number of iterations, and (2) the  $O(1/k)$ sequence decays \emph{extremely slowly} to zero for a large $k$. Based on the  two observations, we    take the gain  to be constant, yet small, to mimic the behavior of the estimates $\hbtheta_k$ within \emph{finite} iterations.  
	Moreover,   the iterates reaching the vicinity of $\bvartheta$ quickly within finite iterations is much more important than convergence after potentially infinite iterations.

\end{itemize}

Let us also mention a few subtitles implied from Assumption C.\ref{assume:Jump}.
\begin{itemize}
	\item The Markovian process  $\ilset{\hbtheta_k}$ becomes   more difficult to analyze  due to the randomness of $ \start_{\stateregime} $ and $\final_{\stateregime}$.   
	To keep it concise, we still assume that $\ilparenthesis{\final_{\stateregime} - \start_{\stateregime}}$ is sufficiently long w.p.1.;  hence, the following assumption. 
	
	\item 
	C.\ref{assume:Jump} is very similar to those imposed in Subsection \ref{subsect:Hybrid}, except that we do not allow dependence of $\bH$ and $\VarianceError$ on the regime $\stateregime$ here. This is because, both $\bH$ and $\bV$ being constant, plays  an important role in Algorithm \ref{algo:adaptiveGain} to appear.

\end{itemize}

Overall, the discussion here  applies  to the following algorithm: 
  \begin{equation}\label{eq:RegimeConstantGainAlgo}
 \hbtheta_{k+1} = \hbtheta_k - \gainRegime \hbg_k^{\SG} \ilparenthesis{\hbtheta_k},\quad \start_{\stateregime}\le k \le  \final_{\stateregime}-1,  
 \end{equation}  
 as long as  $\ilparenthesis{\final_{\stateregime} - \start_{\stateregime}}$ in C.\ref{assume:Hybrid} is sufficiently large.   (\ref{eq:RegimeConstantGainAlgo}) is a straightforward extension of (\ref{eq:ConstantGainSGD}), and let us call  (\ref{eq:RegimeConstantGainAlgo})  ``SGD with regime-wise-constant gain.''
 For brevity's sake, we suppress  the dependence of $\hbtheta_k$ generated by (\ref{eq:ConstantGainSGD}) on $\gain$ under C.\ref{assume:Stationary},  or the dependence of $\hbtheta_k$ generated by (\ref{eq:RegimeConstantGainAlgo}) on $\gainRegime$ under C.\ref{assume:Jump}. Consequently,  $\hbtheta_k$   exhibits a relatively short (compared to the entire regime duration) transient phase and a relatively long steady-state phase, as noted  in Subsection \ref{subsect:Hybrid}.   
 
 Moreover,  part (i) in both  C.\ref{assume:Stationary} and C.\ref{assume:Jump} can
 be relaxed to \cite[Assumption B(ii)]{pflug1986stochastic}, i.e., $ \bg_k\ilparenthesis{\btheta} = \bH  \ilparenthesis{\btheta-\bvartheta_k } +O\ilparenthesis{\norm{\ilparenthesis{\btheta-\bvartheta_{k}}}^2}  $.  Also, part (ii) in both  C.\ref{assume:Stationary} and C.\ref{assume:Jump} can be  relaxed to \cite[Assumption A(iii)]{pflug1986stochastic};  i.e., $ \norm{\noise_k\ilparenthesis{\hbtheta_k}}^2  $ can be upper bounded by  a linear function of $ \norm{\hbtheta_k-\bvartheta_k} $ w.p.1.  We use    stronger assumptions      to present the results more elegantly.

 \subsection{Base Case: Detection of Transient Phase and Steady-State Phase}\label{subsect:BaseCase}

 Even though the constant-gain  SA iterates will  approach $\bvartheta$ in neither a.s.  nor m.s. sense, practitioners still implement SA with a constant gain \cite{dieuleveut2017bridging,spall2005introduction}. 
 As mentioned in Section \ref{sect:AdaptiveReview}, during  the transient phase, the constant-gain SA estimate   generated from (\ref{eq:ConstantGainSGD}) promptly  moves  towards the desired region and forgets the initial condition exponentially fast. Then  during the steady-state phase, the estimate  oscillates around $\bvartheta$ at a region of radius $ O(\sqrt{\gain  }) $.  
 The trade-off is obvious  that  a larger value of $\gain$ shortens the transient phase,  yet simultaneously enlarges the radius of the steady-state phase.  Understanding the transition between the transient phase and the steady-state phase enables us to  enhance  the  empirical performance of the constant-gain algorithm.

 The key puzzle in designing adaptive gain is to determine a statistical test to check the stationarity of the iterates generated from (\ref{eq:ConstantGainSGD}). 
 The  motivation for the stationarity check  comes from a gain-tuning rule in  \emph{deterministic} optimization: 
 increase/decrease the gain  if $  \ilbracket{\bg\ilparenthesis{\hbtheta_{k+1}}}^\transpose\ilbracket{\bg\ilparenthesis{\hbtheta_k}}$ is positive/negative. In a deterministic scenario with  $\noise_k=\zero$ for all $k$,   the recursion $ \hbtheta_{k+1} = \hbtheta_k-\gain  \bg\ilparenthesis{\hbtheta_k} = \hbtheta_k-\gain\bH\ilparenthesis{\hbtheta_k-\bvartheta} 
 $ converges to $\bvartheta$  as long as the gain sequence  $\gain $ is smaller than $ \uplambda_{\min}\ilparenthesis{\bH} $ after some $k$. Note that the convergence of $\hbtheta_k$ to $\bvartheta$ under \emph{noise-free} scenario does \emph{not} require the constant gain  $\gain$ to go to   zero.

 It seems natural to extend the above to use $  \ilbracket{\hbg_{k+1}\ilparenthesis{\hbtheta_{k+1}}}^\transpose\ilbracket{\hbg_k\ilparenthesis{\hbtheta_k}}$ as an 
 indicator for both the transient phase and the steady-state phase. However, we have to handle the \emph{noise} $\noise_k$ in SA problem setting.
 During the transient phase, the observations $\ilset{\hbg_{k}\ilparenthesis{\hbtheta_k}}$ are auto-correlated as successive gradient observations that  are \emph{roughly} pointing to the same direction. During the steady-state phase, successive gradient estimates \emph{tend} to point to opposite directions. To shorten the transient phase, we are better off   increasing the gain $\gain $ by a factor of $\increase$. To move towards the optimum during the steady-state phase, it is advisable to decrease the 
 gain  by a factor of  $\shrink$. 
  To  compensate for  the noise effect, we will alternatively use   the running average of the inner product of the  successive gradient across a  sliding window.  References
 \cite[Sect. 2]{kesten1958accelerated} and  \cite{pflug1988stepsize}\remove{\cite{saridis1970learning} and \cite{ruszczynski1986method}} provide  a high-level discussion on  this statistic.  Nevertheless,  little work has been done in determining the critical values to draw a confident conclusion of  either a  transient or steady-state phase.

 \begin{thm}
 	[Detection of Transient Phase and Steady-State Phase] \label{thm:InnerProduct} 
 	Under C.\ref{assume:Stationary}, let us pick the gain $\gain$ such that 
 	 \begin{equation}\label{eq:gainStationary}
 	\gain <  \LipsPara^{-1}, 
 	\end{equation}
 	where $ \uplambda_{\max}\ilparenthesis{\bH} =  \LipsPara  $. 
 \begin{enumerate}[(1)]
 	\item During the  \emph{steady-state} phase for \emph{large} $k$, we have 
 	\begin{equation}
 	\E \set{  \ilbracket{\hbg_k \ilparenthesis{\hbtheta_k }}^\transpose \hbg_{k-1} \ilparenthesis{\hbtheta_{k-1} }   }  \approx  -\gain\tr\ilparenthesis{\bH\VarianceError} + O\ilparenthesis{\gain^2}, \text{ for large $k$ and   for  (\ref{eq:gainStationary})},
 	\end{equation}
 	and   
 	\begin{equation}\label{eq:InnerProductVariance}
 	\Var\set{ \frac{1}{\window} \sum_{ k-\window+1  }^{k}  \ilbracket{\hbg_i  \ilparenthesis{\hbtheta_i }}^\transpose \ilbracket{\hbg_{i-1} \ilparenthesis{\hbtheta_{i-1} }}   } \le \frac{1}{\window}O\ilparenthesis{\gain }, \text{ for large $k$ and for   (\ref{eq:gainStationary})},
 	\end{equation}
 	where $\window>0$ is  an arbitrary   window size. 
 	\item During  the \emph{transient} phase for \emph{small} $k\ge 1$, we have 
 	\begin{align}\label{eq:PhaseTransient}
 	 	\E \set{  \ilbracket{\hbg_k \ilparenthesis{\hbtheta_k }}^\transpose \hbg_{k-1} \ilparenthesis{\hbtheta_{k-1} }   }&  \le 	 \hbtheta_0^\transpose\bH^2\hbtheta_0 - \gain\bracket{ \hbtheta_0^\transpose\bH^3\hbtheta_0  +  \tr\ilparenthesis{\bH\bV} }+ O(\gain^2), \nonumber\\
 	&\quad\quad\quad\quad\quad\quad\quad\quad \text{ for $\gain$ satisfying (\ref{eq:gainStationary}).}
 	\end{align}
 \end{enumerate} 
 	
 \end{thm}

 \begin{proof}[Proof of Theorem~\ref{thm:InnerProduct}]
Assume that $\bvartheta=\zero$ w.l.o.g., as the following discussion remains to be valid if $\hbtheta_k$ is replaced by $\ilparenthesis{\hbtheta_k-\bvartheta}$ for a nonzero $\bvartheta$.

  	Under C.\ref{assume:Stationary}, we  can rewrite (\ref{eq:ConstantGainSGD}) as follows:
  		\begin{eqnarray}
  	&\quad\quad &\hbtheta_{k+1} = \hbtheta_k - \gain  \bH \hbtheta_k  - \gain \noise_k = \ilparenthesis{\bI-\gain\bH}\hbtheta_k - \gain\noise_k ,\quad\text{ for } k\ge 0, \label{eq:ConstantGainAlgoStationary0} \\ 
  	&\implies &	\hbtheta_{k} = \ilparenthesis{\bI-\gain\bH}^{k} \hbtheta_0   - \gain \sum_{i=0}^{k-1} \bracket{ \ilparenthesis{\bI-\gain\bH}^{i} \noise _{k-1-i}   }\quad \text{ for }k\ge 1.  \label{eq:ConstantGainAlgoStationary}
  	\end{eqnarray} 
  From (\ref{eq:ConstantGainAlgoStationary}) we know that the Markovian process $\hbtheta_k$ generated from (\ref{eq:ConstantGainSGD}) is comprised of a deterministic part   $\ilparenthesis{\bI-\gain\bH}^{k} \hbtheta_0$ (assuming that there is no randomness in $\hbtheta_0$) and a  stochastic   part $- \gain \sum_{i=0}^{k-1} \ilbracket{ \ilparenthesis{\bI-\gain\bH}^{i} \noise _{k-1-i}   }$, which has a mean of $\zero$ under C.\ref{assume:Stationary}. 
  
   With a gain  satisfying (\ref{eq:gainStationary}), 
  the deterministic part       goes to $\zero$ exponentially as $k$ grows, and the stochastic part converges in law to the stationary process $ -\gain \sum_{i=0}^{\infty} \ilparenthesis{\bI-\gain\bH}^{i} \noise_{k-1-i} $. Resultingly, during the transient phase for \emph{small} $k$, the linear convergence of the deterministic part is dominating compared with the stochastic part with a mean of zero; then during the steady-state phase for \emph{large} $k$, the oscillating characteristic of the stationary process $ -\gain \sum_{i=0}^{\infty} \ilparenthesis{\bI-\gain\bH}^{i} \noise_{k-1-i} $ dominates  compared to the deterministic part that decays to $\zero$ exponentially fast in $k$.

  		Similarly, we can also rewrite the \emph{noisy} gradient observation as follows:
  			\begin{eqnarray}
  		&\quad & \hbg_k\ilparenthesis{\hbtheta_k} = \bH \hbtheta_k+\noise_k  \nonumber\\
  		&\implies &  \hbg_k\ilparenthesis{\hbtheta_k} = \bH \ilparenthesis{\bI-\gain\bH}^k\hbtheta_0 + \noise_k - \gain\bH \sum_{i=0}^{k-1}  \bracket{\ilparenthesis{\bI-\gain\bH}^i\noise_{k-1-i} } ,  \label{eq:ConstantGainAlgoStationary2}
  		\end{eqnarray}
  		where the implication in (\ref{eq:ConstantGainAlgoStationary2}) uses  (\ref{eq:ConstantGainAlgoStationary}) directly.  From (\ref{eq:ConstantGainAlgoStationary2}), we see  that   $\hbg_k\ilparenthesis{\hbtheta_k} $ 
  		is comprised of a deterministic part 
  		$\bH \ilparenthesis{\bI-\gain\bH}^k\hbtheta_0 $ (assuming that there is no randomness in $\hbtheta_0$) and a stochastic part $ \ilset{\noise_k-  \gain\bH \sum_{i=0}^{k-1}  \ilbracket{\ilparenthesis{\bI-\gain\bH}^i\noise_{k-1-i} } } $, which has  a   mean of $\zero$ under C.\ref{assume:Stationary}. Again,  with a constant gain $\gain$ such that (\ref{eq:gainStationary}) holds, we see  the deterministic part  goes to $\zero$ exponentially as $k$ grows,      and the stochastic part converges in law to a   stationary process $ \noise_k   -\gain \bH  \sum_{i=0}^{\infty} \ilparenthesis{\bI-\gain\bH}^{i} \noise_{k-1-i}  $. 
  		
  	 \textbf{Let us consider  the \emph{steady-state} phase for \emph{large} $k$.} The  multivariate moving-average process $ -\gain \sum_{i=0}^{\infty} \ilparenthesis{\bI-\gain\bH}^{i} \noise_{k-1-i} $ is mean zero. Denote the covariance matrix for $-\gain \sum_{i=0}^{\infty} \ilparenthesis{\bI-\gain\bH}^{i} \noise_{k-1-i} $ as $\VarianceMA $.  For large $k$, $\VarianceMA $   satisfies the following: 
  			\begin{equation}\label{eq:VarianceMA1}
  			\VarianceMA  =    \ilparenthesis{\bI-\gain\bH}\VarianceMA  {\ilparenthesis{\bI-\gain\bH}}+ \gain^2 \VarianceError, \,\, \text{for } \gain \text{ satisfying (\ref{eq:gainStationary})}, 
  			\end{equation}	by taking the variance on both sides of (\ref{eq:ConstantGainAlgoStationary0}) and then letting $k\to\infty$.  The solution to (\ref{eq:VarianceMA1}) can be explicitly expressed as:  \begin{equation}\label{eq:VarianceMA2}
  			\VarianceMA = \gain ^2 \sum_{i=0}^\infty \ilparenthesis{\bI-\gain\bH}^i \VarianceError \ilparenthesis{\bI-\gain\bH}^{i}. 
  			\end{equation}
  			Meanwhile, the multivariate moving-average process $\noise_k    -\gain \bH  \sum_{i=0}^{\infty} \ilparenthesis{\bI-\gain \bH}^{i} \noise_{k-1-i}  $   also has a mean of zero. Denote the covariance matrix for $\noise_k    -\gain \bH  \sum_{i=0}^{\infty} \ilparenthesis{\bI-\gain \bH}^{i} \noise_{k-1-i}  $ as $\VarianceG _ k $. For large $k$, $\VarianceG _ k $  satisfies the following:  
  				\begin{equation}\label{eq:VarianceG1}
  				\VarianceG _k  =  \gain ^2 \bH \set{\sum_{i=0}^{\infty} \bracket{  \ilparenthesis{\bI-\gain\bH}^{i}  \VarianceError \ilparenthesis{\bI-\gain\bH }^i      } }   \bH  +  \VarianceError = \bH \VarianceMA \bH + \VarianceError, \,\, \text{for }   \text{   (\ref{eq:gainStationary})},
  				\end{equation}
  					by taking the variance on both sides of (\ref{eq:ConstantGainAlgoStationary2})   and using (\ref{eq:VarianceMA2}). Moreover, for $k\ge l$, the covariance of $\hbg_k\ilparenthesis{\hbtheta_k}$ and $\hbg_l\ilparenthesis{\hbtheta_l }$ for large $l$  is 
  					\begin{equation}
  				\label{eq:VarianceG2} \VarianceG_{k:l}\equiv 	\Cov\ilparenthesis{\hbg_k\ilparenthesis{\hbtheta_k}, \hbg_l \ilparenthesis{\hbtheta_l} } = \ilparenthesis{\bI-\gain  \bH}^{k-l}  \VarianceG_k .
  				\end{equation}  
  		When (\ref{eq:gainStationary}) holds, we have $ \norm{\bI-\gain\bH} = 1-\gain\convexPara\in\ilparenthesis{0,1} $, where $\convexPara= \uplambda_{\min}\ilparenthesis{\bH}$. Hence,   
  		\begin{equation}
  		\label{eq:VarianceG3}
  		 \norm{\VarianceG_{k:l} } \le \ilparenthesis{1-\gain\convexPara}^{k-l} \norm{\VarianceG_k} 
  		\end{equation}
  		where the number $\ilparenthesis{1-\gain\convexPara}^{k-l} $ arises due to $ \norm{\ilparenthesis{\bI-\gain\bH}^{k-l}} \le \norm{  \bI-\gain\bH }^{k-l}   =  \ilparenthesis{1-\gain\convexPara}^{k-l} $. 
  		
  		Using (\ref{eq:ConstantGainAlgoStationary2}),   we have the following approximation:	\begin{eqnarray}
  		&\quad & \ilbracket{\hbg_k \ilparenthesis{\hbtheta_k }}^\transpose \hbg_{k-1} \ilparenthesis{\hbtheta_{k-1} }\nonumber\\
  		& \approx & \set{   \noise_k   - \gain \bH \sum_{i=0}^{k-1}  \bracket{\ilparenthesis{\bI-\gain\bH}^i\noise _{k-1-i} } } ^\transpose  \set{  \noise_{k-1 } - \gain\bH \sum_{i=0}^{k-2}  \bracket{\ilparenthesis{\bI-\gain\bH}^i\noise_{k-2-i} } }\nonumber\\
  		&=& \set{  \noise_k  -  \gain \bH \noise_{k-1}  -   \gain\bH   \sum_{i=1}^{k-1}  \bracket{\ilparenthesis{\bI-\gain\bH}^i\noise_{k-1-i} } } ^\transpose \nonumber\\
  		&\quad &\quad  \cdot  \set{   \noise_{k-1 } - \gain\bH \sum_{i=0}^{k-2}  \bracket{\ilparenthesis{\bI-\gain\bH}^i\noise_{k-2-i} } } \nonumber\\ 
  		&=& \set{  \noise_k  -  \gain \bH \noise_{k-1}  -   \gain\bH  \ilparenthesis{\bI-\gain\bH } \sum_{i=0}^{k-2}  \bracket{\ilparenthesis{\bI-\gain\bH}^i\noise_{k-2-i} } } ^\transpose  \nonumber\\
  		&\quad & \quad \cdot  \set{   \noise_{k-1 } - \gain\bH \sum_{i=0}^{k-2}  \bracket{\ilparenthesis{\bI-\gain\bH}^i\noise_{k-2-i} } }, \,\,\,\text{ for large $k$ and   $\gain$ satisfying (\ref{eq:gainStationary}), }\nonumber\\
  		\end{eqnarray}
  		where  the first approximation is claimed after discarding the deterministic part in (\ref{eq:ConstantGainAlgoStationary2}) for \emph{large} $k$.  Combining  the  above observations,  we have the following:
  				 \begin{align}
  				 &   \E \set{  \ilbracket{\hbg_k \ilparenthesis{\hbtheta_k }}^\transpose \hbg_{k-1} \ilparenthesis{\hbtheta_{k-1} }   }\nonumber\\
  				 &\,\,=  \E \parenthesis{-\gain \noise_{k-1}^\transpose\bH \noise_{k-1}  }\nonumber\\
  				 &\,\, \,\,\,\,\,+ \E \parenthesis{  \set{ -\gain   \sum_{i=0}^{k-2}  \bracket{   \ilparenthesis{\bI-\gain\bH}^i\noise_{k-2-i}   }   } ^\transpose \ilparenthesis{\bI-\gain\bH} \bH ^2 \set{ -\gain    \sum_{i=0}^{k-2}  \bracket{   \ilparenthesis{\bI-\gain\bH}^i\noise_{k-2-i}   }   }}   \nonumber\\ 
  				 &\,\,\approx   -\gain\tr\ilparenthesis{\bH\VarianceError} + \tr\ilparenthesis{\ilparenthesis{\bI-\gain\bH} \bH^2   \VarianceMA  }\nonumber\\
  				 &\,\,=  -\gain\tr\ilparenthesis{\bH\VarianceError} + \tr\ilparenthesis{\bH^2\VarianceMA }  - \gain\tr\ilparenthesis{\bH^3\VarianceMA }  \nonumber \\
  				 & \,\,=  -\gain\tr\ilparenthesis{\bH\VarianceError} + O(\gain^2),   \,\,\,\text{ for large $k$ and for $\gain$ satisfying (\ref{eq:gainStationary}), }
  				 \end{align} where the   approximation uses    (\ref{eq:VarianceMA1}), (\ref{eq:ConstantGainAlgoStationary2}) and C.\ref{assume:Stationary}, and 
  				  the last equation   is due to the coefficient $\gain ^2$ on the r.h.s. of  (\ref{eq:VarianceMA2}). 
Furthermore, we also have: 
  	\begin{equation}
  	\Cov\parenthesis{ \ilbracket{\hbg_k \ilparenthesis{\hbtheta_k }}^\transpose \hbg_{k-1} \ilparenthesis{\hbtheta_{k-1} },\ilbracket{\hbg_l \ilparenthesis{\hbtheta_l }}^\transpose \hbg_{l-1} \ilparenthesis{\hbtheta_{l-1} }    }  = O (\gain  \ilparenthesis{1-\gain\convexPara}^{k-l} ), 
  	\end{equation}
  	which follows from (\ref{eq:VarianceG3}). Then (\ref{eq:InnerProductVariance}) immediately follows.

  	 \textbf{	Let us consider  the \emph{transient} phase for \emph{small} $k$.}	  We have the following observation:
  				 	\begin{align}\label{eq:PhaseTransient1}
  	& \E \set{\ilbracket{\hbg_k \ilparenthesis{\hbtheta_k }}^\transpose \hbg_{k-1} \ilparenthesis{\hbtheta_{k-1} }}\nonumber\\
  				 &\,\,  = \hbtheta_0^\transpose\ilparenthesis{\bI-\gain \bH} ^{k} \bH ^2\ilparenthesis{\bI-\gain\bH}^{k-1} \hbtheta_0   -  \E \parenthesis{\gain \noise_{k-1}^\transpose\bH \noise_{k-1}  }\nonumber\\
  				 &\quad  + \E \parenthesis{  \set{ -\gain   \sum_{i=0}^{k-2}  \bracket{   \ilparenthesis{\bI-\gain\bH}^i\noise_{k-2-i}   }   } ^\transpose \ilparenthesis{\bI-\gain\bH} \bH ^2 \set{ -\gain    \sum_{i=0}^{k-2}  \bracket{   \ilparenthesis{\bI-\gain\bH}^i\noise_{k-2-i}   }   }}   \nonumber\\ 
  				 &\,\,  = \hbtheta_0^\transpose\ilparenthesis{\bI-\gain \bH} ^{k} \bH ^2\ilparenthesis{\bI-\gain\bH}^{k-1} \hbtheta_0   -  \gain\tr\ilparenthesis{    \bH   \VarianceError  }     + \gain^2 \sum_{i=0}^ {k-2} \tr\ilparenthesis{  \ilparenthesis{\bI-\gain\bH}^{2i+1}\bH^2 \VarianceError   },\nonumber\\
  				 &\quad\quad\quad\quad\quad\quad\quad\quad\quad\quad\quad\quad\quad\quad\quad\quad\quad\quad\quad\quad\quad\quad\quad\quad\quad\quad\quad   \text{ for $k\ge 2 $, }   
  				 \end{align}  
  				 where the  binomial series $
  				 \ilparenthesis{\bI-\gain\bH}^k  =   \sum_{i=0}^k \binom{k}{i} \ilparenthesis{\gain\bH}^i $ for $k\ge 0$. 
  	For $k=1$, 	we have  	  
  	\begin{align}
  &	\E \set{\ilbracket{\hbg_k \ilparenthesis{\hbtheta_k }}^\transpose \hbg_{k-1} \ilparenthesis{\hbtheta_{k-1} }}\nonumber\\  &\,\, =  \hbtheta_0^\transpose\ilparenthesis{\bI-\gain\bH} \bH^2 \hbtheta_0 - \gain \tr\ilparenthesis{\bH\bV} \nonumber\\
  	&\,\, =  \hbtheta_0^\transpose\bH^2\hbtheta_0 - \gain \hbtheta_0^\transpose\bH^3\hbtheta_0 - \gain \tr\ilparenthesis{\bH\bV} \remove{\nonumber\\
  	&\,\,  \ge \convexPara^2\norm{\hbtheta_0}^2 - \gain \LipsPara^3 \norm{\hbtheta_0}^2 - \gain \LipsPara\tr\ilparenthesis{\bV}\nonumber\\
  	&\,\, \ge( \convexPara^2-\LipsPara^2)\norm{\hbtheta_0}
^2    - \tr\ilparenthesis{\bV},} \quad \text{ for $\gain$ satisfying (\ref{eq:gainStationary}).}
  	\end{align} 
  				 For $k=2$, we have
  	\begin{align}
  & 	\E \set{\ilbracket{\hbg_k \ilparenthesis{\hbtheta_k }}^\transpose \hbg_{k-1} \ilparenthesis{\hbtheta_{k-1} }}\nonumber\\ 
  &\,\, =  \hbtheta_0^\transpose\ilparenthesis{\bI-\gain\bH}^2\bH^2\ilparenthesis{\bI-\gain\bH}\hbtheta_0   - \gain \tr\ilparenthesis{\bH\VarianceError} + O\ilparenthesis{\gain^2}\nonumber\\
  &\,\,  = \hbtheta_0^\transpose\bH^2 \hbtheta_0  - 3\gain \hbtheta_0^\transpose\bH^3\hbtheta_0 - \gain\tr\ilparenthesis{\bH\bV}  + O(\gain^2),   \quad \text{ for $\gain$ satisfying (\ref{eq:gainStationary}).}
    	\end{align}			 
  	In general, for small $k\ge 1 $, the magnitude  deterministic part should dominate the magnitude of the mean-zero stochastic part, and (\ref{eq:PhaseTransient}) holds.	  
 \end{proof}

 Based on Theorem~\ref{thm:InnerProduct}, we propose the following strategy to adapt the gain sequence. 
 Let us recursively define a sequence of   ``critical'' times  $ \ilset{\TimeCritical _{\stopTime}} $ such that: 
 \begin{align}\label{eq:stop1}
 & \TimeCritical_{\stopTime+1}  =\inf \Bigg\{ {k>\TimeCritical _{\stopTime} }\bigg |       \frac{1}{k-\TimeCritical _{\stopTime}  }  \sum_{i=\TimeCritical _{\stopTime}+1}^k    \ilbracket{\hbg_{i}\ilparenthesis{\hbtheta_{i}}}^\transpose\ilbracket{\hbg_{i-1} \ilparenthesis{\hbtheta_{i-1}}} \text{ is either }  \le -  \gain \tr\parenthesis{\widehat{\bH} \widehat{    \VarianceError }     } \nonumber\\
 	&\quad\quad\quad\quad\quad\quad\quad\quad   \text{ or }   \ge \hbtheta_{ k  } ^\transpose \widehat{   \bH   } ^2    \hbtheta_{  k  }    - \gain \bracket{   \hbtheta_{ k  } ^\transpose \widehat{   \bH   } ^3    \hbtheta_{k }    +  \tr\ilparenthesis{ \widehat{   \bH   }   \widehat{    \VarianceError }    }   }            \Bigg\}, \nonumber\\
 &  \text{ with }\TimeCritical _0  =0, 
 \end{align} where $\widehat{   \bH }$ and  $\widehat{  \VarianceError }$are   the estimates   for $\bH$ and $\VarianceError$, respectively. How to construct $\widehat{\bH}$   and $\widehat{    \VarianceError } $ will be discussed momentarily.
 Correspondingly, the gain sequence  is  defined by:  
 \begin{align}\label{eq:dataGain1}
& \gainStopNew =\begin{cases}
 \shrink  \gain_{\stopInd},  \text{ if }      {    \frac{1}{k-\TimeCritical_\stopTime }  \sum_{i=\TimeCritical_{\stopTime}+1}^k    \ilbracket{\hbg_{i}\ilparenthesis{\hbtheta_{i}}}^\transpose\ilbracket{\hbg_{i-1} \ilparenthesis{\hbtheta_{i-1}}}   \le  -  \gain \tr\parenthesis{\bH \VarianceError     }      },\\
 \increase \gainStop, \\
 \,\,  \text{ if }      {    \frac{1}{k-\TimeCritical_\stopTime }  \sum_{i=\TimeCritical_{\stopTime}+1}^k    \ilbracket{\hbg_{i}\ilparenthesis{\hbtheta_{i}}}^\transpose\ilbracket{\hbg_{i-1} \ilparenthesis{\hbtheta_{i-1}}}   \ge   \hbtheta_{ k  } ^\transpose \widehat{   \bH   } ^2    \hbtheta_{  k  }    - \gain \bracket{   \hbtheta_{ k  } ^\transpose \widehat{   \bH   }  ^3    \hbtheta_{ k }    +  \tr\ilparenthesis{ \widehat{   \bH   }  \widehat{    \VarianceError }   }   }    }, 
 \end{cases}   
 \end{align}
 where $\gain  = \gain_{\stopInd}$ for $\TimeCritical_{\stopTime} \le k < \TimeCritical_{\stopTime+1}$.  
 Unfortunately, we do not have any quantification regarding the Type-I and Type-II errors for the phase detection in (\ref{eq:stop1}) at the moment.

 \subsubsection*{Estimation of Hessian Information and Error Covariance }

 We need both $\bH$ and $\VarianceError$ to perform gain adaptation (\ref{eq:dataGain1}), yet they are unknown.  
 Let us briefly obtain        $\widehat{\bH}$ and $\widehat{\VarianceError}$  through the  observable information $\field_{k+1}$.  
 We borrow the SP idea in \cite{spall2000adaptive} to construct $\widehat{\bH}$. Here we will slightly alter the recursion (\ref{eq:ConstantGainSGD}) into the following:
 \begin{equation}\label{eq:sgd2meas}
 \hbtheta_{k+1} = \hbtheta_k - \gain \frac{  \hbg_k\ilparenthesis{  \hbtheta_k + c_k\bDelta_k   } + \hbg_k \ilparenthesis{\hbtheta_k  - c_k\bDelta_k  }  }{2}, 
 \end{equation}
 where the setup for  $c_k$ and $\bDelta_k$   are the same as that in (\ref{eq:g1SP}). At the cost of two  measurements  at each $k$, we can estimate $\bH$ recursively as follows: 
 \begin{align}\label{eq:Hestimate}
 \widehat{\bH}_k &= \frac{k}{k+1} \widehat{\bH}_{k-1}+  \frac{1}{4c_k \ilparenthesis{k+1}}      \ilbracket{  \hbg_k\ilparenthesis{\hbtheta_k+c_k\bDelta_k  } - \hbg_k \ilparenthesis{\hbtheta_k  - c_k\bDelta_k   }    }\bDelta_k ^{-\transpose}  \nonumber\\
 &\,\,+  \frac{1}{4c_k \ilparenthesis{k+1}}   \bDelta_k ^{-1}  \ilbracket{  \hbg_k\ilparenthesis{\hbtheta_k+c_k\bDelta_k  } - \hbg_k \ilparenthesis{\hbtheta_k  - c_k\bDelta_k   }    }^\transpose   ,\,\, k= 1,2,\cdots
 \end{align}  where   $ \bDelta^{-\transpose} = \ilparenthesis{\bDelta^{-1}}^\transpose $. 
  For more details, see \cite{spall2000adaptive} or (\ref{eq:H_hat}) in  Appendix~\ref{chap:2nd}. Note that the initialization for (\ref{eq:Hestimate}) may be a scale matrix ($\mathrm{scale} \cdot \bI_p$ for $\mathrm{scale}>0$), or some other positive-definite matrix reflecting available information (e.g., if one knows that $\btheta$ elements will have very different magnitudes, then the initialization may be chosen to approximately scale for the differences).  
 Similarly, we 
 estimate $\VarianceError$ recursively as follows:
 \begin{align}\label{eq:Eestimate}
 \hat{\VarianceError}_k &= \frac{k}{k+1}\hat{\VarianceError}_{k-1} \nonumber\\
 &\,\,+ \frac{ 1}{ \ilparenthesis{k+1}}    \bracket{\hbg_k\ilparenthesis{ \hbtheta_k+c_k\bDelta_k } - \hbg_k\ilparenthesis{\hbtheta_k-c_k\bDelta_k }    }\bracket{\hbg_k\ilparenthesis{ \hbtheta_k+c_k\bDelta_k } - \hbg_k\ilparenthesis{\hbtheta_k-c_k\bDelta_k }    }^\transpose . 
 \end{align}

 \subsubsection{Summary of Adapted Gain-Tuning Algorithm}
 
 Let us summarize the  aforementioned procedure, including gain adaptation and the estimation of $\bH$ and $\VarianceError$ in   Algorithm  \ref{algo:adaptiveGain} below. 
 
 	\begin{algorithm}[!htbp]
 	\caption{Adaptive Gain Selection for Change Detection Using \emph{Two}-Measurements $ \hbg_k\ilparenthesis{\hbtheta_k\pm c_k\bDelta_k} $ at a Time } 
 	\begin{algorithmic}[1]  
 		\renewcommand{\algorithmicrequire}{\textbf{Input:}}
 		\renewcommand{\algorithmicensure}{\textbf{Output:}}
\Require  initial gain magnitude $\gain $, $ \hbtheta_0 $, increase ratio $\increase$, decrease ratio $\shrink$. 
 		\State \textbf{set} $\stopTime=0$ and $ \TimeCritical_{0} =0$. 
 		\For{$k\ge 1$ or $k\in\ilset{1,\cdots,K}$} \Comment{$K$ is the horizon over which we need to perform tracking.}
 		\State \textbf{collect}  $\hbg_k\ilparenthesis{\hbtheta_k+c_k\bDelta_k} $ and $\hbg_k\ilparenthesis{\hbtheta_k-c_k\bDelta_k}$. \Comment{We may let $c_k$ be the desired minimal change in components of $\hbtheta_k$, and generate $\bDelta_k$ from symmetric Bernoulli  $\pm1$ distribution. }
 		\State \textbf{update} $ \widehat{\bH}_k $ and $\widehat{\bV}_k$ using (\ref{eq:Hestimate}) and (\ref{eq:Eestimate}) respectively. 
 		\If{ $ \frac{1}{k-\TimeCritical_{\stopTime}+1} \sum_{i=\TimeCritical_{\stopTime}+1}^{k}   \ilset{\ilbracket{\hbg_i\ilparenthesis{\hbtheta_i}}^\transpose\ilbracket{\hbg_{i-1}\ilparenthesis{\hbtheta_{i-1}}}}<-\gain  \tr\ilparenthesis{  \widehat{\bH}_k  \widehat{\VarianceError}_k  }    $  } 
 		\State  \textbf{decrease} gain $\gain$ by a factor of $\shrink$.
 		\State \textbf{set} $\stopTime  \gets  \stopTime+1$. 
 		\ElsIf{  $ \frac{1}{k-\TimeCritical_{\stopTime}+1} \sum_{i=\TimeCritical_{\stopTime}+1}^{k}   \ilset{\ilbracket{\hbg_i\ilparenthesis{\hbtheta_i}}^\transpose\ilbracket{\hbg_{i-1}\ilparenthesis{\hbtheta_{i-1}}}}\ge     \hbtheta_{ k  } ^\transpose \widehat{   \bH   }_k ^2    \hbtheta_{  k  }    - \gain \bracket{   \hbtheta_{ k  } ^\transpose \widehat{   \bH   }_k  ^3    \hbtheta_{ k }    +  \tr\ilparenthesis{ \widehat{   \bH   }_k   \widehat{    \VarianceError } _k   }   }     $  }
 		\State    \textbf{increase} gain $\gain $ by a factor of $\increase$.
 		\State \textbf{set} $\stopTime  \gets  \stopTime+1$.  
 		\EndIf 
 		\State \textbf{update} $\hbtheta_k  $ using (\ref{eq:sgd2meas}). 
 		\Ensure $\hbtheta_k$.
 		\EndFor
 	\end{algorithmic}
 	\label{algo:adaptiveGain}
 \end{algorithm}

 \subsubsection*{An Example for Adaptive Gain}
   Here, we again consider $p=2$. The loss function is $\loss \ilparenthesis{\btheta} = \ilparenthesis{\btheta-\bvartheta}^\transpose\bH \ilparenthesis{\btheta-\bvartheta}/2$,   and the gradient function is $ \bg \ilparenthesis{\btheta} = \bH \ilparenthesis{\btheta-\bvartheta} $. 
 Again, $\bH$ is      constructed as $\bH=\bP \bD \bP^\transpose$ in (\ref{eq:SimHessian}), where $\bP$ is   (randomly generated) orthogonal, and $\bD$ is diagonal with diagonal entries $30$ and $5$. For simplicity, we select  $\bvartheta=\zero$. We pick   an increase ratio of $\increase=1.1$ and a decrease ratio $\shrink=0.9$. 
 The observational noise again follows i.i.d. $\mathrm{Normal}\ilparenthesis{\zero, \upsigma^2_1\bI_p}$ with $\upsigma_1=10$. We use an initialization $  \ilparenthesis{100\,\,\, 100}^\transpose $, which is far away from $\bvartheta=\zero$. 
 We can make the following observations from Figures \ref{fig:Compare1} to \ref{fig:Compare3}. 
 \begin{itemize}
 	\item For an appropriately tuned gain, the  estimates generated from  the constant-gain recursion (\ref{eq:ConstantGainSGD})   are capable of getting close to the target and perform  as well  as our adaptive gain algorithm~\ref{algo:adaptiveGain} (see   Figure \ref{fig:Compare1-1}).
 	\item For a gain that is too small, it takes an extremely long time for the estimates generated from  constant gain recursion (\ref{eq:ConstantGainSGD}) to get close to the desired optimum compared with  our adaptive gain algorithm~\ref{algo:adaptiveGain} (see  Figure \ref{fig:Compare2-1}). 
 	\item For a gain that is too large, the constant-gain recursion (\ref{eq:ConstantGainSGD})  will migrate  further and further away  from the target (see   Figure \ref{fig:Compare3-1}). 
 \end{itemize}
 
 In reality, $\LipsPara$ may  \emph{not} be  available to the agent(s), and,  the gain  used in constant-gain recursion (\ref{eq:ConstantGainSGD}) is often \emph{misspecified}. This further manifests the value of the data-dependent gain-tuning strategy summarized in Algorithm \ref{algo:adaptiveGain} and  \cite{zhu2020stochastic}.

 \begin{figure}[!htbp]
 	\centering
 	\begin{subfigure}{.65\textwidth}
 		\centering
 		\includegraphics[width=\linewidth]{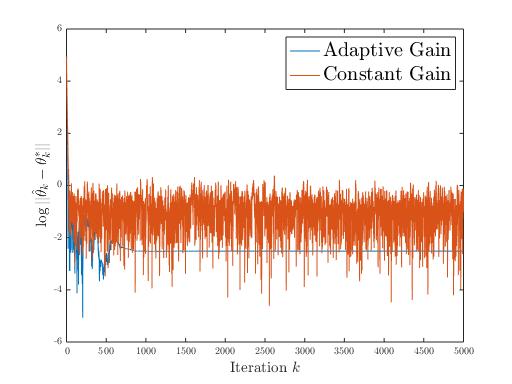}
 		\caption{Log-Euclidean-Distance Between Estimate and True Parameter}  \label{fig:Compare1-1}
 	\end{subfigure}\\
 	\begin{subfigure}{.65\textwidth}
 		\centering
 		\includegraphics[width=\linewidth]{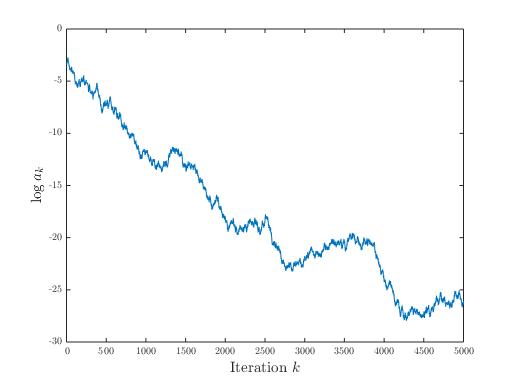}
 \caption{Log-Magnitude of Data-Dependent Gain Sequence} 
  \label{fig:Compare1-2}
 	\end{subfigure}
 \caption{A comparison of adaptive gain used in Algorithm \ref{algo:adaptiveGain} versus constant gain (\ref{eq:ConstantGainSGD}), both of which have   gain initialized at  $1/\LipsPara=0.0333$.   }
 \label{fig:Compare1}
 \end{figure}

 \begin{figure}[!htbp]
 	\centering
 	\begin{subfigure}{.65\textwidth}
 		\centering
 		\includegraphics[width=\linewidth]{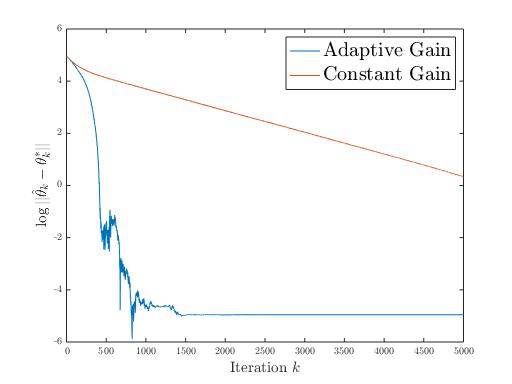}
 	\caption{Log-Euclidean-Distance Between Estimate and True Parameter} 
 	\label{fig:Compare2-1}
 	\end{subfigure}\\
 	\begin{subfigure}{.65\textwidth}
 		\centering
 		\includegraphics[width=\linewidth]{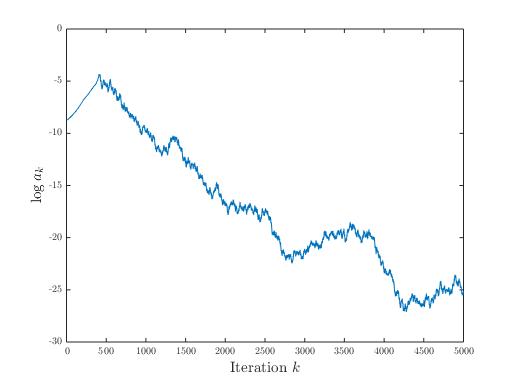}
 	 \caption{Log-Magnitude of Data-Dependent Gain Sequence}  
 	 \label{fig:Compare2-2}
 	\end{subfigure}
 \caption{A comparison of  the adaptive gain used in Algorithm \ref{algo:adaptiveGain} versus the constant gain (\ref{eq:ConstantGainSGD}), both of which have   gain initialized at  $0.005/\LipsPara=1.67\times 10^{-4}$.   }
\label{fig:Compare2}
 \end{figure}

 \begin{figure}[!htbp]
 	\centering
 	\begin{subfigure}{.65\textwidth}
 		\centering
 		\includegraphics[width=\linewidth]{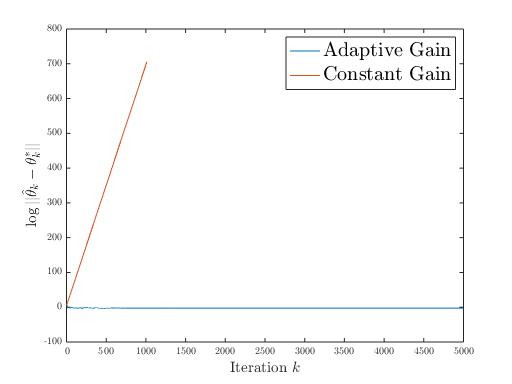}
 	\caption{Log-Euclidean-Distance Between Estimate and True Parameter} 
 	\label{fig:Compare3-1}
 	\end{subfigure}\\
 	\begin{subfigure}{.65\textwidth}
 		\centering
 		\includegraphics[width=\linewidth]{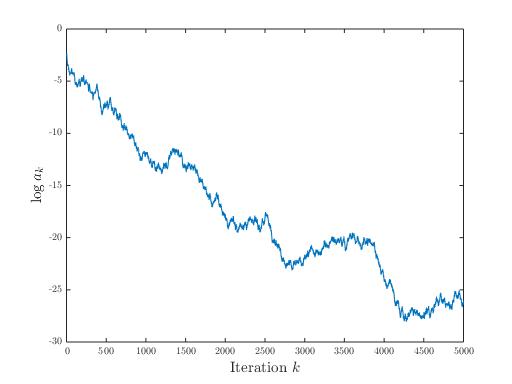}
 		 \caption{Log-Magnitude of Data-Dependent Gain Sequence} 
 		 \label{fig:Compare3-2}
 	\end{subfigure}
 	\caption{Comparison of the adaptive gain used in Algorithm \ref{algo:adaptiveGain} versus the constant gain (\ref{eq:ConstantGainSGD}), both of which have   gain initialized at  $3/\LipsPara=0.1$.   }
 \label{fig:Compare3}
 \end{figure}

 	\subsection{Building Block: Regime Change Detection With Constant Hessian}\label{subsect:BuildingBlock}   We 
 	may apply Algorithm \ref{algo:adaptiveGain} to the scenario where jump structure is allowed.  
 This is a relatively short section as it directly applies Algorithm \ref{algo:adaptiveGain} in 
 Section  \ref{sect:GainOneRegime} to a more general setting C.\ref{assume:Jump} based on the following observations.
 
 \begin{enumerate}
 	\item When regime switches from $\bvartheta_{\stateInd}$ to $ \bvartheta_{\ilparenthesis{\stateregime+1}} $, there will  be a phase of $\hbtheta_k$ steadily approaching the new estimate $ \bvartheta_{\ilparenthesis{\stateregime+1}}$. If an abrupt change is detected, we need to increase the gain  by a factor of $ \increase $, to achieve prompt tracking.

 	\item  When some   oscillating behavior  is detected from the path, we need to decrease the gain   by a factor of $\shrink$, to further approach our desired target.

 	\item For other scenarios (no strong evidence to support a steady-state phase or the transient phase), we simply keep the gain at the most recent level. That is, the gain  $\gain$ is  kept fixed  until we gather strong evidence  in favor of decreasing or increasing the gain.

 \end{enumerate}

 \subsubsection*{An  Example of Adaptive Gain }

 Again consider the same numerical setup as in  Subsubsection \ref{subsect:ExamplePriorBound}, except that the evolution of  $\ilset{\bvartheta_k}$ now changes to (\ref{eq:model1}). In our simulation, the   jump times for the $\ilset{\bvartheta_k}$ are $ 1567, 2949, 3607, 3729$,  and $4498$.

 \begin{figure}[!htbp]
 	\centering
 	\begin{subfigure}{.65\textwidth}
 		\centering
 		\includegraphics[width=\linewidth]{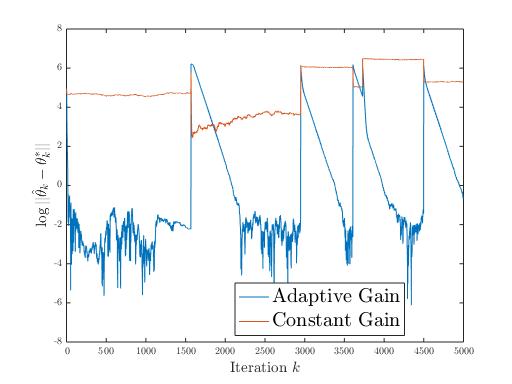}
 			\caption{Log-Euclidean-Distance Between Estimate and True Parameter}  
 			\label{fig:Compare4-1}
 	\end{subfigure}\\
 	\begin{subfigure}{.65\textwidth}
 		\centering
 		\includegraphics[width=\linewidth]{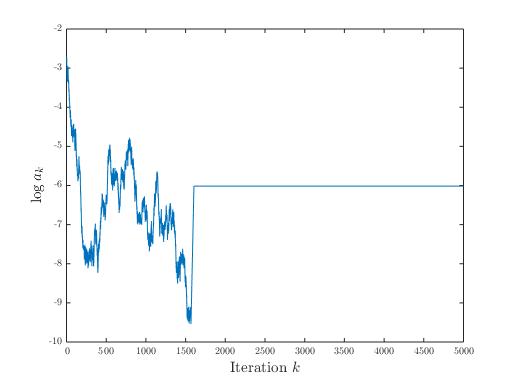} 
 		\caption{Log-Magnitude of Data-Dependent Gain Sequence}  
 		\label{fig:Compare4-2}
 	\end{subfigure}
 	\caption{A comparison of the  adaptive gain used in Algorithm \ref{algo:adaptiveGain} versus constant gain (\ref{eq:ConstantGainSGD}), both of which have   gain initialized at  $2/\LipsPara=0.0667$.  The evolution of $\ilset{\bvartheta_k}$ follows  (\ref{eq:model1}). }
 	\label{fig:Compare4}
 \end{figure}

\remove{
 \begin{figure}[!htbp]
	\centering
	\begin{subfigure}{.65\textwidth}
		\centering
		\includegraphics[width=\linewidth]{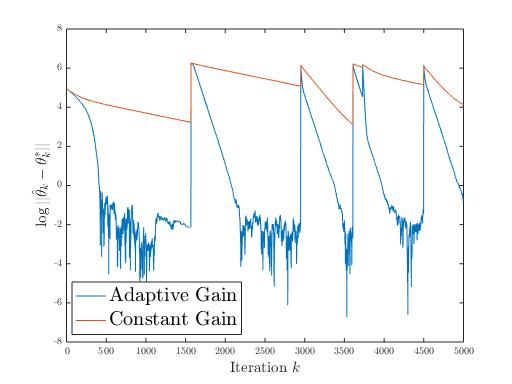}
		\caption{Log-Euclidean-Distance Between Estimate and True Parameter} 
	\end{subfigure}\\
	\begin{subfigure}{.65\textwidth}
		\centering
		\includegraphics[width=\linewidth]{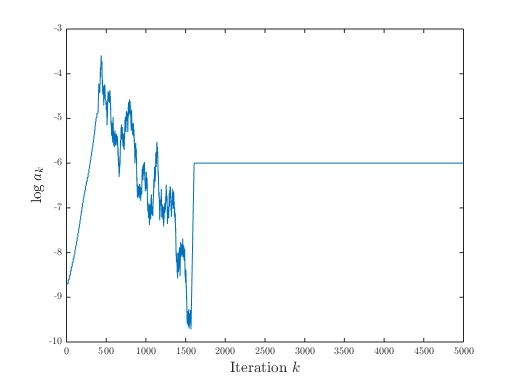} 
		\caption{Log-Magnitude of Data-Dependent Gain Sequence} 
	\end{subfigure}
	\caption{A comparison of the  adaptive gain used in Algorithm \ref{algo:adaptiveGain} versus constant gain (\ref{eq:ConstantGainSGD}), both of which have   gain initialized at  $0.005/\LipsPara=1.67\times 10^{-4}$.  The evolution of $\ilset{\bvartheta_k}$ follows  (\ref{eq:model1}). }
	\label{fig:Compare5}
\end{figure}
}

 \begin{figure}[!htbp]
	\centering
	\begin{subfigure}{.65\textwidth}
		\centering
		\includegraphics[width=\linewidth]{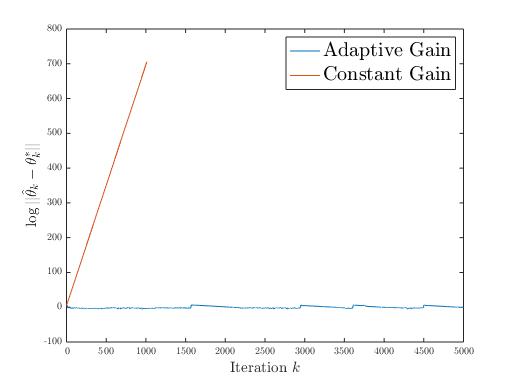}
		\caption{Log-Euclidean-Distance Between Estimate and True Parameter} 
		\label{fig:Compare6-1}
	\end{subfigure}\\
	\begin{subfigure}{.65\textwidth}
		\centering
		\includegraphics[width=\linewidth]{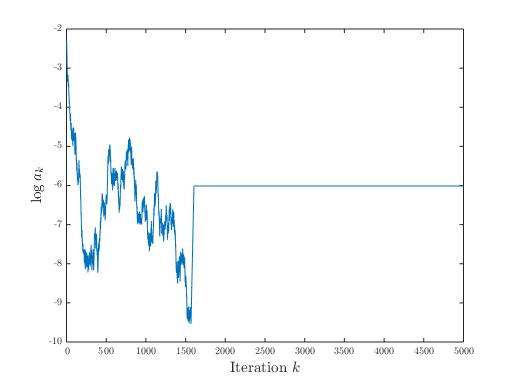} 
		\caption{Log-Magnitude of Data-Dependent Gain Sequence} 
		\label{fig:Compare6-2}
	\end{subfigure}
	\caption{A comparison of the  adaptive gain used in Algorithm \ref{algo:adaptiveGain} versus constant gain (\ref{eq:ConstantGainSGD}), both of which have   gain initialized at  $3/\LipsPara= 0.1$.  The evolution of $\ilset{\bvartheta_k}$ follows  (\ref{eq:model1}). }
	\label{fig:Compare6}
\end{figure}

We can make the following observations from Figures \ref{fig:Compare4} to \ref{fig:Compare6}. 
\begin{itemize}
	\item  Again, we tune the gain for the constant-gain recursion (\ref{eq:ConstantGainSGD}) very carefully, but this ``lazy'' strategy is not  so robust concerning  the jumps in the true $\ilset{\bvartheta_k}$ sequence, compared to the adaptive gain algorithm listed in Algorithm \ref{algo:adaptiveGain}. See  Figure \ref{fig:Compare4-1}. Also, from  Figure \ref{fig:Compare4-2}, we see that the adaptive gain tuning Algorithm \ref{algo:adaptiveGain} does increase the stepsize every time a jump arises  and  decreases  the stepsize every time the estimate is close to the target. 
	\remove{\item For a gain that is too small, it takes an extremely long time for the estimates generated from  constant gain recursion (\ref{eq:CondBound1}) to get close to the desired optimum. See the upper-plot in Figure \ref{fig:Compare2}. }
	\item For a gain that is too large, the constant-gain recursion will migrate  further and further away  from the target. See   Figure \ref{fig:Compare6-1}. Nonetheless, the data-dependent gain-tuning enables the estimates $\hbtheta_k$ to stay close with the moving target.
\end{itemize}

Let us reiterate that $\LipsPara$ may  \emph{not} be  available to the agent(s), and this further  indicates the value  of the data-dependent gain-tuning strategy summarized in Algorithm \ref{algo:adaptiveGain}.

 \section{Concluding Remarks}

 Recall that  
 Chapter~\ref{chap:FiniteErrorBound} provides a computable error bound for non-diminishing gain SA algorithms applied in online learning and dynamic control systems, and naturally gives rise to a gain  selection guidance in Algorithm~\ref{algo:basicSA}, that depends on the   strong convexity parameter $\convexPara_k$, the Lipschitz continuity parameter $\LipsPara_k$, the noise level $\noiseBound_k$, and the drift level $\driftBound_k$. Nonetheless, Chapter~\ref{chap:FiniteErrorBound} only captures the average  performance over possible sample paths.  
 The practical needs to perform well in every sample path drive us to consider data-dependent gain selection strategy, which requires   detecting the jump component in the hybrid system in Section \ref{sect:Detection}, and  estimating  the Hessian information and the  noise level      to adapt the non-diminishing gain sequence intelligently in Section \ref{sect:GainOneRegime}.   The additional restriction C.\ref{assume:Hybrid}   is imposed on the drift, mainly because the error bound discussed in Subsection \ref{subsect:conditional}   requires the availability of  $\convexPara_k$,  $\LipsPara_k $, and  $\noiseBound_k$, to which we  may not have  access   in real-world applications.

 We establish a framework for practical use:  specifically, we can adapt our gain sequence based on our estimate of the  Hessian information and the noise level. This is the key point on which  our work differs from  all the prior work that    require unavailable information, although we   impose stringent assumptions  C.\ref{assume:Hybrid} and so on. 
 The gain adaptation algorithm is developed mainly from the observation that constant-gain SA provides a ``fast transient'' to the vicinity of the solution $\bvartheta$.    Theorem~\ref{thm:InnerProduct}  and Algorithm~\ref{algo:adaptiveGain} are developed to  determine  the critical values to draw a confident conclusion of  either a  transient or steady-state phase, as little work has been done in this direction previously.

\chapter{A Zero-Communication Multi-Agent  Problem   }\label{chap:multiagent}
This chapter      is an illustration of   the tracking capability    of    SA algorithms with non-decaying gains  as applied to the multi-agent multi-target surveillance mission. 
  This  problem of interest is to configure an ensemble of  agents with mobile sensors over a particular region to best\footnote{The quantification of good or bad is according to a set of mission-related metrics, such as the  fraction of
  	targets found, the accuracy of target position estimates and so on. } maintain awareness   of a group of targets within  a specific surveillance
  region.  This     tracking problem is  \emph{dynamic} due to the     motion of both the targets and the agents, and  is \emph{stochastic} due to that only  
  inexact  sensor measurements can be gathered. 
  Given the  two features, this surveillance problem  fits   the  time-varying SA setup (\ref{eq:Minimization}) perfectly, and the loss function $\loss_k\ilparenthesis{\cdot}$ in this chapter will be constructed  in a way such that  the  assumptions A.\ref{assume:ErrorWithBoundedSecondMoment}\textendash A.\ref{assume:BoundedVariation} are met. Again,  there is   no 
  optimal steady-state solution due to  the time-varying characteristic of $\ilset{\bvartheta_k}$.   In fact, this  tracking problem  is what motivates us to solve  (\ref{eq:Minimization}) using general SA algorithms (\ref{eq:basicSA}) with non-decaying gain while making only  modest assumptions on   the error  term  in $\hbg_k\ilparenthesis{\cdot}$  as in A.\ref{assume:ErrorWithBoundedSecondMoment}, the underlying loss function A.\ref{assume:StronglyConvex} and A.\ref{assume:Lsmooth}, and the moving target as in A.\ref{assume:BoundedVariation}, consistent with the main focus of  the entire thesis.

   To ease the  upcoming illustration with graphs, the discussion here will be  on a  two-dimensional $\east$-$\north$ plane with ``E'' and ``N'' representing the east and north directions respectively, i.e., only the latitude and  the longitude are considered. The east and the north directions  can be   relative to the origin of the existing geographic coordinate system, which is currently located in the Gulf of Guinea, or can be relative to any hypothetical origin of the two-dimensional plane.  Nonetheless, they can be readily extended to the three-dimensional space to include the elevation (such as the altitude of the UAV or the depth of the UUV) and other  higher-dimensional problems.

\section{Base Case: One Agent and One Target}\label{sect:basecase}

This section presents the simplest scenario where there are  only one agent and one target. The notation for  this base case can be readily extended to the upcoming general case with multiple targets and multiple agents.

\subsection{Basic Tracking Setup}\label{subsect:base}

We first    define  the     necessary notions for the   tracking problem. Denote the  state vector of the target  at time
$ \uptau_k $     as $ \bx_{k} = \ilparenthesis{  x_{k}^{\east},\,\,\,   x_{k}^{\north },     \,\,\,  \dot{x}_{k}^{\east}, \,\,\,  \dot{x}_{k}^{\north}    }^\transpose \in\real^4 $, where  $ 
\ilparenthesis{  x_{k}^{\east},\,\,\,   x_{k}^{\north}      }^\transpose\in\real^2$ is the coordinate of the target's position  at index $k $, and  $\dot{x}_{k}^{\east}$ and $\dot{x}_{k}^{\north}  $   are the magnitudes  of the target's velocity  in the directions of the east and the north.  
Similarly, the state of the agent  at time $\uptau_k $ will be denoted as $ \by_{k} =   \ilparenthesis{  y_{k}^{\text{E}},\,\,\,   y_{k}^{\text{N}}  ,   \,\,\,  \dot{y}_{k}^{\text{E}} ,\,\,\,  \dot{y}_{k}^{\text{N}}   }^\transpose \in\real^4  $.

Besides, let ${\velocity}_{\bx}^{\max}$ and $ {\velocity}_{\by}^{\max} $ denote the speed limits of    the target and the agent respectively. They set   constraints  on the Euclidean norm of   $ \ilparenthesis{ \dot{x}_k,\,\,\, \dot{x}_k} ^\transpose$ and $ \ilparenthesis{ \dot{y}_k,\,\,\, \dot{y}_k} ^\transpose$ respectively.    Take UUVs as an example: the maximum speed  is typically around  $15$ meters per second. We will correspondingly use one second as the unit for the sampling time $\uptau_k$. For simplicity, we will omit the unit ``meters per second'' for the speed limit, the unit ``meter'' for the distance,  and the unit ``seconds'' for time throughout this chapter.

The     available  information  that can be collected through the agent's sensor at time $\uptau_k $ typically include the \emph{noisy} measurement of the azimuth angle from the agent to the target defined as
\begin{equation}\label{eq:jiaodu} \left. 
\upvarphi\ilparenthesis{\bx,\by} \right| _{\ilparenthesis{\bx,\by}  = \ilparenthesis{\bx_k, \by_k } }= \begin{dcases}
\arctan \parenthesis{  \frac{  x_k^{\text{N}} - y_k ^{\text{N}}   }{   x_k^{\text{E}} - y_k ^{\text{E}}   }  }, &\text{ when }x_k^{\north}>y_k^{\north} \text{ and }x_k^{\east} > y_k^{\east},\\
 \arctan \parenthesis{  \frac{  x_k^{\text{N}} - y_k ^{\text{N}}   }{   x_k^{\text{E}} - y_k ^{\text{E}}   }  } + \uppi\,   , &\text{ when } x_k^{\east} < y_k^{\east}, \\
 \arctan \parenthesis{  \frac{  x_k^{\text{N}} - y_k ^{\text{N}}   }{   x_k^{\text{E}} - y_k ^{\text{E}}   }  }+ 2\uppi   , &\text{ when }x_k^{\north}<y_k^{\north} \text{ and }x_k^{\east} > y_k^{\east},
\end{dcases}
\end{equation}
and the \emph{noisy} measurement of the range between the agent and  the target denoted as 
\begin{equation}  \label{eq:changdu}
  \left. \uprho \ilparenthesis{\bx,\by} \right|      _{\ilparenthesis{\bx,\by}  = \ilparenthesis{\bx_k, \by_k } }= \sqrt{  \parenthesis{x_k^{\text{E}} - y_k ^{\text{E}}}^2 + (x_k^{\text{N}} - y_k ^{\text{N}})^2 },
\end{equation}  
The adjustment $\uppi \, \indicator_{\ilset{ y_k^{\text{E}} - x_k^{\text{E}}>0 }}$ and $ 2\uppi \indicator _{ \set{ x_k^{\north} < y_k^{\north} ,\text{ and }   y_k^{\east}  > y_k^{\east}  } }  $ in (\ref{eq:jiaodu})   serves to enable  $\upvarphi\ilparenthesis{\bx_k,\by_k }$ to be   the direction that the agent needs to move along in order to get closer  to the target. To avoid the issues arising from differentiating   $\arctan(\cdot)$ function (due to its  periodicity) and differentiating  $\sqrt{\,\cdot\, }$ function (due to its  non-differentiability at the origin) in what follows,
the agent's  observable information is rearranged as: 
\begin{equation}\label{eq:noisyMeasurement}
\left. \bz\ilparenthesis{\bx,\by} \right|  _{\ilparenthesis{\bx,\by}  = \ilparenthesis{\bx_k, \by_k } }= \begin{pmatrix}
  \uprho \ilparenthesis{\bx_k,\by_k }\cos\parenthesis{\upvarphi\ilparenthesis{\bx_k,\by_k }}  \\  \uprho\ilparenthesis{\bx_k,\by_k }\sin \parenthesis{\upvarphi\ilparenthesis{\bx_k,\by_k }}
\end{pmatrix}    + \bv_k \in\real^2,
\end{equation}
where $\bv_k\stackrel{\mathrm{i.i.d.}}{\sim }\mathrm{Normal}\ilparenthesis{\zero,\bR_k }$ with  a   covariance matrix of 
\begin{equation}\label{eq:Rmatrix}
\bR_k = \diag(10,10). 
\end{equation}
The covariance matrix for the measurement noise $\bv_k$   in  (\ref{eq:Rmatrix}) is proposed   based on the fact that   the typical   GPS devices  nowadays is   accurate anywhere within  3   to 10 meters.

After obtaining the noisy measurement (\ref{eq:noisyMeasurement}), the agent needs to pick an action  $\btheta_{k} = ( \dot{y}_k^{\east},\, \,\,\dot{y}_k^{\north} )^\transpose\in\real^2$   to determine the magnitude and the direction of its speed at time $\uptau_k$.  
With the aforementioned notation,  let us briefly describe the real-time tracking by iterative updating procedure. 
\begin{enumerate}[i)]
	\item  \label{item:KF1}   At   time  $\uptau_k $, the target is at      state $\bx_k$  according to its desired  motion model (which  is not revealed to the agent), and the agent is at   state $\by_k$.
	
	  The agent is allowed to collect noisy measurements $ \bz(\bx_k,\by_k)  $.
	  Then the agent predicts the next possible position of the target, denoted as  $\hat{\bx}_{\given{k+1}{k}}$ 
	  by making use  of  $\bz(\bx_k,\by_k)$, and the details will be discussed momentarily. With an a priori prediction $\hat{\bx}_{\given{k+1}{k}}$ for the upcoming  state $\bx_{k+1}$ of the target, 
	  the  agent then   picks a direction $\hbtheta_k $ such that the resulting    position of the agent at time $\uptau_{k+1}$ becomes 
\begin{equation}\label{eq:nextstate}
  \begin{pmatrix}
  y_{k+1}^{\east }\\  y_{k+1}^{\north}
  \end{pmatrix} =   \begin{pmatrix}
  y_{k}^{\east }\\  y_k^{\north}
  \end{pmatrix} + \ilparenthesis{\uptau_{k+1}-\uptau_k}\hbtheta_k, \text{ for }\hbtheta_k \in   \bTheta \equiv  \set{ \btheta:  \norm{\btheta
  	}\le{\velocity}_{\by}^{\max} }\subsetneq\real^2,
\end{equation} which should be  as close to $\hat{\bx}_{\given{k+1}{k}}$ as possible. The set  $\bTheta\subsetneq\real^2$ is natural due to the physical  constraints of the  speed limit.

\item \label{item:KF2}  Then at time $\uptau_{k+1}$, the target arrives at state $\bx_{k+1}$ and the agent arrives at $\by_{k+1}$ with the first two components specified as  in 
 (\ref{eq:nextstate}).  We can then repeat  the same procedure  in    \ref{item:KF1}   by setting $k\gets k+1$. 
\end{enumerate}  
 \begin{rem}\label{rem:rotation}
 	For (\ref{eq:nextstate}) to be valid, we need to assume that the agent updates its state according to  the speed   $\hbtheta_k$ at time $\uptau_k $. The  effect of the rotational dynamics are assumed  negligible such that 
 	the UUV can \emph{instantaneously} change direction for all $k$. 	 
 \end{rem}

\subsection{Loss Function  }\label{subsect:loss1}

Now let us discuss the details  of step \ref{item:KF1}   by constructing a  time-varying loss function.
At time $\uptau_k$, an intuitive strategy for the agent is to   pick an action $\hbtheta_k\in\bTheta$ such that the resulting position of the agent  $ \ilparenthesis{ y_{k+1}^{\text{E}} , \,\, y_{k+1}^{\text{N}}} ^\transpose$ computed as (\ref{eq:nextstate}) can be as close to the target's position $ \ilparenthesis{x_{k+1}^{\text{E}}  ,\,\, x_{k+1}^{\text{N}} } ^\transpose$ as possible.  Namely, at time index $k$, we want to find a value of $\btheta\in\real^2$ such that
\begin{align}\label{eq:loss1Case1}
\loss_k\ilparenthesis{\btheta}& \equiv   \frac{1}{2}  \bracket{  \ilparenthesis{   x_{k+1}^{\text{E}} - y_{k+1} ^{\text{E}}  }^2 + \ilparenthesis{   x_{k+1}^{\text{N}} - y_{k+1} ^{\text{N}}  }^2  } \nonumber\\
&= \frac{1}{2} \norm{  \parenthesis{	y_k^{\east} ,\,\,\, y_k^{\north}}^\transpose + (\uptau_{k+1}-\uptau_k) \btheta  -  \parenthesis{x_{k+1}^{\east} ,\,\,\, x_{k+1}^{\north}}^\transpose   }^2
\end{align} is minimized, where the second equality is obtained  by plugging in  (\ref{eq:nextstate}). 
  Note that both the sampling interval $ (\uptau_{k+1}-\uptau_k) $ and the current position of the agent   $ \ilparenthesis{ y_k^{\east},  \,\,\, y_k^{\north}   }^\transpose $ are \emph{known}.  Unfortunately, it is not feasible for the agent to evaluate  the  loss function (\ref{eq:loss1Case1}) at time $\uptau_k$, as  
  the agent does  not know the  \emph{next} position  $ \ilparenthesis{ x_{k+1}^{\east},\,\,\, x_{k+1}^{\north} }^\transpose $ of the target   at time $\uptau_k $. Even at time $ \uptau_{k+1}$,  the agent can only gather noisy information about  $ \ilparenthesis{ x_{k+1}^{\east},\,\,\, x_{k+1}^{\north} }^\transpose $    through  the noisy measurement (\ref{eq:noisyMeasurement}). Nonetheless, the agent can instead use the   approximation in (\ref{eq:loss2Case1}) as a proxy for the true loss function  (\ref{eq:loss1Case1}): 
 \begin{equation}\label{eq:loss2Case1}
\hat{f}_k\ilparenthesis{\btheta}  = \frac{1}{2} \norm{  \parenthesis{	y_k^{\east} ,\,\,\, y_k^{\north}}^\transpose+ (\uptau_{k+1}-\uptau_k) \btheta  -  \parenthesis{	\hat{x}_{\given{k+1}{k}}^{\east} ,\,\,\, \hat{x}_{\given{k+1}{k}}^{\north}}^\transpose   }^2, 
\end{equation}
where $	\hat{x}_{\given{k+1}{k}}^{\east} $ and $\hat{x}_{\given{k+1}{k}}^{\north}$ are the a priori prediction for the first two components of $\bx_{\given{k+1}{k}}$ mentioned  in step \ref{item:KF1}.

The KF scheme is a natural tool to  find the prediction   $ \ilparenthesis{\hat{x}_{\given{k+1}{k}}^{\text{E}}  ,\,\, \hat{x}_{\given{k+1}{k}}^{\text{N}} } ^\transpose$  for  $ \ilparenthesis{x_{k+1}^{\text{E}}  ,\,\, x_{k+1}^{\text{N}} } ^\transpose$. To implement the KF, we need to  impose further assumptions.  
  If the rotational dynamics are  also negligible for the target as it is the case for the agent  discussed in Remark~\ref{rem:rotation}, then the target state evolution  should take the form: 
	\begin{equation}\label{eq:transitiontrue}
	\bx_{k+1}  = \begin{pmatrix}
	1 & 0 &  (\uptau_{k+1} - \uptau_k ) & 0 \\
	0 & 1 & 0 &   (\uptau_{k+1} - \uptau_k ) \\
	0 & 0 & \iota_{k_1} & \iota_{k_2} \\
	0 & 0 & \iota_{k_3} & \iota_{k_4} \\
	\end{pmatrix} \bx_k , \text{ such that }\norm{\bx_{k+1}-\bx_k} \le {\velocity}_{\bx}^{\max}\ilparenthesis{\uptau_{k+1}-\uptau_k},
	\end{equation}
		for some parameters $\iota_{k_1}$, $\iota_{k_2}$, $\iota_{k_3}$, $\iota_{k_4}$ that manifest the change of the speed from $ \ilparenthesis{  \dot{x}_k^{\text{E}} ,\, \,\, \dot{x}_k^{\text{N}}  }^\transpose\in\real^2 $ to $ \ilparenthesis{  \dot{x}_{k+1}^{\text{E}},\, \,\, \dot{x}_{k+1}^{\text{N}}  }^\transpose $. Still, realistically, the agent cannot  access the exact evolution form (\ref{eq:transitiontrue}) of the target. Hence, the following  discrete-time representation  of  linear   dynamics for  the target  is \emph{assumed} by the agent:   
		\begin{equation}\label{eq:timevaryingmodel}
		\bx_{k+1} = \bPhi_k\bx_k + \bw_k,
		\end{equation}where  the state transition matrix is 
	\begin{equation}\label{eq:transition1}
	\bPhi_k = \begin{pmatrix}
	1 & 0 &  (\uptau_{k+1} - \uptau_k ) & 0 \\
	0 & 1 & 0 &   (\uptau_{k+1} - \uptau_k ) \\
	0 & 0 & 1& 0 \\
	0 & 0 & 0& 1   \\
	\end{pmatrix} \in\real^{4\times 4},
	\end{equation}   
	and $\bw_k \stackrel{\mathrm{i.i.d.}}{\sim}\mathrm{Normal}\ilparenthesis{\zero,\bQ_k }  $ with the following  covariance matrix   per  \cite{peterson2014simulation} 
	\begin{equation}\label{eq:Qmatrix}
	\bQ_k = \begin{pmatrix}
	\frac{\ilparenthesis{ \uptau_{k+1} - \uptau_k  }^3}{3} & 0 & \frac{\ilparenthesis{ \uptau_{k+1} - \uptau_k  }^2}{2} & 0 \\
	0 & \frac{\ilparenthesis{ \uptau_{k+1} - \uptau_k  }^3}{3} & 0 & \frac{\ilparenthesis{ \uptau_{k+1} - \uptau_k  }^2}{2}\\
	\frac{\ilparenthesis{ \uptau_{k+1} - \uptau_k  }^2}{2} & 0 & \ilparenthesis{ \uptau_{k+1} - \uptau_k  }& 0 \\
	0 & \frac{\ilparenthesis{ \uptau_{k+1} - \uptau_k  }^2}{2} & 0 & \ilparenthesis{ \uptau_{k+1} - \uptau_k  }
	\end{pmatrix}. 
	\end{equation}
\begin{rem}\label{rem:anticipatedtransition}
	The anticipated form  of (\ref{eq:transition1}) is due to the physical law of inertia, i.e.,   every  vehicle tends to keep its current speed (including both the  direction and  the magnitude). Fortunately,   $\hat{\bx}_{\given{k+1}{k}}$  generated from KF     provides  a reasonable a priori estimation, even if (\ref{eq:timevaryingmodel}) misspecifies (\ref{eq:transitiontrue}). 
\end{rem}

With     assumed form    (\ref{eq:timevaryingmodel}) of the target's motion, we may implement the   KF-based estimation summarized in Algorithm~\ref{algo:EKF}. 
	\begin{algorithm}[!htbp]
	\caption{Using  KF to Predict $ \bx_{k+1} $ at Time $\uptau_k$} 
	\begin{algorithmic}[1]  
		\renewcommand{\algorithmicrequire}{\textbf{Input:}}
		\renewcommand{\algorithmicensure}{\textbf{Output:}}
		\Require  $ \hat{\bx}_{k}\in\real^4 $,  $\uptau_{k+1} - \uptau_k\in\real$,  $\bP_k \in\real^{4\times 4}$, $\bPhi_k\in\real^{4\times 4}$   as   in (\ref{eq:transition1}),   $\bQ_k\in\real^{4\times 4}$ as in  (\ref{eq:Qmatrix}),    $\bR_{k+1}\in\real^{2\times 2} $  as in  (\ref{eq:Rmatrix}), $\KFmeasurementmatrix_{k+1} =  \begin{pmatrix}
		1 & 0 & 0 & 0\\ 0  & 1 & 0 & 0
		\end{pmatrix}\in\real^{2\times 4}$. 
		\State   At time    $\uptau_k $, predict (a priori) state estimate  as  $ \hat{\bx}_{\given{k+1}{k}} = \bPhi_k \hat{\bx}_k $.   \label{line:apriori1}
		\State  At time    $\uptau_k $, also predict (a priori) covariance estimate as   $ \bP_{\given{k+1}{k} }  = \bPhi_ k \bP _k  \bPhi_k^\transpose + \bQ_k $.  \label{line:apriori2}
		\remove{	\State At time $\uptau_{k+1}$, compute the residual/innovation $\tilde{\be}_{k+1} = \bz_{k+1} - \bH_{k+1} \hat{\bx}_{\given{k+1}{k}}$.
		\State At time $\uptau_{k+1}$, compute the   covariance for the residual/innovation as $ \bS _{k+1} = \bH_{k+1} \bP_{\given{k+1}{k}} \bH _{k+1}^\transpose + \bR_{k+1} $.   } 
		\State  At time $\uptau_{k+1}$, compute the   Kalman gain as $ \bK_{k+1} = \bP_{\given{k+1}{k}} \KFmeasurementmatrix_{k+1}^\transpose   \parenthesis{  \KFmeasurementmatrix_{k+1}  \bP_{\given{k+1}{k} }   \KFmeasurementmatrix_{k+1}^\transpose + \bR_{k+1}}        ^{-1} $. \label{line:apost1} 
		\State At time $\uptau_{k+1}$, update  (a posteriori) state estimate $ \hat{\bx}_{k+1} = \hat{\bx}_{\given{k+1}{k}} + \bK _{k+1}  \ilparenthesis{\bz \parenthesis{  \bx_{k+1},  \by_{k+1}  }  - \KFmeasurementmatrix_{k+1} \hat{\bx}_{\given{k+1}{k}}} $, where the binary   function   $\bz\ilparenthesis{\cdot,\cdot}$ is defined in (\ref{eq:noisyMeasurement}). 
		\State At time $\uptau_{k+1}$, update (a posteriori) covariance estimate $ \bP _{k+1} = \ilparenthesis{\bI_ 4  - \bK_{k+1}  \KFmeasurementmatrix _{k+1}} \bP _{\given{k+1}{k}} $.  \label{line:apostend} 
	\end{algorithmic}
	\label{algo:EKF}
\end{algorithm} 
With the a priori estimation  $ \hat{\bx}_{\given{k+1}{k}} $ generated from Algorithm~\ref{algo:EKF},  we    have a way to evaluate the proxy loss function (\ref{eq:loss2Case1}).  We now  discuss using  the SA scheme (\ref{eq:basicSA})  to generate iterative estimate for the minimizer $\bvartheta_k$ of   the  time-varying  loss function  (\ref{eq:loss1Case1}). An unbiased  estimator for 
\begin{equation}\label{eq:trackingG1}
\bg_k\ilparenthesis{\btheta} \equiv \frac{\partial \loss_k\ilparenthesis{\btheta}}{\partial\btheta} =   \ilparenthesis{\uptau_{k+1}-\uptau_k}^2 \btheta+  \ilparenthesis{\uptau_{k+1}-\uptau_k} \begin{pmatrix}
	y_k^{\east} - x_{k+1}^{\east} \\
	y_k^{\north} - x_{k+1}^{\north}
	\end{pmatrix}    
	\end{equation} is
	\begin{equation}\label{eq:trackingG2}
	\hbg_k \ilparenthesis{\btheta} = \ilparenthesis{\uptau_{k+1}-\uptau_k}^2 \btheta+  \ilparenthesis{\uptau_{k+1}-\uptau_k} \begin{pmatrix}
	y_k^{\east} - \hat{x}_{\given{k+1}{k}}^{\east} \\
	y_k^{\north} - \hat{x}_{\given{k+1}{k}}^{\north}
	\end{pmatrix}  .
	\end{equation}
	\begin{rem}
		If the target is moving according to a \emph{prescribed} trajectory, then there is no randomness in $\ilparenthesis{x_{k+1}^{\east}, \,\, x_{k+1}^{\north}}^\transpose$.  If otherwise, $\bg_k\ilparenthesis{\cdot}$ here involves the randomness in   $\ilparenthesis{x_{k+1}^{\east}, \,\, x_{k+1}^{\north}}^\transpose$. 
	\end{rem}
	  With (\ref{eq:trackingG2}), we can update the action of the agent using the scheme (\ref{eq:truncatedSA1})  and the corresponding constraint set  
	  $\bTheta$   is defined in (\ref{eq:nextstate}). 
Furthermore, $ {\partial   {\bg}_k\ilparenthesis{\uptheta}}/{\partial\btheta} = (\uptau_{k+1}-\uptau_k)^2 \bI_2 $  for all $\btheta$ when the sampling interval $\parenthesis{\uptau_{k+1}-\uptau_k}$ is positive, and discrete sampling applies to most  modern sensors.

  We   finish formulating  the loss function for the case where there are  only one agent and one target.  Let us reiterate that the underlying  loss function (\ref{eq:loss1Case1}) is time-varying, as it evolves as the agent and target move with time. The proxy of the underlying loss function (\ref{eq:loss2Case1})  to which   the agent can access is    stochastic as there is random noise in the   measurement  (\ref{eq:noisyMeasurement}), and is information-based given the underlying KF-based prediction  $\hat{\bx} _{\given{k+1}{k}}$ generated from  Algorithm~\ref{algo:EKF}. We summarize the details in implementing step \ref{item:KF1}   in Algorithm~\ref{algo:base}.   
	\begin{algorithm}[!htbp]
	\caption{The Procedure to Generate the $\hbtheta_k$ Sequence For Single-Agent   Single-Target Setting} 
	\begin{algorithmic}[1]  
		\renewcommand{\algorithmicrequire}{\textbf{Input:}}
		\renewcommand{\algorithmicensure}{\textbf{Output:}}
		\Require $ { {\velocity}_{\by}^{\max}} \in\real$, $\hbtheta_0\in\bTheta $,  $ \hat{\bx}_{0}\in\real^4 $,   $\bP_0\in\real^{4\times 4}$,  $\ilparenthesis{\uptau_{k+1}-\uptau_k}\in\real$, $ \bPhi_k \in\real^{4\times 4}   $ as in (\ref{eq:transition1}),  $ \bQ_k  \in\real^{4\times 4}$ as in (\ref{eq:Qmatrix}) for all $k\ge 0$, and $\KFmeasurementmatrix_{k} = \begin{pmatrix}
		1&0&0&0\\
		0&1&0&0
		\end{pmatrix} \in\real^{2\times 4}$,  $\bR_k\in\real^{2\times 2}$ as in (\ref{eq:Rmatrix}) for all $k\ge 1$. 
		\For{$0\le k\le K$}  
		\State \textbf{a priori estimation} $\hat{\bx}_{\given{k+1}{k}} = \bPhi_k \hat{\bx}_k$ and $ \bP_{\given{k+1}{k}}  = \bPhi_k \bP_k \bPhi_k^\transpose + \bQ_k $. 
		\State \textbf{update} $\hat{\btheta}_{k+1} = \Proj_{\bTheta} \ilbracket{\hat{\btheta}_k - \gain_k \hat{\bg}_k\ilparenthesis{\hbtheta_k}}$, where $\hat{\bg}_k\ilparenthesis{\cdot}$ is  given in (\ref{eq:trackingG2}) and $\bTheta\subset\real^2$ given in (\ref{eq:nextstate}).
		\Ensure $\hat{\btheta}_{k+1}$ 
			\State \textbf{update} agent's position $\ilparenthesis{ y_{k+1}^{\east} ,\,\,\,  y_{k+1}^{\north} }^\transpose= \ilparenthesis{ y_{k}^{\east} ,\,\,\,  y_{k}^{\north} }^\transpose+ (\uptau_{k+1} -\uptau_k) \hbtheta_k$ as in (\ref{eq:nextstate}). 
		\State  \textbf{compute}  the Kalman gain   $ \bK_{k+1} = \bP_{\given{k+1}{k}} \KFmeasurementmatrix _{k+1}^\transpose  \ilparenthesis{    \KFmeasurementmatrix_{k+1} \bP_{\given{k+1}{k}} \KFmeasurementmatrix _{k+1}^\transpose + \bR_{k+1}    } ^{-1} $.
		\State   \textbf{a posterior estimation}  $ \hat{\bx}_{k+1} = \hat{\bx}_{\given{k+1}{k}} + \bK _{k+1}  \parenthesis{\bz \ilparenthesis{\bx_{k+1}, \by_{k+1}}   - \KFmeasurementmatrix_{k+1} \hat{\bx}_{\given{k+1}{k}}} $ for the binary function  $\bz\ilparenthesis{\cdot,\cdot}$ as    in  (\ref{eq:noisyMeasurement}), and  $ \bP _{k+1} = \ilparenthesis{\bI_4- \bK_{k+1} \KFmeasurementmatrix_{k+1}} \bP _{\given{k+1}{k}} $. 
		\EndFor
	\end{algorithmic}
	\label{algo:base}
\end{algorithm}

\subsection{Relation With Error Bound Result in Chapter~\ref{chap:FiniteErrorBound}}\label{eq:relatingbackchap3}
Even though the tracking capability results  in Chapter~\ref{chap:FiniteErrorBound} are  derived for iterates $\hbtheta_k$ generated from the unconstrained SA algorithm (\ref{eq:basicSA}), they  can be readily extended to the constrained SA  algorithm (\ref{eq:truncatedSA1})    using the non-expansivity of the projection   $\Proj_{\bTheta}\ilparenthesis{\cdot}$ onto   the feasible region $\bTheta$,     as long as the optimizer $\bvartheta_k\in\bTheta$  for all $k$. 

In the  single-agent single-target setup,   the root of true gradient  function (\ref{eq:trackingG1})   gives  the minimizer of the true loss function (\ref{eq:loss1Case1}) $ \bvartheta_k = \ilparenthesis{  x_{k+1}^{\east}  - y _k ^{\east} , \,\,\,  x_{k+1}^{\north} - y_k^{\north}  }^\transpose     \in\real^2$ when $\bvartheta_k$ falls within the constraint region $\bTheta$. With the speed limit    ${\velocity}_{\bx}^{\max}$  imposed on   the target  in (\ref{eq:transitiontrue}) and the speed  limit ${\velocity}_{\by}^{\max}$ imposed on the  agent in (\ref{eq:nextstate}), we know that
\begin{align}\label{eq:driftBoundKF}
\norm{\bvartheta_{k+1} - \bvartheta_k}  &= \norm{  \parenthesis{	x_{k+2}^{\east} - y_{k+1}^{\east},\,\,\,
		x_{k+2}^{\north} - y_{k+1} ^{\north}}^\transpose  -  \parenthesis{	x_{k+1}^{\east}- y_k^{\east},  \,\,\,\,
		x_{k+1}^{\north}  - y_k^{\north}}^\transpose   } \nonumber\\ &\le \norm{  \parenthesis{x_{k+2}^{\east} - x_{k+1}^{\east}, \,\,\,\,
		x_{k+2}^{\north} - x_{k+1} ^{\north}}^\transpose   } +\norm{    \parenthesis{	y_{k+1}^{\east} - y_k^{\east} ,\,\,\,
		y_{k+1}^{\north}  - y_k^{\north}}^\transpose  } \nonumber\\
&\le{\velocity}_{\bx}^{\max}(\uptau_{k+2}-\uptau_{k+1}) + {\velocity}_{\by}^{\max} (\uptau_{k+1} - \uptau_k)  .
\end{align} Given above,        the assumption  A.\ref{assume:BoundedVariation}   is met with $\driftBound_k =  {\velocity}_{\bx}^{\max}(\uptau_{k+2}-\uptau_{k+1}) + {\velocity}_{\by}^{\max} (\uptau_{k+1} - \uptau_k)  $. 
Also, the assumptions A.\ref{assume:StronglyConvex} and A.\ref{assume:Lsmooth}   are satisfied with   $\convexPara_k=\LipsPara_k=\ilparenthesis{\uptau_{k+1}-\uptau_k}^2$, given the gradient function as in     (\ref{eq:trackingG1}) and the discussion in Subsection~\ref{subsect:ratio}.

Last, we need to consider whether the assumption A.\ref{assume:ErrorWithBoundedSecondMoment} is met. In this single-target single-agent case, the error term defined in (\ref{eq:gGeneral}) becomes  
\begin{align}\label{eq:noiseKF}
\be_k\ilparenthesis{\btheta} &= \hbg_k\ilparenthesis{\btheta} - \bg_k\ilparenthesis{\btheta} =\ilparenthesis{\uptau_{k+1}-\uptau_k} \begin{pmatrix}
x_{k+1}^{\east} - \hat{x}_{\given{k+1}{k}}^{\east}\\
x_{k+1}^{\north} - \hat{x}_{\given{k+1}{k}}^{\north}
\end{pmatrix}, 
\end{align}
where  $\bg_k\ilparenthesis{\cdot}$ is as  (\ref{eq:trackingG1})  and $\hbg_k\ilparenthesis{\cdot}$ is as  (\ref{eq:trackingG2}). \remove{Unfortunately,  the linear dynamical system  (\ref{eq:transitiontrue}) and linear observation (\ref{eq:noisyMeasurement}) only ensure  that the a posterior estimate $ \ilparenthesis{  \hat{x}_{k+1} ^{\east} ,\,\,\, \hat{x}_{k+1} ^{\north} } ^\transpose$ is an unbiased estimator for $\ilparenthesis{x_{k+1}^{\east},\,\,\, x_{k+1}^{\north} }^\transpose$.  Nothing definite can be concluded for the a priori prediction  $ \ilparenthesis{ \hat{x}_{\given{k+1}{k}} ^{\east},\,\,\, \hat{x}_{\given{k+1}{k}}^{\north}  } ^\transpose$ as an estimate for    $\ilparenthesis{x_{k+1}^{\east},\,\,\, x_{k+1}^{\north} }^\transpose$, 
not to mention that the model (\ref{eq:timevaryingmodel}) may not be the accurate representation for (\ref{eq:transitiontrue}), i.e., the transition matrix $\bPhi_k$ may be misspecified. Anyhow, }We will use the upper-left $2$-by-$2$ submatrix of  $ \bP_{\given{k+1}{k}} $, which gives the covariance between $ \hat{x}_{\given{k+1}{k}}^{\east} $ and $ \hat{x}_{\given{k+1}{k}}^{\north} $, as a proxy of the covariance matrix of $\be_k$ in (\ref{eq:noiseKF}). That is, we \emph{assume} that the assumption A.\ref{assume:ErrorWithBoundedSecondMoment} is met with $ \noiseBound_k  $ \emph{approximately} equaling the square root of the sum of the first two diagonal entries of $ \bP_{\given{k+1}{k}} $.

\subsection{Monte Carlo Simulation}
\label{subsect:MCbase}

We consider a time-frame   $0\le k\le 999$. 
Let the sampling frequency   $( \uptau_{k+1} - \uptau_k) = 0.3$ seconds for   $0\le k \le 998$, which is the typical sample interval of the existing sensor. 
 Assume that the target has a speed limit of $\velocity_{\bx}^{\max} =15$ meters per second,  which is  the average speed of the  middle-class submarines.   Assume that the agent has a speed limit of   
 $   \velocity_{\by}^{\max}= 30 $ meters per second,
 which is the average speed of the top-tier submarines.   Assume that the target is moving according to the following  transition law:
   \begin{align}\label{eq:transition2}
  &\quad\,\, \,\,   \bx_{k+1} = \begin{pmatrix}
  1 & 0 &  0.3 & 0 \\
  0 & 1 & 0 &   0.3 \\
  0&0&1&0\\
  0&0&0&1
  \end{pmatrix} \bx_k + \bw_k, \,\,\text{ for } 0\le k \le 499 \text{ and } 501\le k \le 999,\nonumber\\
  &\text{ and } \bx_{k+1}  = \begin{pmatrix}
  1 & 0 &  0.3 & 0 \\
  0 & 1 & 0 &   0.3 \\
  0&0&-1&0\\
  0&0&0&1
  \end{pmatrix}\bx_k+ \bw_k,  \,\, \text{ for }k = 500.
  \end{align}
 The above transition law is certainly    \emph{unknown} to the agent, and the agent will again use the anticipated transition matrix $\bPhi_k$ for the reason explained in Remark~\ref{rem:anticipatedtransition}. The matrix $\bR_k$ will be as in (\ref{eq:Rmatrix})   and the  matrix $\bQ_k$ will be as (\ref{eq:Qmatrix}) after plugging in the value of  the sampling interval, which is  $0.3$ seconds. 
 
 One remaining input for implementing Algorithm~\ref{algo:base} is $\bP_{0}$. We assume that at $k=0$,  the target's location  
 $\bx_0$ and the agent's location $\by_0$ are uniformly-random distributed  within $\bracket{-5,5}\times \bracket{-5,5}$, and assume that the target has an initial speed of  $\ilparenthesis{  \dot{x}_{0}^{\east},\,\,\, \dot{x}_0^{\north} }^\transpose$, whose  Euclidean norm equals $\velocity_{\bx}^{\max} = 15$, and a random direction uniformly sampled from $\mathrm{Uniform}\ilparenthesis{0,2\uppi}$. With such an initialization, the agent picks
 \begin{equation}\label{eq:initP}
 \bP_0 = \begin{pmatrix}
 \frac{100}{12} & 0 & 0 & 0 \\
 0 & \frac{100}{12} & 0 & 0\\
 0 & 0 &   \frac{15^2}{2}     & 0 \\
 0 & 0 & 0 &  \frac{15^2}{2}
 \end{pmatrix} = \diag (\frac{25}{3},  \,\,  \frac{25}{3}, \,\, \frac{225}{2},\,\, \frac{225}{2}  )\in\real^{4\times 4}
 \end{equation}
 as  an initial estimate for the covariance matrix  $ \E\bracket{    \ilparenthesis{  \hat{\bx}_0 - \bx_0 }\ilparenthesis{\hat{\bx}_0-\bx_0}^\transpose  } $. The first two diagonal entries in (\ref{eq:initP}) are    the variance of $\mathrm{Uniform}\bracket{-5,5}$, and the  lower-right $2$-by-$2$ submatrix is given by
 the product of the squared of $\velocity_{\bx}^{\max} $ and the $2$-by-$2$ variance matrix of the cosine and  the sine of a uniform random variable within $\bracket{0,2\uppi}$. 
 
 \subsubsection*{Relation To Results on  Error Bound }
 Subsection~\ref{eq:relatingbackchap3} mentions that the value of 
 $\LipsPara_k$ equals the value of $\convexPara_k$ for all $k$,   so we may pick $\Tobe_k=2.5$ for all $k$ as per line \ref{line:r1-1}  in Algorithm~\ref{algo:basicSA}. We then pick a gain $\gain_k$ of $1.15/\LipsPara_k$   to implement (\ref{eq:truncatedSA1}) as per  line \ref{line:r1-2} in Algorithm~\ref{algo:basicSA}. 
So    $\firstConst_k$  can be computed as in   (\ref{eq:firstConst}), and    $\secondConst_k$ can be computed as  in (\ref{eq:secondConst}). 
  
  We reiterate that the error bound results  (\ref{eq:PropagationLemma3})  and  (\ref{eq:PropagationLemma5}) are obtained after averaging the performance on all the sample  paths. For real-time tracking   in this chapter, the agent will \emph{not} have a chance to repeatedly rehearse the tracking  mission.  As a result,  the tracking error bounds is \emph{not}
 informative for \emph{one}   run, even though all the assumptions A.\ref{assume:ErrorWithBoundedSecondMoment}\textendash A.\ref{assume:BoundedVariation}   are satisfied (as discussed in Subsection~\ref{eq:relatingbackchap3}).

  Here, we use  (\ref{eq:PropagationLemma2})  ``loosely'' as the follows to compute a proxy of the error bound iteratively: 
 \begin{equation}\label{eq:loosebound}
  \norm{  \hbtheta_{k+1} - \bvartheta_k } \le \sqrt{\firstConst_k} \norm{\hbtheta_k -\bvartheta_{k}}  +  \noiseBound_k\sqrt{\secondConst_k} + \driftBound_k,
 \end{equation} 
 where $\driftBound_k$  can be computed    as   in (\ref{eq:driftBoundKF}), and 
 $\noiseBound_k$ can be  approximately    computed as  the square root of the sum of the first two diagonal entries of $ \bP_{\given{k+1}{k}} $ through implementing the recursive procedure described in Algorithm~\ref{algo:base}.

 \subsubsection*{Simulation Results Using Algorithm~\ref{algo:base} With Given Input}
 
 The positions of the target and the agent are plotted on the Cartesian coordinate   in Figure~\ref{fig:KFtrackingvisual}.  The starting/ending position of the target is denoted in the red upward/downward pointing triangle,  and the initial/ending position of the agent is denoted in the blue left/right pointing triangle.  The difference between the position of the target and the agent is plotted in  Figure~\ref{fig:KFtrackingerror}. 
\begin{figure}[!htbp]
	\centering
	\begin{subfigure}{.65\textwidth}
		\centering
		\includegraphics[width=\linewidth]{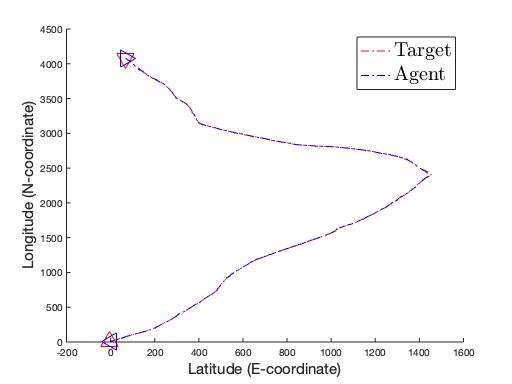}
		\caption{Trajectories of The Target and The Agent in \emph{One} Simulation Run} 	\label{fig:KFtrackingvisual}
	\end{subfigure}\\
	\begin{subfigure}{.65\textwidth}
		\centering
		\includegraphics[width=\linewidth]{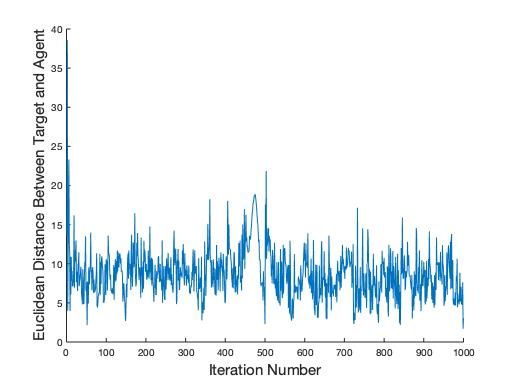}
		\caption{Euclidean Distance Between The Position  of The Target $ \ilparenthesis{x_{k}^{\east},\,\,\, x_k^{\north}}^\transpose $ and The Position of The Agent  $ \ilparenthesis{y_{k}^{\east},\,\,\, y_k^{\north}}^\transpose $.}  	\label{fig:KFtrackingerror}
	\end{subfigure}
	\caption{A Demonstration of Implementing  Algorithm~\ref{algo:base} Using the Inputs Described in This Subsection}
	\label{fig:KF}
\end{figure} 
We also include   Figure~\ref{fig:error}, but it is not very informative as the results in Chapter~\ref{chap:FiniteErrorBound} is valid after averaging the performance across all sample paths.
\begin{figure}[!htbp]
	\centering 
		\centering
		\includegraphics[width=.65\linewidth]{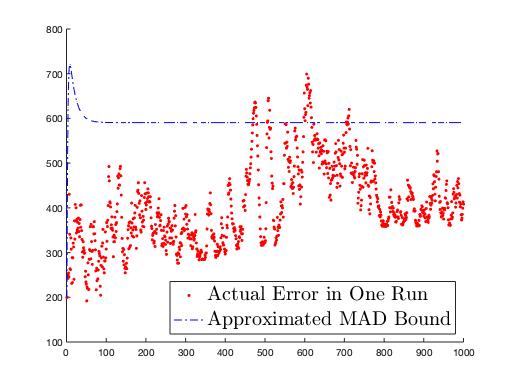}
		\caption{ Actual Error $ \norm{\hbtheta_k-\bvartheta_k} $ for One Simulation Run and   the ``Loose'' Bound Computed Per (\ref{eq:loosebound})   }    
	\label{fig:error}
\end{figure}

\section{Generality: Multi-Agent  Multi-Target Surveillance With Zero-Communication}\label{sect:general}

We  now   consider    the surveillance problem with $I$ targets and $J$ agents,  for $I, J \in \integer$ with $1<I\le J$. The $i$th target's state   and     the $j$th agent's state   at time $\uptau_k$ are  denoted as $\bx_k ^{(i)}$ and $\by_k^{(j)}$ respectively.  Furthermore, we assume \emph{no} communication between agents is  allowed. Each agent should rely on local  awareness and plays 
individually.

 In the   multi-agent multi-target setting with zero-communication, the objective of each agent is two-fold: one is to \emph{track} the  nearby target \emph{if needed}, the other is to \emph{spread} out     to  enlarge  the collective coverage of the area of interest \emph{if otherwise}.  By ``needed'' we mean that the agent   believes that it is closer to a certain target than any other agents.\remove{By ``coverage'' we mean the probability of detecting other targets successfully.  }
As before, the $j$th  target is allowed to obtain noisy observations of all the targets $ \bz \ilparenthesis{\bx_i, \by_j} $ for $i\in I$ and noisy observations of all other agents $ \bz\ilparenthesis{\by_{j'}, \by_j} $ for $j\in J\setminus\set{j}$, where the function $\bz\ilparenthesis{\cdot,\cdot}$ is defined in (\ref{eq:noisyMeasurement}). 
 \subsection{Loss Function  }\label{subsect:loss2}
  Let us explain  the loss function  addressing  the   objective of ``actively-tracking if needed and actively-spreading if otherwise'' that applies to every agent.  If all the targets are \emph{equally} 
  important, the ensemble of the agents is likely to distribute 
  \begin{equation}
  \label{eq:iStar} j^* \le  \left \lfloor{J/I}\right \rfloor  
  \end{equation} 
  agent(s) to track each target. Usually, $j^*=1$.  Algorithm~\ref{algo:trackingAssign} describes the procedure to assign ``actively-tracking'' and ``actively-spreading'' agents.
  
  At time $\uptau_k$, we   assign what we will call  the ``actively-tracking'' agents and the ``actively-spreading'' agents from this point on  using  Algorithm~\ref{algo:trackingAssign}. Specifically, the output       $T_k^{(i)}$ of   Algorithm~\ref{algo:trackingAssign} contains the indices of the agents that are expected to track target $i$ for $1\le i \le I$, and the output $S_k$ (if nonempty) contains the indices of the agents  that are expected to spread out as much as possible. 
 	\begin{algorithm}[!htbp]
 	\caption{Assigning Actively-Tracking and Actively-Spreading Agents At Time $\uptau_k$} 
 	\begin{algorithmic}[1]  
 		\renewcommand{\algorithmicrequire}{\textbf{Input:}}
 		\renewcommand{\algorithmicensure}{\textbf{Output:}}
 		\Require  the  distance matrix $\bD _{k+1}$ with the $(i,j)$-entry being  $(x_{k+1}^{(i), \east}    - y_{k+1} ^{(j),\east}  )^2 +(x_{k+1}^{(i), \north}    - y_{k+1} ^{(j),\north}  )^2 $ for $1\le i \le I$ and $1\le j\le J$, and the desired number  $j^*$ of agents to keep track of a single target satisfying (\ref{eq:iStar}). 
 		\State \textbf{initialize} the   search set  $S_k = \set{1,\cdots,J}$. 
 		\State \textbf{find} a set $T_k^{(1)} \subseteq S_k$ that contains the   indices of the columns   that have the $j^*$ smallest entries within  the first row of $\bD_{k+1}$. \Comment{If there exist equal rankings, just pick any set  such   that cardinality of  $T_k^{(1)}$ is    $j^*$.}
 		\For( $2\le i \le I $)
 		\State \textbf{update} $S_k \gets S_k \setminus T_k^{(i-1)}$.
 			\State \textbf{find} a set $T_k^{(i)} \subseteq S_k$ that contains the indices of the column that have the $j^*$ smallest entries in the $i$th row of $\bD_{k+1}$.
 		\EndFor
 		\State  \textbf{update} the search set $S_k\gets S_k \setminus 
 		T_k^{(I)}$. 
 		\Ensure the  sets $T_k^{(1)},\cdots,T_k^{(I)}$ and the set $S_k$. 
 	\end{algorithmic}
 	\label{algo:trackingAssign}
 \end{algorithm} 

If $S_k$ is nonempty, then the $j$th agent    for $j\in S_k$ is  not expected to actively track any of the targets and should  spread out as much as possible  to maximize the coverage area.  There are many ways to quantify ``spreading'' and 
 we adopt the strategy proposed by  \cite{lee2015multirobot}. The notion of Voronoi cell (a.k.a. Thiessen polygon) is used. Let $\bzeta$ be any canonical point in $\real^2$, then the Voronoi cell $V_k^{(j)}\subsetneq\real^2$ within which the $j$th agent locates is  constructed in a way such that   for any canonical point $\bzeta\in V_k^{(j)}$, the distance between $\bzeta$ and the $j$th agent's position is strictly smaller than the distance between $\bzeta$ and the position   of any other agent     at time $\uptau_k$.  The required input (which is a proper subspace of the two-dimensional Euclidean space)  to  compute the Voronoi cells $ V_k^{(j)} $ for $1\le j \le J$ is
 the convex hull of all the agent's positions $\ilparenthesis{ y_{k+1}^{(j),\east} ,\,\,\,  y_{k+1}^{(j),\north}}^\transpose$ for all $1\le j \le J$.   
 Given that  directly minimizing (\ref{eq:Voronoi1})  is difficult, 
 an intuitive   alternative is to let the ``actively-spreading'' agents   reach the center of the mass of $V_{k}^{(j)}$ for $j\in S_k$, which can be computed as: 
 \begin{equation}\label{eq:Voronoi2}
 \bc_k^{(j)} = \frac{  \int _{\bzeta \in  V_k^{(j)}   }  {   \bzeta   } \diff \bzeta   }{  \text{Area of }V_k^{(j)} }. 
 \end{equation}
 With the notion of Voronoi cell, \cite{lee2015multirobot} minimizes the following loss function
\begin{align}\label{eq:Voronoi1}
& \sum_{j\in S_k} \int _{V_k^{(j)}} \norm{  \bzeta - \ilparenthesis{ y_{k+1}^{(j), \east} ,\,\,\,  y_{k+1}^{(j),\north}}^\transpose  }^2 \diff \bzeta  \nonumber\\
&\quad = \sum_{j\in S_k} \int _{V_k^{(j)}} \norm{   \ilparenthesis{ y_{k}^{(j),\east} ,\,\,\,  y_{k}^{(j),\north}}^\transpose + (\uptau_{k+1}-\uptau_k)\btheta^{(j)} - \bzeta  }^2 \diff \bzeta 
\end{align}
 w.r.t.    the actions   $  \btheta^{(j)} \in\bTheta \subsetneq\real^2  $ for $j\in S_k$ and   for $\bTheta$ defined as in (\ref{eq:nextstate}).
  In (\ref{eq:Voronoi1}), $\bzeta $  is any canonical point in $\real^2$, and $V_k^{(j)}\subsetneq\real^2$ denotes the  Voronoi cell   \cite[Sect. 8.11]{burrough2015principles} within which the $j$th agent's locates. 
   
  To achieve the goal of ``actively-tracking if needed and actively-spreading if otherwise,'' 
   every agent   strives to minimize the  following   loss function 
   \begin{align}\label{eq:lossMultiAgent}
   \loss _k \ilparenthesis{\btheta}  &=\frac{1}{2}\sum_{i=1}^I \sum_{j \in T_k^{(i)}} \parenthesis{(x_{k+1}^{(i), \east}    - y_{k+1} ^{(j),\east}  )^2 +(x_{k+1}^{(i), \north}    - y_{k+1} ^{(j),\north})^2 } \nonumber\\
   &\quad + \frac{1}{2} \sum_{j\in S_k}  \norm{  
   \parenthesis{	y_{k+1}^{(j),\east}, \,\,\,  y_{k+1} ^{(j),\north}}^\transpose
   - \bc_k^{(j)} }^2\nonumber\\
&= \frac{1}{2}\sum_{i=1}^I \sum_{j \in T_k^{(i)}}
\norm{  \ilparenthesis{   y_k^{(j), \east }   - x_{k+1} ^{(i),\east}  ,\,\,\,  y_k^{(j), \north  }   - x_{k+1} ^{(i),\north}     } ^\transpose  + (\uptau_{k+1}-\uptau_k) \btheta^{(j)}     }^2  \nonumber\\
&\quad + \frac{1}{2} \sum_{j\in S_k}  \norm{  
	\ilparenthesis{	y_{k}^{(j),\east}, \,\,\,  y_{k} ^{(j),\north}}^\transpose + \ilparenthesis{\uptau_{k+1}-\uptau_k}\btheta^{(j)}
	- \bc_k^{(j)} }^2
   \end{align} 
   w.r.t.  $\btheta\in\bTheta^J\subsetneq \real^{2J}$, where  $\btheta$  is the concatenation of  $\btheta ^{(j)}\in\bTheta\subsetneq \real^2$ for all $1\le j\le J$.

Nonetheless, under the zero-communication setting, there exists no commander in chief who can      dispatch the corresponding actions $\btheta_k \in\real^{2J}$ to all the agents using the information from 
  $\loss_k\ilparenthesis{\cdot}$. Consequently, the $j$th agent   only gets to  update its   action $\btheta_k^{(j)}\in\bTheta$ by minimizing the following loss function
   \begin{align}\label{eq:loss1Case2}
  &  \loss_k^{(j)} \ilparenthesis{\btheta^{(j)} }\nonumber\\
   &\quad =\frac{1}{2} \sum_{i=1}^{I} \set{
   	\indicator_{\ilset{j\in T_k^{(i)}}}  \times    \ilbracket{(x_{k+1}^{(i), \east}    - y_{k+1} ^{(j),\east}  )^2 +(x_{k+1}^{(i), \north}    - y_{k+1} ^{(j),\north})^2 }} \nonumber\\
   &\quad\quad +  \frac{1}{2} \,  \indicator_{\set{j\in S_k}}   \times  \norm{  
   	\ilparenthesis{	y_{k+1}^{(j),\east}, \,\,\,  y_{k+1} ^{(j),\north}}^\transpose
   	- \bc_k^{(j)} }^2\nonumber\\
   &\quad =\begin{dcases}
   \frac{1}{2} 
   \norm{  \ilparenthesis{   y_k^{(j), \east }   - x_{k+1} ^{(i),\east}  ,\,\,\,  y_k^{(j), \north  }   - x_{k+1} ^{(i),\north}     } ^\transpose  + (\uptau_{k+1}-\uptau_k) \btheta^{(j)}     }^2 , \\
   \quad\quad   \quad\quad      \text{ if } j \in T_k^{(i)} \text{ for \emph{some} }1\le i\le I,\\
   \frac{1}{2}    \norm{  
   	\ilparenthesis{	y_{k}^{(j),\east}, \,\,\,  y_{k} ^{(j),\north}}^\transpose - \bc_k^{(j)} + \ilparenthesis{\uptau_{k+1}-\uptau_k}\btheta^{(j)}
   	 }^2 ,  \text{ if }j\in S_k, 
   \end{dcases}      
   \end{align}
   w.r.t. $\btheta^{(j)}\in\real^2$.  According to Algorithm\ref{algo:trackingAssign}, $T_k^{(1)}, \cdots, T_k^{(I)}, S_k$ are mutually exclusive.

   Of course, at time $\uptau_k$, the agent $j$ does \emph{not} have $\bx_{k+1}^{(i)}$ for $1\le i\le I$ and $\by 
    _{k+1}^{(j')}$ for $j'\neq j$ to determine $T_k^{(i)}$ and $S_k$  for $1\le i \le  I$ using Algorithm~\ref{algo:trackingAssign} and to compute the Voronoi cells $V_k^{(j)}$ and the centers $ \bc_k^{(j)}  $ for $1\le j \le  J$.   Similar to the rationale  behind substituting (\ref{eq:loss1Case1}) for  (\ref{eq:loss2Case1}), 
   the agent $j$ can   use (\ref{eq:loss2Case2})
   as a proxy of (\ref{eq:loss1Case2}):
   \begin{equation}
   \label{eq:loss2Case2}
   \hat{\loss}_k^{(j)}  (\btheta^{(j)})  = \begin{dcases}
   \frac{1}{2} 
   \norm{  \ilparenthesis{   y_k^{(j), \east }   - \hat{x}_{\given{k+1}{k}} ^{(i),\east,(j)}  ,\,\,\,  y_k^{(j), \north  }   - \hat{x}_{\given{k+1}{k}} ^{(i),\north,(j)}     } ^\transpose  + (\uptau_{k+1}-\uptau_k) \btheta^{(j)}     }^2 ,   \text{ if } j \in \hat{T}_k^{(i,j)},\\
   \frac{1}{2}    \norm{  
   	\ilparenthesis{	y_{k}^{(j),\east}, \,\,\,  y_{k} ^{(j),\north}}^\transpose + \ilparenthesis{\uptau_{k+1}-\uptau_k}\btheta^{(j)}
   	- \hat{\bc}_k^{(j,j)} }^2 ,  \text{ if }j\in \hat{S}_k^{(j)}.
   \end{dcases}      
   \end{equation}
   where all the relevant computations      arising in (\ref{eq:loss1Case2}) are executed using   the current state  of the $j$th agent   $\by_k$, and     the a priori approximation $\hat{\bx}_{\given{k+1}{k}}^{(i,j)}$ for $1\le i\le I$ and $\hat{\by}_{\given{k+1}{k}}^{(j',j)} $ for $j'\neq j$ based on the information available to agent $j$, including (1)  finding   $\hat{T}_k^{(i,j)}$ and $\hat{S}_k^{(j)}$ through implementing Algorithm~\ref{algo:trackingAssign} and (2)  generating Voronoi cells $\hat{V}_{k}^{(j,j)}$ using the built-in \textsc{MatLab} function \texttt{voronoi($\cdot$)} and  computing 
   the centers $\hat{\bc}_k^{(j,j)}$ using \texttt{polygem($\cdot$)}.  	$\hat{T}_k^{(i,j)}$,  	 $\hat{S}_k^{(j)}$,   $\hat{V}_{k}^{(j,j)}$,   and $\hat{\bc}_k^{(j,j)}$ 
   in (\ref{eq:loss2Case2})  represent the estimation of       	 ${T}_k^{(i)}$, $S_k$, $V_k^{(j)}$,  and   $\bc_k ^{(j)}$       appearing in  (\ref{eq:loss1Case2})
   based on the estimation obtained by the $j$th agent.

    A ready estimator, which  may be \emph{biased} due to the potential  inconsistency between $T_k^{(i)}$ ($\hat{T}_k^{(i,j)}$) and $S_k$ ($\hat{S}_k^{(j)}$),  for 
   \begin{align}\label{eq:trackingG3}
&   \bg_k^{(j)}\ilparenthesis{\btheta^{(j)}} \nonumber\\
&\quad   \equiv \frac{\partial \loss_k^{(j)}\ilparenthesis{\btheta^{(j)}}}{\partial\btheta^{(j)}} \nonumber\\
   &\quad =\begin{dcases}
  \ilparenthesis{\uptau_{k+1}-\uptau_k}^2 \btheta^{(j)}+  \ilparenthesis{\uptau_{k+1}-\uptau_k} \begin{pmatrix}
  y_k^{(j), \east} - x_{k+1}^{(i),\east} \\
  y_k^{(j),\north} - x_{k+1}^{(i),\north}
  \end{pmatrix}  ,    \text{ if } j \in T_k^{(i)} \text{ for some }1\le i \le I,\\
 \ilparenthesis{\uptau_{k+1}-\uptau_k}^2 \btheta^{(j)}+  \ilparenthesis{\uptau_{k+1}-\uptau_k} \begin{pmatrix}
 	y_k^{(j), \east}  \\
 	y_k^{(j), \north}  
 	\end{pmatrix} - \ilparenthesis{\uptau_{k+1}-\uptau_k}\bc_k^{(j)} ,   \text{ if }j\in S_k.
   \end{dcases}   
   \end{align} is
   \begin{align}\label{eq:trackingG4}
&   \hbg_k^{(j)}\ilparenthesis{\btheta^{(j)}} \nonumber\\
&\quad \equiv \frac{\partial \hat{\loss}_k^{(j)}\ilparenthesis{\btheta^{(j)}}}{\partial\btheta^{(j)}} \nonumber\\
 &\quad=\begin{dcases}
 \ilparenthesis{\uptau_{k+1}-\uptau_k}^2 \btheta^{(j)}+  \ilparenthesis{\uptau_{k+1}-\uptau_k} \begin{pmatrix}
 y_k^{(j), \east} - \hat{x}_{\given{k+1}{k}}^{(i),\east,(j)} \\
 y_k^{(j),\north} - \hat{x}_{\given{k+1}{k}}^{(i),\north,(j)}
 \end{pmatrix}  ,   \\
 \quad\quad \text{ if } j \in \hat{T}_k^{(i,j)} \text{ for some }1\le i \le  I,\\
 \ilparenthesis{\uptau_{k+1}-\uptau_k}^2 \btheta^{(j)}+  \ilparenthesis{\uptau_{k+1}-\uptau_k} \begin{pmatrix}
 y_k^{(j), \east}  \\
 y_k^{(j), \north}  
 \end{pmatrix} - \ilparenthesis{\uptau_{k+1}-\uptau_k}\hat{\bc}_k^{(j,j)} ,   \text{ if }j\in \hat{S}_k^{(j)}.
 \end{dcases}   
   \end{align} 
   A natural strategy   to decide    the action $\btheta^{(j)}$ of the agent $j$ is the truncated SA algorithm  (\ref{eq:truncatedSA1})  with   
   $\bTheta$   as    in (\ref{eq:nextstate}). 
   Furthermore, $ {\partial   {\bg}_k^{(j)}\ilparenthesis{\uptheta}}/{\partial\btheta} = (\uptau_{k+1}-\uptau_k)^2 \bI_2 $  for all $\btheta$ when the sampling interval $\uptau_{k+1}-\uptau_k$ is positive. Hence, all the discussion  in Subsection~\ref{eq:relatingbackchap3}   regarding  the loss function $\loss_k\ilparenthesis{\cdot}$ in (\ref{eq:loss1Case1}) is applicable for the loss function $\loss_k^{(j)}\ilparenthesis{\cdot}$ in (\ref{eq:loss1Case2}) for all $1\le j\le J$.

   We   finish stating the loss function for the general case where there are multiple agents and multiple targets. Again the underlying loss function   (\ref{eq:loss1Case2}) for agent $j$ is time-varying, and it only gets access to the noisy evaluation     $\bz\ilparenthesis{\bx_i,\by_j}$ for $1\le i \le I$ and $\bz\ilparenthesis{\by_{j'},\by_j}$ for $j'\neq j$.

  	\begin{algorithm}[!htbp]
  	\caption{The Procedure to Generate $\hbtheta_k^{(j)}$ Sequence For $j$th Agent  In Multi-Agent     Multi-Target Setting} 
  	\begin{algorithmic}[1]  
  		\renewcommand{\algorithmicrequire}{\textbf{Input:}}
  		\renewcommand{\algorithmicensure}{\textbf{Output:}}
  		\Require $ { {\velocity}_{\by}^{\max}} \in\real$, $\hbtheta_0^{(j)}\in\bTheta $,  $ \ilparenthesis{\uptau_{k+1}-\uptau_k}\in\real $ for all $k\ge 0 $, $ \hat{\bx}_{0}^{(i,j)}\in\real^4 $ and  $\bP_0^{(i,j)}\in\real^{4\times 4}$ for all $1\le i \le I$,  $ \hat{\by}_0^{(j',j)} \in\real^4 $ and $ \tilde{\bP}_0^{(j',j)} $ for all $j'\neq j$, $ \bPhi_k\in\real^{4\times 4}$ as in as in (\ref{eq:transition1}),   $ \bQ_k  \in\real^{4\times 4}$ as in (\ref{eq:Qmatrix}) for all $k\ge 0$, and $\KFmeasurementmatrix_k=\begin{pmatrix}
  		1&0&0&0\\ 0 & 1 & 0 & 0
  		\end{pmatrix} \in\real^{2\times 4}$,  $\bR_k\in\real^{2\times 2}$ as in  (\ref{eq:Rmatrix}) for all $k\ge 1$. 
  		\For{$0\le k\le K$}  
  		\For{$1\le i \le I$} \label{line:priorStart}
  		\State \textbf{a priori estimation} $\hat{\bx}_{\given{k+1}{k}}^{(i,j)} = \bPhi_k \hat{\bx}_k^{(i,j)}$ and $ \bP_{\given{k+1}{k}} ^{(i,j)}  = \bPhi_k \bP_k^{(i,j)} \bPhi_k^\transpose + \bQ_k $.
  		\EndFor
  		\For{$1\le j'\le J$ and $j'\neq j'$}
  		\State \textbf{a priori estimation} $\hat{\by}_{\given{k+1}{k}}^{(j',j)} = \bPhi_k \hat{\by}_k^{(j',j)}$ and $ \tilde{\bP}_{\given{k+1}{k}} ^{(j',j)}  = \bPhi_k \tilde{\bP}_k^{(j',j)} \bPhi_k^\transpose + \bQ_k $.
  		\EndFor \label{line:priorEnd}
  		\State   \textbf{generate} $\hat{T}_{k}^{(i,j)}$ and $\hat{S}_{k}^{(j)}$ (via Algorithm~\ref{algo:trackingAssign}) and \textbf{compute} $\hat{V}_k^{(j,j)}$ and $\bc_k^{(j,j)}$ using $\hat{\bx}_{\given{k+1}{k}}^{(i,j)}$ for $1\le i\le I$ and $\hat{\by}_{\given{k+1}{k}}^{(j',j)} $ for $j'\neq j$ as input.  \label{line:pickStart}
  		\State \textbf{update} $\hat{\btheta}_{k+1} ^{(j)}= \Proj_{\bTheta} \ilbracket{\hat{\btheta}_k^{(j)} - \gain_k^{(j)} \hat{\bg}_k^{(j)}\ilparenthesis{\hbtheta_k^{(j)}}}$ where $\hat{\bg}_k^{(j)}\ilparenthesis{\cdot}$ is  given in (\ref{eq:trackingG4}). \label{line:pickEnd}
  		\Ensure $\hat{\btheta}_{k+1}^{(j)}$ 
  		\State \textbf{update} the $j$th agent's position $\ilparenthesis{ y_{k+1}^{(j),\east} ,\,\,\,  y_{k+1}^{(j),\north} }^\transpose= \ilparenthesis{ y_{k}^{(j),\east} ,\,\,\,  y_{k}^{(j),\north} }^\transpose+ (\uptau_{k+1} -\uptau_k) \hbtheta_k^{(j)}$.  \label{line:positionUpdate}
  	\For{$1\le i \le I$} \label{line:postStart}
  	\State \textbf{compute} the Kalman gain $ \bK_{k+1}^{(i,j)} = \bP_{\given{k+1}{k}}^{(i,j)} \KFmeasurementmatrix _{k+1}^\transpose\ilparenthesis{ \KFmeasurementmatrix_{k+1}  \bP_{\given{k+1}{k}} ^{(i,j)}  \KFmeasurementmatrix _{k+1}^\transpose + \bR_{k+1} }^{-1} $.  
  	\State   \textbf{a posterior estimation}  $ \hat{\bx}_{k+1}^{(i,j)} = \hat{\bx}_{\given{k+1}{k}} ^{(i,j)}+ \bK _{k+1}^{(i,j)}  \ilparenthesis{  \bz \ilparenthesis{  \bx_{k+1}^{(i)}, \by_{k+1}^{(j)}  }  - \KFmeasurementmatrix_{k+1} \hat{\bx}_{\given{k+1}{k}} ^{(i,j)} } $ for the binary function $\bz\ilparenthesis{\cdot,\cdot}$ as in (\ref{eq:noisyMeasurement}),  and  $ \bP _{k+1}^{(i,j)} = \ilparenthesis{\bI_4 - \bK_{k+1}^{(i,j)} \bH _{k+1}} \bP _{\given{k+1}{k}} ^{(i,j)}$. 
  	\EndFor
  	\For{$1\le j'\le J$ and $j'\neq j$}
  		\State \textbf{compute} the Kalman gain $ \tilde{\bK}_{k+1}^{(j',j)} = \tilde{\bP}_{\given{k+1}{k}}^{(j',j)} \KFmeasurementmatrix _{k+1}^\transpose\ilparenthesis{ \KFmeasurementmatrix_{k+1}  \tilde{\bP}_{\given{k+1}{k}} ^{(j',j)}  \KFmeasurementmatrix _{k+1}^\transpose + \bR_{k+1} }^{-1} $.   
  	\State   \textbf{a posterior estimation}  $ \hat{\by}_{k+1}^{(j',j)} = \hat{\by}_{\given{k+1}{k}} ^{(j',j)}+ \tilde{\bK} _{k+1}^{(j',j)}  \ilparenthesis{     \bz\ilparenthesis{  \by_{k+1}^{(j')} , \by_{k+1} ^{(j)} } - \KFmeasurementmatrix_{k+1} \hat{ {\by}}_{\given{k+1}{k}} ^{(j',j)} } $ for the binary function $\bz\ilparenthesis{\cdot,\cdot}$ as in (\ref{eq:noisyMeasurement}),   and  $ \tilde{\bP} _{k+1}^{(j',j)} = \ilparenthesis{\bI _4- \tilde{\bK}_{k+1}^{(j',j)} \bH _{k+1}} \tilde{\bP} _{\given{k+1}{k}} ^{(j',j)}$. 
  	\EndFor \label{line:postEnd}
  		\EndFor
  	\end{algorithmic}
  	\label{algo:general}
  \end{algorithm}  Let us summarize the estimation procedure in Algorithm~\ref{algo:general}. Lines \ref{line:priorStart}\textemdash \ref{line:priorEnd}   compute the a priori  estimate for the states of all the  targets and all the other agents. 
Lines \ref{line:pickStart}     decides whether the $j$th agent  is ``actively-tracking'' (i.e., $j\in \hat{T} _k^{(i,j)}$ for some $1\le i \le  I$) or is ``actively-spreading'' (i.e., $j \in \hat{ S} _k^{(j)}$). Line  \ref{line:pickEnd} is to  pick a decision $\hbtheta_k^{(j)}$ using  the truncated SA scheme (\ref{eq:truncatedSA1}).  Then line \ref{line:positionUpdate} is to update the  $j$th agent's position according to $\hbtheta_{k}^{(j)}$ and (\ref{eq:nextstate}).  Lines \ref{line:postStart}\textemdash \ref{line:postEnd}   update the a posterior estimate for the states of  all the   targets and all the other agents.

 \subsection{Monte Carlo Simulation}\label{subsect:MCgeneral}

This subsection will use the same initialization as   Subsection~\ref{subsect:MCbase}, except that  $\velocity_{\by} ^{\max} $  becomes the same as $\velocity_{\bx}^{\max}=15$ meters per seconds. This change is made 
 in  the hope that the requirement on the agent's UUV speed in the multi-agent setting  
 with the joint effort with an ensemble of agents will not be as stringent as the requirement  in the single-agent setting. 

For graphical illustration, we use $J=4$ agents to track $I=2$ targets, and we pick $j^* $ to be $1$ per (\ref{eq:iStar}). 
 The positions of two  targets  and four agents  from $0\le k \le 999$ are plotted on the two-dimensional plane   in Figure~\ref{fig:KFtrackingmulti}.  The staring/ending positions of the  first target are denoted in the red upward/downward pointing triangles, and those of the second target are denoted in black.  The staring/ending positions of four agents  is denoted in the   left/right pointing triangles, and they are in   blue, magenta, yellow, and cyan respectively.   We can see that two agents are ``actively-tracking'' as they follow the two targets closely, and two agents are ``actively-spreading''  as they are randomly moving to somewhere in the middle of the simulation runs and end up in the positions that are  not   close to any   of the targets. This is what an ensemble of agents would look like as they are   all trying to achieve ``actively-tracking if needed and actively-spreading if otherwise.''
 \begin{figure}[!htbp]
 	\centering 
 	\centering
 	\includegraphics[width=.93\linewidth]{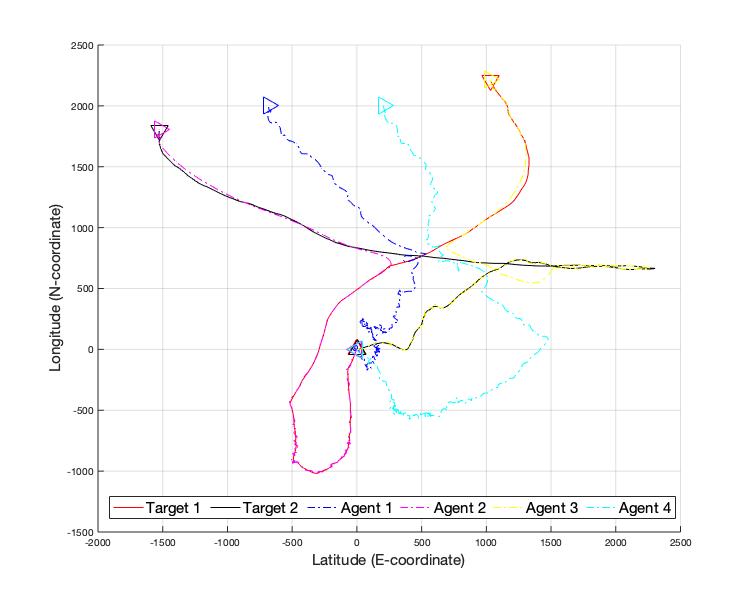}
 	\caption{ Trajectories of Two Targets and Four Agents in \emph{One} Simulation Run.     }    
 	\label{fig:KFtrackingmulti}
 \end{figure} 
\remove{The difference between the position of the target and the agent is plotted on Figure~\ref{fig:KFtrackingerror}. }

\section{Further Discussion} 
This numerical chapter presents an  investigation on the performance of the general SA algorithm  (\ref{eq:truncatedSA1}) in this multi-agent multi-target setting by simulating their dynamics.  Here, several agents in UUV need to perform a surveillance task within a certain area of interest such that any ``intruders'' in the coverage area can be tracked. Subsections \ref{subsect:loss1} and \ref{subsect:loss2}   formulate the surveillance problem as a stochastic optimization problem under time-varying setting.  Algorithms~\ref{algo:base} and \ref{algo:general} demonstrate   how the SA algorithms  with  non-decaying  gain  is  applied  to conduct the time-varying SO task. Besides, the numerical results in Subsections  \ref{subsect:MCbase} and  \ref{subsect:MCgeneral}     partly manifest the error bound  results in  Chapter~\ref{chap:FiniteErrorBound}. The data-dependent gain-tuning strategy proposed  in Chapter~\ref{chap:AdaptiveGain} can also be applied in the  multi-agent application.         Nonetheless, there are many other subtitles we avoid on purpose to present a clean story.  The real-world application may be different from the procedure described in this chapter due to various factors, e.g., the agent can collect a four-dimensional (as opposed to the two-dimensional reading (\ref{eq:noisyMeasurement})) reading including the speed of the agent using Doppler radar, or the number of agents $J$ is smaller than the number of targets $I$, and so on. 
We mention a few of them that need to be dealt with  in real-world tracking problems here.

\subsection{Detection Model}\label{subsect:detection}

In both Section \ref{sect:basecase} and Section \ref{sect:general}, we assume that each agent has an infinite detection range, i.e., each  agent can  collect noisy measurements (\ref{eq:noisyMeasurement}) between itself and any other object (either a target or an agent), so as   to present the loss functions (\ref{eq:loss1Case1}) and (\ref{eq:loss1Case2}) concisely. In reality, each  agent is only allowed to gather readings   (\ref{eq:noisyMeasurement}) from its nearby surroundings and use these readings   to estimate the position of the detectable objects. That is, the agent's sensor can only detect its surroundings up to a certain distance denoted as $\upmu$.  If the object  falls within the $\upmu$-neighborhood of the agent, it can be detected by the agent; otherwise, it is practically  ``invisible'' to the agent.  The typical value of the detection range $\upmu$ can be as small as  $30$ meters or as large as   $1600$ meters.

In addition to the detection range, the detection accuracy is also an important factor that affects  the loss function formulation. Intuitively, the further the object is away from the agent, the less informative is the noisy measurement between the object and the agent itself, even \emph{if} the target stays  within  the detection range of the agent. For example, we can borrow the idea in \cite{kim2005tracking}  to model the detection accuracy, which is measured by the magnitude of the covariance matrix  $\bR_k$  of the measurement noise $\bv_k$ arising  in (\ref{eq:noisyMeasurement}). By taking the detection accuracy into consideration, the diagonal entries of $\bR_k$ can be a non-decreasing function of the actual distance between the agent and the object (either a target or another agent). The domain of this non-decreasing function would be $ \bracket{0,\upmu} $ for the detection range of $\upmu$, and the function form  can be linear, exponential, and so on.

\remove{Our question is: will the agent be able to keep the target within the detection range for all $k$? If yes, what are  the necessary conditions? 
	Suppose that the target is within the agent's detection range $\upmu$  at  the beginning $k=0$. 
}

\subsection{Communication Within Range }
Another factor comes into play when the detection range in Subsection~\ref{subsect:detection} is taken into account. If two agents are close to each other, i.e., one falls within the detection range of the other, then  the  communication between these two agents  is generally allowed in reality. With the allowable communication within range at time $\uptau_k$, two neighboring  agents can exchange information, including their  estimations for the states  of all the targets and all other agents. The neighboring agents  can    rely on their mutual  awareness and   plan   in a  local team (as opposed to individually in Section~\ref{sect:general})   as a fully connected ensemble.

 Let  $j_1$ and $j_2$ be the indices of two neighboring agents at time $\uptau_k$. They can share the following with each other:
\begin{equation}\label{eq:share}
\hat{\bx}_{k}^{(i,j_1)}, \hat{\bx}_{k}^{(i,j_2)} \text{ for all }1\le i\le I, \text{ and } \hat{ {\by}}_{k}^{(j',j_1)}, \hat{{\by}}_{k}^{(j',j_2)} \text{ for all }j' \in \set{1,\cdots,J}\setminus\set{j_1,j_2}. 
\end{equation} 
It is natural to take advantage of the shared information (\ref{eq:share}) to jointly improve  their estimations. For example,  to obtain a better estimate of  $\hat{ \bx} _k^{(i)}$, the two neighboring agents can take a \emph{weighted} average between $\hat{\bx}_{k}^{(i,j_1)}$ and $ \hat{\bx}_{k}^{(i,j_2)}$. The weighted average is  proportional to $\ilparenthesis{\bP_{k}^{(i,j_1)}}^{-1}   \hat{\bx}_{k}^{(i,j_1)} +    \ilparenthesis{\bP_{k}^{(i,j_2)}}^{-1} \hat{\bx}_{k}^{(i,j_2)} $ 
where the $ \ilparenthesis{\bP_{k}^{(i,j_1)}}^{-1} $ measures the degree of confidence in  using     $\hat{\bx}_{k}^{(i,j_1)} $ as an estimate for $\hat{\bx}_k^{(i)}$, and $ \ilparenthesis{\bP_{k}^{(i,j_2)}}^{-1} $ measures the degree of confidence of using     $\hat{\bx}_{k}^{(i,j_2)} $ as an estimate for $\hat{\bx}_k^{(i)}$. Similarly, agent $j_1$ and $j_2$ can jointly improve their estimate for $\hat{ {\by}}_k^{(j')}$ for 
$j' \in \set{1,\cdots,J}\setminus\set{j_1,j_2}$. 

 Of course, the weighted average of multiple agents' information is also possible when they are within the detection range of each other. We haven't implemented this weighted average idea in our simulation study, as the notion of detection range in Subsection~\ref{subsect:detection} is not considered in this numerical study. It is also ``dangerous'' in some situations where agents need to stay hidden--minimal (or no) transmissions are preferred.

		\remove{
		This motivates
		the need for applying decentralized resource optimization
		methods to the decentralized
		motion planning problem. In the chosen formulation, each
		agent controls its motion by deciding on its action  at each
		time step via minimization of an appropriate loss function.
	}

\remove{ 
\section{Other Concerns}

\subsection{Some Radar Also Gives Speed}

\subsection{Tagging Issues}  
Samer Saab, ``Shuffled Linear Regression with Erroneous Observations'' 2019
}


\chapter{Summary and Possible Future Work} 
The thesis considers the general stochastic approximation setup, under a practical situation where the scalar-valued objective function $\loss_k\ilparenthesis{\cdot}$ or the vector-valued function $\bg_k\ilparenthesis{\cdot}$  may be perpetually time-varying.    The method we investigated is the general SA recursive schemes (\ref{eq:basicSA}) with non-decaying gains. The time-varying problem setting and SA framework   have  presented several
new issues, both theoretical and practical.

Chapter \ref{chap:FiniteErrorBound} develops     bounds for both MAD and RMS: the unconditional version is obtained by averaging all possible sample paths, and the conditional version is gathered by observing actual noisy gradient information. Both error bounds are computable as long as we have access to the noise level and the Hessian matrix of the underlying loss function.  Note that our quantification of  tracking capability within finite-iterations of the non-diminishing gain SA algorithm is  in terms of a probabilistically \emph{computable} error bound, which  may  also apply to the general \emph{nonlinear} SA literature. These two characteristics make our work different from \cite{eweda1985tracking} which focuses on linear models, and \cite{wilson2019adaptive} which provides big-$O$ bounds. Moreover, to the best of our knowledge, there are no existing approaches in estimation theory that solve  a sequence of the  time-varying problem, under only Assumption      A.\ref{assume:BoundedVariation}  (the expected distance between two consecutive optima  are bounded from above)  without any further  stringent state evolution assumption.\remove{
	The case of interest requires the strong convexity of the time-varying loss function (sequence), but  our  tracking error bound  is favorably informative under reasonable assumptions on the evolution of the true parameter being estimated.  }  
  A.\ref{assume:BoundedVariation} is a  fairly modest  assumption on the evolution of the underlying time-varying parameter to be identified:   the average distance between two consecutive minimizers $ \norm{\bvartheta_{k+1}-\bvartheta_k} $ is bounded uniformly across $k$.  
Finally, as a consequence of the MAD bound, we can characterize the stability of the SA algorithm in response to the drift $\ilset{\bvartheta_k}$  in terms of determining the allowable region for the non-diminishing gain $\gain  $, which embraces many more general SA algorithms including the special case of SGD discussed in \cite{zhu2016tracking}.    In short, the tracking performance 
of  non-diminishing gain SA algorithms is guaranteed by  a computable  bound on MAD, which is useful in     \emph{finite-sample} performance.

To supplement the tracking capability discussed in Chapter~\ref{chap:FiniteErrorBound}, Chapter~\ref{chap:Limiting} focuses on the concentration behavior in terms of the  probabilistic bound of the recursive estimates generated from the  constant-gain SGD algorithm over a finite time frame.    
 The weak convergence limit of a suitably interpolated sequence  of the iterates is shown to follow the trajectory of a non-autonomous ordinary differential equation, and the discussion there applies to constrained optimization in Section~\ref{sect:ConstantGain}. The weak convergence limit is taken w.r.t.  the constant gain $\gain$. It should be interpreted that for some nonzero constant gain $\gain$, which needs not to  go to zero, the continuation of $\hbtheta_k$ will stay close in the sense of weak limit to a  non-autonomous ODE  when the underlying data change with time on a scale that is commensurate with what is determined by the gain.  To make the bound of the probability for the event that $\hbtheta_k$ deviates  from $\bvartheta_k $ computable, Section~\ref{sect:ProbBound} imposes further assumptions and utilizes the formula for variation of parameters.
The probabilistic bound there provides a general sense of the likelihood of $\hbtheta_k$ staying close to $\bvartheta_k$ for a constant gain $\gain$ under certain conditions. 
    Note that  the upper bound for the probability of the iterates deviating from the target is  valid for all time, which is useful for finite-sample analysis.

Even though Chapter \ref{chap:FiniteErrorBound} develops a gain tuning strategy based upon the MAD bound, the strategy is derived after averaging out all possible sample paths of the random sequence $\ilset{\bvartheta_k}$. Even though Chapter~\ref{chap:Limiting} discusses the weak convergence limit and a bound of the event that $\hbtheta_k$ deviates from $\bvartheta_k$, it only characterizes the small probability of the rare event of $\bvartheta_k$ deviates from $\bvartheta_k$ beyond a certain threshold.  Both of these are probabilistic arguments      and may not provide much help in tuning the non-decaying gain in practical implementations.   In reality, we hope to detect the changes in $\ilset{\bvartheta_k}$ as promptly and accurately  as possible. Moreover, we have to deal with the situation where the Hessian and error information that governs the MAD bound are \emph{unavailable}. These two reasons motivate us to direct our attention to a data-dependent gain tuning strategy. Taking advantage of observable data  helps improve the tracking performance on each  specific sample-path.  Thus, Section \ref{sect:Detection} develops a change detection strategy, using the test statistic in the  multivariate Behrens\textendash Fisher problem, although the detection relies on the approximately normal distribution of the estimates $\ilset{\hbtheta_k}$ when it reaches steady-state phase  and oscillates around $\bvartheta_k$.  We, unfortunately,  cannot provide exact type-I and type-II errors  for such a test. Nonetheless, the detection scheme does help to detect regime change robustly and avoid the burden  of estimating Hessian and noise level adaptively.  
Based on  the change detection testing  in Sections \ref{sect:Detection} and    \ref{sect:GainOneRegime}, we  continue to develop a gain adaptation strategy   to adaptively adjust the gain sequence  by detecting whether a jump has occurred or not.   To perform better with each sample-path, we have to adjust the gain sequence adaptively based on the given data stream. Here, we handle the issues  of the Hessian and noise levels  being unknown by using  simultaneous perturbation methods, which is efficient and inexpensive.

In a nutshell, this work partly answers the questions ``what is the estimate for the dynamical system $\bvartheta_k$'' and ``how much   we can trust  $\hbtheta_k$ as an estimate for $\bvartheta_k$.'' 
To the best of our knowledge, there are no existing approaches in estimation theory that solve  a sequence of time-varying problems, under only Assumption      A.\ref{assume:BoundedVariation}    in Chapter~\ref{chap:FiniteErrorBound} or B.\ref{assume:gSequenceRegularity}  (the average distance between two consecutive optima is proportional to the sampling time elapsed)   in Chapter~\ref{chap:Limiting}  without any further  stringent state evolution assumption. Moreover, the probabilistic arguments  in  Chapter~\ref{chap:FiniteErrorBound}   and  Chapter~\ref{chap:Limiting}  are non-asymptotic.    Additionally, a data-dependent gain-tuning strategy is proposed in Chapter~\ref{chap:AdaptiveGain}.

Some possible future work includes:

\begin{itemize}
	
	\remove{\item  Admittedly, the mean-squared distance in A.\ref{assume:BoundedVariation} can be applicable in general. Nonetheless, it is still beneficial to extend to general ``$L_p$-stability'' distance notion in \cite[Def. 2.1]{guo1994stability}. }
	
		\item It appears unlikely that the bounds in Chapter~\ref{chap:FiniteErrorBound} that use Lipschitz constants and strong convexity parameters can be improved much. But how to efficiently estimate  these  needed parameters in an online  fashion remains unresolved.

	\item  Most existing works focus on the case where   $\bvartheta_k$  is a singleton for each $k$ for unconstrained optimization.
	 The extension to constrained optimization and multiple minimizers scenarios will help the practical implementation.  
	 
	  In the constrained or nonsmooth context, the optimum point $\bvartheta_ k$ may not lie within the interior of the feasible region, implying that the gradient at $\bvartheta_k $ may not be zero. Namely, $\bg_k\ilparenthesis{\bvartheta_k}=\zero$ is no longer  a necessary and sufficient condition for  determining $\bvartheta_k$, and other optimality condition should be discussed.  
		\item 
	It would be of interest to  extend the discussion of Theorem~\ref{thm:ODE} to 
	more   involved scenarios such as correlated noise, multi-scale, state-dependent noise processes, decentralized/asynchronous algorithms,  and discontinuities in the algorithms.
	\item
	Future work on computable probabilistic bound as in Theorem~\ref{thm:ProbBound} may consider the extension of  the bound to the case where  FDSA or  SPSA (instead of SGD) is used in time-varying problems (e.g., \cite{spall1998model}). 
	The main benefit is that only noisy measurements of the loss function $\loss_k\ilparenthesis{\cdot}$ are needed, but the main theoretical  complication introduced by FDSA or SPSA is that the gradient estimate is biased.
	
	\item Even though Chapter~\ref{chap:AdaptiveGain} discusses a data-dependent gain-tuning, more  theoretical   and  practical work is still needed  to  effectively tuning the constant gain to regulate the tracking capability and stability needs.  Some unresolved questions relative to gain tuning are listed below. 
	
	\begin{itemize}
			\item The critical value for the change detection in Section \ref{sect:Detection} is data-dependent, which forces us to estimate   unknown  covariance matrix $\VarianceLimiting$ in (\ref{eq:LimitingVariance}) on the fly. If some distribution-free test statistic with high power can be adapted to meet the change detection purpose, that may help  streamline the  change detection procedure. 
			\item Assumption C.\ref{assume:Hybrid} (the optimum remains constant within each regime), in some real-world applications, may still be restrictive. The extension to the scenario C.\ref{assume:HybridRelaxed} (the optimum remains stationary within each regime) will  be very much desirable, yet it requires more in-depth understanding of the limiting distribution of $\hbtheta_k$,  which is currently unavailable. 
			
		\remove{		\item Is it possible to extend the discussion beyond C.\ref{assume:Jump}?
		\item How to make the discussion in  Chapter~\ref{chap:AdaptiveGain} extend-able from SGD to general SA algorithms?
		}

			\item Both  Sect.  \ref{sect:Detection} and Sect.  \ref{sect:GainOneRegime}    require that the sequence of  loss functions take  the quadratic form $ \ilparenthesis{\btheta-\bvartheta_k}\bH_k\ilparenthesis{\btheta-\bvartheta_k}/2 $.  Can we extend the form of loss functions   $\ilset{\loss_k\ilparenthesis{\cdot
			}}$   to more general nonlinear form? 
				\item The explicit form of (\ref{eq:InnerProductVariance}), which  pertains to the variance of the moving average of the inner product of two consecutive noisy gradient estimates, is difficult to derive. Nonetheless,  if that is available, it does help to improve    Algorithm~\ref{algo:adaptiveGain} that adaptively changes the gain based on observed data $\bvartheta_k$. 
			 
		\item How can we select an optimal gain while estimating the drift term and the noise level in an online fashion?

		\item Are there  any values of $\increase$ and $\shrink$ (the parameters that govern  the increase and the decrease of the gain sequence) that are optimal in a certain statistical sense, i.e., the resulting estimate $\hbtheta_k$  achieves  the information-theoretic Cramer-Rao lower bound for SA contexts \cite{fabian1968asymptotic}? 
		
			\item  When $\bvartheta_k=\bvartheta$ for all $k$, can the idea of determining whether $\hbtheta_k$ reaches  proximity to stationarity be formalized in a way such that the resulting iterates in Algorithm~\ref{algo:adaptiveGain}  converge to $\bvartheta$  a.s.?
		\item Can we extend the scalar gain to a matrix gain, without incurring much extra computational cost? (Appendix \ref{chap:2nd} or  \cite{zhu2019efficient} demonstrate a reduction of $O(p)$ for the standard SA setup without time variation.)
	\end{itemize}

	\item Throughout our discussion, we promote \emph{few} measurements of the loss function or the gradient  at each sampling  time $\uptau_k$: only one or two parallel measurements  are allowed. An increased number of design points at each $k$ can likely  produce a tighter bound for the tracking error $ \norm{\hbtheta_k-\bvartheta_k} $, even though this goes against   the general philosophy of SA. Is there a way to  measure the efficiency trade-off for increased sampling?

\end{itemize}

 There are many unresolved questions, especially for the field of data-dependent gain tuning. This work is   a   step towards fully understanding  how $\hbtheta_k$ generated from general SA schemes with non-decaying gains,  tracks the time variation in $\bvartheta_k$ and how much we can trust $\hbtheta_k$ as an estimate of $\bvartheta_k$.


\begin{appendices}

\chapter{Second-Order SA in High-Dim Problems}\label{chap:2nd}

\section{Introduction}
 SA algorithms have been widely applied in minimization problems where the loss functions and/or the gradient are only accessible through noisy evaluations. Among all the SA algorithms, the second-order simultaneous perturbation stochastic approximation (2SPSA)\label{acronym:2SPSA}  and the second-order stochastic gradient (2SG)\label{acronym:2SG} are particularly efficient in high-dimensional problems covering both gradient-free and gradient-based scenarios. However, due to the necessary matrix operations, the per-iteration   FLOPs  of the original 2SPSA/2SG are  $ O(p^3) $ with $ p $ being the dimension of the underlying parameter. Note that the $O(p^3)$ FLOPs are  distinct from the classical SPSA-based per-iteration $O(1)$ cost in terms of the number of noisy function evaluations. In \cite{zhu2019efficient},  we propose a technique to efficiently implement the 2SPSA/2SG algorithms via the symmetric indefinite matrix factorization such  that the per-iteration floating-point operations (FLOPs)\label{acronym:FLOPs}    are reduced from $ O(p^3) $ to $ O(p^2) $. The almost sure convergence and rate of convergence for the newly-proposed scheme are naturally inherited from the original 2SPSA/2SG. The numerical improvement manifests its superiority in numerical studies in terms of computational complexity and numerical stability. 
\subsection{Problem Context}

SA has been widely applied in minimization and/or root-finding problems, when only  {noisy} loss function and/or gradient evaluations are accessible. Consider minimizing a differentiable loss function $ \loss (\btheta): \real^p \to \real $, where only noisy evaluations of $\loss \parenthesis{\cdot}$ and/or its gradient $\bg\parenthesis{\cdot} $ are accessible. The key distinction between SA and classical deterministic optimization is the presence of noise, which is largely inevitable when the function measurements are collected from either physical experiments or computer simulation. Furthermore, the noise term comes into play when the loss function is only evaluated on a small subset of an entire (inaccessible) dataset as in online training methods popular with neural network and machine learning. In the era of big-data, we deal with applications where   solutions are   data-dependent such that the cost is minimized over a  given set of sampled data rather than the entire distribution. Overall, SA algorithms have numerous applications in adaptive control, natural language processing, facial recognition, and collaborative filtering, just to name but a few.

In modern machine learning, there is a growing need for algorithms to handle high-dimensional problems. Particularly for deep learning, the need arises as the number of parameters (including both weights and bias) explodes quickly as the network depth and width increase. First-order methods based on back-propagation are widely applied, yet they suffer from slow convergence rate in later iterations after a sharp decline during the early iterations. Second-order methods are occasionally utilized to speed up convergence in terms of the  number of iterations, but, still, at a computational burden of $O(p^3)$ per-iteration FLOPs.

To achieve a faster convergence rate at a reasonable computational cost, we present a second-order SP method that incurs only   $O(p^2)$ per-iteration FLOPs  in contrast to the standard  $O(p^3)$. The idea of SP is an elegant generalization of a   finite difference (FD)\label{acronym:FD} scheme and can be applied in both first-order and second-order SA algorithms. Our proposed method rests on the  factorization of symmetric indefinite matrices.

\subsection{Relevant Prior Works}

The adaptive second-order methods here differ in fundamental ways from stochastic quasi-Newton and other similar methods in the machine learning literature. First, most of the machine learning-based methods are designed for loss functions of the  ERF form; namely,  for functions represented as summations, where each summand represents the contribution of one data vector. Such a structure, together with an assumption of strong convexity, has been exploited in \cite{johnson2013accelerating,martens2015optimizing}, and others for stronger convergence results. Second, first- or second-order derivative information is often assumed to be directly available on the summands in the loss function (e.g., \cite{byrd2016stochastic,sohl2014fast,schraudolph2007stochastic}). Ref.  \cite{saab2019multidimensional}  also assumes direct information on the Hessian is available in a second-order stochastic method, but allows for loss functions more general than the ERF.   Ref. \cite{byrd2016stochastic} applies  the   BFGS method to SO, but under a nonstandard setup where noisy Hessian information can be gathered. In our work, we assume that only  noisy loss function evaluations or noisy gradient information are available.  Third, notions of convergence and rates of convergence are in line with those in deterministic optimization when the loss function (the ERF) is composed of a finite (although possibly large) number of summands. For example, rates of convergence are linear or quadratic as a measure of iteration-to-iteration improvement in the ERF. In contrast, we follow the traditional notion of stochastic approximation, including applicability to general noisy loss functions, no availability of direct derivative information, and stochastic notions of convergence and rates of convergence based on sample-points (in almost surely sense) and convergence in distribution.

Among various SA schemes, SP algorithms are particularly efficient compared with   FD methods. Under certain regularity conditions, \cite{spall1992multivariate} shows that the SPSA algorithm uses only $ 1/p $ of the required number of loss function observations needed in the FD form to achieve the same level of MSE for the SA iterates. To further explore the potential of SP algorithms, \cite{spall2000adaptive} presents the second-order SP-based methods, including the  2SPSA for applications in the gradient-free case and the 2SG  for applications in the gradient-based case. Those methods estimate the  Hessian matrix to achieve near-optimal or optimal convergence rates and can be viewed as the stochastic analogs of the deterministic Newton-Raphson algorithm. Ref. \cite{spall2009feedback} incorporates both a feedback process and an optimal weighting mechanism in the averaging of the per-iteration Hessian estimates to improve the accuracy of the cumulative Hessian estimate in enhanced second-order simultaneous perturbation stochastic approximation (E2SPSA)\label{acronym:E2SPSA}   and enhanced second-order stochastic gradient (E2SG)\label{acronym:E2SG}. The guidelines for practical implementation details and the choice of gain coefficients are available in \cite{spall1998implementation}. More details on the related methods are discussed in \cite[Chaps. 7\textendash 8]{bhatnagar2012stochastic}.

\subsection{Our Contribution}
Refs. \cite{spall2000adaptive,spall2009feedback} show that the 2SPSA/2SG methods can achieve near-optimal or optimal convergence rates with a much smaller number (independent of dimension $p$) of loss or gradient function evaluations relative to other second-order stochastic methods in \cite{fabian1971stochastic, ruppert1985newton}. However, after obtaining   function evaluations, the per-iteration FLOPs  to update the estimate are  $ O(p^3) $, as discussed below. The computational burden becomes more severe as $p$ gets larger. This is usually the case in many modern machine learning applications. Here we propose a scheme to implement 2SPSA/2SG efficiently via the symmetric indefinite factorization, which reduces the per-iteration FLOPs   from $ O(p^3) $ to $ O(p^2) $. We also show that the proposed scheme inherits the almost sure convergence and the rate of convergence from the original 2SPSA/2SG in \cite{spall2000adaptive}.

The remainder of the chapter is as follows. Section~\ref{sec:2SPSA} reviews the original 2SPSA/2SG in \cite{spall2000adaptive} along with the computational complexity analysis. Section~\ref{sec:efficient_implementation} discusses the proposed efficient implementation, while  Section~\ref{sec:theory} covers the almost sure convergence and asymptotic normality.  Numerical studies   are  in Section \ref{sec:numerical}. Section \ref{sec:discussion} concludes with a discussion of  some practical issues.

\section{Review of 2SPSA/2SG}\label{sec:2SPSA}

Before proceeding, let us review the original 2SPSA/2SG algorithms and explain their $ O(p^3) $ per-iteration FLOPs.

\subsection{2SPSA/2SG Algorithm}
Following the routine SA framework, we find the root(s) of $\bg\parenthesis{\btheta}\equiv \partial \loss\parenthesis{\btheta}/\partial\btheta$   to solve the problem of finding $ \arg\min \loss\parenthesis{\btheta} $. \remove{ 
	Consider the root-finding problem, the proxy problem for minimizing $ L(\btheta) $ is to find $ \btheta $ such that
	\begin{equation*}
	\bg(\btheta) \equiv \frac{\partial \loss (\btheta)}{\partial\btheta} = \bm{0}\,. 
	\end{equation*} }

Our central task is to streamline the computing procedure, so we do not dwell on differentiating the global minimizer(s) from the local ones. Such root-finding formulation is widely used in the neural network training and other machine learning literature. We consider optimization under two different settings:
\begin{enumerate}
	\item Only noisy measurements of the loss function, denoted by $ y(\btheta)$ as in Section~\ref{sect:SAview}, are available.
	\item Only noisy measurements of the gradient function, denoted by $ \bY(\btheta) $ as in Section~\ref{sect:SAview}, are available. 
\end{enumerate}
 The conditions for noise can be found in \cite[Assumptions C.0 and C.2]{spall2000adaptive}, which include various types of noise such as Gaussian, multiplicative and impulsive noise as special cases.
The main updating recursion for 2SPSA/2SG in \cite{spall2000adaptive} is 
\begin{equation} \label{eq:theta_update}
\hbtheta_{k+1} = \hbtheta_k - a_k \oobH_k^{-1} \bG_k(\hbtheta_k), k = 0, 1, \cdots,
\end{equation}
where $ \{a_k\}_{k\geq0} $ is a positive decaying scalar gain sequence, $ \bG_k(\hbtheta_k) $ is the direct noisy observation or the approximation of the gradient information, and $ \oobH_k $ is the approximation of the Hessian information. The true gradient $ \bg(\hbtheta_k) $ is estimated by: 
\begin{numcases}
{\bG_k(\hbtheta_k)=}
\frac{y(\hbtheta_k+c_k\bDelta_k)-y(\hbtheta_k-c_k\bDelta_k)}{2c_k\bDelta_k}\,, &\hspace{-.25in}\text{for 2SPSA,}\label{eq:gradient_estimate_2SPSA}\\
\bY_k(\hbtheta_k)\,, &\hspace{-.25in}\text{for 2SG,}\label{eq:gradient_estimate_2SG}
\end{numcases}  
	where $ \bDelta_k = [\Delta_{k1}, \dots, \Delta_{kp}]^\transpose $ is a mean-zero $p$-dimensional stochastic perturbation vector with bounded inverse moments
\cite[Assumption B.6$^{\prime\prime}$ on pp. 183]{spall2005introduction}, $ 1 / \bDelta_k = \bDelta_k^{-1} \equiv (\Delta_{k1}^{-1}, \cdots, \Delta_{kp}^{-1})^\transpose $ is a vector of reciprocals of each nonzero components of $ \bDelta_k $ ($\bDelta_k^{-\transpose}$ is the transpose of $\bDelta_k^{-1}$), and $ \{c_k\}_{k\geq0} $ is a positive decaying scalar gain sequence satisfying conditions in \cite[Sect. 7.3]{spall2005introduction}. A valid choice for $ c_k $ is  $ c_k = 1 / (k+1)^{1/6}$. For the Hessian estimate $ \oobH_k $, \cite{spall2000adaptive} proposes: 
\begin{numcases}{}
\oobH_k = \bm{f}_k(\obH_k)\,, & \label{eq:H_ooverline}\\
\obH_k = (1-w_k) \obH_{k-1} + w_k \hbH_k \,,& \label{eq:H_overline}\\
\hbH_k=\frac{1}{2}\left[\frac{\updelta\bG_k}{2c_k}\bDelta_k^{-\transpose}+\left(\frac{\updelta\bG_k}{2c_k}\bDelta_k^{-\transpose}\right)^\transpose \right] \,, \label{eq:H_hat}&\\
\updelta\bG_k=\bG_k^{(1)}(\hbtheta_k+c_k\bDelta_k)-\bG_k^{(1)}(\hbtheta_k-c_k\bDelta_k) \,, \nonumber&
\end{numcases}
where $ \bm{m}_k\hspace{-0.04in}: \real^{p\times p} \to \{$positive definite $p\times p$ matrices$\}$ is a preconditioning step to guarantee the positive-definiteness of $ \oobH_k $,  $ \{w_k\}_{k\geq0} $ is a positive decaying scalar weight sequence, and $ \bG_k^{(1)}(\hbtheta_k\pm c_k\bDelta_k) $ are one-sided gradient estimates calculated by:
\begin{align*}
&\bG_k^{(1)}(\hbtheta_k\pm c_k\bDelta_k)=\begin{dcases}
\frac{y(\hbtheta_k\pm c_k\bDelta_k+\tilde{c}_k\tbDelta_k)-y(\hbtheta_k\pm c_k\bDelta_k)}{\tilde{c}_k\tbDelta_k}, &\hspace{-.1in}\text{in 2SPSA,}\\
\bY_k(\hbtheta_k\pm c_k\bDelta_k), &\hspace{-.1in}\text{in 2SG,}
\end{dcases}
\end{align*}
where $ \{\tilde{c}_k\}_{k\geq0} $ is another positive decaying gain sequence, and $ \tbDelta_k = (\tilde{\Delta}_{k1}, \cdots, \tilde{\Delta}_{kp})^\transpose $ is generated independently from $ \bDelta_k $, but in the same statistical manner as $ \bDelta_k $.  Some valid choices for $w_k$ include $w_k=1/(k+1)$ and the asymptotically optimal choices in  \cite[Eq. (4.2) or Eq. (4.3)]{spall2009feedback}.  Ref. \cite{spall2000adaptive} considers the special case where $ w_k=1/\parenthesis{k+1} $, i.e., $\obH_k$ is a sample average of the $ \hbH_j $ for $ j = 1, \cdots, k $.  Later \cite{spall2009feedback} proposes the E2SPSA and E2SG to obtain more accurate Hessian estimates by taking the optimal selection of weights and feedback-based terms in (\ref{eq:H_overline}) into account. While the focus of this paper is the original 2SPSA/2SG in \cite{spall2000adaptive}, we also  discuss the applicability of the ideas to the E2SPSA/E2SG algorithms in \cite{spall2009feedback}. Note that, independent of $p$, one iteration of 2SPSA/E2SPSA uses four noisy measurements $ y(\cdot) $,  and one iteration of 2SG/E2SG uses three noisy measurements $ \bY(\cdot) $.

\subsection{Per-Iteration Computational Cost of $ O(p^3) $} \label{subsect:p3cost}

The per-iteration computational cost of $ O(p^3) $ arises from two steps: one is from the preconditioning step in (\ref{eq:H_ooverline}), i.e., obtaining $ \oobH_k $; the other is from the descent direction step in (\ref{eq:theta_update}), i.e., obtaining $ \oobH_k^{-1}\bG_k(\hbtheta_k) $. We now discuss the per-iteration computational cost of these two steps in more detail.

\textbf{Preconditioning} The preconditioning step in (\ref{eq:H_ooverline}) is to guarantee the positive-definiteness of the Hessian estimate $ \oobH_k $. This step is necessary  because the updating of $ \obH_k $ in (\ref{eq:H_overline}) does not necessarily yield a positive-definite matrix (but $ \obH_k $ is guaranteed to be symmetric). One straightforward way is to perform the following transformation:
\begin{equation} \label{eq:f_k_sqrtm}
\bm{m}_k(\obH_k) = (\obH_k \obH_k + \delta_k \bI)^{1/2} \,,
\end{equation}
where $ \delta_k > 0 $ is a small \emph{decaying} scalar coefficient \cite{spall2000adaptive} and superscript ``1/2" denotes  the symmetric matrix square root. Let $ \lambda_i(\cdot) $ denote the $i$th eigenvalue of the argument. In that  $ \lambda_i(\bA+c\bI) = \lambda_i(\bA)+c $ for any matrix $\bA$ and constant $c$ \cite[Obs. 1.1.7]{horn1990matrix}, we see that (\ref{eq:f_k_sqrtm}) directly modifies the eigenvalues of $ \obH_k\obH_k $ such that $ \lambda_i(\obH_k\obH_k + \delta_k\bI) = \lambda_i(\obH_k\obH_k) + \delta_k $ for $ i = 1, ..., p $. When $ \delta_k > 0 $, all the eigenvalues of $ \obH_k\obH_k + \delta_k\bI $ are strictly positive and,  therefore,  the resulting $ \oobH_k $ is positive definite. However, (\ref{eq:f_k_sqrtm}) has  a computational cost of $ O(p^3) $ due to both the matrix multiplication in $ \obH_k \obH_k $ and the matrix square root computing \cite{higham1987computing}. Another intuitive transformation is 
\begin{equation} \label{eq:f_k_add}
\bm{m}_k(\obH_k) = \obH_k + \delta_k \bI
\end{equation} for a positive and sufficiently large $ \delta_k $. Again, applying eigen-decomposition on $ \obH_k $, we see that $ \lambda_i(\oobH_k) = \lambda_i(\obH_k) + \delta_k $ for $ i = 1, \cdots, p $. Take $ \uplambda_{\min}\parenthesis{\cdot}= \min_{1\le i\le p } \uplambda_i\parenthesis{\cdot} $ for any argument matrix in $\real^{p\times p}$. Any $ \delta_k > |\lambda_{\min}(\obH_k)| $ will result in $ \lambda_{\min}(\oobH_k) > \remove{\lambda_{\min}(\obH_k) + |\lambda_{\min}(\obH_k)| \geq} 0 $, and, therefore, the output $ \oobH_k $ is positive definite. Unfortunately, (\ref{eq:f_k_add}) cannot avoid the $O(p^3)$ cost in estimating $ \lambda_{\min}(\obH_k) $. 

In addition to  the $O(p^3)$ cost in (\ref{eq:f_k_sqrtm}) and (\ref{eq:f_k_add}), the Hessian estimate $ \oobH_k $ may be ill-conditioned, leading to slow convergence. Ref. \cite{zhu2002modified} proposes to replace all negative eigenvalues of $ \obH_k $ with values proportional to its smallest positive eigenvalue. Such modification is shown to improve the convergence rate for problems with ill-conditioned Hessian and achieve smaller mean square errors for problems with better-conditioned Hessian compared with original 2SPSA \cite{zhu2002modified}. However, those benefits are gained at the  price of computing the eigenvalues of $ \obH_k $, which still costs $ O(p^3) $.

\textbf{Descent direction} Another per-iteration computational cost of $ O(p^3) $ originates from the descent direction computing in (\ref{eq:theta_update}), which is typically computed by solving the linear system for $ \bd_k: \oobH_k\bd_k = \bG_k(\hbtheta_k) $. The estimate is updated recursively  as following: 
\begin{equation}\label{eq:theta_update_s}
\hbtheta_{k+1} = \hbtheta_k - a_k\bd_k \,.
\end{equation} 
With the matrix left-division, it is possible to efficiently solve for $ \bd_k $. However, the computation costs of typical methods, such as $LU$ decomposition or singular value decomposition, are still dominated by $ O(p^3) $.

\begin{table*}[!htbp]
	\renewcommand{\arraystretch}{2}
	\caption{Expressions for terms in (\ref{eq:two_rank_one_update})--(\ref{eq:v_k_tilde}). See \cite[Sect. 7.8.2]{spall2005introduction} for detailed suggestions.}
	\label{table:u_k_v_k}
	\centering
	\begin{tabular}{|l|c|c|c|c|}
		\hline
		Algorithm & $ t_k $ & $ b_k $ & $ \bu_k $ & $ \bv_k $\\
		\hline\hline	
		2SPSA \cite{spall2000adaptive} & $ 1 - w_k $ & $ w_k \delta y_k / (4c_k\tilde{c}_k) $ & $ \tbDelta_k^{-1} $ & \multirow{4}{*}{$ \bDelta_k^{-1}$} \\
		\cline{1-4}
	\multirow{2}{*}{E2SPSA \cite{spall2009feedback} } & 	\multirow{2}{*}{$1$} & $w_k[\updelta y_k / (2c_k\tilde{c}_k) ]/2 $  &  \multirow{2}{*}{$ \tbDelta_k^{-1} $} & \\ 
	& & $- w_k[  \bDelta_k^\transpose\obH_{k-1}\tbDelta_k]/2 $   & &  \\
		\cline{1-4}
		2SG \cite{spall2000adaptive} & $ 1 - w_k $ & $ w_k / (4c_k) $ & $ \delta\bG_k $ & \\
		\cline{1-4}
		E2SG \cite{spall2009feedback} & $ 1 $ & $ w_k / 2 $ & $ \delta\bG_k/(2c_k) - \obH_{k-1} \bDelta_k $ & \\
		\hline
	\end{tabular}
\end{table*}

To speed up the original 2SPSA/2SG, \cite{rastogi2016efficient} proposes to rearrange (\ref{eq:H_overline}) and (\ref{eq:H_hat}) into the following two sequential rank-one modifications:
\begin{numcases}{}
\obH_k = t_k\obH_{k-1} + b_k\tbu_k\tbu_k^{\transpose} - b_k\tbv_k\tbv_k^{\transpose}\,, & \label{eq:two_rank_one_update}\\
\tbu_k = \sqrt{\frac{\norm{\bv_k}}{2\norm{\bu_k}}} \parenthesis{\bu_k + \frac{\norm{\bu_k}}{\norm{\bv_k}}\bv_k} \,, & \label{eq:u_k_tilde}\\
\tbv_k = \sqrt{\frac{\norm{\bv_k}}{2\norm{\bu_k}}} \parenthesis{\bu_k - \frac{\norm{\bu_k}}{\norm{\bv_k}}\bv_k} \,, &\label{eq:v_k_tilde}
\end{numcases}
where the scalar terms $ t_k $ and $ b_k $ (\ref{eq:two_rank_one_update}), and vectors $ \bu_k $ and $ \bv_k $ in (\ref{eq:u_k_tilde}) and (\ref{eq:v_k_tilde}) are listed in Table \ref{table:u_k_v_k}. 
Applying the matrix inversion lemma \cite[pp. 513]{spall2005introduction}, \cite{rastogi2016efficient} shows that $ \obH_k^{-1} $ can be computed from $ \obH_{k-1}^{-1} $ with a cost of $ O(p^2) $. However, the positive-definiteness of $ \obH_k^{-1} $ is not guaranteed, and an additional eigenvalue modification step similar to either (\ref{eq:f_k_sqrtm}) or (\ref{eq:f_k_add}) is required. As discussed before, for any direct eigenvalue modifications, the computational cost of $ O(p^3) $ is inevitable due to the lacking knowledge about the eigenvalues of $ \obH_{k-1}^{-1} $.

In short, no prior works can fully streamline the entire second-order SP procedure with an $O\parenthesis{p^2}$ per-iteration FLOPs, which motivates the elegant procedure below.

\section{Efficient Implementation of 2SPSA/2SG}\label{sec:efficient_implementation}

\subsection{Introduction}\label{subsect:Introduction}
With the motivation for proposing
an efficient implementation scheme for 2SPSA/2SG laid out in Subsection~\ref{subsect:p3cost}, we now explain our methodology in more detail. Note that none of the prior attempts on 2SPSA/2SG methods can bypass the end-to-end computational cost of $O(p^3)$ per iteration in high-dimensional SO problems. Therefore, we propose   replacing $ \obH_k $ by its symmetric indefinite factorization, which enables us to implement the 2SPSA/2SG at a per-iteration computational cost of $O(p^2)$. Our work helps alleviate  the notorious curse of dimensionality by achieving the fastest possible second-order methods based on Hessian estimation, to the best of our knowledge. Moreover, note that the techniques in \cite{rastogi2016efficient} are no longer applicable because our scheme keeps track of the matrix factorization instead of the matrix itself, so we propose new algorithms to establish our claims.

\tikzstyle{block} = [rectangle, draw, fill=blue!20, text centered, rounded corners, minimum height=2em]
\tikzstyle{line} = [draw, -latex']
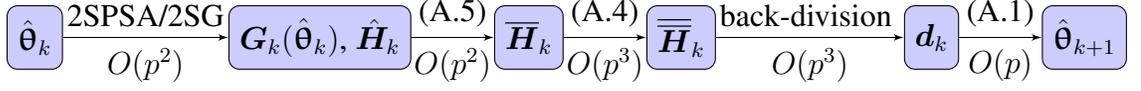
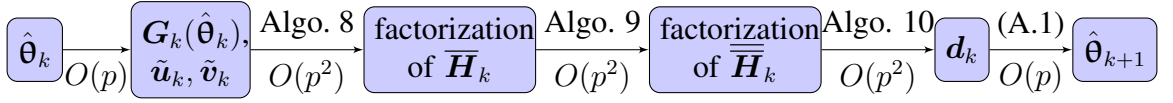
\begin{figure}[!bpht]
		\centering 
\begin{subfigure}{\linewidth}
\begin{tikzpicture}[auto]
\node [block] (hbtheta) {$ \hbtheta_k $};
\node [block, right = 2.2cm of hbtheta] (bG) {$ \bG_k(\hbtheta_k)$, $\hbH_k $};
\node [block, right = 1.1cm of bG] (obH) {$ \obH_k $};
\node [block, right = 1.1cm of obH] (oobH) {$ \oobH_k $};
\node [block, right = 2.5cm of oobH] (direction) {$ \bd_k $};
\node [block, right = 1.1cm of direction] (hbthetanew) {$ \hbtheta_{k+1} $};
\path [line] (hbtheta) -- node {2SPSA/2SG} node[below] {$ O(p^2) $} (bG);
\path [line] (bG) -- node {(\ref{eq:H_overline})} node[below] {$ O(p^2) $} (obH);
\path [line] (obH) -- node[align=center] {   (\ref{eq:H_ooverline})} node[below] {$ O(p^3) $} (oobH);
\path [line] (oobH) -- node[align=center]{  back-division  } node[below] {$ O(p^3) $} (direction);
\path [line] (direction) -- node[align=center] { (\ref{eq:theta_update})} node[below] {$ O(p) $} (hbthetanew);
\end{tikzpicture} 
\caption{Flow chart for the original 2SPSA/2SG}
\label{fig:Original}
\end{subfigure}
\hfil
\begin{subfigure}{\linewidth} 
\begin{tikzpicture}[auto]
\node [block] (hbtheta) {$ \hbtheta_k $};
\node [block, right = .9cm of hbtheta, text width=1.3cm] (bG) {$ \bG_k(\hbtheta_k)$, \\ $\tbu_k, \tbv_k $};
\node [block, right = 1.5cm of bG, text width=2cm] (obH) {factorization of $ \obH_k $};
\node [block, right = 1.5cm of obH, text width=2cm] (oobH) {factorization of $ \oobH_k $};
\node [block, right = 1.5cm of oobH] (direction) {$ \bd_k $};
\node [block, right = 1.1cm of direction] (hbthetanew) {$ \hbtheta_{k+1} $};
\path [line] (hbtheta) -- node {\remove{2SPSA/2SG}} node[below] {$ O(p) $} (bG);
\path [line] (bG) -- node[align=center] {Algo.~\ref{algo:two_rank_one_update}} node[below] {$ O(p^2) $} (obH);
\path [line] (obH) -- node[align=center] { Algo. \ref{algo:preconditioning}} node[below] {$ O(p^2) $} (oobH);
\path [line] (oobH) -- node[align=center] {   Algo. \ref{algo:descent_direction}} node[below] {$ O(p^2) $} (direction);
\path [line] (direction) -- node {(\ref{eq:theta_update})} node[below] {$ O(p) $} (hbthetanew);
\end{tikzpicture}
\caption{Flow chart for the proposed efficient implementation of 2SPSA/2SG (see Section~\ref{subsec:complexity} for detailed description)}
\label{fig:Proposed}
\end{subfigure} 
	\captionsetup{justification=centering}
\caption{Flow charts showing FLOPs cost at each stage of   the original 2SPSA/2SG and the proposed 2SPSA/2SG. Algorithms \ref{algo:two_rank_one_update}--\ref{algo:descent_direction} in the lower path are described in Section \ref{subsec:algo}.}
\label{fig:Flow_chart}
\end{figure}

To better illustrate our scheme and to be consistent with the original 2SPSA/2SG, we decompose our approach into the following three main steps and discuss the efficient implementation step by step.
\begin{enumerate}
	\item[i)] \textbf{Two rank-one modifications}: Update the symmetric indefinite factorization of $ \obH_k $ by the two sequential rank-one modifications in (\ref{eq:two_rank_one_update}) 
	\item[ii)] \textbf{Preconditioning}: Obtain the symmetric indefinite factorization of a positive definite $ \oobH_k $ from the symmetric indefinite factorization of $ \obH_k $
	\item[iii)] \textbf{Descent direction}: Update $ \hbtheta_{k+1} $ by the recursion (\ref{eq:theta_update_s})
\end{enumerate}

Note that $ \obH_k $ is guaranteed to be symmetric by (\ref{eq:two_rank_one_update}) as long as $ \obH_0 $ is chosen symmetric. For the sake of comparison, we list the flow-charts of the original 2SPSA and that of our proposed scheme in Figure~\ref{fig:Flow_chart} 
along with the per-iteration and per-step computational cost. The comparison of the flow-charts helps to put the extra move of indefinite factorization into perspective.

The remainder of this section is as follows. We introduce the symmetric indefinite factorization in Subsection~\ref{subsec:IMF} and derive the efficient algorithm in Subsection~\ref{subsec:algo}. The per-iteration computational complexity analysis is included in Subsection~\ref{subsec:complexity}.

\subsection{Symmetric Indefinite Factorization}\label{subsec:IMF}
This subsection briefly reviews the symmetric indefinite factorization, also called $ \bL\bB\bL^\transpose $ factorization, introduced in \cite{bunch1971direct}, which applies to any symmetric matrix $ \obH $ regardless of the positive-definiteness:
\begin{equation}
\label{eq:LBL}
\bP \obH \bP^\transpose = \bL\bB\bL^\transpose, 
\end{equation}
where $ \bP $ is a permutation matrix, $ \bB $ is a block diagonal matrix with diagonal blocks being symmetric with size $ 1\times1 $ or $ 2\times2 $, and $ \bL $ is a lower-triangular matrix. Furthermore, the matrices $ \bL $ and $ \bB $ satisfy the following properties \cite[Sect. 4]{bunch1971direct}, which are fundamental for carrying out subsequent steps i) -- iii) at a computational cost of $ O\parenthesis{p^2} $:
\begin{itemize}
	\item The magnitudes of the entries of $ \bL $ are bounded by a fixed positive constant. Moreover, the diagonal entries of $ \bL$ are all equal to $1$.
	\item $ \bB $ has the same number of positive, negative, and zero eigenvalues as $ \obH $.
	\item The number of negative eigenvalues of $ \obH $ is the sum of the number of blocks of size $ 2\times2 $ on the diagonal and the number of blocks of size $ 1\times1 $ on the diagonal with negative entires of $ \bB $. (Note: There are no guarantees for the signs of the entries in the $ 2\times2 $ blocks.)
\end{itemize}

The bound on the magnitudes of the entries of $ \bL $ is approximately $ 2.7808 $ per \cite{bunch1977some} and it is \textit{independent} of the size of $ \obH $. As shown in Theorem \ref{thm:H_barbar_property}--\ref{thm:H_barbar_uniform_bound}, such a constant bound is useful in practice to perform   a  \emph{quick} sanity check  regarding the appropriateness of the symmetric indefinite factorization and to provide useful bounds for the eigenvalues of $ \oobH_k $. From (\ref{eq:LBL}), $ \obH $ can be expressed as $ \obH = (\bP^\transpose\bL)\bB(\bP^\transpose\bL)^\transpose $. Then the second bullet point above can be easily  shown   by \textit{Sylvester's law of inertia}, which states that two congruent matrices have the same number of positive, negative,  and zero eigenvalues ($ \bA $ and $\bB $ are congruent if $ \bA = \bP\bB\bP^\transpose $ for some nonsingular matrix $ \bP $) \cite{sylvester1852xix}. From the third bullet point, if $ \obH $ is positive semidefinite, the corresponding $\bB$ is a diagonal matrix with nonnegative diagonal entries.

\subsection{Algorithm Description}\label{subsec:algo}

We now illustrate how the $ \bL\bB\bL^\transpose $ factorization can be of use in 2SPSA/2SG and discuss steps i) -- iii) in Section \ref{subsect:Introduction} in detail.

\textbf{Two rank-one modifications} Although the direct calculation of $ \obH_k $ in (\ref{eq:two_rank_one_update}) only costs $ O(p^2) $, the subsequent preconditioning step incurs a computational cost of $ O(p^3) $ when not using any factorization of $ \obH_k $. Therefore, in anticipation of the subsequent necessary preconditioning, we propose  monitoring the $ \bL\bB\bL^\transpose $ factorization of $ \obH_k $ instead of the matrix itself. That is, the two direct rank-one modifications in (\ref{eq:two_rank_one_update}) are transformed into two non-trivial modifications on the $ \bL\bB\bL^\transpose $ factorization, which also incurs a computational cost of $ O(p^2) $. It is not necessary that $ \obH_k $ is  explicitly computed in the algorithm, thereby avoiding the $ O(p^3) $ cost arising from matrix-associated necessary multiplications in the preconditioning.

Lemma \ref{lem:rank_one_update}   states that the $ \bL\bB\bL^\transpose $ factorization can be updated for rank-one modification at a computational cost of $O(p^2)$. The detailed algorithm is established in \cite{sorensen1977updating}. We adopt that algorithm to our two rank-one modifications in (\ref{eq:two_rank_one_update}) and present the result in Theorem \ref{thm:two_rank_one_update}.

\begin{lem}\label{lem:rank_one_update}
\textup{	\cite[Thm. 2.1]{sorensen1977updating}}\textbf{.}
	Let $ \bA \in \real^{p \times p} $ be symmetric (possibly indefinite) and \emph{non-singular} with $ \bP\bA\bP^\transpose = \bL\bB\bL^\transpose $. Suppose that $ \bz \in \real^p, \sigma \in \real $ are such that: 
	\begin{equation}\label{eq:IMF_rank_one_update}
	\tbA = \bA + \sigma\bz\bz^\transpose  
	\end{equation}
	is also \emph{nonsingular}. Then the factorization $ \tilde{\bP}\tbA\tilde{\bP}^\transpose = \tilde{\bL}\tilde{\bB}\tilde{\bL}^\transpose $ can be obtained from the factorization $ \bP\bA\bP^\transpose = \bL\bB\bL^\transpose $ with a computational cost of $ O(p^2) $.
\end{lem}

\begin{thm}\label{thm:two_rank_one_update}
	Suppose $ \obH_k $ is given in (\ref{eq:two_rank_one_update}). Further, assume that both $ \obH_{k-1} $ and $ \obH_k $ are \emph{nonsingular} and the factorization $ \bP_{k-1}\obH_{k-1}\bP_{k-1}^\transpose = \bL_{k-1}\bB_{k-1}\bL_{k-1}^\transpose $ is available. Then the factorization, 
	\begin{equation}\label{eq:IMF_rwo_rank_one_update}
	\bP_k\obH_k\bP_k^\transpose = \bL_k\bB_k\bL_k^\transpose\,,
	\end{equation}
	can be obtained at a computational cost of $ O(p^2) $.
\end{thm}

\begin{proof} 	 
	With Lemma \ref{lem:rank_one_update}, we see that (\ref{eq:IMF_rwo_rank_one_update}) can be obtained by applying (\ref{eq:IMF_rank_one_update}) twice with $ \sigma = b_k, \bz = \tbu_k $ and $ \sigma = -b_k, \bz = \tbv_k $, respectively. In as  much as  each update requires a computational cost of $ O(p^2) $, the total computational cost remains $ O(p^2) $.
\end{proof}

\begin{rem}
	The nonsingularity (\emph{not} necessarily positive-definiteness) of $ \obH_k $ is a modest assumption for the following three reasons: i) $ \obH_0 $ is often initialized to be a positive definite matrix satisfying the nonsingularity assumption. For example, $ \obH_0 = c\bI $ for some constant $ c > 0 $. ii) Whenever $ \obH_k $ violates the nonsingularity assumption due to the two rank-one modifications in (\ref{eq:two_rank_one_update}), a new pair of $ \bDelta_k $ and $ \tbDelta_k $ along with the noisy measurements can be generated to redo the modifications in (\ref{eq:two_rank_one_update}). In practice, the singularity of $ \obH_k $ can be detected via the entry-wise bounds of $ \bL_k $ per \cite{bunch1977some}. Namely, if $ \bL_k $ has an entry exceeding $ 2.7808 $, the nonsingularity assumption of $ \obH_k $ is violated. It is indeed possible to compute the probability of getting a singular $ \obH_k $; however, we deem it as a minor practical issue and do not pursue further analysis in this work. iii) In that  the second-order method is often recommended to be implemented only after $ \hbtheta_k $ reaches the vicinity of $ \btheta^* $, and the true Hessian matrix of $ \btheta^* $ is assumed to be positive definite \cite{spall2000adaptive}, the estimate $ \obH_k $ is ``pushed" towards nonsingularity. The bottom line is that we can  run second-order methods at any iteration $ k $, but are more interested when $\hbtheta_k$ is near $\btheta^*$.
\end{rem}

We summarize the two rank-one modifications of $ \obH_k $ in  Algorithm~\ref{algo:two_rank_one_update} that follows. The outputs of Algorithm~\ref{algo:two_rank_one_update} are used to obtain a computational cost of $ O(p^2) $ in the preconditioning step as the eigenvalue modifications on $ \bB_k $, a diagonal block matrix, is more efficient than the direct eigenvalue modifications in (\ref{eq:f_k_sqrtm}) and (\ref{eq:f_k_add}). Algorithm~\ref{algo:two_rank_one_update} is the key that renders   steps ii) and iii)  in Subsection~\ref{subsect:Introduction}  achievable at a computational cost of $ O(p^2) $. 

\begin{algorithm}[!htbp]
	\caption{Two rank-one updates of $ \obH_k $}
	\label{algo:two_rank_one_update}
	\begin{algorithmic}[1]
			\renewcommand{\algorithmicrequire}{\textbf{Input:}}
		\renewcommand{\algorithmicensure}{\textbf{Output:}}
	\Require   matrices $ \bP_{k-1}, \bL_{k-1}, \bB_{k-1} $ in the symmetric indefinite factorization of $\obH_{k-1}$, scalars $ t_k, b_k $, and vectors $ \bu_k, \bv_k $ computed per Table~\ref{table:u_k_v_k}.
\Ensure   matrices $ \bP_k, \bL_k, \bB_k $ in the symmetric indefinite factorization of $ \obH_k $ per (\ref{eq:LBL}).
		\State \textbf{set} $ \bP_k \gets \bP_{k-1}, \bL_k \gets \bL_{k-1}, \bB_k \gets t_k\bB_{k-1} $. 
		\State  \textbf{update} $ \bP_k, \bL_k, \bB_k $ with the rank-one modifications $ b_k\tbu_k\tbu_k^\transpose $ with $ \tbu_k $ computed in (\ref{eq:u_k_tilde}) and $ -b_k\tbv_k\tbv_k^\transpose $ with $ \tbv_k $ computed in (\ref{eq:v_k_tilde}), using the updating procedure outlined in \cite{sorensen1977updating}. Code is available at   \url{https://github.com/jingyi-zhu/Rank1FactorizationUpdate}. 
		\State \Return  matrices $ \bP_k, \bL_k, \bB_k $. 
	\end{algorithmic}
\end{algorithm}

\begin{rem}
	Though $ \obH_k $ is not explicitly computed during each iteration, whenever needed, it can be computed easily from its $ \bL\bB\bL^\transpose $ factorization, though with a computational cost of $ O(p^3) $; i.e, $ \obH_k = \bP_k^\transpose\bL_k\bB_k\bL_k^\transpose\bP_k $. This calculation yields the same $ \obH_k $ as (\ref{eq:H_overline}) or (\ref{eq:two_rank_one_update}). The $ \bL\bB\bL^\transpose $ factorization of $ \obH_0 $ requires a computational cost of, at most\remove{Straightforward if choose $ \obH_0=\bI $.}, $ O(p^3) $ \cite[Table 2]{bunch1971direct}. However, as a one-time sunk-in cost, it does not compromise the overall computational cost. Of course, we can avoid this bothersome issue by initializing $ \obH_0 $ to a diagonal matrix, which immediately gives $ \bP_0 = \bL_0 = \bB_0 = \bI $. Generally, the cost for initialization is trivial if $\obH_0$ is a diagonal matrix.
\end{rem}

\textbf{Preconditioning}
Given the factorization of the estimated Hessian information $ \obH_k $, which is symmetric yet potentially \emph{indefinite} (especially during early iterations), we aim to output a factorization of the Hessian approximation $ \oobH_k $ such that $ \oobH_k $ is symmetric and sufficiently positive definite, i.e., $ \lambda_{\min}(\oobH_k) \geq \tau $ for some constant $ \tau > 0 $. With the above $ \bL\bB\bL^\transpose $ factorization associated with $\obH_k$ obtained from the previous two rank-one modification steps, we can modify the eigenvalues of $ \bB_k $. Note that $\bB_k$ is a block diagonal matrix, so any eigenvalue modification can be  carried out inexpensively. This is in contrast to directly modifying the eigenvalues of $\obH_k$ to obtain $ \oobH_k $, which is computationally-costly as laid out in Subsection \ref{subsect:p3cost}. Denote $ \obB_k $ as the modified matrix from $ \bB_k $. Note that $ \oobH_k $ and $ \obB_k $ are congruent as $ \oobH_k = (\bP_k^\transpose\bL_k)\obB_k(\bP_k^\transpose\bL_k)^\transpose $. By   Sylvester's law of inertia, the positive definiteness of $ \oobH_k $ is guaranteed as long as $ \obB_k $ is positive definite. 

To modify the eigenvalues of $ \bB_k $, we borrow the ideas from the modified Newton's method \cite[pp. 50]{nocedal2006numerical} to set $$ \lambda_j(\obB_k) = \max \set{\tau_k, |\uplambda_j(\bB_k)|} $$ for $ j = 1, ..., p $,  where $ \tau_k $ is a user-specified stability threshold, which is possibly   data-dependent. A possible choice of the uniformly bounded $ \set{\tau_k} $ sequence   in the Section \ref{sec:numerical} is to set $ \tau_k = \max\{10^{-4}, 10^{-4}p\max_{1\leq j \leq p} |\uplambda_j(\bB_k)| \} $.   The intuition behind the eigenvalue modification in Algorithm~\ref{algo:preconditioning} is to make $ \obB_k $ well-conditioned while behaving similarly to $ \bB_k $. The pseudo-code of the preconditioning step is listed in Algorithm~\ref{algo:preconditioning}.

\begin{algorithm}[htbp]
	\caption{Preconditioning}
	\label{algo:preconditioning}
	\begin{algorithmic}[1]
			\renewcommand{\algorithmicrequire}{\textbf{Input:}}
		\renewcommand{\algorithmicensure}{\textbf{Output:}}
\Require  user-specified stability-threshold $ \tau_k > 0 $ and matrix $ \bB_k $ in the symmetric indefinite factorization of $ \obH_k $.
	\Ensure  matrix $ \bQ_k $ in the eigen-decomposition of $ \bB_k $ and the modified matrix $ \obLambda_k $.
		\State \textbf{apply} eigen-decomposition of $ \bB_k = \bQ_k \bLambda_k \bQ_k^\transpose $, where $ \bLambda_k = \text{diag}(\lambda_{k1}, ..., \lambda_{kp}) $ and $ \lambda_{kj} \equiv \lambda_j(\bB_k) $ for $ j = 1, ..., p $.
		\State \textbf{update} $ \obLambda_k = \text{diag}(\bar{\lambda}_{k1}, ..., \bar{\lambda}_{kp}) $ with $ \bar{\lambda}_{kj} = \max \set{\tau_k, |\uplambda_{kj}|} $ for $j = 1, ..., p$.
		\State \Return  eigen-decomposition of $ \obB_k = \bQ_k \obLambda_k \bQ_k^\transpose $.
	\end{algorithmic}
\end{algorithm}

\begin{rem}
		Although the eigen-decomposition, in general, incurs an $ O(p^3) $ cost, the block diagonal structure of $ \bB_k $ allows such operation to be implemented relatively inexpensively. In the worst-case scenario, $ \bB_k $ consists of $ p/2 $ diagonal blocks of size $ 2 \times 2 $, where eigen-decompositions are applied on each block separately leading to a total computational cost of $ O(p) $. For the sake of efficiency, the matrix $ \oobH_k $ is not explicitly computed. Whenever needed, however, it can be computed by $ \oobH_k = \bP_k^\transpose\bL_k\bQ_k\obLambda_k\bQ_k^\transpose\bL_k^\transpose\bP_k $ at a cost of $O(p^3)$. 
\end{rem}

Algorithm~\ref{algo:preconditioning} makes our approach different from \cite{spall2000adaptive}. We only modify the eigenvalues of $ \bLambda_k $ (or equivalently of $ \bB_k $), which indirectly affects the eigenvalues of $ \oobH_k $ in a non-trivial way. However, if one constructs $ \obH_k $ and $ \oobH_k $ from their factorization (formally unnecessary as mentioned above), Algorithm~\ref{algo:preconditioning} can be viewed as a function that maps $ \obH_k $ to a positive-definite $ \oobH_k $. In this sense, Algorithm~\ref{algo:preconditioning} is just a special choice of $ f_k(\cdot) $ in (\ref{eq:H_ooverline}) even though such a $ f _k(\cdot) $ is non-trivial and difficult to find.

\textbf{Descent direction} After the preconditioning step, the descent direction $ \bd_k: \oobH_k \bd_k = \bG_k(\hbtheta_k) $ can be computed readily via one forward substitution  w.r.t. the lower-triangular matrix $\bL_k$ and one backward substitution w.r.t. the upper-triangular matrix $\bL_k^\transpose$, as the decomposition $\oobH_k=\bP_k^\transpose\bL_k\bQ_k\obLambda_k\bQ_k^\transpose\bL_k^\transpose\bP_k $ is available. The estimate $ \hbtheta_k $ can then be updated as in (\ref{eq:theta_update_s}). Note that $ \oobH_k $ is not directly computed in any iteration, and the forward and backward substitutions are implemented through the terms in the $ \bL\bB\bL^\transpose $ factorization. Algorithm~\ref{algo:descent_direction} below summarizes the details.

\begin{algorithm}[H]
	\caption{Descent Direction Step}
	\label{algo:descent_direction}
	\begin{algorithmic}[1]
			\renewcommand{\algorithmicrequire}{\textbf{Input:}}
		\renewcommand{\algorithmicensure}{\textbf{Output:}}
		\Require gradient estimate $ \bG_k(\hbtheta_k) $, and matrices $ \bP_k, \bL_k, \bQ_k, \obLambda_k $ in the $ \bL\bB\bL^\transpose $ factorization of $ \oobH_k $.
		\Ensure descent direction $ \bd_k $.
		\State \textbf{Solve} $ \bz$ by forward substitution such that $\bL_k\bz = \bP_k\bG_k(\hbtheta_k) $.
		\State \textbf{Compute} $ \bw $ such that $ \bw = \bQ_k\obLambda_k^{-1}\bQ_k^\transpose \bz $.
		\State \textbf{Solve} $ \by $ by backward substitution such that $ \bL_k^\transpose\by = \bw $.
\State 	\Return  $ \bd_k = \bP_k^\transpose\by $.
	\end{algorithmic}
\end{algorithm}
Given the triangular structure of $\bL_k$ and that both $\bP_k$ and $\bQ_k$ are permutation matrices, the computational cost of Algorithm \ref{algo:descent_direction} is dominated by $O(p^2)$.

\subsection{Overall Algorithm (Second-Order SP) and Computational Complexity }
\label{subsec:complexity}

With   the aforementioned steps, we present the \emph{complete} algorithm for implementing second-order SP in Algorithm~\ref{algo:2SP} below, which applies to 2SPSA/2SG/E2SPSA/E2SG. A complete computational complexity analysis for 2SPSA is also stated, and the suggestions for the user-specified inputs are listed in \cite[Sect. 7.8.2]{spall2005introduction}. Results for 2SG/E2SPSA/E2SG can be obtained similarly.

\begin{algorithm}[!htbp]
	\caption{Efficient Second-order SP (applies to 2SPSA, 2SG, E2SPSA, and E2SG)}
	\label{algo:2SP}
	\begin{algorithmic}[1]
				\renewcommand{\algorithmicrequire}{\textbf{Input:}}
		\renewcommand{\algorithmicensure}{\textbf{Output:}}
		\Require  initialization $ \hbtheta_0 $ and $ \bP_0, \bQ_0, \bB_0 $ in the symmetric indefinite factorization of $ \obH_0 $; user-specified stability-threshold $ \tau_k > 0 $; coefficients $ a_k, c_k, w_k $ and, for 2SPSA/E2SPSA, $ \tilde{c}_k $.
		\Ensure  terminal estimate $ \hbtheta_k $.
		\State \textbf{set} iteration index $ k = 0 $.
		\While{terminating condition for $ \hbtheta_k $ has not been satisfied}
		\State \textbf{estimate} gradient $ \bG_k(\hbtheta_k) $ by (\ref{eq:gradient_estimate_2SPSA}) or (\ref{eq:gradient_estimate_2SG}). 
		\State \textbf{compute} $ t_k, b_k, \tbu_k $ and $ \tbv_k $ by (\ref{eq:u_k_tilde}), (\ref{eq:v_k_tilde}) and Table~\ref{table:u_k_v_k}.
		\State \textbf{update} the symmetric indefinite factorization of $ \obH_k $ by Algorithm~\ref{algo:two_rank_one_update}.
		\State \textbf{update} the symmetric indefinite factorization of $ \oobH_k $ by Algorithm~\ref{algo:preconditioning}.
		\State \textbf{compute} the descent direction $ \bd_k $ by Algorithm~\ref{algo:descent_direction}.
		\State \textbf{update} $ \hbtheta_{k+1} = \hbtheta_k - a_k\bd_k $.
		\State $ k \leftarrow k + 1 $
		\EndWhile 
		\State \Return  $ \hbtheta_k $.
	\end{algorithmic}
\end{algorithm}

For the terminating condition, the algorithm is set to stop when a pre-specified total number of function  evaluations (applicable for 2SPSA and E2SPSA)  or gradient  measurements (applicable for 2SG and E2SG) is reached or  the norm  of the differences between several consecutive estimates is less than a pre-specified threshold.  Note that, for each iteration, four noisy loss function measurements are required in the gradient-free case, whereas  three noisy gradient measurements are required in the gradient-based case.

The corresponding computational complexity analysis for Algorithm~\ref{algo:2SP} under the gradient-free case is summarized in Table~\ref{table:computational_complexity}. Analogously, the analysis can be carried out for the gradient-based case  and the feedback-based case (E2SPSA or E2SG).

\begin{table}[!t]
	\caption{Computational complexity analysis in gradient-free case (2SPSA in Algorithm \ref{algo:2SP}) Complexity cost shown in FLOPs.}
	\centering
	\begin{tabular}{|c|c|c|}
		\hline
		Leading Cost & Original 2SPSA & Proposed Implementation\\
		\hline\hline
		Update $ \obH_k $ & $ 7p^2 $ & $ 3.67p^2 + O(p) $ \\
		\hline
		Precondition $ \oobH_k $ & $ 17.67p^3 + O(p^2) $ & $ 8p $ \\
		\hline
		Descent direction $ \bd_k $ & $ 0.33p^3 + O(p^2) $ & $ 4p^2 + O(p) $ \\
		\hline\hline
		Total Cost & $ 18p^3 + O(p^2) $ & $7.67p^2 + O(p) $ \\
		\hline
	\end{tabular}\label{table:computational_complexity}
\end{table}

Let us now show how we obtain the terms in Table~\ref{table:computational_complexity}. A floating-point operation (FLOP)  is assumed to be either a summation or a multiplication, while transposition requires no FLOPs. For the updating $ \obH_k $ step in the  original 2SPSA, $ 3p^2 $ FLOPs are required per (\ref{eq:H_overline}) and $ 4p^2 $ FLOPs are required per (\ref{eq:H_hat}). In the proposed implementation, $ 10p $ FLOPs are required to get $ \tbu_k $ and $ \tbv_k $ per (\ref{eq:u_k_tilde}) and (\ref{eq:v_k_tilde}), respectively, and $ 22p^2/6 + O(p) $ FLOPs are required to update the symmetric indefinite factorization of $ \obH_k $ \cite[Thm. 2.1 ]{sorensen1977updating}. For the preconditioning step in the  original 2SPSA, if using (\ref{eq:f_k_sqrtm}), $ p^3 + p $ FLOPs are required to get $ \obH_k\obH_k + \delta_k\bI $ and an additional $ 50p^3/3 + O(p^2) $ FLOPs are required for the matrix square root operation \cite{higham1987computing}. In the proposed implementation, at most $ 7p $ FLOPs are required to get an eigenvalue decomposition on $ \bB_k $ ($ 14 $ FLOPs for at most $ p/2 $ blocks of size $ 2 \times 2 $),  and $ p $ FLOPs are required to update the eigenvalues of $ \bB_k $. For computing the descent direction $ \bd_k $ in the original 2SPSA, $ p^3/3 $ FLOPs are required to apply Cholesky decomposition for $ \oobH_k $, and $ 2p^2 $ FLOPs are required for the backward substitutions. In the proposed implementation, $ 4p^2 + 2p $ FLOPs are required to backward substitutions.

Table~\ref{table:computational_complexity} may not provide the lowest possible computational complexities  because a great deal of existing work on parallel computing\textemdash such as \cite{george1986parallel} on parallelization of Cholesky decomposition, \cite{deadman2012blocked} for computing principal matrix square root, and \cite{dongarra1987fully} for the symmetric eigenvalue problem\textemdash have tremendously accelerated the matrix-operation computing speed in modern data analysis packages. Nonetheless, even with such enhancements, the FLOPS counts remain $ O(p^3) $ in the standard methods. The bottom line is that our proposed implementation reduces the overall computational cost from $ O(p^3) $ to $ O(p^2) $.

\section{Theoretical Results and Practical Benefits}\label{sec:theory}

This section presents the  theoretical foundation  related to  the almost sure convergence and the asymptotic normality of $ \hbtheta_k $. We also offer comments on  the practical benefits  of  the proposed scheme. Lemma~\ref{lem:Ostrowski} provides the theoretical guarantee to connect the eigenvalues of $ \oobH_k $ and $ \obLambda_k $, which are important for proving Theorem~\ref{thm:H_barbar_property}--\ref{thm:H_barbar_uniform_bound} related to the matrix properties of $\obH_k$ and $\oobH_k$.

\begin{lem}  \label{lem:Ostrowski} \textup{\cite[Thm. 4.5.9]{horn1990matrix}}\textbf{.}
	  Let $ \bA, \bS \in \real^{p \times p} $, with $ \bA $ being symmetric and $ \bS $ being \emph{nonsingular}. Let the eigenvalues of $ \bA $ and $ \bS\bA\bS^\transpose $ be arranged in \emph{nondecreasing} order. Let $ \sigma_1 \geq \cdots \geq \sigma_p > 0 $ be the singular values of $ \bS $. For each $ j = 1, \cdots, p $, there exists a positive number $ \zeta_j \in [\sigma_p^2, \sigma_1^2] $ such that $ \lambda_j(\bS\bA\bS^\transpose) = \zeta_j\lambda_j(\bA)$.
\end{lem}

Before presenting the main theorems, we first discuss the singular values of $ \bL_k $. Denote $ \{\sigma_i(\bL_k)\}_{i=1}^p $ as the singular values of $ \bL_k $ and let $ \sigma_{\min}(\cdot) = \min_{1\leq i \leq p} \sigma_i(\cdot)$, $ \sigma_{\max}(\cdot) = \max_{1\leq i \leq p} \sigma_i(\cdot) $. Since $ \bL_k $ is a unit lower triangular matrix, we have $ \lambda_j(\bL_k) = 1 $ for $ j = 1, .., p $ and $ \det(\bL_k) = 1 $. From the entry-wise bounds of $ \bL_k $ in Subsection~\ref{subsec:IMF}, we see that $ p \leq \norm{\bL_k}_F \leq 3p^2/2 - p/2 $ for all $ k $, where $\norm{\cdot}_F$ is the Frobenius norm of the argument matrix in $ \real^{p\times p} $. With the lower bound of $ \sigma_{\min}(\bL_k) $ \cite{yi1997note}, there exists a constant $ \underline{\sigma} > 0 $ such that $ \sigma_{\min}(\bL_k) \geq \underline{\sigma} $ for all $ k $. On the other hand, by the equivalence of the matrix norms, i.e, $ \sigma_{\max}(\bL_k) = \norm{\bL_k}_2 \leq \norm{\bL_k}_F $ for $ \norm{\cdot}_2 $ being the spectral norm, there exists a constant $ \overline{\sigma} > 0 $ such that $ \sigma_{\max}(\bL_k) \leq \overline{\sigma} $ for all $ k$. Both $ \underline{\sigma} $ and $ \overline{\sigma} $ are independent of the sample path for $ \bL_k $. By the Rayleigh-Ritz theorem \cite[Thm. 4.2.2]{horn1990matrix}, $ \be_1^\transpose(\bL_k\bL_k^\transpose)\be_1 = 1 $ implies that $ \sigma_{\min}(\bL_k) \leq 1 $ and $ \sigma_{\max}(\bL_k) \geq 1 $. Combined, all the singular values of $ \bL_k $ are bounded uniformly across $k$; i.e., $ \underline{\sigma} < \sigma_{\min}(\bL_k) \leq 1 \leq \sigma_{\max}(\bL_k) \leq \overline{\sigma} $. Let $ \kappa(\bL_k) $ be the condition number of $ \bL_k $, then $ 1 \leq \kappa(\bL_k) \leq \overline{\sigma} / \underline{\sigma} $.

As  the focus of Algorithm~\ref{algo:preconditioning} is to generate a positive definite $ \obB_k $ (or equivalently its eigen-decomposition), we replace $ \tau_k $  in Theorem~\ref{thm:H_barbar_property}--\ref{thm:H_barbar_uniform_bound}   with some constant $ \underline{\tau}\in\left(0,\uptau_k\right] $ independent of the sample path for $ \bB_k $ for all $k$. Note that the substitution is solely for succinctness and does not affect the theoretical result that $\obB_k$ is positive definite. Theorem~\ref{thm:H_barbar_property} presents the key theoretical properties of $ \oobH_k $ satisfying the regularity conditions in \cite[C.6]{spall2000adaptive}. Based on Theorem~\ref{thm:H_barbar_property}, the strong convergence, $ \hbtheta_k \to \btheta^* $ and $ \obH_k \to \bH(\btheta^*) $, can be established conveniently. See Remark~\ref{rmk:strong_convergence}.

\begin{thm}\label{thm:H_barbar_property} Assume there exists a symmetric indefinite factorization $ \obH_k = \bP_k^\transpose\bL_k\bB_k\bL_k^\transpose\bP_k$.
	Given any constant $\underline{\uptau} \in\left(0,\uptau_k\right] $ for all $ k $, the matrix $ \oobH_k = \bP_k^\transpose\bL_k\bQ_k\obLambda_k\bQ_k^\transpose\bL_k^\transpose\bP_k $ with $ \bQ_k$ and $\obLambda_k $ returned from Algorithm~\ref{algo:preconditioning} satisfies the following properties:
	\begin{itemize}
		\item[(a)] $ \lambda_{\min}(\oobH_k) \geq \underline{\sigma}^2\underline{\tau} > 0 $.
		\item[(b)] $\oobH_k^{-1}$ exists a.s., $ c_k^2 \oobH_k ^{-1}\to \bm{0}$ a.s., and for some constants $\updelta, \uprho>0$, $ \mathbb{E} [ \| \oobH_k^{-1} \|^{2+\updelta} ]\le \uprho $.
	\end{itemize}
\end{thm}

\begin{proof}
	For all $k$, it is easy to see that $ \lambda_{\min}(\obLambda_k) \geq \underline{\tau} > 0 $ implying $ \obLambda_k $ is positive definite. Since both $ \bQ_k $ and $ \bL_k $ are nonsingular, by Sylvester's law of inertia \cite{sylvester1852xix}, $ \oobH_k $ is also positive definite as $ \obLambda_k $ is positive definite. Moreover, by Lemma~\ref{lem:Ostrowski},
		\begin{equation}\label{eq:lambda_min_H_barbar}
	\lambda_{\min}(\oobH_k) \geq \sigma_{\min}^2(\bL_k)\lambda_{\min}(\obLambda_k) \geq \underline{\sigma}^2\underline{\tau} > 0 \,.
	\end{equation}
	Because $ \oobH_k $ has a constant lower bound for all its eigenvalues across $k$, property (b) follows.
\end{proof}

\begin{rem}
	\label{rmk:strong_convergence}
	Theorem~\ref{thm:H_barbar_property} guarantees that $ \oobH_k $ is positive definite, and,  therefore, the estimates of $\btheta$ in the second-order method move in a descent-direction on average. Meeting property (b) is also necessary for  showing the convergence results. Suppose the routine regularity conditions in \cite[Sect. III and IV]{spall2000adaptive} hold. To depict  the strong convergence, $ \hbtheta_k \to \btheta^* $ and $ \obH_k \to \bH(\btheta^*) $, we  need only   verify that $ \oobH_k $ satisfies the regularity conditions in \cite[C.6]{spall2000adaptive} because the key difference between the original 2SPSA/2SG and our proposed method is effectively the preconditioning step. Theorem~\ref{thm:H_barbar_property} verifies the Assumption C.6 in \cite{spall2000adaptive} directly, and therefore we have $ \hbtheta_k \to \btheta^* $ a.s. and $ \obH_k \to \bH(\btheta^*) $ a.s. under both the 2SPSA and 2SG settings by \cite[Thms. 1 and 2]{spall2000adaptive}. 
\end{rem}

Theorem~\ref{thm:H_bar_property} discusses the connection between $ \obH_k $ and $ \oobH_k $ when $ k $ is sufficiently large. It also verifies a key condition when proving the asymptotic normality of $ \hbtheta_k $. See Remark~\ref{rmk:asymptotic_normality}.

\begin{thm}\label{thm:H_bar_property}
	Assume $ \bH(\btheta^*) $ is positive definite. When choosing $ 0 < \underline{\tau} \leq \lambda_{\min}(\bH(\btheta^*)) / (2\overline{\sigma}^2) $, there exists a constant $ K_1 $ such that for all $ k > K_1 $, we have $ \oobH_k = \obH_k $.
\end{thm}

\begin{proof}
	By Remark~\ref{rmk:strong_convergence}, since $ \obH_k \to \bH(\btheta^*) $ a.s., there exists an integer $K$ such that for all $ k > K_1 $, $ \lambda_{\min}(\obH_k) \geq \lambda_{\min}(\bH(\btheta^*)) / 2 > 0 $. By Lemma~\ref{lem:Ostrowski}, we can achieve a lower bound for the eigenvalues of $ \bLambda_k $ as
	\begin{equation*}
	\lambda_{\min}(\bLambda_k) \geq \frac{\lambda_{\min}(\obH_k)}{\sigma_{\max}^2(\bL_k)} \geq \frac{\lambda_{\min}(\obH_k)}{\overline{\sigma}^2} \geq \underline{\tau}\,.
	\end{equation*}
	Therefore, for all $ k > K_1$, $ \obLambda_k = \bLambda_k $ and,  consequently, $ \oobH_k = \obH_k $.
\end{proof}

\begin{rem}
	\label{rmk:asymptotic_normality}
	Theorem~\ref{thm:H_bar_property} shows that when $ k $ is large (the estimated Hessian $ \obH_k $ is sufficiently positive definite), the proposed preconditioning step will automatically make $ \oobH_k = \obH_k $, which satisfies one of the key required conditions for the asymptotic normality of $ \hat{\btheta}_k $ in \cite{spall2000adaptive}. Apart from the additional regularity conditions in \cite[C.10--12]{spall2000adaptive}, we are required to verify that
	$
	\oobH_k - \obH_k \to \bm{0}~\text{a.s.},
	$
	which can be inferred by Theorem~\ref{thm:H_bar_property}. Following \cite[Thm. 3]{spall2000adaptive}, when the gain sequences have the standard form $ a_k = a/(A+k+1)^\alpha $ and $ c_k = c/(k+1)^\gamma $, the asymptotic normality of $ \hbtheta_k $ gives:
	\begin{equation*}
	\begin{split}
	k^{(\alpha-2\gamma)/2}(\hbtheta_k - \btheta^*) \stackrel{\text{dist}}{\longrightarrow} N(\bm{\upmu}, \bm{\Omega}) &\quad \text{for 2SPSA,}\\
	k^{\alpha/2}(\hbtheta_k - \btheta^*) \stackrel{\text{dist}}{\longrightarrow} N(\bm{0}, \bm{\Omega'}) &\quad\text{for 2SG,}
	\end{split}
	\end{equation*}
	where the specifications of $ \alpha, \gamma, \bm{\upmu}, \bm{\Omega}$ and $ \bm{\Omega'} $ are available in \cite{spall2000adaptive}. Under E2SPSA/E2SG settings, the convergence and asymptotic results can be derived analogously from \cite[Thms. 1--4]{spall2009feedback}.
\end{rem}

As  an ill-conditioned matrix may cause an excessive step-size in recursion (\ref{eq:theta_update_s}) leading to slow a convergence rate \cite{li2018preconditioned}, we need to make sure that the resulting $\oobH_k $ (or its equivalent factorization) is not only positive definite but also numerically favorable. Theorem~\ref{thm:H_barbar_uniform_bound} below shows that changing the eigenvalues of $ \bLambda_k $ does not lead to the eigenvalues of $ \oobH_k $ becoming either too large or too small.

\begin{thm}\label{thm:H_barbar_uniform_bound} Assume the eigenvalues of $\bH\parenthesis{\btheta^*}$ are bounded uniformly such that $0< \underline{\uplambda}^*<\abs{\uplambda_j\parenthesis{\bH\parenthesis{\btheta^*}}}<\overline{\uplambda}^*<\infty $ for $j=1,...,p$ for all $k$. Then there exists some $K_2$ such that for $ k>K_2 $, the eigenvalues and condition number of $ \oobH_k $ are also bounded uniformly.
\end{thm}

\begin{proof}
	Again by Remark \ref{rmk:strong_convergence}, in that  $\obH_k\to \bH\parenthesis{\btheta^*}$ a.s.; therefore,  for all $k>K_2$, the eigenvalues of $ \obH_k $ are bounded uniformly in the sense that $ \underline{\lambda} < |\lambda_j(\obH_k)| < \overline{\lambda} $ for $ j = 1, ..., p $, where $ \underline{\lambda} = \underline{\uplambda}^*/2 $ and $ \overline{\lambda} =2 \overline{\uplambda}^*$ are constants independent of the sample path for $ \obH_k $. 
	Given $ \obH_k = \bP_k\bL_k\bB_k\bL_k^T\bP_k $, by Lemma~\ref{lem:Ostrowski},
		\begin{equation*}
	\frac{\lambda_{\min}(\obH_k)}{\sigma_{\max}^2(\bL_k)} \leq \lambda_{\min}(\bB_k) \leq \frac{\lambda_{\min}(\obH_k)}{\sigma_{\min}^2(\bL_k)}\,,
	\end{equation*}
	and \begin{equation*}
	\frac{\lambda_{\max}(\obH_k)}{\sigma_{\max}^2(\bL_k)} \leq \lambda_{\max}(\bB_k) \leq \frac{\lambda_{\max}(\obH_k)}{\sigma_{\min}^2(\bL_k)}\,.
	\end{equation*}
	Similarly, since $ \oobH_k = \bP_k\bL_k\obB_k\bL_k^T\bP_k $,
	\begin{equation*}
	\begin{split}
	\lambda_{\min}(\oobH_k) &
	\geq \sigma_{\min}^2(\bL_k) \lambda_{\min}(\obB_k)\\
	&\geq \sigma_{\min}^2(\bL_k)\max\left\{\underline{\tau},\frac{\lambda_{\min}(\obH_k)}{\sigma_{\max}^2(\bL_k)}\right\}\\
	&
	\geq \underline{\sigma}^2\max\left\{\underline{\tau},\frac{\underline{\lambda}}{\overline{\sigma}^2}\right\}\,,
	\end{split}
	\end{equation*}
	\begin{equation*}
	\begin{split}
	\lambda_{\max}(\oobH_k)&
	\leq \sigma_{\max}^2(\bL_k) \lambda_{\max}(\obB_k)\\
	&\leq \sigma_{\max}^2(\bL_k)\max\left\{\underline{\tau},\frac{\lambda_{\max}(\obH_k)}{\sigma_{\min}^2(\bL_k)}\right\}\\
	&\leq \overline{\sigma}^2\max\left\{\underline{\tau},\frac{\overline{\lambda}}{\underline{\sigma}^2}\right\}\,,
	\end{split}
	\end{equation*}
		where $ \kappa(\cdot) $ is the condition number of the matrix argument. In as much as  $ \underline{\sigma}^2, \overline{\sigma}^2, \underline{\lambda} $, and $ \overline{\lambda} $ are all constants specified before running the algorithm, the eigenvalues of $ \oobH_k $ are bounded uniformly across $ k > K_2 $.
	
	Moreover, for the condition number of $ \oobH_k $, we have: 
		\begin{equation*}
	\kappa(\oobH_k)
	\leq \frac{\sigma_{\max}^2(\bL_k)}{\sigma_{\min}^2(\bL_k)}\frac{\max\left\{\underline{\tau},\lambda_{\max}(\obH_k)/\sigma_{\min}^2(\bL_k)\right\}}{\max\left\{\underline{\tau},\lambda_{\min}(\obH_k)/\sigma_{\max}^2(\bL_k)\right\}}.
	\end{equation*}
	Hence, the condition number of $ \oobH_k $ is also bounded uniformly across $k>K_2$.
\end{proof}

\begin{rem}
	Theorem~\ref{thm:H_barbar_uniform_bound} is highly desired for the preconditioning step as  it ensures the numerical stability. Recall that the preconditioning step listed in Algorithm~\ref{algo:preconditioning} modifies the eigenvalues of $ \obH_k $ by modifying the eigenvalues of $ \bB_k $. This modification is desirable because  the eigenvalues of $ \oobH_k $ are controllable; i.e., a bound for $ \lambda_j(\oobH_k) $ uniformly for sufficiently large $k$ under a given size $ p $ can be obtained. The controlled condition number in Theorem~\ref{thm:H_barbar_uniform_bound} demarcates the original preconditioning procedure as in Eq. (\ref{eq:f_k_add}), which does not control the condition number of $\overline{\overline{\bH}}_k$. 
\end{rem}

\section{Numerical Studies}\label{sec:numerical}

In this section, we demonstrate the strength of the proposed algorithms by minimizing the skewed-quartic function \cite{spall2000adaptive} using   efficient 2SPSA/E2SPSA and training a neural network using    efficient 2SG.
\subsection{Skewed-Quartic Function}

We consider the following skewed-quartic function used in \cite{spall2000adaptive} to show the performance of the efficient 2SPSA/E2SPSA:
\begin{equation*}
\loss (\btheta) = \btheta^\transpose\bB^\transpose\bB\btheta+0.1
\sum_{i=1}^{p} (\bB\btheta)_i^3 +0.01 \sum_{i=1}^{p}
(\bB\btheta)_i^4 \,,
\end{equation*}
where $ (\cdot)_i $ is the $i$th component of the argument vector, and $ \bB $ is such that $ p\bB $ is an upper triangular matrix of all $1$'s. The additive noise in $ y(\cdot) $ is independent $\mathcal{N}\parenthesis{0,0.05^2}$; i.e., $ y(\btheta) = \loss (\btheta) + \varepsilon$, where $ \varepsilon \sim \mathcal{N}\parenthesis{0,0.05^2} $. \remove{The noise level $\upsigma=0.05$ is relatively large when compared to the values of the elements in $ L(\hbtheta_0) = $. }It is easy to check that that $ L(\btheta) $ is strictly convex with a unique minimizer $ \btheta^* = \bm{0} $ such that $L(\btheta^*) = 0$.

For the preconditioning step in the original 2SPSA/E2SPSA, we choose $ \oobH_k = \bm{m}_k(\obH_k) = (\obH_k\obH_k + 10^{-4}e^{-k} \bI)^{1/2} $, which satisfies the definition of $ \bm{m}_k(\cdot) $ in (\ref{eq:f_k_sqrtm}) as  $ \delta_k = 10^{-4}e^{-k} \to 0 $. In the efficient 2SPSA/E2SPSA, we choose $ \obLambda_k = \text{diag}(\bar{\lambda}_{k1}, ..., \bar{\lambda}_{kp}) $ with $ \bar{\lambda}_{kj} = \max \{10^{-4}$, $10^{-4}p\max_{1\leq i \leq p} |\lambda_{ki}|, |\lambda_{kj}|\} $ for all $ j $, which is consistent with the suggestion in \cite[pp. 118]{sorensen1977updating} and satisfies Theorem~\ref{thm:H_barbar_property}. To guard against unstable steps during the iteration process, a blocking step is added to reset $ \hbtheta_{k+1} $ to $ \hbtheta_k $ if $ \norm{\hbtheta_{k+1} - \hbtheta_k} \geq 1 $. We choose an initial value $ \hbtheta_0 = [1, 1, \dots, 1]^\transpose $.

We show three plots below. Figures \ref{fig:loss_2SPSA} and \ref{fig:loss_E2SPSA} illustrate how the efficient method here provides essentially the same solution in terms of the loss function values as the $O(p^3)$ methods in \cite{spall2000adaptive} and \cite{spall2009feedback} (2SPSA and feedback and weighting-based E2SPSA). Figure 4 illustrates how the $O(p^3)$ vs. $O(p^2)$ FLOPS-based cost in Table \ref{table:computational_complexity} above is manifested in overall runtimes.

Figure~\ref{fig:loss_2SPSA} plots the normalized loss function values $ [\loss(\hbtheta_k) - \loss (\btheta^*)] / [\loss(\hbtheta_0) - \loss (\btheta^*)] $ of the original 2SPSA and the efficient 2SPSA averaged over 20 independent replicates for $ p = 100 $ and  the number of iterations $ N = 50,000 $. Similar to the numerical studies in \cite{spall2009feedback}, the gain sequences of the two algorithms are chosen to be $ a_k = a/(A+k+1)^{0.602} $, $ c_k = \tilde{c}_k = c/(k+1)^{0.101} $, and $ w_k = w/(k+1)^{0.501} $, where $ a = 0.04, A = 1000, c = 0.05 $, and $ w = 0.01 $ following the standard guidelines in \cite{spall1998implementation}.

\begin{figure}[!htbp]
	\centering
	\includegraphics[width=\linewidth]{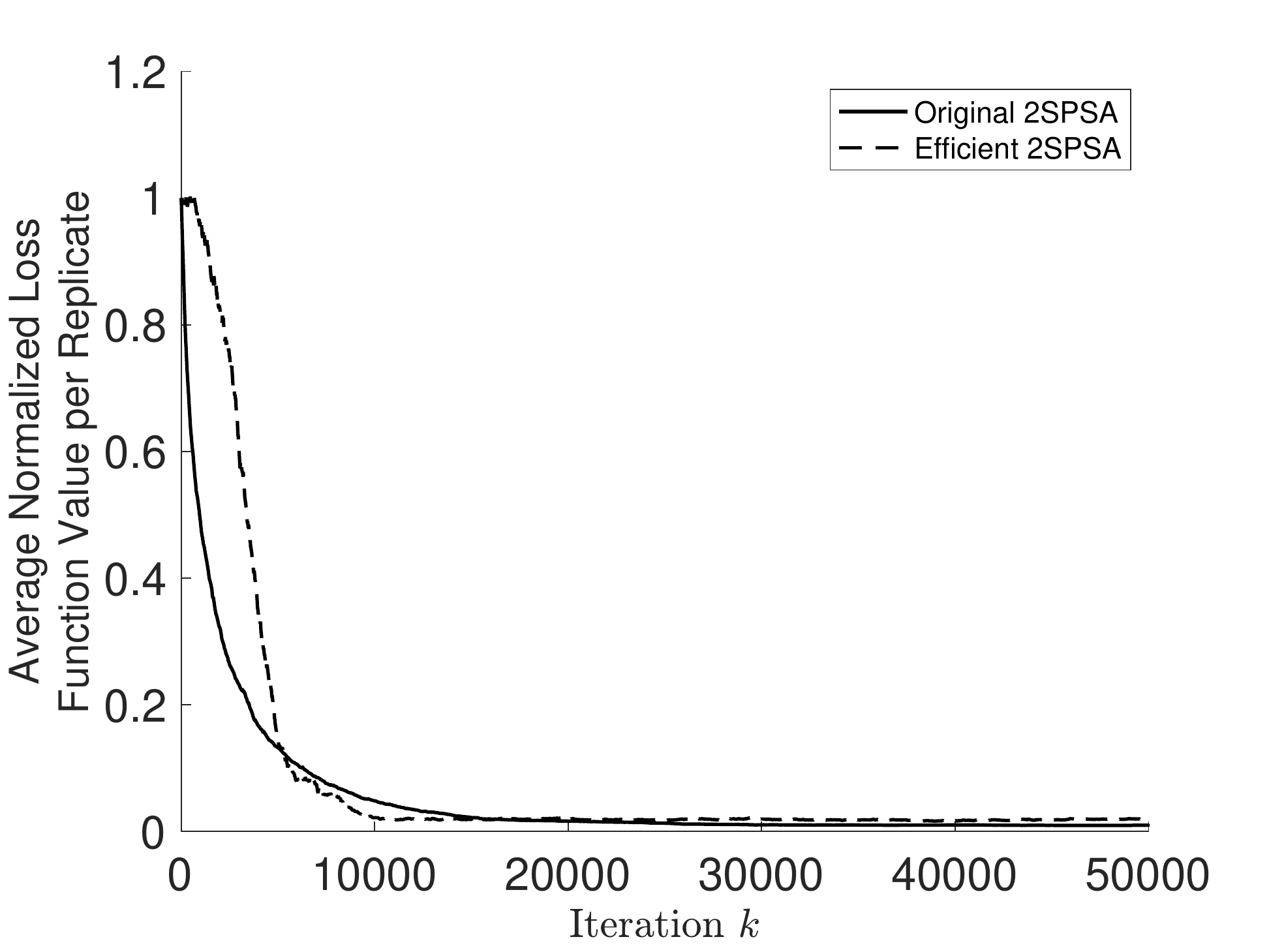}
	\caption{Similar performance of algorithms with respect to loss values (\emph{different} run times). Normalized terminal loss $ [\loss (\hbtheta_k) - \loss (\btheta^*)] / [\loss (\hbtheta_0) - \loss (\btheta^*)] $ of the original 2SPSA and the efficient 2SPSA averaged over 20 replicates for $ p = 100 $.}
	\label{fig:loss_2SPSA}
\end{figure}

Figure~\ref{fig:loss_E2SPSA} compares the normalized loss function values $ [\loss (\hbtheta_k) - \loss (\btheta^*)] / [\loss(\hbtheta_0) - \loss(\btheta^*)] $ of the standard E2SPSA and the efficient E2SPSA averaged over 10 independent replicates for $ p = 10 $ and number of iterations $ N = 10,000 $. The gain sequences of the two algorithms are chosen to have the form $ a_k = a/(A+k+1)^{0.602} $, $ c_k = \tilde{c}_k = c/(k+1)^{0.101} $, and $ w_k = w/(k+1)^{0.501} $, where $ a = 0.3, A = 50 $, and $ c = 0.05 $. The weight sequence $ w_k = \tilde{c}_k^2c_k^2/[\sum_{i=0}^{k}(\tilde{c}_i^2c_i^2 )]$ is set according to the optimal weight in \cite[Eq. (4.2)]{spall2009feedback}. 

\begin{figure}[!htbp]
	\centering
	\includegraphics[width=\linewidth]{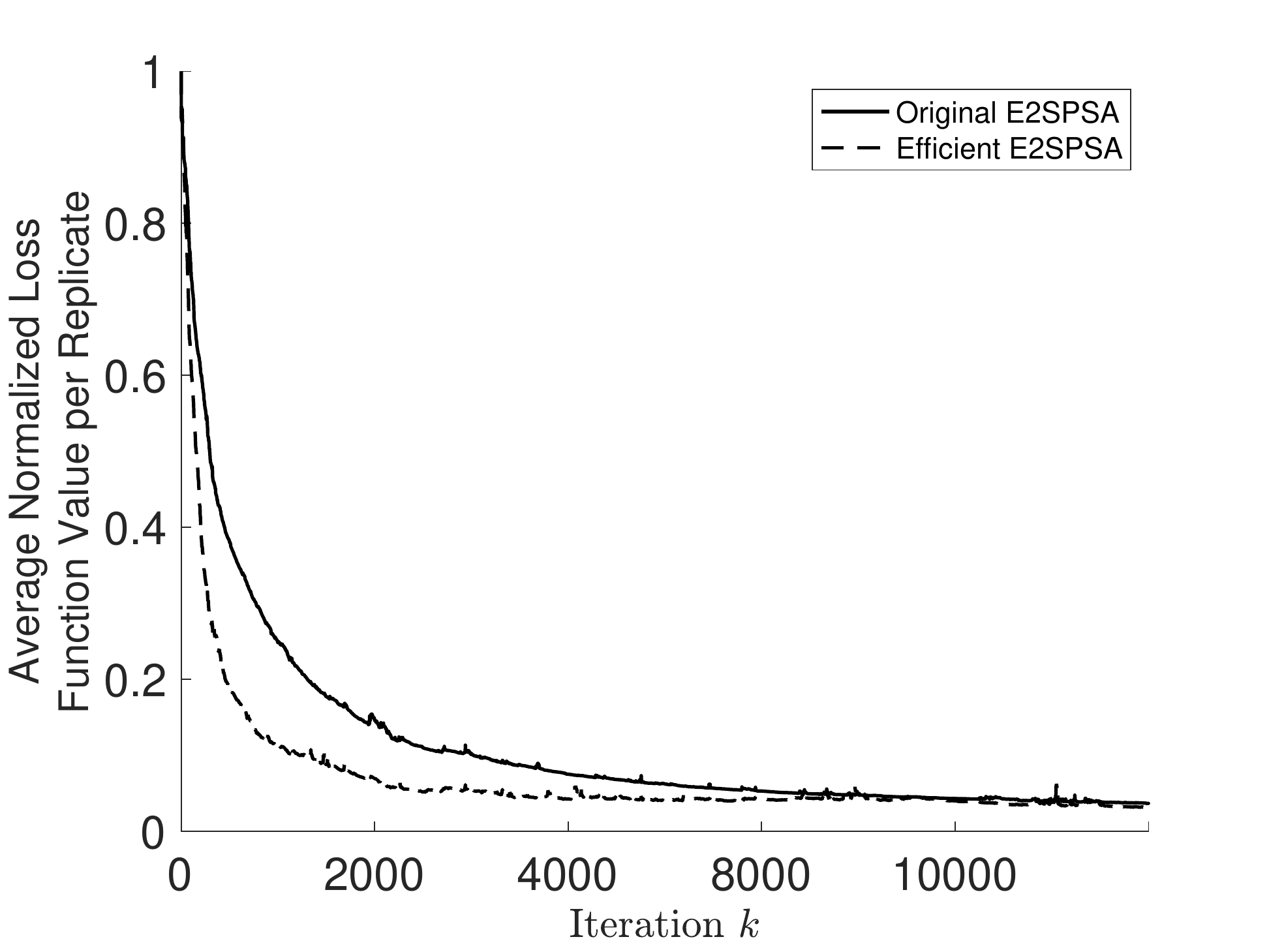}
	\caption{Similar performance of algorithms with respect to loss values (\emph{different} run times).  Normalized terminal loss $ [\loss(\hbtheta_k) - \loss(\btheta^*)] / [\loss(\hbtheta_0) - \loss(\btheta^*)] $ of the original E2SPSA and the efficient E2SPSA averaged over 10 replicates for $ p = 10 $. }
	\label{fig:loss_E2SPSA}
\end{figure}

In the above comparisons, the loss function decreases significantly for all the dimensions with only noisy loss function measurements available. We see that the two implementations of E2SPSA provide close to the same accuracy for $1000$ or more iterations, although at a computing cost difference of $O(p^2)$ versus $O(p^3)$. Note that the differences (across $k$) between the original 2SPSA and the efficient 2SPSA/E2SPSA in Figure~\ref{fig:loss_E2SPSA} can be made arbitrarily small by picking an appropriate $ \bm{m}_k(\cdot) $ (or equivalently $ \oobH_k $) in the original 2SPSA, although such a choice might be non-trivial.

\begin{figure}[!htbp]
	\centering
	\includegraphics[width=\linewidth]{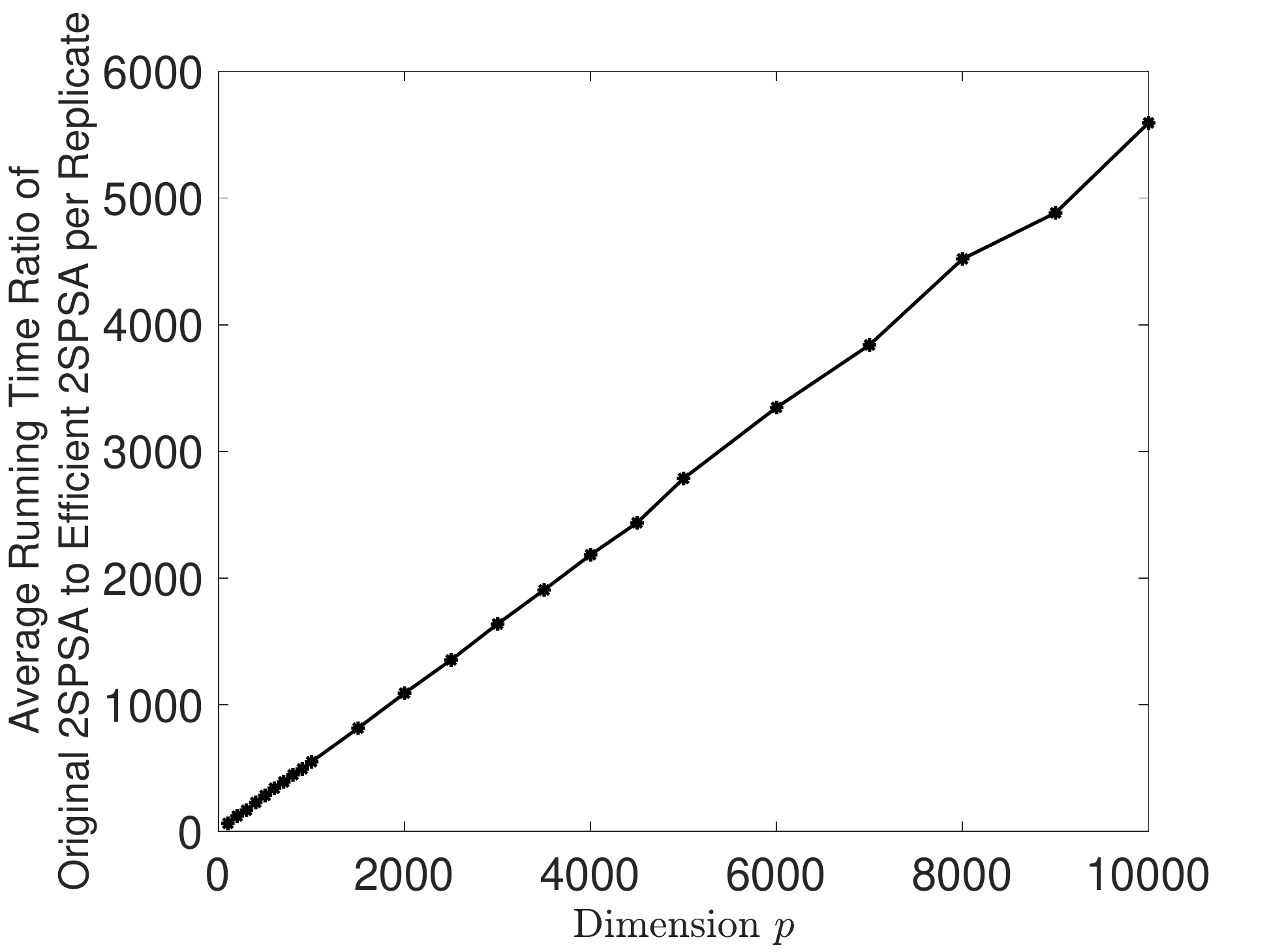}
	\caption{Running time ratio of the original 2SPSA to the efficient 2SPSA averaged over 10 replicates, where the same skewed-quartic loss function is used, and the total number of iterations is fixed at 10 for each run. The trend is close to the theoretical linear relationship as a function of dimension $p$.}
	\label{fig:time}
\end{figure}

To measure the computational time, Figure~\ref{fig:time} plots the running time (measured by the built-in \textsc{C++} function \texttt{clock()} with no input) ratio of the original 2SPSA to the efficient 2SPSA averaged over 10 independent replicates with dimension up to 10000. It visualizes the practicality of the efficient 2SPSA over the original 2SPSA. In terms of the general trend, the linear relationship between the running time ratio and the dimension number is consistent with the $ O(p^3) $ cost for the original 2SPSA and the $ O(p^2) $ cost for the efficient 2SPSA.  From Figure~\ref{fig:time}, it is clear that the computational benefit of the efficient 2SPSA is more apparent as the dimension $ p $ increases. The slope in Figure~\ref{fig:time} is roughly 0.56, which is consistent with  the theoretical FLOPs ratio of 2.35 in Table~\ref{table:computational_complexity}, when accounting for differences due to the storage costs and code efficiency. With a more dedicated programming language, it is expected that the running time ratio will be closer to  the  theoretical FLOPs ratio in Table~\ref{table:computational_complexity}.

\subsection{Real-Data Study: Airfoil Self-Noise Data Set}

In this subsection, we compare the efficient 2SG with the SGD and ADAM \cite{kingma2014adam}   in training a one-hidden-layer feed-forward neural network to predict sound levels over an airfoil.  Although there are many gradient-based   methods to train a neural network, we select SGD and  ADAM because they are popular and representative of algorithms within the  machine learning community. Comparison of  efficient 2SG and the two aforementioned  algorithms is appropriate as all of  them use the noisy gradient evaluations \emph{only}, despite their different forms. Aside from the application here, neural networks have been widely used as function approximators in the field of aerodynamics and aeroacoustics. Recent applications include airfoil design \cite{rai2000aerodynamic},\remove{ pressure prediction \cite{irani2011application},} and aerodynamic prediction \cite{perez2000prediction}.

The dataset used in this example is the NASA data of the  NACA 0012 airfoil self-noise data set \cite{brooks1981trailing, brooks1989airfoil}, which is also available on the UC Irvine Machine Learning Repository at  \url{https://archive.ics.uci.edu/ml/datasets/Airfoil+Self-Noise}. This NASA dataset is obtained from a series of aerodynamic and acoustic tests of two and three-dimensional airfoil blade sections conducted in an anechoic wind tunnel. The inputs contain five variables: frequency (in Hertz); angle of attack (in degrees, not in radians); chord length (in meters); free-stream velocity (in meters per second); and suction side displacement thickness (in meters). The output contains the scaled sound pressure level (in decibels). Readers may refer to \cite{brooks1989airfoil} and \cite[Sect. 3]{errasquin2009airfoil} for further details.

Now that  the number of samples is  $ n = 1503 $, we   fit the dataset using a one-hidden-layer neural network with 150 hidden neurons and with sigmoid activating functions. Other choices in  neural network structures, that use a different number of layers or different activation functions, have been implemented in \cite{errasquin2009airfoil}. Here, we use a neural network with a greater number of neurons than in \cite{errasquin2009airfoil}   to demonstrate the strength of the efficient 2SG in high-dimensional problems. The value of $ p $ is 1051, calculated as $ 5 \times 150 $ weights and $ 150 $ bias parameters for the hidden neurons along with 150 weights and 1 bias parameters for the output neuron.

Following the principles in \cite{wilson2003general}, we train the neural network in an online manner, where only one training sample is evaluated during each iteration. Denote the dataset as $ \{(y_i, \bx_i)\}_{i=1}^n $ and the parameters in the neural network as $ \btheta $. The loss function is chosen to be the ERF; i.e., $ \loss(\btheta) = (1/n) \sum_{i=1}^n (y_i - \hat{y}_i)^2 $, where $ \hat{y}_i $ is the neural network output based on input $ \bx_i $ and parameter $ \btheta $. Consistent with the online training of an ERF in machine learning, the loss function based on that one training sample can be deemed as a noisy measurement of the loss function based on the entire dataset.

We implement SGD and ADAM with 10 epochs, each corresponding  to 1503 iterations (one iteration per data point), resulting in a total   of 15030 iterations.   The gain sequence is  chosen to be $ a_k = a / (k + 1 + A)^\alpha $ and  $ A = 1503$ being 10\% of the total number of iterations with  $ \alpha = 1 $ following \cite[pp. 113\textendash 114]{spall2005introduction}. After tuning for optimal performance, we choose $ a = 1 $ for SGD and ADAM \cite{kingma2014adam} . Other hyper-parameters for ADAM are determined from the default settings in \cite{kingma2014adam}. There is no ``re-setting'' of $ a_k $ imposed at the beginning of each epoch   so that the gain sequence goes down consecutively across iterations and epochs. The initial value $ \hbtheta_0 = \bm{0} $. Recall that efficient 2SG requires  three back-propagations per iteration, where SGD and ADAM only require  one   per iteration. Therefore, for fair comparison, we implement the efficient 2SG under two different scenarios: (1) serial computing, and (2) concurrent computing.

Within each iteration of efficient 2SG, the three gradient measurements, $ \bY_k(\hbtheta_k),\bY_k(\hbtheta_k +c_k\bDelta_k) $ and $ \bY_k(\hbtheta_k-c_k\bDelta_k) $ can be computed simultaneously in as much as  they do not rely on each other. Using this concurrent implementation, the time spent in back-propagation can be reduced to one-third of the original time. All the remaining steps are unchanged.  Although the efficient 2SG takes time in performing Algorithm~\ref{algo:preconditioning}, numerical studies indicate that the majority of the time is spent on   back-propagation. Therefore, under the concurrent implementation, the efficient 2SG has roughly the same running time per iteration as SGD and ADAM. Figure~\ref{fig:real_data_log_MSE_per_iteration} shows the value of   ERF under the concurrent implementation. In the efficient 2SG, the gain sequences are chosen to be $ a_k = a/(A+k+1)^\alpha $, $ w_k = 1/(k+1) $, and $ c_k = c/(k+1)^\gamma $ with $A=1503, \alpha=1 $ and $ \gamma = 1/6 $ following \cite{spall1998implementation}. Other parameters of $ a = 0.1 $ and $ c = 0.05 $ are tuned for optimal performance. The matrix $ \obLambda_k $ is computed the same as in the skewed-quartic function above. For better practical performance, training data is normalized to the range $ [0,1] $. As  all the inputs and outputs are positive,   normalization is simply performed  by dividing the data by their corresponding maximums. Figure~\ref{fig:real_data_log_MSE_per_iteration} shows that the efficient 2SG converges much quicker and obtains a better terminal value. One explanation for this phenomenon is that the Hessian information helps the speed of convergence, similar to the benefits of Newton-Raphson relative to the gradient-descent method.

\begin{figure}[!htbp]
	\centering
	\includegraphics[width=\linewidth]{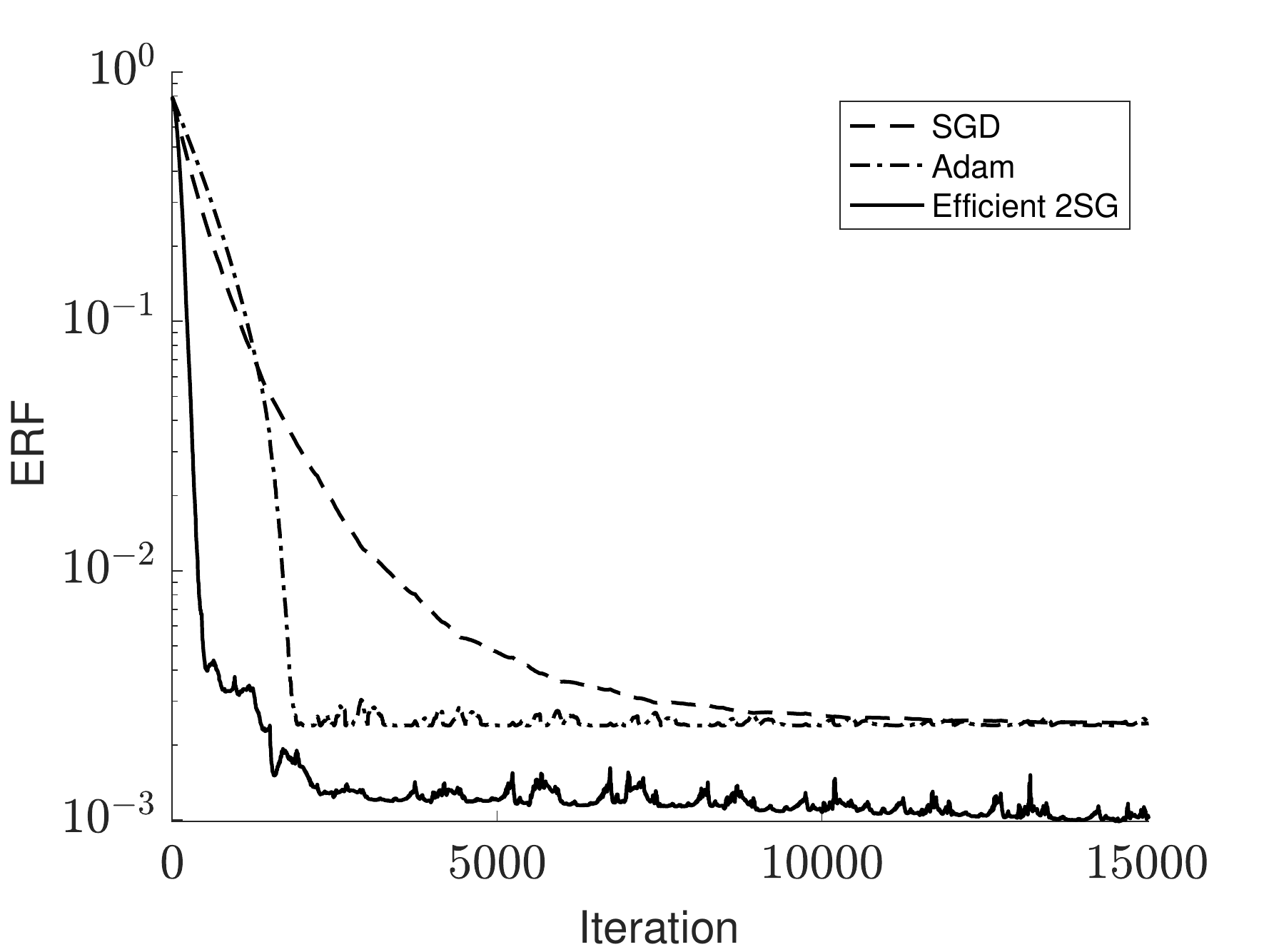}
	\caption{ERF of training samples in SGD, ADAM, and the efficient 2SG under concurrent implementation.}
	\label{fig:real_data_log_MSE_per_iteration}
\end{figure}

Figure~\ref{fig:real_data_log_MSE_per_iteration_adjusted} compares the ERF of the two algorithms in terms of the number of gradient evaluations. Note that each iteration of SGD and ADAM takes one gradient evaluation, while the efficient 2SG necessitates  three. This comparison is suitable for the non-concurrent implementation because one iteration of the efficient 2SG has roughly the cost of three iterations of the SGD. It is shown in Figure~\ref{fig:real_data_log_MSE_per_iteration_adjusted} that the efficient 2SG still outperforms SGD and ADAM even without any concurrent implementation. There is less than a 7\% difference in running time among SGD, ADAM, and  the efficient 2SG under the concurrent implementation. \remove{For running time, SGD takes 2068 seconds, ADAM takes 2105 second and the efficient 2SG takes 2211 seconds under the concurrent implementation.}

\begin{figure}[!htbp]
	\centering
	\includegraphics[width=\linewidth]{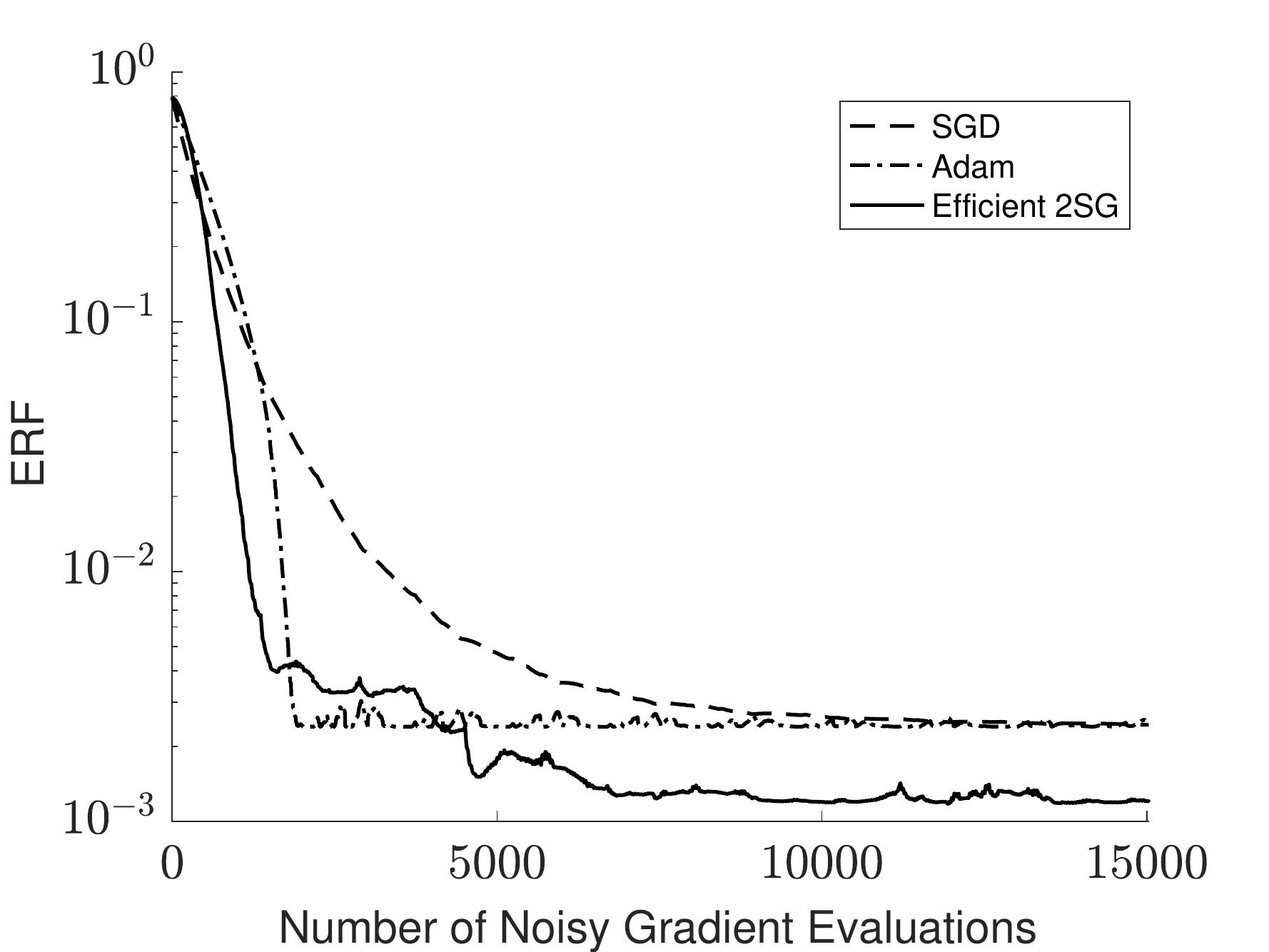}
	\caption{ERF of training samples in SGD, ADAM, and the efficient 2SG per gradient evaluation under \emph{serial} (non-concurrent) computing. SGD and ADAM have  three times the number of iterations of 2SG.}
	\label{fig:real_data_log_MSE_per_iteration_adjusted}
\end{figure}

\section{Practical Issues and Concluding Remarks}\label{sec:discussion}

Let us discuss two practical questions regarding   Algorithm~\ref{algo:2SP}. i) What is the difference between the standard adaptive SPSA-based method and the proposed algorithm if $ \obB_k $ (or $ \obH_k $) is sufficiently positive definite? ii) How to recover $ \obH_k $ at any $ k $?

In the ideal case, if $ \bB_k $ (or $ \obH_k $) is assumed to always be positive definite, the preconditioning step becomes unnecessary and we can directly set the symmetric indefinite factorization of $ \obH_k $ as the symmetric indefinite factorization of $ \oobH_k $; i.e., $ \bLambda_k = \obLambda_k $. In this scenario, the proposed method is identical to the original 2SPSA. However, because of the symmetric indefinite factorization, the overall computational cost remains at $ O(p^2) $ as in Table~\ref{table:computational_complexity}, and it is still favorable relative to the original 2SPSA, which incurs a computational cost of $ O(p^3) $ due to the Gaussian elimination of $ \oobH_k $ in computing the descent direction $ \bd_k $. As mentioned in Section \ref{subsect:p3cost}, however, \cite{rastogi2016efficient} uses the matrix inversion lemma to show that the computational cost can be reduced to $ O(p^2) $ as well. Compared  with \cite{rastogi2016efficient}, which directly updates the matrix $ \obH_k^{-1} $ using the matrix inverse lemma, our proposed method has more control over the eigenvalues of $ \obH_k $ and performs well even when $ \obH_k $ is ill-conditioned.

The second aspect is that both $ \obH_k $ and $ \oobH_k $ are never explicitly computed during each iteration. By maintaining the corresponding factorization, we avoid   expensive matrix multiplications and gain a much faster way to achieve second-order convergence. However, whenever needed, either $ \obH_k $ or $ \oobH_k $ can be directly computed from the factorizations at a cost of $ O(p^3) $. See Subsection~\ref{subsec:algo}.

To the best of our knowledge, 2SPSA, 2SG, E2SPSA, and E2SG are the fastest possible second-order stochastic Newton-type algorithms based on the estimations of the Hessian matrix from either noisy loss measurements or noisy gradient measurements.
This paper shows how symmetric indefinite matrix factorization may be used to reduce the per-iteration FLOPs of the algorithms from $ O(p^3) $ to $ O(p^2) $. The approach guarantees both a positive definite estimation of the Hessian matrix (``preconditioned") as well as a valid stochastic Newton-type update of the parameter vector, both in $ O(p^2) $. This implementation scheme serves to improve practical performance in high-dimensional problems, such as deep learning. In our proposed scheme,   formal convergence and convergence rates for $\hbtheta_k$ and $\obH_k$ are maintained,   following the prior work \cite{spall2000adaptive,spall2009feedback}. 

Apart from   the theoretical guarantee, numerical studies show that the efficient implementation of second-order SP methods provides a promising convergence rate at a tolerable computing cost  compared with the  stochastic gradient descent method. Note that second-order methods do not provide global convergence in general, and, therefore, the second-order method is recommended to be implemented after reaching the  vicinity of the optimizer.

Overall, our proposed scheme of second-order SA methods has values in high-dimensional optimization and learning problems. In that  a key step of this work is the symmetric indefinite factorization, the proposed algorithm might be useful for other algorithms whenever updating an estimated Hessian matrix is involved, such as second-order random directions stochastic approximation \cite{prashanth2017adaptive}, natural gradient descent \cite{amari2000adaptive}, and stochastic variants of the BFGS quasi-Newton methods \cite{schraudolph2007stochastic}. In all these  methods, instead of directly updating the matrix of interest (usually the Hessian matrix), one might consider updating its corresponding symmetric indefinite factorization in the manner of this paper in order to speed up any matrix inverse operation or matrix eigenvalue modification. Overall, the proposed approach provides a practical second-order method that can be used following first-order or other methods that can place  the iterate in at least the vicinity of the solution.

\end{appendices}

\addcontentsline{toc}{chapter}{Bibliography}
\bibliographystyle{apalike}
\bibliography{thesis}  

\begin{vita}  
	  Jingyi Zhu was born in November 1991 in Heyuan City, Guangdong Province, People's Republic of China. 
	 Jingyi  is currently a    Ph.D. candidate in  the Department of Applied Mathematics and Statistics (AMS) at the Johns Hopkins University (JHU). Her doctoral study  was supported in part by  the  Paul V. Renoff Fellowship from  JHU,     the Charles and Catherine Counselman Fellowship from  the  Department of AMS, the Acheson J. Duncan Fund for the Advancement of Research in Statistics from  the  Department of AMS,   and   Navy contract N00024-13-D6400 via the  Office of Naval Research. Her efforts in teaching were recognized through the Professor Joel Dean Award for Excellence in Teaching in 2017.  In summer 2018, she conducted research on mixed-variable constrained optimization at the JHU  Applied Physics Laboratory.  In winter 2017, she was an Academic Cooperation Program intern in the Lawrence Livermore National Lab, funded by the  National Science Foundation  Mathematical Sciences Graduate Internship Program.   She received an M.S.E. degree in Computer Science and an M.S.E. degree in Financial Mathematics from   JHU, in May 2019 and December 2014, respectively.   She   received an   B.S. degree in Mathematics and  Applied Mathematics from Tongji University, Shanghai, China in May  2013.    
	 Broadly, her research lies
	 in the intersection of stochastic optimization and control\textemdash particularly the theoretical foundation for
	 stochastic optimization algorithms to be applied in nonstationarity tracking.     She also works on  general stochastic approximation \cite{blakney2019compare}, second-order methods \cite{zhu2019efficient}, and  mixed-variable optimization \cite{wang2018mixed}. In her leisure time, she   plays badminton,   participates in aerial dancing and yoga, and cooks Cantonese and Hakka cuisine. 
\end{vita} 
\end{document}